\documentclass[11pt,reqno]{amsart}
\usepackage{amsmath,amsthm,amssymb,mathrsfs,stmaryrd,color}
\usepackage[all]{xy}
\usepackage{url}
\usepackage{geometry}
\usepackage{url}
\usepackage[utf8]{inputenc}
\usepackage[T1]{fontenc}

\usepackage{relsize}
\usepackage[bbgreekl]{mathbbol}
\usepackage{amsfonts}
\usepackage{slashed}
\usepackage{amsmath, tikz-cd}
\usetikzlibrary{decorations.pathreplacing, calc}

\tikzcdset{
  diagrams={>={Straight Barb[scale=0.5]}}
}
\tikzset{
  altstackar/.style={decorate, decoration={show path construction,
    lineto code={
      \path (\tikzinputsegmentfirst); \pgfgetlastxy{\xstart}{\ystart}
      \path (\tikzinputsegmentlast); \pgfgetlastxy{\xend}{\yend}
      \path ($(0,0)!2.5pt!(\ystart-\yend,\xend-\xstart)$); \pgfgetlastxy{\xperp}{\yperp}
      \foreach \n[evaluate=\n as \k using .5*#1-\n+.5] in {1,...,#1}{
        \ifodd\n{\draw[->, shorten <=2pt, shift={($\k*(\xperp,\yperp)$)}](\xstart,\ystart)--(\xend,\yend);}
        \else{\draw[<-, shorten >=2pt, shift={($\k*(\xperp,\yperp)$)}](\xstart,\ystart)--(\xend,\yend);}\fi
      }
    }
  }}, altstackar/.default={1}
}
\DeclareSymbolFontAlphabet{\mathbb}{AMSb}
\DeclareSymbolFontAlphabet{\mathbbl}{bbold}
\DeclareMathOperator{\WCart}{WCart}
\DeclareMathOperator*{\Cone}{Cone}

\DeclareMathOperator{\Qcoh}{Qcoh}
\newcommand{\Prism}{{\mathlarger{\mathbbl{\Delta}}}}
\usepackage{enumitem}
\usepackage[colorlinks=true,hyperindex, linkcolor=magenta, pagebackref=false, citecolor=cyan,pdfpagelabels]{hyperref}
\usepackage[capitalize]{cleveref}
\usepackage{mathtools}
\usepackage{tikz-cd}
\setlist[enumerate]{itemsep=2pt,parsep=2pt,before={\parskip=2pt}}

\newcommand{\cosimp}[3]{\xymatrix@1{#1 \ar@<.4ex>[r] \ar@<-.4ex>[r] & {\ }#2 \ar@<0.8ex>[r] \ar[r] \ar@<-.8ex>[r] & {\ } #3 \ar@<1.2ex>[r] \ar@<.4ex>[r] \ar@<-.4ex>[r] \ar@<-1.2ex>[r] & \cdots }}

\newcommand{\adjunction}[4]{\xymatrix@1{#1{\ } \ar@<0.3ex>[r]^{ {\scriptstyle #2}} & {\ } #3 \ar@<0.3ex>[l]^{ {\scriptstyle #4}}}}

\begin{document}

\numberwithin{equation}{section}
\newtheorem{theorem}{Theorem}[subsection]
\newtheorem*{theorem*}{Theorem}
\newtheorem*{definition*}{Definition}
\newtheorem{proposition}[theorem]{Proposition}
\newtheorem{lemma}[theorem]{Lemma}
\newtheorem{corollary}[theorem]{Corollary}

\theoremstyle{definition}
\newtheorem{definition}[theorem]{Definition}
\newtheorem{question}[theorem]{Question}
\newtheorem{remark}[theorem]{Remark}
\newtheorem{warning}[theorem]{Warning}
\newtheorem{example}[theorem]{Example}
\newtheorem{notation}[theorem]{Notation}
\newtheorem{convention}[theorem]{Convention}
\newtheorem{construction}[theorem]{Construction}
\newtheorem{claim}[theorem]{Claim}
\newtheorem{assumption}[theorem]{Assumption}

\crefname{assumption}{assumption}{assumptions}
\crefname{construction}{construction}{constructions}

\def\todo#1{\textcolor{red}%
{\footnotesize\newline{\color{red}\fbox{\parbox{\textwidth}{\textbf{todo: } #1}}}\newline}}

\def\commentbox#1{\textcolor{red}%
{\footnotesize\newline{\color{red}\fbox{\parbox{\textwidth}{\textbf{comment: } #1}}}\newline}}

\newcommand{\dR}{{\mathrm{dR}}}
\newcommand{\qc}{q-\mathrm{crys}}
\newcommand{\Ainf}{{A_{\mathrm{inf}}}}
\newcommand{\Ain}{{\mathbb{A}_{\mathrm{inf}}}}
\newcommand{\Shv}{\mathrm{Shv}}
\newcommand{\Vect}{\mathrm{Vect}}
\newcommand{\et}{\mathrm{\acute{e}t}}
\newcommand{\eh}{\mathrm{\acute{e}h}}
\newcommand{\proet}{\mathrm{pro\acute{e}t}}
\newcommand{\crys}{\mathrm{crys}}
\renewcommand{\inf}{\mathrm{inf}}
\newcommand{\Hom}{\mathrm{Hom}}
\newcommand{\RHom}{\mathrm{RHom}}
\newcommand{\Sch}{\mathrm{Sch}}
\newcommand{\fSch}{\mathrm{fSch}}
\newcommand{\Rig}{\mathrm{Rig}}
\newcommand{\Spf}{\mathrm{Spf}}
\newcommand{\Spa}{\mathrm{Spa}}
\newcommand{\Spec}{\mathrm{Spec}}
\newcommand{\Bl}{\mathrm{Bl}}
\newcommand{\perf}{\mathrm{perf}}
\newcommand{\Perf}{\mathrm{Perf}}
\newcommand{\Pic}{\mathrm{Pic}}
\newcommand{\qsyn}{\mathrm{qsyn}}
\newcommand{\perfd}{\mathrm{perfd}}
\newcommand{\arc}{{\rm arc}}
\newcommand{\conj}{\mathrm{conj}}
\newcommand{\rad}{\mathrm{rad}}
\newcommand{\Id}{\mathrm{Id}}
\newcommand{\coker}{\mathrm{coker}}
\newcommand{\im}{\mathrm{im}}
\newcommand{\TXA}{X_{/A}^{\Prism}}
\newcommand{\TXAn}{X_{/A, n}^{\Prism}}
\newcommand{\Cond}{\mathrm{Cond}}
\newcommand{\CHaus}{\mathrm{CHaus}}
\newcommand{\cg}{\mathrm{cg}}
\newcommand{\topo}{\mathrm{top}}
\newcommand{\Top}{\mathrm{Top}}
\newcommand{\ac}{\mathbb{A}_{\mathrm{crys}}}
\newcommand{\bc}{\mathbb{B}_{\mathrm{crys}}^+}
\newcommand{\oc}{\mathcal{O}_{\Prism}[[\frac{\mathcal{I}_{\Prism}}{p}]]}
\newcommand{\ocn}{\mathcal{O}_{\Prism}[[\frac{\mathcal{I}_{\Prism}}{p}]]/(\mathcal{I}_{\Prism}/p)^n}
\newcommand{\ocm}{\mathcal{O}_{\Prism}[[\frac{\mathcal{I}_{\Prism}}{p}]]/(\mathcal{I}_{\Prism}/p)^m}
\newcommand{\Ab}{\mathrm{Ab}}
\newcommand{\Set}{\mathrm{Set}}
\newcommand{\Pro}{\mathrm{Pro}}
\newcommand{\Ind}{\mathrm{Ind}}
\newcommand{\Sm}{\mathrm{SmRig}}
\newcommand{\HTlog}{\operatorname{HTlog}}
\newcommand{\HT}{\operatorname{HT}}
\newcommand{\Mod}{\mathrm{Mod}}
\newcommand{\AT}{A\left\langle \underline{T}^{}\right\rangle}
\newcommand{\MIC}{\mathrm{MIC}}
\newcommand{\nil}{\mathrm{Nil}}
\newcommand{\bdr}{{\mathbb{B}_{\mathrm{dR}}^{+}}}
\newcommand{\an}{{\mathrm{an}}}
\newcommand{\uHom}{\underline{\mathrm{Hom}}}
\newcommand{\Sym}{\operatorname{Sym}}
\newcommand{\Lie}{\operatorname{Lie}}
\newcommand{\fib}{\operatorname{fib}}
\newcommand{\Eq}{\operatorname{Eq}}
\newcommand{\TX}{\widetilde{X^{\Prism}}}
\newcommand{\TXn}{\widetilde{X^{\Prism}_{[n]}}}
\newcommand{\TXL}{\widetilde{\Spf(\mathcal{O}_L)^{\Prism}_{[n]}}}
\newcommand{\TXM}{\widetilde{X^{\Prism}_{[n-1]}}}
\newcommand{\qd}{q\text{-}\nabla}
\newcommand{\qh}{q\text{-}\mathrm{Higgs}}
\newcommand{\qe}{q\text{-}\mathrm{EHIG}}
\newcommand{\TXV}{\widetilde{X^{\Prism}_{[1]}}}
\newcommand{\XY}{\Prism_X(Y)}
\newcommand{\TR}{\widetilde{\Spf(R^+)^{\Prism}}}
\newcommand{\TRn}{\widetilde{\Spf(R^+)_n^{\Prism}}}
\setcounter{tocdepth}{1}

\title{A stacky approach to prismatic crystals via $q$-prism charts}
\author{Zeyu Liu}
\address{Department of Mathematics, University of California Berkeley, 970 Evans Hall, MC 3840 Berkeley, CA 94720}
\email{zeliu@berkeley.edu}
\begin{abstract}
Let $Y$ be a locally complete intersection over $\mathcal{O}_K$, the ring of integers of a $p$-adic field containing a $p$-power root of unity $\zeta_p$. We classify $\mathcal{D}(Y_{\Prism}, \mathcal{O}_{\Prism})$ by studying quasi-coherent complexes on $Y^{\Prism}$ (the prismatization of $Y$) via $q$-prism charts. We also develop a Galois descent mechanism to remove the assumption on $\mathcal{O}_K$. As an application, we classify $\mathcal{D}(\WCart)$ and give a purely algebraic calculation of the cohomology of the structure sheaf on $(\mathbb{Z}_p)_{\Prism}$. Along the way, for $Y$ a locally complete intersection over $\overline{A}$ with $A$ lying over a $q$-prism, we classify $\mathcal{D}(Y_{/A}^{\Prism}, \mathcal{O}_{\Prism})$, i.e. quasi-coherent complexes on the relative prismatization of $Y$.
\end{abstract}

\maketitle
\setcounter{tocdepth}{2}
\tableofcontents
\section{Introduction}
In this paper, we work with a $p$-adic field $K$. More precisely, let $\mathcal{O}_K$ be a complete discrete valuation ring of mixed characteristic with fraction field $K$ and perfect residue field $k$ of characteristic $p$. 

Introduced by Bhatt and Scholze in \cite{bhatt2022prisms}, prismatic cohomology theory turns out to be a ground breaking work in $p$-adic geometry and $p$-adic Hodge theory. It connects various known $p$-adic cohomology theories via specialization, and when passing to coefficients, various prismatic crystals are closely related to $p$-adic Galois representations and $p$-adic local systems. For example, in \cite{bhatt2021prismatic} Bhatt-Scholze have identified $\Vect^{\varphi}(X_{\Prism},\mathcal{O}_{\Prism})$ with the category of crystalline $\mathbb{Z}_p$-local systems on the generic fiber of $X$ for $X=\Spf(\mathcal{O}_K)$, giving new methods of studying $\mathbb{Z}_p$-crystalline representations of $G_K$, a key object in the area of $p$-adic Hodge theory. By working with analytic prismatic $F$-crystals instead, such a result was later generalized to $X$ a smooth $p$-adic formal scheme over $\Spf(\mathcal{O}_K)$ due to the work of Guo-Reinecke \cite{guo2024prismatic} or that of Du-Liu-Moon-Shimizu \cite{du2024completed}. Hence it becomes both important and urgent to understand the underlying prismatic crystals, viewed as objects in $\mathcal{D}(X_{\Prism}, \mathcal{O}_{\Prism})$.

Later prismatic cohomology theory and its variants have been packaged into coherent information on certain stacks, furnished by the pioneering work of Drinfeld \cite{drinfeld2020prismatization} and Bhatt-Lurie \cite{bhatt2022absolute, bhatt2022prismatization}, \cite{bhatt2023gauge}. More precisely, to any (quasi-syntomic) $p$-adic formal scheme $X$, they associate $X^{\Prism}$, the prismatization of $X$, whose quasi-coherent complexes parameterize prismatic crystals. Also, they introduce $X^{\mathcal{N}}$ and $X^{\mathrm{Syn}}$, which admit canonical morphisms $X^{\Prism}\to X^{\mathcal{N}}\to X^{\mathrm{Syn}}$ and further encode information on Nygaard filtrations and Frobenius structures. Under such a stacky language, Bhatt and Lurie have proved that the category of reflexive $F$-gauges on $\mathcal{O}_K^{\mathrm{Syn}}$ is equivalent to the category of lattices in crystalline
representations of $G_K$. Also, recently in  \cite{terentiuk2024prismatic} Terentiuk-Vologodsky-Xu have established an equivalence between a subcategory of quasi-coherent complexes on $W(k)^{\mathrm{Syn}}$ and certain derived category of Fontaine-Laffaille modules. For the purpose of understanding such $F$-gauges, we need to know more about their underlying complexes, obtained by pulling back along $X^{\Prism}\to X^{\mathrm{Syn}}$. It thus becomes a fundamental question in $p$-adic Hodge theory to understand $\mathcal{D}(X^{\Prism})$. 

Towards this direction, Bhatt and Lurie \cite{bhatt2022absolute,bhatt2022prismatization} have done a detailed study on the structure of the Hodge-Tate stack of $X$, which is a closed substack inside $X^{\Prism}$ and is denoted as $X^{\HT}$. Their work identifies $X^{\HT}$ with the classifying stack of certain groupoids for $X$ a smooth $p$-adic formal scheme over $\mathcal{O}_K$. Such a geometric description directly leads to the classification of (derived) Hodge-Tate crystals on the prismatic site of $X$, recorded in  \cite{bhatt2022absolute, bhatt2022prismatization} and \cite{anschutz2022v,anschutz2023hodge}.

    On the other hand, while the geometry of the Hodge-Tate locus $X^{\HT}$ is usually well understood, little is known about the structure of $X^{\Prism}$, even in the simplest case $X=\Spf(\mathcal{O}_K)$, leaving the description of $\mathcal{D}(X^{\Prism})$ still mysterious\footnote{Indeed, the only known result so far is Gros-Le Stum-Quir\'os's \cite{gros2023absolute} classification for $\Vect(X_{\Prism},\mathcal{O}_{\Prism})$ when $X=W(k)[\zeta_p]$.
    }. In \cite{liu2024prismatization} we made partial progress on this question. More precisely, we constructed certain nilpotent thickenings of $X^{\HT}=X_{1}^{\Prism}$ inside $X^{\Prism}=X^{\Prism}_{\infty}$, denoted as $X_{n}^{\Prism}$ in \textit{loc. cit.}, such that $\mathcal{D}(X_{n}^{\Prism})\simeq \mathcal{D}(X_{\Prism},\mathcal{O}_{\Prism}/\mathcal{I}_{\Prism}^n)$. For $X=\Spf(W(k))$ and $n\leq p$, by pulling back to the Breuil-Kisin chart along $\rho: \mathfrak{S}\to \Spf(W(k))^{\Prism}$, we then classify $\mathcal{D}(\Spf(W(k)_{n}^{\Prism})$ using linear algebraic data (roughly it means that $\rho^*\mathcal{E}$ is equipped with a monodromy operator $\theta_{\mathcal{E}}$ for $\mathcal{E}\in \mathcal{D}(\Spf(W(k)_{n}^{\Prism})$), see \cite[Theorem 1.1]{liu2024prismatization} for details. For a general $\mathcal{O}_K$ and $n$, we further observe that a monodromy operator exists on $\rho^*\mathcal{E}$ only when restricted to a slight shrinking of $\Spf(\mathcal{O}_K)^{\Prism}_n$, namely the locus obtained by adding $\frac{\mathcal{I}_{\Prism}}{p}$. Moreover, the proof in \cite{liu2024prismatization} actually implies that this is \textit{optimal} if working with the Breuil-Kisin prism. As a consequence, we classify $\mathcal{D}((\mathcal{O}_{K})_{\Prism}, \ocn)$ for all $n$ (hence also $\mathcal{D}((\mathcal{O}_{K})_{\Prism}, \oc)$) in \cite[Theorem 1.6]{liu2024prismatization}, which could be viewed as integral models for (truncated)-de Rham prismatic crystals over $\mathcal{O}_K$ studied in \cite{liu2023rham}. 

    After realizing the mystery of studying $\mathcal{D}(X^{\Prism})$ via Breuil-Kisin prism charts, it is natural to ask whether the strategy developed in \cite{liu2024prismatization} can be applied to some other charts to get a complete understanding of $\mathcal{D}(X^{\Prism})$, for $X$ a general quasi-syntomic $p$-adic scheme. Motivated by the recent work of Michel Gros, Bernard Le Stum and Adolfo Quir\'os \cite{gros2023absolute}, which classifies $\operatorname{Vect}((W(k)[\zeta_p])_{\Prism}, \mathcal{O}_{\Prism})$ (hence also $\operatorname{Vect}((W(k)[\zeta_p])_{\Prism}, \mathcal{O}_{\Prism}/{\mathcal{I}}^{m})$ for all $m$) via absolute $q$-calclus, we turn our attention to the $q$-prism charts in this paper. Our first result is the classification of $\mathcal{D}(\Spf(\mathcal{O}_K)^{\Prism})$ for $\mathcal{O}_K=W(k)[\zeta_{p^{\alpha+1}}]$ ($\alpha\geq 0$), the cyclotomic ring obtained by adjoining $p^{\alpha+1}$-th root of unity to $W(k)$. To state it, we fix some notations first.
    \begin{notation}
        Let $X=\Spf(W(k)[\zeta_{p^{\alpha+1}}])$ and $(A, d)=(W(k)[[q-1]], [p]_{q^{p^\alpha}})\in X_{\Prism}$. Let $\gamma_A: A\to A$ be the $W(k)$-linear ring automorphism sending $q$ to $q^{p^{\alpha+1}+1}$, then $\partial_A:=\frac{\gamma_A-\Id}{q(q^{p^\alpha}-1)}$ is a $\gamma_A$-derivation of $A$, i.e. $\partial_A(x_1x_2)=\gamma_A(x_1)\partial_A(x_2)+\partial_A(x_1)x_2$. We define the \textit{Ore extension} $A[\partial; \gamma_A, \partial_A]$ to be the noncommutative ring obtained by giving the ring of polynomials $A[\partial]$ a new multiplication law, subject to the identity
    $$\partial r=\gamma_A(r)\partial+\partial_A(r), ~~~\forall r\in A.$$
    \end{notation}
    \begin{theorem}[{\cref{thmt.main classification}}. Quasi-coherent complexes on the prismatization of cyclotomic rings]\label{intro.main thm1}
        Assume that $p>2$ or $\alpha>0$. With the preceding notations, for $n\in \mathbb{N}\cup \{\infty\}$, the pullback along the covering $\rho: \Spf(A/d^n)\to X_n^{\Prism}$\footnote{By abuse of notation, when $n=\infty$, this just means $\rho: \Spf(A)\to X^{\Prism}$.} induces a fully faithful functor 
    \begin{align*}
        &\beta_n^{+}: \mathcal{D}(X_n^{\Prism}) \rightarrow  \mathcal{D}(A/d^n[\partial; \gamma_A, \partial_A]), \qquad \mathcal{E}\mapsto (\rho^{*}(\mathcal{E}),\partial_{\mathcal{E}}),
    \end{align*}
    whose essential image consists of those objects $M\in \mathcal{D}(A/d^n[\partial; \gamma_A, \partial_A])$ 
     satisfying the following pair of conditions:
    \begin{itemize}
        \item $M$ is $(p,d)$-complete.
        \item The action of $\partial$ on the cohomology $\mathrm{H}^*(M\otimes^{\mathbb{L}}k)$ \footnote{Here the derived tensor product means the derived base change along $A/d^n\to A/(d,q-1)=k$.} is locally nilpotent. 
    \end{itemize}
    In particular, per \cref{rem. cohomology of the structure sheaf}, for $\mathcal{E}\in \mathcal{D}(X_n^{\Prism})$,
    \begin{equation*}
        \mathrm{R} \Gamma(X_n^{\Prism}, \mathcal{E})\xrightarrow{\simeq} \fib(\rho^*\mathcal{E}\stackrel{\partial_{\mathcal{E}}}{\longrightarrow} \rho^*\mathcal{E}).
    \end{equation*}
    \end{theorem}
\begin{remark}\label{intro.remark compare with work}
\begin{itemize}
    \item \cref{intro.main thm1} fails for $p=2$ and $\alpha=0$. Indeed, the assumption that $p>2$ or $\alpha>0$ precisely separates out the ramified case and Bhatt-Lurie's work \cite{bhatt2022absolute,bhatt2022prismatization} shows that $\Spf(\mathcal{O}_K)^{\HT}$ for $\mathcal{O}_K$ ramified behaves quite differently from $\Spf(W(k))^{\HT}$. It is still \textit{false} when $p=2, \alpha=0$ even if we modify the nilpotence condition in the statement, namely instead requiring $\partial^2-\partial$ being locally nilpotent. Actually, the failure already shows up on the Hodge-Tate locus for cohomology reasons, see \cref{rem. failure for p=2} for details.
    \item One can restrict \cref{intro.main thm1} to perfect complexes (resp. vector bundles) to get the corresponding classification for $\Perf(X_n^{\Prism})$ (resp. $\Vect(X_n^{\Prism})$), see \cref{corot.perfect version of the remain theorem} for details.
    \item Our results are new for $\alpha>0$. When $\alpha=0$ and restricted to the abelian level (i.e. working with vector bundles), \cite{gros2023absolute} identifies $\operatorname{Vect}((W(k)[\zeta_p])_{\Prism}, \mathcal{O}_{\Prism})$ with the category of finite projective $A$ modules equipped with a $q(p)$-connection, denoted as $\nabla_{q(p)} (A)$ in \textit{loc.cit.}. Given $(M,\partial_M)\in \nabla_{q(p)} (A)$, one can check that the Leibnitz rule on $\partial_{M}$ given in \cite{gros2023absolute} precisely promotes such a pair to an $A[\partial; \gamma_A, \partial_A]$-module, whose underlying module is just $M$ and $\partial$ acts on it via $\partial_M$. In this sense, our \cref{intro.main thm1} is a generalization of the main theorem in \cite{gros2023absolute} to derived coefficients when $\alpha=0$. However, our methods are different. While they use the prismatic site, we adopt the perspective furnished by Drinfeld and Bhatt–Lurie’s stacks as well as the strategy working beyond the Hodge-Tate locus developed in \cite{liu2024prismatization}.
    \item When $n=1$, note that $\mathcal{E}\in \mathcal{D}(X^{\HT})$ is equipped with two operators: the $q$-connection $\partial_{\mathcal{E}}$ from \cref{intro.main thm1} (constructed via the $q$ prism) as well as the Sen operator $\theta_{\mathcal{E}}$ from \cite[Theorem 2.5]{anschutz2022v} (constructed utilizing the Breuil-Kisin prism), hence so is $\rho^*\mathcal{E}$. The relations between the action of $\partial_{\mathcal{E}}$ and $\theta_{\mathcal{E}}$ on 
    $\rho^*\mathcal{E}$ can be summarized as follows:
    \begin{equation*}
\partial_{\mathcal{E}}=d^{\prime}(q)\cdot\frac{(1+p^{\alpha+1})^{\theta_{\mathcal{E}}/E^{\prime}(\pi)}-1}{p^{\alpha+1}}.
    \end{equation*}
    see \cref{rem. compare two operators} for details.
    \item In $p$-adic Hodge theory, the existence of the Frobenius structure typically forces the monodromy operator to be nilpotent. Taking \cref{intro.main thm1} as an input, arguing similarly as that in \cite[Corollary 10.9, Corollary 10.11]{gros2023absolute} (which treats the case $\alpha=0$), one can show that the category of $q$-connections in vector bundles over $A$ equipped with an additional Frobenius structure is equivalent to the category of prismatic $F$-crystals in vector bundles on $(\mathcal{O}_K)_{\Prism}$ for $\mathcal{O}_K=(W(k)[\zeta_{p^{\alpha+1}}]$, hence also parameterize crystalline $\mathbb{Z}_p$-representations of $G_K$ by invoking \cite[Theorem 5.6]{bhatt2021prismatic}.
\end{itemize}
\end{remark}

We now explain the idea of proving \cref{intro.main thm1} in more detail, as this strategy works uniformly to prove several upcoming theorems. For a moment we assume $X$ is a smooth affine $p$-adic formal scheme and $n\in \mathbb{N}\cup \{\infty\}$.

As pointed out at the very beginning, it's typically difficult to describe the structure of $X_{n}^{\Prism}$ unless $n=1$ (For $X=\Spf(W(k))$, we do know the structure of $X^{\Prism}_n$ for $n\leq p$ due to a result of Sasha Petrov, see \cite[Proposition 3.17]{liu2024prismatization}). This roughly means that it is hard to understand $\square$, the fiber product of the following diagram  
\[ \xymatrix{ \square \ar[r] \ar[d] & \Spf(B/J^n) \ar[d]^{\eta} \\
\Spf(B/J^n) \ar[r]^{\eta} & X^{\Prism}_n}\]
for any transversal prism $(B,J)\in X_{\Prism}$ covering the final object in the topos and $n>1$, here $\eta_n$ is the induced covering map. When translating this picture into the prismatic site, this means $B^{(1)}/J^n$ is too complicated to understand for $n>1$ \footnote{When $n=1$, it is always a divided power polynomial algebra over $B/J$ due to the work of Tian \cite{tian2023finiteness}.}, where $B^{(1)}$ is the self-product of $(B,J)$ in $X_{\Prism}$. All of the known work for classifying various prismatic crystals via the prismatic site relies on an explicit calculation of $B^{(1)}$ (and its variants) for suitable $B$, from which they deduce a Sen/monodromy operator on the evaluation of prismatic crystals at $(B,J)$ by identifying prismatic crystals with stratification over the simplicial diagram $B^{(\bullet)}$, which (plus certain nilpotence condition on this operator) is essentially enough for reconstructing the corresponding prismatic crystals. To name a few, \cite{morrow2010generalised}, \cite{tian2023finiteness}, \cite{gao2023hodge}, \cite{min2022p}, \cite{ogus2022crystalline}, \cite{liu2023rham}, \cite{gao2022rham}, \cite{gros2023absolute}, \cite{tsuji2024prismatic}, \cite{gao2024prismatic}.

A key observation in \cite{liu2024prismatization}, which is inspired by the recent work of Bhatt-Lurie \cite{bhatt2022absolute} as well as that of Anschütz-Heuer-Le Bras \cite{anschutz2022v} for $n=1$, is that to produce a monodromy operator on $\eta^{*}\mathcal{E}$ for $\mathcal{E}\in \mathcal{D}(X_n^{\Prism})$, it actually doesn't require the full understanding of $\square$. Instead, it suffices to work with an infinitesimal neighborhood of the diagonal embedding $\Spf(B/J^n)\xhookrightarrow{} \square$. More precisely, as long as we can construct certain nontrivial ring homomorphism $\psi=\Id+\epsilon \nabla: B/J^n\to B/J^n\oplus B/J^n\cdot \epsilon$\footnote{The algebraic structure on the target is to be determined later. Indeed, $\epsilon^2$ is not always $0$ unless we work with the Breuil-Kisin prism.} and specify an isomorphism $\gamma$ between two functors $\eta\circ \psi$ and $\eta\circ \iota$\footnote{Here $\iota: B/J^n\to B/J^n\oplus B/J^n\cdot \epsilon$ is the canonical inclusion.}, i.e. the following commutative diagram 
\[\xymatrixcolsep{5pc}\xymatrix{\Spf(B/J^n\oplus B/J^n\cdot \epsilon)\ar[d]^{\iota}\ar[r]^{\psi}& \Spf(B/J^n)\ar@{=>}[dl]^{\gamma} \ar[d]_{}^{\eta}
\\\Spf(B/J^n)  \ar^{\eta}[r]&X_n^{\Prism},}\]
then for any $\mathcal{E}\in \mathcal{D}(X_n^{\Prism})$, $\gamma$ will induce an isomorphism  $$\gamma_{}: \psi^{*}\eta^{*}(\mathscr{E}) \stackrel{\simeq}{\longrightarrow} \iota^{*}\eta^{*}\mathscr{E},$$
which essentially leads to a monodromy operator $\nabla_{\mathcal{E}}: \eta^{*}\mathcal{E}\to \eta^{*}\mathcal{E}$ satisfying certain Leibnitz rules which depends on the algebraic structure on $B/J^n\oplus B/J^n\cdot \epsilon$. 

From such a uniform stacky point of view, 
\begin{itemize}
    \item In \cite{liu2024prismatization}, we considered the case $X=\Spf(W(k))$ and constructed such a diagram for the Breuil-Kisin prism $(B,J)=(W(k)[[\lambda]], (\lambda))$ when $n\leq p$ (this bound is optimal, see \cite[Remark 3.2]{liu2024prismatization}). Also, again using the Breuil-Kisin prism, for a general $X=\Spf(\mathcal{O}_K)$ and $n\in \mathbb{N}\cup \{\infty\}$, we constructed such a diagram for $\TXn$ (see \cite[Definition 4.3]{liu2024prismatization}). Such constructions essentially lead to the classification results \cite[Theorem 1.1, Theorem 1.6]{liu2024prismatization} stated at the beginning of the introduction. 
    \item Let $X=\Spf(W(k)[\zeta_{p^{\alpha+1}}])$ and $(B, J)=(W(k)[[q-1]], [p]_{q^{p^\alpha}})\in X_{\Prism}$. We construct the above diagram for $\psi: B\to B[\epsilon]/(\epsilon^2-q(q^{p^\alpha}-1)\epsilon)$, from which we obtain the $q$-connections in \cref{intro.main thm1}. The complexity of the $q$-twisted Leibnitz rule comes from the complicated algebraic structure on the target of $\psi$, which is a new phenomenon when working with the $q$-prism. 
    \item When $X=\Spf(R)$ is small affine over $\overline{A}$ (in the sense of Construction \ref{intro.con}) with $(A,d)$ lying over a $q$-prism, we construct the above diagram for $\psi: \Tilde{R}\to \Tilde{R}\oplus \epsilon \Omega_{\Tilde{R}/A}$ with $(\epsilon dT_i)^2=(q-1)T_i\cdot\epsilon dT_i$ on the target, from which we extract the $q$-Higgs derivations in \cref{intro.main 3}.
\end{itemize}
\begin{remark}
    Comparing the arithmetic case and the geometric case above, we see $q$ plays the role of a coordinate $T_i$. 
\end{remark}
\begin{remark}\label{intro. rmk advantage}
    Beyond the benefit of dealing with derived coefficients uniformly, there are several other advantages of developing such a stacky method.
    \begin{itemize}
    \item First, it's usually much easier to observe certain integral structures from the stacky perspective. For example, in \cite{liu2024prismatization} we were able to classify $\mathcal{D}(\Spf(W(k)_{n}^{\Prism})$ for $n\leq p$ as well as $\mathcal{D}((\mathcal{O}_{K})_{\Prism}, \ocn)$ for all $n$ and $\mathcal{O}_{K}$, while working with prismatic cite in \cite{liu2023rham} and \cite{gao2022rham} only deals with the isogeny category $\Vect((\mathcal{O}_K)_{\Prism},\mathcal{O}_{\Prism}/\mathcal{I}_{\Prism}^n[\frac{1}{p}])$ (i.e. with $p$-inverted) for $n\geq 2$ and it's unclear to see how to refine arguments there to integral levels.
    \item  Moreover, the stacky approach allows us to work beyond the smooth case without much extra effort, cf. \cref{intro.main 3}, \cref{intro.main thm4}, which is even necessary to understand $\mathcal{D}(X^{\Prism})$ for $X=\Spf(\mathcal{O}_K)$ a point, see the discussion after \cref{intro. rem cohomology} for details. On the other hand, in every work studying prismatic crystals via the prismatic site we mentioned, a local description is given only over a framed smooth lifting except Ogus' \cite{ogus2022crystalline} and Tsuji's \cite{tsuji2024prismatic}.
        \item Finally, in comparison with the classical methods working with the prismatic site and hence requiring the calculation of $B^{(1)}$, we only need to construct the above diagram, which is usually much simpler, see the proof of \cref{propt.key automorphism of functors} for a typical example. Moreover, a key step when working with the prismatic site is to reconstruct a prismatic crystal starting with a monodromy operator, which amounts to checking certain cocycle conditions that always involve heavy calculations. On the other hand, following a strategy of Bhatt-Lurie, the essential image of $\beta_n^+$ in \cref{intro.main thm1} can be determined easily once we prove full faithfulness and know the generators of the source. 
    \end{itemize}
\end{remark}

On the other hand, the $q$-prism chart was used slightly differently by Bhatt-Lurie \cite{bhatt2022absolute} when studying quasi-coherent complexes on the Cartier-Witt stack, namely they realize the latter as a quotient of the $q$-prism. At the very end of the preparation of this project, we realize that we can play a similar game for $X^{\Prism}$ when $X=\Spf(W(k)[\zeta_{p^{\alpha+1}}])$ and give an alternative classification of $\mathcal{D}(X^{\Prism})$, more akin to Bhatt-Lurie's approach in \cite[Section 3.8]{bhatt2022absolute} (we will explain the relation of this description with \cref{intro.main thm1} in \cref{rem. relation with BL}). Indeed, for $(A, I)=(W(k)[[q-1]], [p]_{q^{p^\alpha}})\in X_{\Prism}$, it admits an action by $\mathbb{Z}_p^{\times}$, where an element $u\in \mathbb{Z}_p^{\times}$ acts by the automorphism $q\mapsto q^u$, which will be denoted as $\gamma_u$. In \cref{prop. group quotient}, we show that the cover map $\rho: \Spf(A)\to X^{\Prism}$ factors through the quotient stack $[\Spf(A)/(1+p^{\alpha+1}\mathbb{Z}_p)^{\times}]$ and moreover, it sits in the following commutative diagram:
\[\xymatrixcolsep{5pc}\xymatrix{[\Spf(A/(q^{p^{\alpha}}-1))/(1+p^{\alpha+1}\mathbb{Z}_p)^{\times}]\ar[d]^{}\ar[r]^{q^{p^{\alpha}}\mapsto 1}& [\Spf(A)/(1+p^{\alpha+1}\mathbb{Z}_p)^{\times}] \ar[d]_{}^{\rho}
\\\Spf(A/(q^{p^{\alpha}}-1)) \ar[r]^{\rho_{}}&X^{\Prism}.}\]
Then an alternative description of $\mathcal{D}(X^{\Prism})$ is given as follows:
\begin{theorem}[{\cref{thmt.main classification II}}]\label{into thmt.main classification II}
    Let $X=\Spf(W(k)[\zeta_{p^{\alpha+1}}])$. Assume that $p>2$ or $\alpha>0$. Then the diagram above induces a fully faithful functor of $\infty$-categories 
$$F: \mathcal{D}(X^{\Prism}) \longrightarrow \mathcal{D}([\Spf(A)/(1+p^{\alpha+1}\mathbb{Z}_p)^{\times}])\times_{\mathcal{D}([\Spf(A/(q^{p^{\alpha}}-1))/(1+p^{\alpha+1}\mathbb{Z}_p)^{\times}])} \mathcal{D}(\Spf(A/(q^{p^{\alpha}}-1))).$$
The essential image of $F$ consists of those pairs $(M,\gamma_M)$ with $M\in \mathcal{D}(A)$ and $\gamma_M$ being an automorphism of $M$ which is congruent to the identity modulo $q(q^{p^{\alpha}}-1)$ such that the following holds:
\begin{itemize}
    \item $M\in \mathcal{D}(A)$ is $(p,d)$-complete.
    \item the action of $\partial_M$ on the cohomology $\mathrm{H}^*(M\otimes^{\mathbb{L}}k)$ \footnote{Here the derived tensor product means the derived base change along $A/d^n\to A/(d,q-1)=k$.} is locally nilpotent, where $\partial_M: M\to M$ is an operator satisfying $\gamma_M=\Id+q(q^{p^{\alpha}}-1)\partial_M$.
\end{itemize}
    In particular, for $\mathcal{E}\in \mathcal{D}(X^{\Prism})$, the diagram 
\begin{equation}\label{intro diagram}
\begin{gathered}
\xymatrixcolsep{5pc}\xymatrix{\mathrm{R} \Gamma(X^{\Prism}, \mathcal{E})\ar[d]^{}\ar[r]^{}&\mathrm{R} \Gamma([\Spf(A)/(1+p^{\alpha+1}\mathbb{Z}_p)^{\times}], \rho^{*} \mathcal{E}) 
\ar[d]_{}^{}
\\\mathrm{R} \Gamma(\Spf(A/(q^{p^{\alpha}}-1)), \rho^{*} \mathcal{E})  \ar^{}[r]&\mathrm{R} \Gamma([\Spf(A/(q^{p^{\alpha}}-1))/(1+p^{\alpha+1}\mathbb{Z}_p)^{\times}], \rho^{*} \mathcal{E})}
\end{gathered}
\end{equation}
    is a pullback square in the $\infty$-category $\hat{\mathcal{D}}(\mathbb{Z}_p)$.
\end{theorem}
\begin{remark}[{The relation between \cref{into thmt.main classification II} with \cref{intro.main thm1}}]\label{rem. relation with BL}
 Let $u=1+p^{\alpha+1}$, which is a topological generator of $(1+p^{\alpha+1}\mathbb{Z}_p)^{\times}$, then $\gamma_u=\gamma_A$ and for $\mathcal{E}\in \mathcal{D}(X^{\Prism})$, our construction of $\partial_{\mathcal{E}}$  implies that $\gamma_u$ acts on $\rho^*\mathcal{E}$ via $1+q(q^{p^{\alpha}}-1)\partial_{\mathcal{E}}$ (see \cref{prop. monodromy and group action}). Consequently, the pullback from $\mathcal{D}(X^{\Prism})$ to $\mathcal{D}([\Spf(A)/(1+p^{\alpha+1}\mathbb{Z}_p)^{\times}])$ along
 $\rho: [\Spf(A)/(1+p^{\alpha+1}\mathbb{Z}_p)^{\times}]\to X^{\Prism}$ is not fully faithful yet (see \cref{rem. decalage} for details), and to eliminate such a difference, we need to utilize the above diagram. However, it is \cref{intro.main thm1} which can be more easily generalized to the high dimensional case, as the forthcoming \cref{intro.main 3} and \cref{intro.main thm4} justifies. 
\end{remark}

Before moving on to work with a general $X$, we discuss how to apply \cref{intro.main thm1} or \cref{into thmt.main classification II} to understand $\mathcal{D}(\WCart)$ via Galois descent theory, as suggested by \cite{bhatt2022absolute}. For this purpose, we tentatively assume $X=\mathbb{Z}_p[\zeta_p]$ and consider the $q$-prism $(A,d)=(\mathbb{Z}_p[[q-1]], [p]_q)\in X_{\Prism}$. As discussed in \cite[Section 3.5]{bhatt2022absolute} $A=\mathbb{Z}_p[[q-1]]$ admits an action of $\mathbb{F}_p^{\times}$, which carries each element $e\in \mathbb{F}_p^{\times}$ to the automorphism of $\mathbb{Z}_p[[q-1]]$ given by $\gamma_e: q\mapsto q^{[e]}$, for which $[e]$ is the Teichmuller lift of $e$. One can see that such an action preserves the ideal $(d)$, hence induces an automorphism $\gamma_e: A/d^n\to A/d^n$ for any $n\geq 1$. As an upshot, $\mathbb{F}_p^{\times}$ acts on $X^{\Prism}$. Moreover, the natural morphism $X^{\Prism}\to \WCart$ is $\mathbb{F}_p^{\times}$-equivariant, with the target given the trivial action. Hence it induces a natural morphism $\pi: [X^{\Prism}_n/\mathbb{F}_p^{\times}] \to \WCart_n$ (see \cref{prop. AD descent morphism} for details). Then our main result is the following:
\begin{theorem}[{\cref{thm. descent for the cartier witt stack to q prism}}]\label{intro. main theorem 2}
    Assume $p>2$. Then pullback along $\pi$ induces a functor 
    $$F: \mathcal{D}(\WCart) \longrightarrow \mathcal{D}_{}([X^{\Prism}/\mathbb{F}_p^{\times}])\times_{[\Spf(\mathbb{Z}_p)/\mathbb{F}_p^{\times}]} \mathcal{D}(\Spf(\mathbb{Z}_p)),$$ which turns out to be fully faithful. Moreover, the essential image of $F$ consists of those $\mathcal{E}$ in the target whose underlying complex in $ \mathcal{D}_{}([X^{\Prism}/\mathbb{F}_p^{\times}])$ (still denoted as $\mathcal{E}$ by abuse of notation) satisfies the following additional condition:
    \begin{itemize}
        \item the $q$-Higgs connection $\partial_{\mathcal{E}_0}$ of $\mathcal{E}_0:=\mathcal{E}|_{X^{\HT}}$ can be factored as $e\circ \partial_{\mathcal{E}_0}^{\prime}$ such that the action of $(\partial_{\mathcal{E}_0}^{\prime})^p-\partial_{\mathcal{E}_0}^{\prime}$ on the cohomology $\mathrm{H}^*(\rho^*\mathcal{E}\otimes^{\mathbb{L}}k)$ is locally nilpotent, here $e=d^{\prime}(q)\in W(k)[\zeta_p]$.
    \end{itemize}
\end{theorem}

Combining \cref{intro.main thm1} with \cref{intro. main theorem 2}, we can understand $\mathcal{D}(\WCart)$ quite well.
\begin{corollary}[{\cref{cor. main cor 1}}]\label{intro. main cor 1}
We have a fully faithful functor 
$$F: \mathcal{D}(\WCart) \longrightarrow \mathcal{D}_{}([\Spf(A[\partial; \gamma_A, \partial_A])/\mathbb{F}_p^{\times}])\times_{[\Spf(\mathbb{Z}_p)/\mathbb{F}_p^{\times}]} \mathcal{D}(\Spf(\mathbb{Z}_p)).$$
Moreover, the essential image of $F$ consists of those $\mathcal{E}$ in the target whose underlying complex in $ \mathcal{D}_{}([\Spf(A[\partial; \gamma_A, \partial_A])/\mathbb{F}_p^{\times}])$ (still denoted as $\mathcal{E}$ by abuse of notation) satisfies the following additional condition:
    \begin{itemize}
        \item the action of $\partial$ on $\mathcal{E}_0:=\mathcal{E}|_{X^{\HT}}$ can be factored as $e\circ \partial_{}^{\prime}$ such that the action of $(\partial_{}^{\prime})^p-\partial_{}^{\prime}$ on the cohomology $\mathrm{H}^*(\rho^*\mathcal{E}\otimes^{\mathbb{L}}k)$ is locally nilpotent, here $e=d^{\prime}(q)\in W(k)[\zeta_p]$.
    \end{itemize}
\end{corollary}
\begin{remark}\label{intro.remark compare with BL}
One can compare our \cref{intro. main theorem 2} and \cref{intro. main cor 1} 
with \cite[Theorem 3.8.3]{bhatt2022absolute}, in which they prove that there is a fully faithful functor 
\begin{equation*}
    F^{\prime}: \mathcal{D}(\WCart)\to \mathcal{D}([\Spf(A)/\mathbb{Z}_p^{\times}])\times_{\mathcal{D}([\Spf(\mathbb{Z}_p)/\mathbb{Z}_p^{\times}])} \mathcal{D}(\Spf(\mathbb{Z}_p)),
\end{equation*}
which is not essentially surjective (see \cite[Warning 3.7.12]{bhatt2022absolute}). More or less, given $\mathcal{E}\in \mathcal{D}(\WCart)$ corresponding to $(M,\partial_M)$ under \cref{intro. main cor 1}, the information of $\partial_M$ is hidden in the $(1+p\mathbb{Z}_p)^{\times}$ action on $F^{\prime}(\mathcal{E})$ as $\mathbb{Z}_p^{\times}=(1+p\mathbb{Z}_p)^{\times}\rtimes \mathbb{F}_p^{\times}$, see \cite[Propostition 3.7.1]{bhatt2022absolute}. In this sense, \cite[Theorem 3.8.3]{bhatt2022absolute} could be viewed as a combination of \cref{into thmt.main classification II} and \cref{intro. main theorem 2}, while \cref{intro. main cor 1} combines \cref{intro.main thm1} and \cref{intro. main theorem 2} instead.
\end{remark}

Another application of \cref{intro. main theorem 2} is the calculation of the cohomology of the structure sheaf on $(\mathbb{Z}_p)_{\Prism}$.
\begin{proposition}[{\cref{prop. cohomology of Zp}}]\label{intro.prop cohomology of the structure sheaf}
     Assume that $p>2$. The Cartier-Witt stack $\WCart$ is of cohomological dimension $1$. Moreover, 
    \begin{equation*}
        \mathrm{H}^{0}(\WCart, \mathcal{O})=\mathbb{Z}_p, ~~~\quad \mathrm{H}^{1}(\WCart, \mathcal{O})\simeq \prod_{n\in \mathbb{N}, n\not\equiv p\mod p+1} \mathbb{Z}_p.
    \end{equation*}
\end{proposition}
\begin{remark}\label{intro. rem cohomology}
    The cohomology of the structure sheaf on $\WCart$ was known before, for example, $\mathrm{H}^{0}(\WCart, \mathcal{O})=\mathbb{Z}_p$ is stated in \cite[Corollary 4.7.2]{drinfeld2020prismatization}. Via topological methods, in \cite{bhatt2019topological} Bhatt-Morrow-Scholze show that $\mathrm{H}^{1}((\mathbb{Z}_p)_{\Prism}, \mathcal{O}_{\Prism})=\prod_{n\in \mathbb{N}} \mathbb{Z}_p.$ Note that there is a bijection between the sets $\{n\in \mathbb{N}| n\not\equiv p\mod p+1\}$ and $\mathbb{N}$.
\end{remark}

Next we proceed to discuss the classification of $\mathcal{D}(\Spf(\mathcal{O}_K)^\Prism)$ for a general $\mathcal{O}_K$ and we can separate this question into three cases.
\begin{itemize}
    \item If $\mathcal{O}_K=W(k)[\zeta_{p^{\alpha+1}}]$ is a cyclotomic ring, then this was already solved by \cref{intro.main thm1}.
    \item If $\mathcal{O}_K$ doesn't contain $\zeta_p$, then we can consider $\mathcal{O}_K^{\prime}=\mathcal{O}_K[\zeta_p]$ first, and then run Galois descent similar as that in \cref{intro. main theorem 2} provided we can understand $\mathcal{D}(\Spf(\mathcal{O}_K^{\prime})^\Prism)$.
    \item If $\mathcal{O}_K$ contains $\zeta_p$, then $\mathcal{O}_K$ could be viewed as a locally complete intersection over $\mathcal{O}_{K_0}=W(k)[\zeta_p]$. Indeed, if we choose a uniformizer $\pi\in \mathcal{O}_K$ and suppose $E(u)$ is the minimal polynomial of $\pi$ over $W(k)[\frac{1}{p}]$, suppose $E(u)$ factored as $\prod_{i=0}^{t} E_i$ over $\mathcal{O}_{K_0}$ with $E_i$ monic and irreducible, WLOG we may assume $\pi$ is a root of $E_0$ (then $E_i(\pi)\neq 0$ for $i>0$ as $E$ is irreducible), then $\mathcal{O}_K=\mathcal{O}_{K_0}[u]/E_0(u)$.
\end{itemize}

With the discussion above, the question turns into classifying $\mathcal{D}(Y^\Prism)$ for $Y$ a locally complete intersection over $\mathcal{O}_K=W(k)[\zeta_{p^{}}]$. Fix  $(A,d)=(W(k)[[q-1]], [p]_q)\in (W(k)[\zeta_{p^{}}])_{\Prism}$ as usual. By considering the following pullback square,
\[ \xymatrix{  Y_{/A}^{\Prism} \ar[d]^{}\ar[r]^{\rho_A} & Y^{\Prism} \ar[d]\\  \Spf(A) \ar[r]^{\rho_A} &\Spf(\mathcal{O}_K)^{\Prism},}\]
where $Y_{/A}^{\Prism}$ is the relative prismatization, we could proceed in two steps. First, we need to understand how $\mathcal{D}(\Spf(\mathcal{O}_K)^{\Prism})$ differs from $\mathcal{D}(\Spf(A))$, which was exactly illustrated in \cref{intro.main thm1}. Secondly, we would like to have a characterization of $\mathcal{D}(Y_{/A}^{\Prism})$. It turns out for the purpose of classifying complexes on the relative prismatization, we can work under a much more general setting: we don't need to assume the base prism is exactly the $q$-prism $(W(k)[[q-1]], [p]_q)$, instead, any prism lying over it works. More precisely,
\begin{construction}\label{intro.con}
Fix $(A,d)$ to be a transversal prism over $(W(k)[[q-1]], [p]_q)$. Let $X=\Spf(R)$ be a small affine over $\Bar{A}=A/([p]_q)$ with a fixed chart $\square:=\overline{A}\left\langle T_1, \ldots, T_m\right\rangle \rightarrow R$. Moreover, this \'etale chart map uniquely lifts to a prism $(\Tilde{R}, I)$ over $(A\left\langle \underline{T}\right\rangle, [p]_q)$. Let $Y=\Spf(R/(\Bar{x}_1,\cdots,\Bar{x}_r)) \hookrightarrow X$ ($x_i\in \Tilde{R}$ and $\Bar{x}_i$ is its image in $\Bar{A}$) be a closed embedding such that the prismatic envelope with respect to the morphism of $\delta$-pairs $(\Tilde{R}, (d))\to (\Tilde{R}, (d,x_1,\cdots,x_m))$ exists and is exactly given by $\Tilde{R}\{\frac{x_1,\cdots,x_r}{d}\}^{\wedge}_{\delta}$\footnote{It is obtained by freely adjoining $\frac{x_i}{d}$ to $\Tilde{R}$ in the category of derived $(p,I)$-complete simplicial $\delta$-$A$-algebras, see \cite{bhatt2022prisms}[Corollary 2.44] for the precise definition. It is a derived ring that might not be concentrated on degree $0$.}, denoted as $\Prism_X(Y)$ later. One can show that (\cref{prop. lci case base to R}) such data induces a covering map 
\[\rho: \Spf(\Prism_X(Y)/d^n)\to Y_{/A,n}^{\Prism}.\]
\end{construction}
\begin{theorem}[{\cref{thmt. relative main classification},~\cref{thmt. lci main classification}}]\label{intro.main 3}
    Let $Y$ be as that in the previous construction or $Y=X$. For $n\in \mathbb{N}\cup \{\infty\}$. Then the covering map $\rho$ above induces a fully faithful functor\footnote{By abuse of notation, when $Y=X$, we use $\Prism_X(X)$ to denote $\Tilde{R}$.} 
    \begin{align*}
        &\beta_n^{+}: \mathcal{D}(Y_{/A, n}^{\Prism}) \rightarrow \mathcal{D}(\XY/d^n[\underline{\nabla}; \gamma_{\XY}, \nabla_{\XY}]), \qquad \mathcal{E}\mapsto (\rho^{*}(\mathcal{E}),\nabla_{\mathcal{E}}),
    \end{align*}
   whose essential image consists of those objects $M\in \mathcal{D}(\XY/d^n[\underline{\nabla}; \gamma_{\XY}, \nabla_{\XY}])$ 
     satisfying the following pair of conditions:
    \begin{itemize}
        \item $M$ is $(p,d)$-adically complete.
        \item The action of $\nabla_{i}$ on the cohomology $\mathrm{H}^*(M\otimes_{\XY/d^n}^{\mathbb{L}}\XY/(d,p))$ 
        is locally nilpotent for all $i$. 
    \end{itemize}
    In particular, by \cref{rem. rel cohomology of q higgs}, for $\mathcal{E}\in \mathcal{D}(Y_{/A,n}^{\Prism})$,
    \begin{equation*}
        \mathrm{R} \Gamma(Y_{/A, n}^{\Prism}, \mathcal{E})\xrightarrow{\simeq} (\rho^*\mathcal{E} \xrightarrow{\nabla_{\mathcal{E}}} \rho^*\mathcal{E}\otimes_{\XY/d^n}(\XY\otimes_{\Tilde{R}}\Omega_{(\Tilde{R}/d^n)/(A/d^n)})\cdots \xrightarrow{\nabla_{\mathcal{E}}\wedge \nabla_{\mathcal{E}}} \cdots 0).
    \end{equation*}
\end{theorem}

The non-commutative ring $\XY/d^n[\underline{\nabla}; \gamma_{\XY}, \nabla_{\XY}]$ appearing above is defined as follows:
\begin{definition}\label{intro. def Rel q higgs}
 We define the \textit{Ore extension} $\XY[\underline{\nabla}; \gamma_{\XY}, \nabla_{\XY}]$ to be the noncommutative ring obtained by giving the ring of polynomials $\XY[\nabla_1,\cdots,\nabla_m]$ a new multiplication law, subject to the identity
    \begin{equation*}
        \begin{split}
           \nabla_i\cdot \nabla_j&=\nabla_j\cdot \nabla_i,  ~~~\forall ~ 1\leq i,j\leq m
           \\ \nabla_i \cdot r&=\gamma_{\XY,i}(r)\cdot\nabla_i+\nabla_{\XY,i}(r), ~~~\forall~ r\in \XY, 1\leq i\leq m .
        \end{split}
    \end{equation*}
    Here $\gamma_{\XY,i}$ (resp. $\nabla_{\XY,i}$) is the ring automorphism (resp. $\gamma_{\XY,i}$-derivation) on $\XY$ defined in \cref{rem. lci replace psi by i}, extending $\gamma_{\Tilde{R},i}$ (resp. $\nabla_{\Tilde{R},i}$) on $\Tilde{R}$ defined above \cref{lem. twist S in the relative case}. 
\end{definition}
\begin{remark}\label{intro.rem Tsuji}
    Recently Tsuji studied $\Vect((Y/A)_{\Prism},\mathcal{O}_{\Prism})$ via relative prismatic site in \cite{tsuji2024prismatic}. He proved that when $X$ is small affine over the $m$-dimensional torus over $A$, there is a canonical equivalence between $\Vect((Y/A)_{\Prism},\mathcal{O}_{\Prism})$ and the category of $D$-modules with quasi-nilpotent $q$-Higgs field. Moreover, he was also able to calculate the cohomology of a prismatic crystal in terms of the cohomology of the associated $q$-Higgs complex, see \cite[Theorem 0.1, Theorem 0.2]{tsuji2024prismatic}. Such results should be closely related to \cref{intro.main 3} when restricted to the abelian level. But again, our methods are totally different as we follow the stacky roadmap directed by Drinfeld, Bhatt-Lurie and apply the strategy working beyond the Hodge-Tate locus developed in \cite{liu2024prismatization}.
\end{remark}

Following the picture illustrated before Construction \ref{intro.con}, essentially by combing the tool-kits proving \cref{intro.main thm1} and \cref{intro.main 3}, we end up with the following classification of $\mathcal{D}(Y^{\Prism})$ for $Y$ a locally complete intersection over $W(k)[\zeta_{p^{\alpha+1}}]$.
\begin{theorem}[{\cref{thmt. AR main classification}, \cref{thmt. ARL main classification}}]\label{intro.main thm4}
    Let $Y$ be as that in Construction \ref{intro.con} or $Y=X$. We further assume $(A,d)=(W(k)[[q-1]], [p]_{q^{p^\alpha}})$ with $p>2$ or $\alpha>0$. Then for $n\in \mathbb{N}\cup \{\infty\}$, the covering map $\rho_Y: \Spf(\Prism_X(Y)/d^n)\to Y_{n}^{\Prism}$ induces a fully faithful functor\footnote{Similarly as before, when $Y=X$, we use $\Prism_X(X)$ to denote $\Tilde{R}$.}  
    \begin{align*}
        &\beta_n^{+}: \mathcal{D}(Y_{ n}^{\Prism}) \rightarrow \mathcal{D}(\XY/d^n[\partial, \underline{\nabla}; \gamma_{\Tilde{R}}]), \qquad \mathcal{E}\mapsto (\rho_Y^{*}(\mathcal{E}), \partial_{\mathcal{E}}, \nabla_{\mathcal{E}}),
    \end{align*}
    whose essential image consists of those objects $M\in \mathcal{D}(\XY/d^n[\partial, \underline{\nabla}; \gamma_{\Tilde{R}}])$ satisfying the following pair of conditions:
    \begin{itemize}
        \item $M$ is $(p,d)$-adically complete.
        \item The action of $\partial$  and $\nabla_{i}$ on the cohomology $\mathrm{H}^*(M\otimes_{\XY/d^n}^{\mathbb{L}}\XY/(d,q-1))$ 
        is locally nilpotent for all $i$.
    \end{itemize}
    
     In particular, by \cref{absolute rem. rel cohomology of q higgs}, for $\mathcal{E}\in \mathcal{D}(Y_{n}^{\Prism})$,
     \begin{equation*}
        \mathrm{R} \Gamma(Y_{n}^{\Prism}, \mathcal{E})\xrightarrow{\simeq} \fib(\mathrm{DR}(\rho_Y^*\mathcal{E},\nabla_{\mathcal{E}})\xrightarrow{\Tilde{\partial}_{\mathcal{E}}^{[\bullet]}} \mathrm{DR}(\rho_Y^*\mathcal{E},\nabla_{\mathcal{E}})).
    \end{equation*}
\end{theorem}
The non-commutative ring $\XY[\partial, \underline{\nabla}; \gamma_{\XY}]$ just adds one variable $\partial$ to $\XY[\underline{\nabla}; \gamma_{\XY}, \nabla_{\XY}]$ to encode the arithmetic information, and the relations with other variables are given below.
\begin{definition}[{\cref{absolute def.skew polynomial}, \cref{rem. ARL RELATIONS OF GAMMA}}] We define the \textit{Ore extension} $\XY[\partial, \underline{\nabla}; \gamma_{\XY}]$ to be the noncommutative ring obtained by giving the ring of polynomials $\XY[\partial, \nabla_1,\cdots,\nabla_m]$ a new multiplication law, subject to the identity
    \begin{equation*}
        \begin{split}
           \nabla_i\cdot \nabla_j&=\nabla_j\cdot \nabla_i,  ~~~\forall ~ 1\leq i,j\leq m,
           \\ \nabla_i \cdot r&=\gamma_{\XY,i}(r)\cdot\nabla_i+\nabla_{\XY,i}(r), ~~~\forall~ r\in \Tilde{R}, 1\leq i\leq m,
           \\ \partial\cdot r&=\gamma_{\XY,0}(r)\cdot\partial+\partial_{\XY}(r), ~~~\forall~ r\in \Tilde{R},
           \\ \partial\cdot \nabla_{i}&=s_0^{-1}(1+\beta q \mathscr{D}(\nabla_{i}))\nabla_{i}\cdot \partial+(s_0^{-1}\mathscr{D}(\nabla_{i})-s_1)\cdot \nabla_{i},~~~\forall ~ 1\leq i\leq m.
        \end{split}
    \end{equation*}
    Here  $\gamma_{\XY,i}$ (resp. $\partial$ for $i=0$ and $\nabla_{\XY,i}$ for $i\neq 0$) is the ring automorphism (resp. $\gamma_{\XY,i}$-derivation) on $\XY$ defined in \cref{rem. lci replace psi by i}, extending $\gamma_{\Tilde{R},i}$ (resp. $\partial_{\Tilde{R}}$, $\nabla_{\Tilde{R},i}$) on $\Tilde{R}$ defined above \cref{lem. AR relation between 0 and i}. $s_0, s_1, \mathscr{D}(\nabla_i)$ are defined in \cref{lem. AR relation between 0 and i}.
\end{definition}
\begin{remark}
    \begin{itemize}
        \item To the best knowledge of the author, \cref{intro.main thm4} is the most general result so far. In the interesting case that $Y$ is a locally complete intersection over $\mathcal{O}_K$ (for example, $Y$ is a small affine over $\mathcal{O}_K$) with $\mathcal{O}_K$ containing $\zeta_p$, then following the discussion after \cref{intro. rem cohomology}, we can write $Y$ as a locally complete intersection over $\mathcal{O}_{K_0}=W(k)[\zeta_p]$, hence we could apply \cref{intro.main thm4} to describe $\mathcal{D}(Y_{n}^{\Prism})$, i.e. $\mathcal{D}(Y_{\Prism}, \mathcal{O}_{\Prism}/\mathcal{I}_{\Prism}^n)$ for $n\in \mathbb{N}\cup \{\infty\}$. Even in the case that $\mathcal{O}_K$ doesn't contain $\zeta_p$, we could run Galois descent and reduce to the previous case. An analog of \cref{intro. main theorem 2} to this full generality will appear in an upcoming work. 
        \item For $X$ a smooth $p$-adic formal scheme over $\mathcal{O}_K$, recently Gao-Min-Wang \cite{gao2024prismatic} classify $\Vect(X_{\Prism},\mathcal{O}_{\Prism}/\mathcal{I}_{\Prism}^n[\frac{1}{p}])$ for any $n$ (and hence $\Vect(X_{\Prism},\bdr)$) using prismatic site and Breuil-Kisin prisms. On the other hand, it is known that the results in \cite{liu2024prismatization} still hold for such a general $X$. Actually, in an upcoming joint work, for $Y$ a locally complete intersection over $\mathcal{O}_K$, we give a classification of $\mathcal{D}(Y_{\Prism}, \ocn)$ for all $n$ (hence also $\mathcal{D}(Y_{\Prism}, \oc)$) via Breuil-Kisin charts, which are integral models for de Rham prismatic crystals. Moreover, we will relate them with local systems over certain integral period sheaves on the $v$-site of the generic fiber of $Y$ studied in \cite{liu2024stacky}, which appears naturally in \textit{loc.cit.} when they study a stacky $p$-adic Riemann-Hilbert correspondence.
        \item Considering the special case  $Y=\Spf(W(k)[\zeta_{p^{\alpha+1}}])/\pi^m$\footnote{It is a non-trivial result to see such a scheme satisfies our assumption on $Y$, see \cref{lem. ARL discrete case}.}, where $\pi=\zeta_{p^{\alpha+1}}-1$ is a uniformizer. Then \cref{intro.main thm4} provides a concrete complex to calculate the cohomology of the structure sheaf on $Y_{\Prism}$ (this is exactly how we prove \cref{intro.prop cohomology of the structure sheaf}). After a detailed study of certain filtered version of this complex, we believe it can be used to calculate the $K$ groups of $Y$. Note that recently Antieau-Krause-Nikolaus study prismatic cohomology of $\delta$-rings in \cite{antieau2023prismatic}, from which they construct concrete complexes to calculate the syntomic cohomology of $\mathcal{O}_K/\pi^m$ in \cite{antieau2024k}, and hence the $K$ groups of $\mathcal{O}_K/\pi^m$ via trace methods. We expect our complex to offer an alternative calculation. Moreover, by varying $Y$ in \cref{intro.main thm4}, it might shed some light on the calculation of $K$ groups of a more general $Y$.
        \item Heuristically one can think of $\mathcal{D}(\XY[\partial, \underline{\nabla}; \gamma_{\Tilde{R}}])$ as the derived category of \textit{enhanced $q$-Higgs modules}, where the notation latter is encouraged by the work of Min-Wang \cite{min2022p}, where they identify Hodge-Tate crystals, i.e. $\Vect(\Spf(R)_{\Prism}, \overline{\mathcal{O}}_{\Prism})$ with the category of enhanced Higgs modules over $R$ for $R$ smooth affine over $\mathcal{O}_K$. 
    \end{itemize}
\end{remark}

\subsection*{Outline}
The paper is organized as follows. In section 2 we recollect some constructions and properties of the (truncated) prismatization functor. Section 3 explains the construction of the $q$-Higgs connection on $\rho^{*}(\mathscr{E})$ for $\mathscr{E} \in \mathcal{D}((W(k)[\zeta_{p^{\alpha+1}}])^{\Prism})$ and we prove \cref{intro.main thm1} and \cref{into thmt.main classification II}. Moreover we develop the Galosi descent machine and prove \cref{intro. main theorem 2} as well as several related results. In section 4 we study quasi-coherent complexes on the relative prismatization of $Y$ over $A$, for $A$ a transversal prism over the $q$-prism and prove \cref{intro.main 3}. Finally in section 5, we study $\mathcal{D}(Y^{\Prism})$ for $Y$ a locally complete intersection over a cyclotomic ring and obtain \cref{intro.main thm4}.

\subsection*{Notations and conventions}
\begin{itemize}
    \item In this paper $\mathcal{O}_K$ is a complete discrete valuation ring of mixed characteristic with fraction field $K$ and perfect residue field $k$ of characteristic $p$.
    \item For $X$ a $p$-adic bounded formal scheme, $X^{\Prism}$ (resp. $X^{\HT}$) is the prismatization of $X$ (resp. the Hodge-Tate stack of $X$) defined as $\WCart_X$ (resp. $\WCart_X^{\HT}$) in \cite{bhatt2022absolute} and \cite{bhatt2022prismatization}. But when $X=\mathbb{Z}_p$, we stick to the original notion $\WCart$ and $\WCart^{\HT}$. Moreover, if $X$ is over a base prism $(A,I)$, we write $X_{/A}^{\Prism}$ for the relative prismatization of $X$ over $A$, denoted as $\WCart_{X/A}$ in \cite{bhatt2022prismatization}.
    \item $[n]_{q}:=\frac{q^n-1}{q-1}=1+q+\cdots +q^{n-1}$.
    \item When $R$ is a non-commutative ring, we write $\mathcal{D}(R)$ for the derived category of left $R$-modules.
\end{itemize}
\subsection*{Acknowledgments}
The influence of the work of Bhatt and Lurie \cite{bhatt2022absolute,bhatt2022prismatization} on this paper should be obvious to readers, we thank them for their pioneering and wonderful work. We express our gratitude to Benjamin Antieau, Bhargav Bhatt, Kiran Kedlaya, Shizhang Li, and Noah Olander for helpful discussions. Special thanks to Johannes Anschütz and Sasha Petrov for their interest in this project as well as useful discussions and suggestions on various parts of this paper. The author is indebted to his mentor Martin Olsson for constant help and guidance throughout the writing of this paper. During the preparation of the project, the author was partially supported by the Simons Collaboration grant on Perfection under Professor Olsson.

\section{Recollections on the prismatization functor}
Motivated by Bhatt and Lurie's definition of $\WCart^{\mathrm{HT}}$ (see \cite{bhatt2022absolute}), in \cite{liu2024prismatization} we define certain nilpotent thickenings of $\WCart^{\mathrm{HT}}$ inside $\WCart$. 
\begin{definition}[{\cite[Definition 2.1]{liu2024prismatization}}]\label{defi.truncated}
    Let $R$ be a $p$-nilpotent commutative ring. Fix a positive integer $n$, We let $\WCart^{\mathrm{HT}}_{n}(R)$ denote the full subcategory of $\WCart(R)$
spanned by those Cartier-Witt divisors $\alpha: I \rightarrow W(R)$ for which the composite map $I^{\otimes n} \xrightarrow{\alpha^{\otimes n}} W(R) \twoheadrightarrow R$
is equal to zero. The construction $R \mapsto \WCart^{\mathrm{HT}}_n(R)$ determines a closed substack of the Cartier-Witt stack $\WCart$.
We denote this closed substack by $\WCart_n$. 
\end{definition}

In particular, when $n=1$, $\WCart_1$ coincides with $\WCart^{\mathrm{HT}}$ and we will switch freely between these two notations. In general, $\WCart_{n}$ could be viewed as a infinitesimal thickening of $\WCart^{\mathrm{HT}}$. To unify the notation, sometimes we write $\WCart_{\infty}$ for $\WCart$ in this paper.
\begin{remark}\label{remark.second definition}
As pointed out in \cite[Remark 2.3]{liu2024prismatization}, there is an alternative definition for $\WCart_{n}$. We briefly review it here. Recall that given a Cartier-Witt divisor $I\rightarrow W(R)$, its base change along the restriction map $W(R)\rightarrow R$ is a generalized Cartier divisor. Consequently this determines a morphism of stacks $\mu: \WCart \to [\mathbb{A}^1/\mathbb{G}_m]$, which actually factors through the substack $[\widehat{\mathbb{A}}^1/\mathbb{G}_m]$ as the image of $I$ in $R$ is nilpotent (see \cite[Remark 3.1.6]{bhatt2022absolute} for details). From this point of view, unwinding \cref{defi.truncated}, we see that the diagram 
    $$ \xymatrix@R=50pt@C=50pt{ \WCart_n \ar[d] \ar@{^{(}->}[r] & \WCart \ar[d]^{\mu}\\
[\Spf(\mathbb{Z}[[t]]/t^n)/\mathbb{G}_m] \ar@{^{(}->}[r]^-{} & [\widehat{\mathbb{A}}^1/\mathbb{G}_m] }$$
is a pullback square, which gives an equivalent definition of $\WCart_n$.
\end{remark}

Following \cite[Construction 3.7]
{bhatt2022prismatization}, we can generalize the previous construction to any bounded $p$-adic formal scheme $X$.
\begin{construction}[{\cite[Construction 2.7]{liu2024prismatization}}]\label{ConsAbsHT}
Let $X$ be a bounded $p$-adic formal scheme. Given $n\in \mathbb{N}\cup \{\infty\}$, form a fiber square
\[ \xymatrix{ X_{n}^{\Prism} \ar[r] \ar[d] & X^{\Prism} \ar[d] \\
\mathrm{WCart}_n \ar[r] & \mathrm{WCart} }\]
defining the closed substack $X_{n}^{\Prism}$ inside $X^{\Prism}$, the prismatization of $X$. We will call $X_{n}^{\Prism}$ \textit{the $n$-truncated prismatization of $X$}.
\end{construction}

With such a definition in hand, \cite[Proposition 8.13, Proposition 8.15]{bhatt2022prismatization} imply that quasi-coherent complexes on $X_n^{\Prism}$ corresponds to prismatic crystals on $X_{\Prism}$ under mild assumption on $X$:
\begin{proposition}[{\cite[Proposition 8.13, Proposition 8.15]{bhatt2022prismatization}, see \cite[Proposition 2.11]{liu2024prismatization}}]\label{prop. recollect}
    Assume that $X$ is a quasi-syntomic $p$-adic formal scheme, then there is an equivalence
    $$\mathcal{D}_{}(X_{n}^{\Prism}) \xrightarrow{\sim} \lim _{(A, I)\in X_{\Prism}} \widehat{\mathcal{D}}(A/I^{n})=:\mathcal{D}_{\crys}(X_{\Prism}, \mathcal{O}_{\Prism}/\mathcal{I}_{\Prism}^n) $$
    of symmetric monoidal stable $\infty$-categories. Similar results hold if we consider perfect complexes or vector bundles on both sides instead.
\end{proposition}

Finally we quickly explain how to pass from a prism in $X_{\Prism}$ to an affine point of $X_n^{\Prism}$ following \cite[Construction 3.10]{bhatt2022prismatization}. Such a construction will be used frequently in this paper without further notice. 
\begin{construction}\label{construction. covers}
    Let $X$ be a bounded $p$-adic formal scheme. Fix an object $(\operatorname{Spf}(A) \leftarrow \operatorname{Spf}(\bar{A}) \rightarrow X) \in X_{\Prism}$. Given $(p,I)$-nilpotent $A$-algebra $R$, the structure morphism $A\to R$ uniquely lifts to $\delta$-algebra homomorphism $A\rightarrow W(R)$, along which the base change of the inclusion $I\rightarrow A$ determines a Cartier-Witt divisor $I\otimes_A W(R)\rightarrow W(R)$ together with a map $\eta: \Spec(\overline{W(R)})\to \Spf(\Bar{A})\to X$ of derived schemes. Letting $R$ vary, such construction induces a morphism $\rho_{A}: \Spf(A) \rightarrow X^{\Prism}$. Moreover, 
Then $\rho_{A}$ carries the formal subscheme $\Spf(A/I^n) \subset \Spf(A)$ to $X_n^{\Prism}$, and therefore restricts to a morphism $\rho_{n,A}: \Spf(A/I^n) \rightarrow X_n^{\Prism}$. Later we will omit $n$ and $A$ when both of them are clear in the context.
\end{construction}

Given $\rho_{n,A}$ above, one could define the \textit{$n$-truncated relative prismatization}.
\begin{construction}
    Let $X$ be a bounded $p$-adic formal scheme. Given $n\in \mathbb{N}\cup \{\infty\}$, form a fiber square
\[ \xymatrix{ X_{/A, n}^{\Prism} \ar[r] \ar[d] & X_n^{\Prism} \ar[d] \\
\Spf(A/I^n) \ar[r]^{\rho_{A,n}} & \Spf(\Bar{A})_n^{\Prism} }\]
defining the closed substack $X_{/A, n}^{\Prism}$ inside $X_{/A}^{\Prism}$, the relative prismatization of $X$. We will call $X_{/A, n}^{\Prism}$ \textit{the $n$-truncated relative prismatization of $X$}.
\end{construction}

Then the following analog of \cref{prop. recollect} is due to \cite[Theorem 6.5]{bhatt2022prismatization}.
\begin{proposition}
    Let $(A,I)$ be a prism and $X$ be a quasi-syntomic $p$-adic formal scheme over $\Bar{A}$, then there is an equivalence
    $$\mathcal{D}_{qc}(X_{/A,n}^{\Prism}) \xrightarrow{\sim} \lim _{(A, I)\in (X/A)_{\Prism}} \widehat{\mathcal{D}}(A/I^{n})=:\mathcal{D}_{\crys}((X/A)_{\Prism}, \mathcal{O}_{\Prism}/\mathcal{I}_{\Prism}^n) $$
    of symmetric monoidal stable $\infty$-categories. Similar results hold if we consider perfect complexes or vector bundles on both sides instead.
\end{proposition}

\section{$q$-connections on the prismatization of the absolute cyclotomic ring}
In this section we work with $\mathcal{O}_K=W(k)[\zeta_{p^{\alpha+1}}]$ ($\alpha\geq 0$) and the $q$-prism $(A, I)=(W(k)[[q-1]], [p]_{q^{p^\alpha}})$. $X$ in this section refers to $\Spf(\mathcal{O}_K)$. As $(A,I)$ defines a transversal prism in $X_{\Prism}$, we have a faithfully flat cover $\rho: \Spf(A)\to X^{\Prism}$ as explained in Construction \ref{construction. covers}. Our goal is to construct a $q$-connection $\partial$ on $\rho^*\mathscr{E}$ for $\mathscr{E} \in \Qcoh(X^{\Prism})$ following the strategy explained in the introduction and then study its behaviors. For that purpose, we need several preliminaries.
\subsection{Construction of the $q$-connection after pulling back to the $q$-prism}
\begin{lemma}\label{lemt.extend delta}
Let $R=A[\epsilon]/(\epsilon^2-q(q^{p^\alpha}-1)\epsilon)=A\oplus \epsilon A$. Then the $W(k)$-linear ring homomorphism $\psi: A \to  R$ sending $q$ to $q+\epsilon [p]_{q^{p^{\alpha}}}$ induces a ring homomorphism $ A/([p]_{q^{p^\alpha}}^n) \to R/([p]_{q^{p^\alpha}}^n)$ for any $n\in \mathbb{N}$, which will still be denoted as $\psi$ by abuse of notation.
\end{lemma}
\begin{proof}
    As $A$ is topologically free over $W(k)$, any $W(k)$-linear ring homomorphism from $A$ to $R$ is uniquely determined by its image of $q$. Next by induction we calculate that for any $k\in \mathbb{N}$, $$\psi(q^k)=q^k+\epsilon q^{k-1}[p]_{q^{p^\alpha}} [k]_{q^{p^{\alpha+1}}}=q^k+\epsilon [pk]_{q^{p^\alpha}} q^{k-1}.$$ For $k=1$, this holds by the definition of $\psi$. Suppose it holds up to $k$, then a direct induction shows that 
    \begin{equation*}
    \begin{split}
        \psi(q^{k+1})&=\psi(q^{k})\cdot \psi(q)=(q^k+\epsilon q^{k-1}[p]_{q^{p^\alpha}} [k]_{q^{p^{\alpha+1}}})(q+\epsilon [p]_{q^{p^\alpha}})
        \\&=q^{k+1}+\epsilon [p]_{q^{p^\alpha}} q^k(1+[k]_{q^{p^{\alpha+1}}})+\epsilon^2 q^{k-1}[p]_{q^{p^\alpha}}^2 [k]_{q^{p^{\alpha+1}}}
        \\&=q^{k+1}+\epsilon [p]_{q^{p^\alpha}} q^k(1+[k]_{q^{p^{\alpha+1}}})+\epsilon q^{k}\cdot (q^{p^\alpha}-1)\cdot (\frac{q^{p^{\alpha+1}}-1}{q^{p^\alpha}-1})^2\cdot \frac{q^{kp^{\alpha+1}}-1}{q^{p^{\alpha+1}}-1}
        \\&=q^{k+1}+\epsilon [p]_{q^{p^\alpha}} q^k(1+[k]_{q^{p^{\alpha+1}}}+q^{kp^{\alpha+1}}-1)
        \\&=q^{k+1}+\epsilon [p]_{q^{p^\alpha}} q^k [k+1]_{q^{p^{\alpha+1}}},
    \end{split}
    \end{equation*}
    hence the claim works for all positive integers.

    It then follows that $\psi([p]_{q^{p^\alpha}})=\psi(\sum_{i=0}^{p-1} q^{p^{\alpha }i})=[p]_{q^{p^\alpha}}+\epsilon [p]_{q^{p^\alpha}} (\sum_{i=0}^{p-1} q^{p^{\alpha }i-1}[p^{\alpha }i]_{q^{p^{\alpha+1}}})$, hence $\psi([p]_{q^{p^\alpha}})\in ([p]_{q^{p^\alpha}})R$, which further implies that $\psi$ induces a ring homomorphism $ A/([p]_{q^{p^\alpha}}^n) \to R/([p]_{q^{p^\alpha}}^n)$ for any $n\in \mathbb{N}$. 
\end{proof}
\begin{remark}\label{rem.relation with automorphism}
    One can check that the composition of $\psi$ with the quotient map $R\to A$ sending $\epsilon$ to $q(q^{p^\alpha}-1)$ is an automorphism of $A$ under which $q$ is taken to $q^{p^{\alpha+1}+1}$. We will denote this automorphism as $\gamma_A$ for later use. In this sense, $\psi(f)=f+\epsilon \partial_A(f)$ for 
    \begin{equation*}
        \begin{split}
            \partial_A: A&\longrightarrow A\\
            f&\longmapsto \frac{\gamma_A(f)-f}{q(q^{p^\alpha}-1)}.
        \end{split}
    \end{equation*}
    Note that $\partial_A$ is only $W(k)$-linear.
\end{remark}
\begin{remark}
    If $\alpha=0$, then $\psi(q^k)=q^k+\epsilon [pk]_q q^{k-1}$. One can compare $\psi$ with the $q(p)$-derivation in \cite{gros2023absolute}. Roughly speaking, $\psi=\Id+\epsilon \partial_{q(p)}$, here $\partial_{q(p)}$ is the \textit{absolute $q(p)$-derivation} defined in loc. cit.
\end{remark}
\begin{remark}\label{rem. delta ring structure on R}
We upgrade $R$ to a $\delta$-ring extending that structure on $A$ by requiring that $\varphi(\epsilon)=q^{p-1}\frac{q^{p^{\alpha+1}}-1}{q^{p^\alpha}-1}\epsilon$. Actually, as $(p,q^{p^\alpha}-1)$ forms a regular sequence in $R$, one see that $$\varphi(\epsilon)-\epsilon^p=\frac{q^{p-1}\epsilon}{q^{p^\alpha}-1}(q^{p^{\alpha+1}}-1-(q^{p^\alpha}-1)^p)$$
is zero modulo $p$, hence $\varphi$ is indeed a lift of the Fribenius modulo $p$.
    
\end{remark}
\begin{proposition}\label{lemT.construct b}
    For $R=A[\epsilon]/(\epsilon^2-q(q^{p^\alpha}-1)\epsilon)=A\oplus \epsilon A$ as in the previous lemma, there exists a unique $b$ in $W(R)^{\times}$ such that the following holds:
    \begin{itemize}
        \item $\Tilde{g}([p]_{q^{p^\alpha}})=\Tilde{f}([p]_{q^{p^\alpha}})\cdot b$, where $\Tilde{g}$ is the unique $\delta$-ring map such that the following diagram commutes:
    \[\begin{tikzcd}[cramped, sep=large]
& W(R) \arrow[d, "p_0"]\\ 
A \arrow[ru, "\Tilde{g}"] \arrow[r, "\psi"]& R
\end{tikzcd}\]
and $\Tilde{f}$ is defined similarly with the bottom line in the above diagram replaced with the canonical embedding $A\hookrightarrow R$.
    \end{itemize}
\end{proposition}
\begin{proof}
    Notice that the $W(k)$-linear algebra morphism $\Tilde{g}$ is uniquely characterized by the following two properties:
\begin{itemize}
    \item $p_0(\Tilde{g}(q))=q+[p]_{q^{p^\alpha}}$.
    \item $\varphi(\Tilde{g}(q))=\Tilde{g}(\varphi(q))$.
\end{itemize}

Now we wish to construct $b=(b_0,b_1,\ldots)$ such that $\Tilde{g}([p]_{q^{p^\alpha}})=\Tilde{f}([p]_{q^{p^\alpha}})\cdot b$. As $R$ is $p$-torsion free, the ghost map is injective, hence this identity is equivalent to that 
\begin{align}\label{equa.ghost identity}
    \forall n\geq 0, w_n(\Tilde{g}([p]_{q^{p^\alpha}}))=w_n(b)\cdot w_n(\Tilde{f}([p]_{q^{p^\alpha}})).
\end{align}
Here $w_n$ denotes the $n$-th ghost map.

We first make the three terms showing up in \cref{equa.ghost identity} more explicit. In the following, we denote $[p]_{q^{p^\alpha}}$ as $d$ for simplicity. Notice that 
\begin{equation*}
    w_n(\Tilde{f}(d))=w_0(\varphi^n(\Tilde{f}(d)))=w_0(\Tilde{f}(\varphi^n(d)))=\varphi^n(d),
\end{equation*}
Here the second equality follows as $\Tilde{f}$ commutes with $\varphi$. Similarly, using the formula that $\psi(q^k)=q^k+\epsilon q^{k-1}[p]_{q^{p^\alpha}} [k]_{q^{p^{\alpha+1}}}$ obtained in \cref{lemt.extend delta}, for $n\geq 1$, we have that 
\begin{equation}\label{equa.dq cal}
\begin{split}
    w_n(\Tilde{g}((d))&=w_0(g(\varphi^n(d)))=\psi([p]_{q^{p^{n+\alpha}}})=\varphi^n(d)+\epsilon [p]_{q^{p^\alpha}} \sum_{i=0}^{p-1}q^{ip^{n+\alpha}-1}[ip^{n+\alpha}]_{q^{p^{\alpha+1}}}
    \\&=\varphi^n(d)+\epsilon [p]_{q^{p^\alpha}} [p]_{q^{p^{\alpha+n}}} \sum_{i=0}^{p-1}q^{ip^{n+\alpha}-1}[ip^{\alpha}]_{q^{p^{\alpha+n+1}}}\cdot [p^{n-1}]_{q^{p^{\alpha+1}}}
    \\&=\varphi^n(d)(1+\epsilon [p]_{q^{p^\alpha}} [p^{n-1}]_{q^{p^{\alpha+1}}} \sum_{i=0}^{p-1}q^{ip^{n+\alpha}-1}[ip^{\alpha}]_{q^{p^{\alpha+n+1}}})
    \\&=\varphi^n(d)(1+\epsilon [p^n]_{q^{p^\alpha}} \sum_{i=0}^{p-1}q^{ip^{n+\alpha}-1}[ip^{\alpha}]_{q^{p^{\alpha+n+1}}}) .
\end{split}
\end{equation}
Here the second line holds by \cref{lem.units sn}. Then a direct calculation shows that $w_n(\Tilde{g}((d))=\varphi^n(d)(1+\epsilon [p^n]_{q^{p^\alpha}} \sum_{i=0}^{p-1}q^{ip^{n+\alpha}-1}[ip^{\alpha}]_{q^{p^{\alpha+n+1}}})$ also holds for $n=0$.

Define $r_n=1+\epsilon [p^n]_{q^{p^\alpha}} \sum_{i=0}^{p-1}q^{ip^{n+\alpha}-1}[ip^{\alpha}]_{q^{p^{\alpha+n+1}}}$. As $R$ is $\varphi^n(d)$-torsion free for all $n$ (for example, see \cite[Lemma 2.1.7]{anschutz2023prismatic}), \cref{equa.ghost identity} holds if and only if the ghost coordinates of $b$ is precisely given by $(r_0, r_1,\cdots)$. In other words, we aim to show that $(r_0, r_1,\cdots)\in R^{\mathbb{N}}$ is the ghost coordinate for some element $b\in W(R)$ (such $b$ is unique if exists as ghost map is injective). As $R$ is a $\delta$-ring thanks to \cref{rem. delta ring structure on R}, by invoking Dwork's lemma (see \cite[4.6 on page 213]{lazard2006commutative} for details), it suffices to show that for any $n\geq 0$,
$$p^{n+1}|\varphi(r_n)-r_{n+1}.$$

Given our specific formulas for $r_n$, we just need to check that for any $1\leq i\leq p-1$, $$\varphi(\epsilon [p^n]_{q^{p^\alpha}} q^{ip^{n+\alpha}-1}[ip^{\alpha}]_{q^{p^{\alpha+n+1}}})-\epsilon [p^{n+1}]_{q^{p^\alpha}} q^{ip^{n+\alpha+1}-1}[ip^{\alpha}]_{q^{p^{\alpha+n+2}}}$$
is divisible by $p^{n+1}$, but actually it vanishes:
\begin{equation*}
    \begin{split}
        &\varphi(\epsilon [p^n]_{q^{p^\alpha}} q^{ip^{n+\alpha}-1}[ip^{\alpha}]_{q^{p^{\alpha+n+1}}})-\epsilon [p^{n+1}]_{q^{p^\alpha}} q^{ip^{n+\alpha+1}-1}[ip^{\alpha}]_{q^{p^{\alpha+n+2}}}
        \\=&\epsilon q^{p-1}\frac{q^{p^{\alpha+1}}-1}{q^{p^\alpha}-1} \frac{q^{p^{\alpha+n+1}}-1}{q^{p^{\alpha+1}}-1}q^{ip^{n+\alpha+1}-p}[ip^\alpha]_{q^{p^{\alpha+n+2}}}-\epsilon \frac{q^{p^{n+\alpha+1}}-1}{q^{p^\alpha}-1} q^{ip^{n+\alpha+1}-1}[ip^{\alpha}]_{q^{p^{\alpha+n+2}}}
        \\=&0.
    \end{split}
\end{equation*}
\end{proof}
\begin{remark}\label{rem.first strenthern}
    The above proof implies that $b_0=1+\epsilon \sum_{i=1}^{p-1} q^{ip^{\alpha}-1}[ip^\alpha]_{q^{p^{\alpha+1}}}$. In particular, the image of it after modulo $d$ is $1+\epsilon \sum_{i=1}^{p-1} ip^\alpha q^{ip^{\alpha}-1}$, which will be used later.
\end{remark}
The following lemma is used in the above proof:
\begin{lemma}\label{lem.units sn}
    Keep notations as in the above lemma. For any $n\geq 1, i\geq 0$, $$[ip^{n+\alpha}]_{q^{p^{\alpha+1}}}=[p]_{q^{p^{\alpha+n}}}\cdot [ip^{\alpha}]_{q^{p^{\alpha+n+1}}}\cdot [p^{n-1}]_{q^{p^{\alpha+1}}}.$$
\end{lemma}
\begin{proof}
    This follows from the observation that
    \begin{equation*}
        \begin{split}
            [ip^{n+\alpha}]_{q^{p^{\alpha+1}}}&=\frac{(q^{p^{\alpha+1}})^{ip^{n+\alpha}}-1}{q^{p^{\alpha+1}}-1}=\frac{q^{ip^{n+1+2\alpha}}-1}{q^{p^{\alpha+n+1}}-1}\cdot \frac{q^{p^{\alpha+n+1}}-1}{q^{p^{n+\alpha}}-1}\cdot \frac{q^{p^{n+\alpha}}-1}{q^{p^{\alpha+1}}-1}\\&=[ip^{\alpha}]_{q^{p^{\alpha+n+1}}}\cdot [p]_{q^{p^{\alpha+n}}} \cdot [p^{n-1}]_{q^{p^{\alpha+1}}}.
        \end{split}
    \end{equation*}
\end{proof}
Next we state a result which will be used together with \cref{lemT.construct b} to construct the desired Sen operator.
\begin{proposition}\label{lemt.construct homotopy}
    Keep notations as in \cref{lemT.construct b}, there exists a unique $c$ in $W(R)$ such that $$\Tilde{g}(q)-\Tilde{f}(q)=\Tilde{f}(d)\cdot c.$$
\end{proposition}
\begin{proof}
    We wish to construct $c=(c_0,c_1,\ldots)$ such that $\Tilde{g}(q)-\Tilde{f}(q)=\Tilde{f}(d)\cdot c$. As $R$ is $p$-torsion free, the ghost map is injective, hence this identity is equivalent to that 
\begin{align}\label{equa for c.ghost identity}
    \forall n\geq 0, w_n(\Tilde{g}(q))-w_n(\Tilde{f}(q))=w_n(c)\cdot w_n(\Tilde{f}(d)),
\end{align}
where $w_n$ denotes the $n$-th ghost map.
Notice that
\begin{equation*}
    \begin{split}
    &w_n(\Tilde{f}(q))=w_0(\varphi^n(\Tilde{f}(q)))=w_0(\Tilde{f}(\varphi^n(q)))=q^{p^n},
        \\&w_n(\Tilde{g}(q))=w_0(\varphi^n(\Tilde{g}(q)))=w_0(\Tilde{g}(\varphi^n(q)))=\psi(q^{p^n})=q^{p^n}+\epsilon q^{p^n-1}[p]_{q^{p^\alpha}} [p^n]_{q^{p^{\alpha+1}}},
        \\& w_n(\Tilde{f}(d))=\varphi^n([p]_{q^{p^\alpha}})=[p]_{q^{p^{n+\alpha}}}.
    \end{split}
\end{equation*}
Consequently,
\begin{equation*}
    \begin{split}
        &w_n(\Tilde{g}(q))-w_n(\Tilde{f}(q))=\epsilon q^{p^n-1}[p]_{q^{p^\alpha}} [p^n]_{q^{p^{\alpha+1}}}
        \\=&\epsilon q^{p^n-1} \cdot \frac{q^{p^{\alpha+1}}-1}{q^{p^{\alpha}}-1}\cdot \frac{q^{p^{\alpha+n+1}}-1}{q^{p^{\alpha+1}}-1}
        \\=& \epsilon q^{p^n-1} \cdot \frac{q^{p^{\alpha+n+1}}-1}{q^{p^{\alpha+n}}-1} \cdot \frac{q^{p^{\alpha+n}}-1}{q^{p^{\alpha}}-1}
        \\=&\epsilon q^{p^n-1} [p]_{q^{p^{n+\alpha}}} \cdot [p^n]_{q^{p^{\alpha}}}
    \end{split}
\end{equation*}
We define $r_n=\epsilon q^{p^n-1} \cdot [p^n]_{q^{p^{\alpha}}}$. As $R$ is $\varphi^n(d)$-torsion free for all $n$, \cref{equa for c.ghost identity} holds if and only if the ghost coordinates of $c$ is precisely given by $(r_0, r_1,\cdots)$. In other words, we aim to show that $(r_0, r_1,\cdots)\in R^{\mathbb{N}}$ is the ghost coordinate for some element $c\in W(R)$ (such $c$ is unique if exists as ghost map is injective). As $R$ is a $\delta$-ring thanks to \cref{rem. delta ring structure on R}, by invoking Dwork's lemma (see \cite[4.6 on page 213]{lazard2006commutative} for details), it suffices to show that for any $n\geq 0$,
$$p^{n+1}|\varphi(r_n)-r_{n+1}.$$
But $\varphi(r_n)-r_{n+1}$ actually vanishes as
\begin{equation*}
    \begin{split}
        &\varphi(r_n)-r_{n+1}=\varphi(\epsilon q^{p^n-1} \cdot [p^n]_{q^{p^{\alpha}}})-\epsilon q^{p^{n+1}-1} \cdot [p^{n+1}]_{q^{p^{\alpha}}}
        \\=&\epsilon\cdot q^{p-1}\cdot \frac{q^{p^{\alpha+1}}-1}{q^{p^\alpha}-1}q^{p^{n+1}-p}\cdot \frac{q^{p^{\alpha+n+1}}-1}{q^{p^{\alpha+1}}-1}-\epsilon q^{p^{n+1}-1} \cdot \frac{q^{p^{\alpha+n+1}}-1}{q^{p^\alpha}-1}
        \\=&0.
    \end{split}
\end{equation*}
Hence we finish the proof.
\end{proof}

Now we are ready to construct the $q$-Higgs connection for quasi-coherent complexes on $X^{\Prism}$, based on the following key proposition.
\begin{proposition}\label{propt.key automorphism of functors}
    The elements $b$ and $c$ constructed in \cref{lemT.construct b} and \cref{lemt.construct homotopy} together induce an isomorphism $\gamma_{b,c}$ between functors $\rho: \Spf(R) \xrightarrow{\iota} \Spf(A) \rightarrow X^{\Prism}$
    and $\rho\circ \psi: \Spf(R) \xrightarrow{\psi} \Spf(A)\rightarrow X^{\Prism}$, i.e. we have the following commutative diagram:
\[\xymatrixcolsep{5pc}\xymatrix{\Spf(R)\ar[d]^{\iota}\ar[r]^{\psi}& \Spf(A)\ar@{=>}[dl]^{\gamma_{b,c}} \ar[d]_{}^{\rho}
\\\Spf(A)  \ar^{\rho}[r]&X^{\Prism}}\]
\end{proposition}
\begin{proof}
      Given a test $(p,d)$-nilpotent $R$-algebra $T$ via the structure morphism $h: R\to T$, we denote the induced morphism $W(R)\to W(T)$ by $\Tilde{h}$. Then $\rho\circ h(T)$ corresponds to the point $$(\alpha:(d) \otimes_{A,\Tilde{h}\circ \Tilde{f}} W(T)\to W(T),\eta: \Cone((d)\to A)\xrightarrow{\Tilde{h}\circ \Tilde{f}} \Cone(\alpha))$$
      in $X^{\Prism}(T)$, while $(\rho\circ \psi)\circ  h(T)$ corresponds to the point $$(\alpha^{\prime}:(d) \otimes_{A,\Tilde{h}\circ \Tilde{g}} W(T)\to W(T), \eta^{\prime}: \Cone((d)\to A) \xrightarrow{\Tilde{h}\circ \Tilde{g}} \Cone(\alpha^{\prime})).$$ 
      
      We need to specify an isomorphism $\gamma_{b}: \alpha^{\prime}\xrightarrow{\simeq} \alpha$ as well as a homotopy $\gamma_{c}$ between $\gamma_{b}\circ \eta^{\prime}$ and $\eta$ which are both functorial in $T$.
      
      Utilizing \cref{lemT.construct b},  we construct the desired isomorphism $\gamma_{b}$ as follows: 
     \[\xymatrixcolsep{5pc}\xymatrix{(d) \otimes_{A,\Tilde{h}\circ \Tilde{g}} W(T) \ar[d]^{(d)\otimes x\mapsto (d)\otimes \Tilde{h}(b)x}\ar[r]^{\iota}& W(T) \ar[d]_{}^{\Id}
\\(d) \otimes_{A,\Tilde{h}\circ \Tilde{f}} W(T)\ar[r]^{\iota}&W(T)}\]
Here the left vertical map is $W(T)$-linear and the commutativity of the diagram follows from \cref{lemT.construct b}.

Then we draw a diagram illustrating $\gamma_{b}\circ \eta^{\prime}$ and $\eta$ (as maps of quasi-ideals):
\[\xymatrixcolsep{5pc}\xymatrix{(d) \ar@<-.5ex>[d]_{\gamma_{b}\circ \eta^{\prime}} \ar@<.5ex>[d]^{\eta}\ar[r]^{\iota}& A \ar[d]_{\gamma_{b}\circ \eta^{\prime}}
 \ar@<-.5ex>[d]^{\eta}
\\(d) \otimes_{\mathfrak{S},\Tilde{h}\circ \Tilde{f}} W(T)\ar[r]^{\iota}&W(T)}\]
Here the two left vertical maps are given by $\gamma_{b}\circ \eta^{\prime}: x\cdot (d)\mapsto (d)\otimes \Tilde{h}(b)\Tilde{h}(\Tilde{g}(x))$ and $\eta: x\cdot (d)\mapsto (d)\otimes \Tilde{h}\circ \Tilde{f}(x)$, the two right vertical maps are given by $\gamma_{b}\circ \eta^{\prime}$ and $\eta=\Tilde{h}\circ \Tilde{f}$.

To construct the desired homotopy, we need to specify a map $\gamma_c: A \to (d) \otimes_{\mathfrak{S},\Tilde{h}\circ \Tilde{f}} W(T)$ such that $\iota\circ \gamma_c=\gamma_{b}\circ \eta^{\prime}-\eta$ and that $\gamma_c \circ \iota =\gamma_{b}\circ \eta^{\prime}-\eta$. Without loss of generality, we assume $T=R$. In this case, $\iota: (d) \otimes_{A, \Tilde{f}} W(R) \to W(R)$ is injective as $W(R)$ is $d$-torsion free by the uniqueness in \cref{lemT.construct b} (otherwise if there exists a nontrivial $d$-torsion $q$, $b+q\neq b$ is another objects in $W(R)$ satisfying $(b+q)\cdot d=d$, a contradiction with the uniqueness of $b$). Consequently, it suffices to construct $\gamma_c$ and check that $\iota\circ \gamma_c=\gamma_{b}\circ \eta^{\prime}-\eta$.

Inspired by \cref{lemt.construct homotopy}, we just define $\gamma_c(q)$ to be $d\otimes_{\mathfrak{S}, \Tilde{f}} c$. For a general $s(q)\in A$, denote $k_s(u,v) \in W(k)[[u,v]]$ to be the unique power series such that $s(u)-s(v)=(u-v)\cdot k_s(u,v)$. Then 
\[(\gamma_{b}\circ \eta^{\prime}-\eta)(s(u))=\Tilde{g}(s(u))-\Tilde{f}(s(u))=s(\Tilde{g}(u))-s(\Tilde{f}(u))=(\Tilde{g}(u)-\Tilde{f}(u))\cdot k_s=\Tilde{f}(d)c\cdot k_s\]
for $k_s=k_s(\Tilde{g}(u), \Tilde{f}(u))\in W(R)$. Here the last quality follows from our construction of $c$ in \cref{lemt.construct homotopy}.

Hence if we define 
\[\gamma_c(s(q))=(d)\otimes_{A, \Tilde{f}} c k_s,\]

Then the desired identity $\iota\circ \gamma_c=\gamma_{b}\circ \eta^{\prime}-\eta$ follows.

Finally it is clear that $\gamma_b$ and $\gamma_c$ are all constructed via a base change from $W(R)$ to $W(T)$, hence they are all natural in $T$. We win.
\end{proof}

When restricted to the locus where $(t)^n=0$ inside $[\Spf(\mathbb{Z}[[ t]])/ \mathbf{G}_m]$, we obtain the following truncated version of \cref{propt.key automorphism of functors}.
\begin{corollary}\label{cort.truncated}
    The $\gamma_{b,c}$ constructed in \cref{propt.key automorphism of functors} induces an isomorphism between functors $\rho: \Spf(R/d^n) \to \Spf(A/d^n) \rightarrow X_n^{\Prism}$
    and $\rho\circ \psi: \Spf(R/d^n) \xrightarrow{\psi} \Spf(A/d^n)\rightarrow X_n^{\Prism}$, i.e. we have the following commutative diagram:
\[\xymatrixcolsep{5pc}\xymatrix{\Spf(R/d^n)\ar[d]^{\iota}\ar[r]^{\psi}& \Spf(A/d^n)\ar@{=>}[dl]^{\gamma_{b,c}} \ar[d]_{}^{\rho}
\\ \Spf(A/d^n)  \ar[r]^{\rho}&X_n^{\Prism}}\]
\end{corollary}
Let $n\in \mathbb{N}$. Now we are ready to construct a $q$-Higgs derivation on $\rho^*\mathscr{E}$ for $\mathscr{E}\in \Qcoh(X^{\Prism})$ (resp. $ \Qcoh(X_n^{\Prism})$). Based on \cref{propt.key automorphism of functors} (resp. \cref{cort.truncated}), we have an isomorphism $\gamma_{b,c}: \rho\circ \psi \stackrel{\simeq}{\longrightarrow} \rho$. Consequently, for $\mathscr{E}\in \Qcoh(X^{\Prism})$ (resp. $ \Qcoh(X_n^{\Prism})$), we have an isomorphism
\[\gamma_{b,c}: \psi^{*}\rho^{*}(\mathscr{E}) \stackrel{\simeq}{\longrightarrow} \rho^{*}\mathscr{E}\otimes_A R.\]
Unwinding the definitions, this could be identified with a $\psi$-linear morphism 
\begin{align}\label{equat.truncated induced morphism}
    \gamma_{b,c}: \rho^{*}(\mathscr{E})\rightarrow \rho^{*}(\mathscr{E})\otimes_A R.
\end{align} 
Moreover, our definition of the element $b$ and $c$ in \cref{lemT.construct b} and \cref{lemt.construct homotopy} implies that $\gamma_{b,c}$ in \cref{equat.truncated induced morphism} reduces to the identity modulo $\epsilon$, hence could be written as $\Id+\epsilon\partial_{\mathscr{E}}$ for some operator $\partial_{\mathscr{E}}: \rho^{*}\mathscr{E} \rightarrow \rho^{*}\mathscr{E}$ (as $R=A[\epsilon]/(\epsilon^2-q(q^{p^\alpha}-1)\epsilon)=A\oplus A\epsilon$), which we will refer to as the \textit{$q$-connection} on the complex $\rho^*\mathscr{E}$.  

Next we explain how the pullback along $\rho$ induces a functor from the category of quasi-coherent complexes on $X^{\Prism}$ to the derived category of modules over a certain non-commutative ring.

First we observe that given a pair $(M,\gamma_M)$ such that $M$ is a discrete $A$-module and $\gamma_M: \psi^*M \stackrel{\simeq}{\rightarrow} M\otimes_{A,\iota} R$ is an isomorphism reducing to the identify after modulo $\epsilon$, then we can extract a monodromy operator on $M$ satisfying a twisted Leibnitz rule.
\begin{lemma}\label{lem.leibniz}
   Given a pair $(M,\gamma_M)$ as above, we can write $\gamma_M$ as $\Id+\epsilon \partial_M$ when restricted to $M$ for a operator $\partial_M: M\to M$, then we have that $$\partial_M(ax)=a\partial_M(x)+\partial_A(a)x+q(q^{p^\alpha}-1) \partial_A(a)\partial_M(x)=\gamma_A(a)\partial_M(x)+\partial_A(a)x$$ for $a\in A$ and $x\in M$. Here $\partial_A: A\to A$ (defined in \cref{rem.relation with automorphism}) is $W(k)$-linear and sends $q^i$ to $[pi]_{q^{p^\alpha}} q^{i-1}$. In particular, 
   $$\partial_M(qx)=[p]_{q^{p^\alpha}}x+q^{p^{\alpha+1}+1}\partial_M(x).$$
\end{lemma}
\begin{proof}
    The characterization of $\partial$ on $A$ follows as $\gamma$ is $\psi$-linear and $\psi (q^i)=q^i+\epsilon [pi]_{q^{p^\alpha}} q^{i-1}$. For a general pair $(M,\gamma_M)$,
    \begin{equation*}
        \begin{split}
            &\gamma_M(ax)=\psi(a)\gamma_M(x)=(a+\epsilon\partial(a))(x+\epsilon\partial_M(x))\\=&ax+\epsilon(a\partial_M(x)+\partial(a)x)+\epsilon^2 \partial(a)\partial_M(x)\\=&ax+\epsilon(a\partial_M(x)+\partial(a)x)+\epsilon q(q^{p^\alpha}-1) \partial(a)\partial_M(x)
            \\=&ax+\epsilon(a\partial_M(x)+\partial(a)x+q(q^{p^\alpha}-1) \partial(a)\partial_M(x)).
        \end{split}
    \end{equation*}
    This implies the desired result.
\end{proof}
\begin{remark}\label{rem. preserve d filtatration}
    $\partial_M$ preserves the $d$-adic filtration, i.e. $\partial(dM)\subseteq dM$ as $\partial(d)\in d A$.
\end{remark}
\begin{remark}\label{rem.twisted leibnitz in general}
    More generally, let $\mathscr{E}$ and $\mathscr{E}^{\prime}$ be quasi-coherent complexes on $X_n^{\Prism}$, and let $\mathscr{E}\otimes \mathscr{E}^{\prime}$ denote their derived tensor product. Then the same argument shows that the $q$-connection $\partial_{\mathscr{E}\otimes \mathscr{E}^{\prime}}$ on $\rho^*(\mathscr{E}\otimes \mathscr{E}^{\prime})\cong \rho^*(\mathscr{E}) \otimes \rho^*(\mathscr{E}^{\prime})$ can be identified with $\partial_{\mathscr{E}}\otimes\Id_{\mathscr{E}^{\prime}}+\Id_{\mathscr{E}}\otimes \partial_{\mathscr{E}^{\prime}}+q(q^{p^{\alpha}}-1)\partial_{\mathscr{E}}\otimes \partial_{\mathscr{E}^{\prime}}$.
\end{remark}

Motivated by \cref{lem.leibniz}, the following skew polynomial enters the picture.
\begin{definition}\label{def.skew polynomial}
    Let $\gamma_A: A\to A$ be the ring automorphism defined in \cref{rem.relation with automorphism}, then $\partial_A: A\to A$ introduced in \cref{rem.relation with automorphism} is a $\gamma_A$-derivation of $A$, i.e. $\partial_A(x_1x_2)=\gamma_A(x_1)\partial_A(x_2)+\partial_A(x_1)x_2$. We define the \textit{Ore extension} $A[\partial; \gamma_A, \partial_A]$ to be the noncommutative ring obtained by giving the ring of polynomials $A[\partial]$ a new multiplication law, subject to the identity
    $$\partial r=\gamma_A(r)\partial+\partial_A(r), ~~~\forall r\in A.$$
\end{definition}
\begin{remark}\label{rem. mod d ore extension}
    For any $n\geq 1$, as $\gamma_A$ (resp. $\partial_A$) induces a ring automorphism $A/d^n \to A/d^n$ (resp. a $\gamma_A$-derivation of $A/d^n$), we can similarly define $A/d^n[\partial; \gamma_A, \partial_A]$\footnote{We want to point out that by \cref{rem. preserve d filtatration}, $d^n$ is a two sided ideal inside $A[\partial; \gamma_A, \partial_A]$ and $A/d^n[\partial; \gamma_A, \partial_A]$ defined in this way is actually isomorphic to $A[\partial; \gamma_A, \partial_A]/d^n$.}. Note that one $n=1$, $A/d[\partial; \gamma_A, \partial_A]$ is just the usual (commutative) polynomial ring $A/d[\partial]$ as $\gamma_A$ reduces to the identity after modulo $d$ and $\partial_A$ vanishes on $A/d$.
\end{remark}
\begin{remark}\label{rem. cohomology of the structure sheaf}
   One can check that $A$ is isomorphic to the quotient of $A[\partial; \gamma_A, \partial_A]$ by the left ideal generated by $\partial$. More precisely, there is a canonical resolution for $A$ by finite free $A[\partial; \gamma_A, \partial_A]$-modules:
   $$[A[\partial; \gamma_A, \partial_A] \stackrel{f\mapsto f\cdot \partial}{\longrightarrow} A[\partial; \gamma_A, \partial_A]]\simeq A.$$
   Consequently, for $M\in \mathcal{D}(A[\partial; \gamma_A, \partial_A])$, 
   $$\RHom_{\mathcal{D}(A[\partial; \gamma_A, \partial_A])}(A, M)\simeq[M\stackrel{ \partial}{\longrightarrow} M].$$
   Such an observation will be used to calculate the cohomology later.
\end{remark}
For the purpose of upgrading the pullback along $\rho$ to a functor from $\mathcal{D}(X^{\Prism})$ to $\mathcal{D}(A[\partial; \gamma_A, \partial_A])$, the last missing piece is to show that $\mathcal{D}(X^{\Prism})$ is equivalent to the derived category of its heart. We first construct a $t$-structure on $\mathcal{D}(X_n^{\Prism})$ following \cite[Section 2.1]{guo2023frobenius}, where they study $t$-structures on perfect complexes of prismatic crystals.
\begin{construction}\label{construct. t structure}
    Let $n\in \mathbb{N}\cup \{\infty\}$ and $(A,I)$ be the $q$-prism as usual. We denote the Cech nerve of this cover by $\Spf(A^{\bullet})$. By \cref{prop. recollect}, $\mathcal{D}(X_{n}^{\Prism}) \xrightarrow{\sim} \lim _{(A, I)\in X_{\Prism}} \widehat{\mathcal{D}}(A/I^n)$, which is then equivalent to the cosimplicial limit of the derived category of quasi-coherent complexes on $\Spf(A^{\bullet}/I^n)$ as the latter forms a cofinal system. we define a $t$-sturcture on $\mathcal{D}(X_{n}^{\Prism})$ by requiring that $\mathcal{F}\in \mathcal{D}(X_{n}^{\Prism})$ lives in $\leq 0$ part (resp. $\geq 0$ part) if the underlying complexes on $\Spf(A^{\bullet})$ is concentrated on degree $\leq 0$ (resp. $\geq 0$). We call this the \textit{standard $t$-structure}, which deserves the name by the following proposition.
\end{construction}
\begin{proposition}\label{prop. t independ}
    Construction \ref{construct. t structure} defines a $t$-structure on $\mathcal{D}_{}(X_{n}^{\Prism})$. Moreover, given any $(B,J)\in X_{\Prism}$ such that $\Spf(B/J)\to X$ is $p$-adically flat, the pullback functor functor $\mathcal{D}(X_{n}^{\Prism})\to \mathcal{D}(B/J^n)$ is $t$-exact.
\end{proposition}
\begin{proof}
    The proof of \cite[Lemma 2.9, Proposition 2.11]{guo2023frobenius} works verbatim here. Note that the $q$-prism $(A,I)$ belongs to the category they denote as ``Breuil-Kisin prisms" in \cite[Definition 2.4]{guo2023frobenius}.
\end{proof}
\begin{remark}
    \cref{prop. t independ} tells us being in the heart of the $t$-structure with respect to our chosen family $\Spf(A^{\bullet})$ guarantees that its evaluation at any other prism $(B,J)\in X_{}\Prism$ covering the final object is concentrated on degree $0$ as well, hence is independent of our chosen cover $\Spf(A^{\bullet})$ (we can run Construction \ref{construct. t structure} and \cref{prop. t independ} to any chosen cover).
\end{remark}

From now on we denote the heart of $\mathcal{D}_{}(X_{n}^{\Prism})$ with respect to the standard $t$ structure as $\mathcal{C}_n$. Note that for $m<n$, there is a natural functor $i_{m,n}: \mathcal{C}_m \to \mathcal{C}_n$ which admits a left adjoint functor given by tensoring with the structure sheaf $\mathcal{O}_{X_{n-1}^{\Prism}}$. It induces a derived functor $Li_{m,n}^*: \mathcal{D}(\mathcal{C}_n)\to \mathcal{D}(\mathcal{C}_m)$, which is compatible with the natural functors $r_n: \mathcal{D}(\mathcal{C}_n) \to \mathcal{D}(X_n^{\Prism})$ as $r_n$ preserves projective objects.

\begin{proposition}\label{prop. heart generator}
For $n\in \mathbb{N}\cup \{\infty\}$, the the natural functor $r_n: \mathcal{D}(\mathcal{C}_n) \to \mathcal{D}(X_n^{\Prism})$ is an equivalence.
\end{proposition}
\begin{proof}
It suffices to show full faithfulness as the essential surjectivity will then follow immediately from \cref{propt.generation}. For this purpose, we first explain how to reduce to the Hodge-Tate case by doing induction on $n$. Suppose we have proved the statement up to $n-1$ ($n\geq 2$). Then we wish to show that for any $\mathcal{F}, \mathcal{G}\in \mathcal{D}(\mathcal{C}_n)$,  
    \begin{equation}\label{equa.induction heart}
        \RHom_{\mathcal{D}(\mathcal{C}_n)}(\mathcal{F}, \mathcal{G})\stackrel{\simeq}{\longrightarrow} \RHom_{\mathcal{D}(X_n^{\Prism})}(r_n(\mathcal{F}), r_n(\mathcal{G})).
    \end{equation}
    For this purpose, first we notice that there is a fiber sequence (see \cite[Proposition 3.15]{liu2024prismatization} for details)
    \begin{equation}\label{induction eqaution}
        \mathcal{G}\otimes ((i_{n-1,n})_{*}\mathcal{O}_{X^{\Prism}_{n-1}}\{1\})\rightarrow \mathcal{G}\otimes \mathcal{O}_{X^{\Prism}_{n}} \rightarrow \mathcal{G}\otimes ((i_{1,n})_{*}\mathcal{O}_{\mathrm{WCart}^{\HT}}).
    \end{equation}
    Then we observe that 
    \begin{equation*}
        \begin{split}
            &\RHom_{\mathcal{D}(\mathcal{C}_n)}(\mathcal{F}, \mathcal{G}\otimes ((i_{n-1,n})_{*}\mathcal{O}_{X^{\Prism}_{n-1}}\{1\}))=\RHom_{\mathcal{D}(\mathcal{C}_n)}(\mathcal{F}, (i_{n-1,n})_{*}L(i_{n-1,n})^{*}\mathcal{G}\{1\})\\=&\RHom_{\mathcal{D}(\mathcal{C}_{n-1})}(Li_{n-1,n}^{*}\mathcal{F}, Li_{n-1,n}^{*}\mathcal{G}\{1\})=\RHom_{\mathcal{D}(X^{\Prism}_{n-1})}(r_{n-1}(Li_{n-1,n}^{*}\mathcal{F}), r_{n-1}(Li_{n-1,n}^{*}\mathcal{G}\{1\})),
        \end{split}
    \end{equation*}
   where the first identity is due to the projection formula, the second equation follows from adjunction and the last identity holds by induction. A similar calculation for $\RHom_{\mathcal{D}(\mathcal{C}_n)}(\mathcal{F}, \mathcal{G}\otimes ((i_{1,n})_{*}\mathcal{O}_{\mathrm{WCart}^{\HT}}))$ implies the desired \cref{equa.induction heart}.

    For $n=\infty$, we observe that any $\mathcal{G}\in \mathcal{D}(\mathcal{C}_{\infty})$ is complete with respect to the Hodge-Tate ideal sheaf $\mathcal{I}$, hence $\mathcal{G}=\varprojlim \mathcal{G}/\mathcal{I}^n \mathcal{G}$, which implies that for any $\mathcal{F} \in \mathcal{D}(\mathcal{C}_{\infty})$, $$\RHom_{\mathcal{D}(\mathcal{C}_{\infty})}(\mathcal{F}, \mathcal{G})=\RHom_{\mathcal{D}(\mathcal{C}_{\infty})}(\mathcal{F}, \varprojlim\mathcal{G}/\mathcal{I}^n \mathcal{G})=\varprojlim \RHom_{\mathcal{D}(\mathcal{C}_{n})}(\mathcal{F}/\mathcal{I}^n \mathcal{F}, \mathcal{G}/\mathcal{I}^n \mathcal{F}).$$
    One then similarly calculates $\RHom_{\mathcal{D}(X_{\infty}^{\Prism})}(r_{\infty}(\mathcal{F}), r_{\infty}(\mathcal{G}))$ and the desired results follows from induction.
    
    To prove the statement for $X^{\HT}$, a similar reduction process as above implies that it suffices to work with $X^{\HT}\otimes k$ instead. But as $X^{\HT}$ is the classifying stack of $G_{\pi}$ over $X$ by \cite[Proposition 9.5]{bhatt2022prismatization} (see \cite[Example 9.6]{bhatt2022prismatization} for an explicit description of $G_{\pi}$), hence $X^{\HT}\otimes k=BG_{\pi,k}$. Then the proof of \cite[Theorem 2.5]{anschutz2022v} implies that the pullback along the covering map $\Spec(k)\to X^{\HT}\otimes k$ identifies  $\mathcal{D}(X^{\HT}\otimes k)$ with $\mathcal{D}_{\mathcal{B}}(k[\theta])$, where $\mathcal{B}$ is the full subcategory of the category of (discrete) $k[\theta]$-modules consisting of those objects which are $\theta$-power torsion (see \cite[\href{https://stacks.math.columbia.edu/tag/05E6}{Tag 05E6}]{stacks-project} for the precise definition) and $\mathcal{D}_{\mathcal{B}}(k[\theta])$ is the full subcategory of $\mathcal{D}(k[\theta])$ whose objects are those $M\in \mathcal{D}(k[\theta])$ such that $\mathrm{H}^n(M)\in \mathcal{B}$. Moreover, under such an identification $\mathcal{D}(X^{\HT}\otimes k)^{\heartsuit}$ is matched with $\mathcal{B}$. Hence for our purpose, it is enough to check that the natural map $\mathcal{D}(\mathcal{B})\to \mathcal{D}_{\mathcal{B}}(k[\theta])$ is an equivalence, which follows from \cite[\href{https://stacks.math.columbia.edu/tag/0955}{Tag 0955}]{stacks-project}.
\end{proof}

Given \cref{prop. heart generator}, the discussion in this subsection so far can be summarized as follows:
\begin{proposition}\label{prop. induce functor}
  For $n\in \mathbb{N}\cup \{\infty\}$, the pullback along $\rho: \Spf(A/d^n)\to X_n^{\Prism}$ induces a functor 
  \begin{equation*}
  \begin{split}
      \mathcal{D}(X_n^{\Prism}) &\to \mathcal{D}(A/d^n[\partial; \gamma_A, \partial_A])\\
      \mathcal{E} &\mapsto (\rho^{*}\mathcal{E}, \partial_{\mathcal{E}})
  \end{split}
  \end{equation*}
  which will be denoted as $\beta_n^+$ later. Here $\partial_{\mathcal{E}}$ is defined after \cref{cort.truncated}.
\end{proposition}
\begin{proof}
    By the discussion above \cref{lem.leibniz}, for $\mathcal{E}\in \mathcal{D}(X_n^{\Prism})^{\heartsuit}$, $\rho^*\mathcal{E}$ is an $A$-module concentrated on degree $0$ equipped with an operator $\partial_{\mathcal{E}}: \rho^*\mathcal{E}\to \rho^*\mathcal{E}$. Moreover, $\partial_{\mathcal{E}}$ satisfies the Leibnitz rule stated in \cref{lem.leibniz}, hence $\rho^*\mathcal{E}$ can be viewed as a (left) $A[\partial; \gamma_A, \partial_A]$-module, whose underlying module is $\rho^*\mathcal{E}$ with $\partial$ acting via $\partial_{\mathcal{E}}$. Consequently it induces a functor $\mathcal{D}(\mathcal{D}(X_n^{\Prism})^{\heartsuit})\to \mathcal{D}(A[\partial; \gamma_A, \partial_A])$. By \cref{prop. heart generator} and functoriality, it coincides with the functor stated in the proposition. 
\end{proof}


\begin{remark}\label{lem.sen operator on pullback}
    Recall that in \cite{bhatt2022absolute}, the Sen operator for complexes on the Hodge-Tate locus is defined on the complex itself, while our $q$-connection is defined on $\rho^{*}(\mathscr{E})$ other than on $\mathscr{E}$ itself, this is because our construction of the isomorphism $\gamma_{b,c}$ relies on certain coordinates $q$ in the $q$-prism when $n>1$. This is a feature but not a bug as we expect that the $q$-connection $\partial$ for the structure sheaf should act nontrivially on $q$.
\end{remark}
Let $n\in \mathbb{N}\cup \{\infty\}$. Recall that for $\mathcal{E}\in \mathcal{D}(X_n^{\Prism})$, the global section of $\mathcal{E}$ is defined as 
\begin{equation*}
    \mathrm{R} \Gamma(X_n^{\Prism}, \mathcal{E}):=\lim_{f: \Spec(R)\to X_n^{\Prism}} f^{*}\mathcal{E}
\end{equation*}

In particular, the cover $\rho: \Spf(A/I^n) \rightarrow X_n^{\Prism}$ induces a natural morphism $$\mathrm{R} \Gamma(X_n^{\Prism}, \mathcal{E})\to \rho^{*}\mathcal{E}.$$

The next proposition shows that it actually factors through the fiber of $\partial_{\mathcal{E}}$.
\begin{proposition}\label{propt. factor through fiber}
    For any $\mathcal{E}\in \mathcal{D}(X_n^{\Prism})$, the natural morphism $\mathrm{R} \Gamma(X_n^{\Prism}, \mathcal{E})\to \rho^{*}\mathcal{E}$ induces a canonical morphism $$
    \mathrm{R} \Gamma(X_n^{\Prism}, \mathcal{E})\to \fib(\rho^*\mathcal{E}\stackrel{\partial_{\mathcal{E}}}{\longrightarrow} \rho^*\mathcal{E}).$$
\end{proposition}
\begin{proof}
By \cref{propt.key automorphism of functors} and \cref{cort.truncated}, we have an isomorphism $\gamma_{b,c}: \rho\circ \psi \stackrel{\simeq}{\longrightarrow} \rho$ as functors  $\Spf(R) \rightarrow X_n^{\Prism}$. Then the definition of $ \mathrm{R} \Gamma(X_n^{\Prism}, \mathcal{E})$ implies that the natural morphism $\mathrm{R} \Gamma(X_n^{\Prism}, \mathcal{E})\to \rho^{*}\mathcal{E}$ factors through the equalizer of $$\rho^*\mathcal{E}\stackrel{\Id\otimes 1}{\longrightarrow} 
 \rho^*\mathcal{E}\otimes_{A/d^n} R/d^n$$
and 
$$\rho^*\mathcal{E}\stackrel{\gamma_{b,c}}{\longrightarrow} 
 \rho^*\mathcal{E}\otimes_{A/d^n, \psi} R/d^n,$$
where $\gamma_{b,c}=\Id+\epsilon \partial_{\mathcal{E}}$. This produces a canonical morphism 
\begin{equation*}
    \mathrm{R} \Gamma(X_n^{\Prism}, \mathcal{E})\to \fib(\rho^*\mathcal{E}\stackrel{\partial_{\mathscr{E}}}{\longrightarrow} \rho^*\mathcal{E}).
\end{equation*}
\end{proof}
\begin{remark}
    Eventually we will show the above morphism actually identifies $\mathrm{R} \Gamma(X_n^{\Prism}, \mathcal{E})$ with $\fib(\partial_{\mathcal{E}})$. For that purpose, we need to investigate the behavior of $\partial$ on the Hodge-Tate locus first, which is the task of the next subsection.
\end{remark}
\subsection{$q$-connections on the Hodge-Tate locus}
Now we study the isomorphism $\gamma_{b,c}$ and hence the behavior of the $q$-connection $\partial$ when restricted to the Hodge-Tate locus (i.e. take $n=1$ in \cref{cort.truncated}) in more detail. First notice that $\psi: A/d\to R/d$ coincides with the natural inclusion as $\psi(x)=x, \forall x\in A/d$, we see $\gamma_{b,c}$ descends to an automorphism of
    \[\Spf(\mathcal{O}_K)^{\HT}\times_{\Spf(\mathcal{O}_K)} \Spf(\mathcal{O}_K[\epsilon]/(\epsilon^2-q(q^{p^\alpha}-1)\epsilon)),\]
    for simplicity here and in the following in this subsection we still write $q$ for its image in $\mathcal{O}_K$, which is actually $\zeta_{p^{\alpha+1}}$. This implies that for $\mathcal{E}\in \mathcal{D}(X^{\HT})$, $\partial_{\mathcal{E}}$ also descends to a functor from $\mathcal{E}$ to itself.
    
    As $\Spf(\mathcal{O}_K)^{\HT}$ is the classifying stack of $G_{\pi}$ over $\Spf(\mathcal{O}_K)$ by \cite[Proposition 9.5]{bhatt2022prismatization}, where $G_{\pi}$ is calculated in \cite[Example 9.6]{bhatt2022prismatization}. More explicitly\footnote{Here we emphasize that although the calculation in \cite[Proposition 9.5, Example 9.6]{bhatt2022prismatization} used the Breuil-Kisin prism, one could check by hand that it still works with the $q$-prism instead. In particular, the $E^{\prime}(\pi)$ in \textit{loc. cit.} could be replaced with the derivative of $d=[p]_{q^{p^\alpha}}$ (with respect to $q$), denoted as $d^{\prime}(q)$ later.} 
    \[G_\pi=\{(t,a)\in \mathbb{G}_m^\sharp\ltimes \mathbb{G}_a^\sharp\ |\ t-1=d^{\prime}(q)\cdot a\}\]
Under this identification, $\gamma_{b,c}$ corresponds to an element in $G_\pi(\mathcal{O}_K[\epsilon]/(\epsilon^2-q(q^{p^\alpha}-1)\epsilon))$, and we claim this element is precisely $(1+d^{\prime}(q)\epsilon,\epsilon)$.\footnote{ Here by our definition we see that $\epsilon^k=\epsilon q^{k-1}(q^{p^\alpha}-1)^{k-1}$. In particular, it doesn't vanish for $k\geq 1$, which is a key difference from its analog in the Breuil-Kisin setting studied in \cite{anschutz2022v}. However, the divided powers of $\epsilon$ still exist in $\mathcal{O}_K[\epsilon]/(\epsilon^2-q(q^{p^\alpha}-1)\epsilon)$ as $v_p(q^{p^\alpha}-1)=\frac{1}{p-1}$, hence $v_p(n!)\leq \frac{n-1}{p-1}=v_p(q^{n-1}(q^{p^\alpha}-1)^{n-1})$. Consequently, $\epsilon \in \mathbb{G}_a^\sharp(\mathcal{O}_K[\epsilon]/(\epsilon^2-q(q^{p^\alpha}-1)\epsilon))$.}. 
To see this, unwinding the construction of $\gamma_{b,c}$ from $b$ and $c$, we just need to verify that the image of $b$ (resp. $c$) under the natural morphism $$W(A[\epsilon]/(\epsilon^2-q(q^{p^\alpha}-1)\epsilon))\to W(\mathcal{O}_K[\epsilon]/(\epsilon^2-q(q^{p^\alpha}-1)\epsilon))$$ lies in $\mathbb{G}_m^\sharp(\mathcal{O}_K[\epsilon]/(\epsilon^2-q(q^{p^\alpha}-1)\epsilon))$ (resp. $\mathbb{G}_a^\sharp(\mathcal{O}_K[\epsilon]/(\epsilon^2-q(q^{p^\alpha}-1)\epsilon))$) and is precisely given by $1+d^{\prime}(q)\epsilon$ (resp. $\epsilon$). 
As $\mathcal{O}_K[\epsilon]/(\epsilon^2-q(q^{p^\alpha}-1)\epsilon)$ is $p$-torsion free, it suffices to check that $b_0=1+d^{\prime}(q)\epsilon$ and $c_0=\epsilon$ in $\mathcal{O}_K[\epsilon]/(\epsilon^2-\epsilon q(q^{p^\alpha}-1)\epsilon)$, which are both clear from our construction of $b$ and $c$ in \cref{lemT.construct b} (see \cref{rem.first strenthern}) and \cref{lemt.construct homotopy}. 

For simplicity we rename $d^{\prime}(q)\in \mathcal{O}_K$ as $e$ and denote $q(q^{p^\alpha}-1)$ as $\beta \in \mathcal{O}_K$. First we would like to calculate $e$ explicitly.
\begin{lemma}\label{lem. calculate e}
    $e\beta=p^{\alpha+1}$.
\end{lemma}
\begin{proof}
By definition, $e=\sum_{i=1}^{p-1} ip^\alpha q^{ip^{\alpha}-1}=\frac{p^{\alpha}}{q}\sum_{i=1}^{p-1} i (q^{p^\alpha})^i$. Let $t=q^{p^\alpha}$ and $S=\sum_{i=1}^{p-1} it^i$, then $$tS-S=\sum_{i=2}^{p} (i-1)t^i-\sum_{i=1}^{p-1} it^i=pt^p-\sum_{i=1}^p t^i=p,$$
hence $e=\frac{p^\alpha}{q}\cdot\frac{p}{q^{p^\alpha}-1}=\frac{p^{\alpha+1}}{\beta}$.
\end{proof}

The next lemma helps us understand how $\partial$ acts on $\rho^{*}\rho_{*} \mathcal{O}_{X}$ under the covering morphism $\rho: X\to X^{\HT}$, in parallel to the behavior of the Sen operator $\theta$ studied in \cite{bhatt2022absolute} and \cite{anschutz2022v} (see \cite[Lemma 2.6]{anschutz2022v} for details).
\begin{lemma}\label{lem. calculate partial on HT}
    Let $\mathcal{E}=\rho_{*} \mathcal{O}_{X}$, we have the following 
    \begin{itemize}
        \item $\rho^{*}\mathcal{E}\cong \mathcal{O}_{G_\pi}\cong \widehat{\bigoplus}_{n\geq 0} \mathcal{O}_K \cdot \frac{a^n}{n!}$.
        \item Suppose that $p>2$ or $p=2$ and $\alpha>0$, then the sequence $0\to 
        \mathcal{O}_K \to \rho^{*}\mathcal{E}\xrightarrow{\partial} \rho^{*}\mathcal{E} \to 0$ is exact.
    \end{itemize}
\end{lemma}
\begin{proof}
 The projection formula tells us that $\rho^{*}\mathcal{E}\cong \mathcal{O}_{G_\pi}$. Moreover, the projection $G_{\pi}\to \mathbb{G}_a^{\sharp}$ sending $(t,a)$ to $a$ identifies $G_{\pi}$ with $\mathbb{G}_a^{\sharp}$, under which the group law on $G_{\pi}$ is transferred to the following formal group law on $\widehat{\mathcal{O}}_{\mathbb{G}_a^{\sharp}}$: $$\Delta: \widehat{\mathcal{O}}_{\mathbb{G}_a^{\sharp}} \to \widehat{\mathcal{O}}_{\mathbb{G}_a^{\sharp}} \hat{\otimes} \widehat{\mathcal{O}}_{\mathbb{G}_a^{\sharp}}, ~~~~a\mapsto a+b+e\cdot ab.$$

    Under the above identification, the isomorphism $\gamma_{b,c}$ (hence also the $q$-connection $\partial$) is just constructed via $\epsilon \in \mathbb{G}_a^{\sharp}(\mathcal{O}_K[\epsilon]/(\epsilon^2-\beta\epsilon))$, which precisely means that for $f(a)\in \widehat{\mathcal{O}}_{\mathbb{G}_a^{\sharp}}$, $\gamma_{b,c}(f)=f+\epsilon\partial(f)=f(a+\epsilon+ea\epsilon)$, hence $\epsilon\partial(f)=f(a+\epsilon+ea\epsilon)-f(a)$.

    First we show that $\ker(\partial)=\mathcal{O}_K$, i.e. $\partial(f)=0$ if and only if $f$ is a constant function. One direction is obvious. On the other hand, if $\partial(f)=0$, then $f(a)\equiv f(a+\epsilon+ea\epsilon)$. Take $a=0$, we see that $g(\epsilon)=0$ for $g(a):=f(a)-f(0)$. Via induction we could produce a sequence $\{y_n\}$ such that $y_0=0$, $y_n=(e\epsilon+1)^n(\epsilon+\frac{1}{e})-\frac{1}{e}$ (hence $y_i\neq y_j$ for $i\neq j$), such that $g(y_i)=0$ for all $i$. As all of calculations happen at $\mathcal{O}_K[\epsilon]/(\epsilon^2-\beta \epsilon)$, hence still hold in $\mathcal{O}_K$, which is $(\mathcal{O}_K[\epsilon]/(\epsilon^2-\beta \epsilon))/(\epsilon-\beta)$. Consequently we get a sequence $\{z_n\}$ in $\mathcal{O}_K$ such that $z_0=0, z_n=(e\beta+1)^n(\beta+\frac{1}{e})-\frac{1}{e}=\frac{(p^{\alpha+1}+1)^{n+1}-1}{e}$ and $g(z_n)=0$. Let $t_n=\frac{z_n}{\beta} \in \mathcal{O}_K$ and define $\kappa(a)$ to be $g(\beta a)$. Then $\kappa(a)\in \mathcal{O}_K \langle a\rangle$ and $\kappa(t_n)=0, \forall n\geq 0$.  An easy application of the Weierstrass preparation theorem (as $t_i\neq t_j$ for $i\neq j$) implies that $\kappa \equiv 0$, which also implies the vanishing of $g$, hence $f$ must be a constant function.

    Next we prove that $\partial$ is surjective. We first observe that $\partial$ is $\mathcal{O}_K$-linear. Indeed, this follows from the Leibniz rule given in \cref{lem.leibniz} and the vanishing of $\partial$ on $A/d=\mathcal{O}_K$. Given $g\in \widehat{\mathcal{O}}_{\mathbb{G}_a^{\sharp}}$, we wish to construct $f\in \widehat{\mathcal{O}}_{\mathbb{G}_a^{\sharp}}$ such that $\partial(f)=g$. We do this by induction modulo $\pi^{n}$ for the uniformizer $\pi=q-1$. More precisely, we wish to construct a sequence $\{h_i\}_{i\geq 1}$ such that $h_i\in \pi^{i-1} \widehat{\mathcal{O}}_{\mathbb{G}_a^{\sharp}}$ and $f_i:=\sum_{k=1}^i f_k$ satisfies that $\partial(h_i)\equiv g$ modulo $\pi^i$, then $f:=\sum_{i=1}^{\infty} f_i$ converges and is a solution for $\partial(f)=g$. Assuming the statement holds for $i=1$ for a moment, then we could do the induction as follows. 
    Suppose $k\geq 1$ and we have find $h_i$'s up to $i=k$ satisfying the desired properties, then $\partial(f_k)-g=\pi^k l_k$ for some $l_k\in \widehat{\mathcal{O}}_{\mathbb{G}_a^{\sharp}}$ by assumption, where $f_k=\sum_{i=1}^k h_i$. By the base case, there exists $h_{k+1}^{\prime} \in \widehat{\mathcal{O}}_{\mathbb{G}_a^{\sharp}}$ such that $\partial(h_{k+1}^{\prime})\equiv -l_k$ mod $\pi$. Hence if we take $h_{k+1}=\pi^k h_{k+1}^{\prime}$, then $\partial(f_{k+1})-g=\pi^{k}(l_k+h^{\prime}_{k})\in \pi^{k+1}\widehat{\mathcal{O}}_{\mathbb{G}_a^{\sharp}}$ for $f_{k+1}=f_k+h_{k+1}$. We thus finish the induction process.

    Now we focus on the base case and would like to show that for any $g\in \widehat{\mathcal{O}}_{\mathbb{G}_a^{\sharp}}$, there exists $f\in \widehat{\mathcal{O}}_{\mathbb{G}_a^{\sharp}}$ such that $\partial(f)\equiv g$ modulo $\pi$. It might be tempting to apply \cite[Lemma 2.6]{anschutz2022v} directly as $\epsilon^2=\beta \epsilon=0$ after modulo $\pi$. However, this doesn't work as the higher divided powers of $\epsilon$ don't vanish. Indeed, since $v_{p}(\beta)=\frac{1}{p-1}$, $\frac{\epsilon^{p^n}}{p^n !}=\frac{\beta^{n-1}}{n!}\epsilon \neq 0$ in $\widehat{\mathcal{O}}_{\mathbb{G}_a^{\sharp}}/\pi$ for any $n\geq 1$. For our purpose, we need to analyze $\partial(f)$ in more detail. First we observe that for $f=\sum_{n=0}^{\infty} c_n\frac{a^n}{n!}\in \widehat{\mathcal{O}}_{\mathbb{G}_a^{\sharp}}$, 
    \begin{equation*}
        \begin{split}
            \gamma_{b,c}(f)&=\sum_{n=0}^{\infty} c_n\frac{(a+(1+ea)\epsilon)^n}{n!}=\sum_{n=0}^{\infty} \frac{c_n}{n!}(a^n+\sum_{i=1}^n \binom{n}{i}a^{n-i}(1+ea)^i\epsilon^i)
            \\&=f(a)+\epsilon\sum_{n=1}^{\infty}\frac{c_n}{n!}(\sum_{i=1}^n\binom{n}{i}a^{n-i}(1+ea)^i\beta^{i-1}).
        \end{split}
    \end{equation*}
On the other hand, since $\gamma_{b,c}(f)=f+\epsilon\partial(f)$, hence
\begin{equation*}
    \begin{split}
        \partial(f)&=\frac{1}{\beta}\sum_{n=1}^{\infty}c_n\frac{(a+\beta e a +\beta)^n-a^n}{n!}=\frac{1}{\beta}\sum_{n=1}^{\infty}\frac{c_n}{n!}(((1+\beta e)^n-1)a^n+\sum_{i=0}^{n-1}\binom{n}{i}(1+\beta e)^i\beta^{n-i}a^i)
        \\&=\frac{1}{\beta}(\sum_{i=0}^{\infty}(\sum_{n=i+1}^{\infty}\frac{c_n}{n!} \binom{n}{i}(1+\beta e)^i\beta^{n-i})a^i)\\&=\frac{1}{\beta}((\sum_{n=1}^{\infty}\frac{c_n\beta_n}{n!})+\sum_{i=1}^{\infty}(\frac{c_i}{i!}((1+\beta e)^i-1)+\sum_{n=i+1}^{\infty}\frac{c_n}{n!} \binom{n}{i}(1+\beta e)^i\beta^{n-i})a^i)
    \end{split}
\end{equation*}
Consequently if we let $\partial(f)=\sum_{i=0}^{\infty} b_i\frac{a^i}{i!}$, then by comparing the coefficients, we have that
\begin{equation*}
    \begin{split}
        b_i=\frac{1}{\beta}(c_i((1+\beta e)^i-1)+(1+\beta e)^i(f^{(i)}(\beta)-c_i))=\frac{1}{\beta}(-c_i+(1+\beta e)^if^{(i)}(\beta)).
    \end{split}
\end{equation*}
From now on we work in $\widehat{\mathcal{O}}_{\mathbb{G}_a^{\sharp}}/\pi$. As $p>2$ or $\alpha>0$, $v_p(e)=(\alpha+1)-\frac{1}{p-1}>0$, which implies that 
\begin{equation}\label{equa. find c solutions}
    b_i=\frac{1}{\beta}(-c_i+f^{(i)}(\beta))=\sum_{n=1}^{\infty} c_{n+i}\frac{\beta^{n-1}}{n!}=\sum_{t=0}^{\infty} c_{p^t+i}u_{p^t},
\end{equation}
where $u_{p^t}:=\frac{\beta^{p^t-1}}{p^t!}$ is a unit. Here the last equality follows from the fact that $v_p(\beta)=\frac{1}{p-1}$.

Our goal is to show that for a fixed sequence $\{b_i\}$, $b_i\in k$ and $b_N=0$ for $N$ large enough, there exists a sequence $\{c_i\}$ with $c_i\in k$ such that \cref{equa. find c solutions} holds and $c_M=0$ for $M\geq N+1$. We do induction on $N$. For the base case $N=1$, then we take $c_1=u_1^{-1}b_0$ and $c_i=0$ for $i>1$, which already guarantees that \cref{equa. find c solutions} holds. Suppose the statement holds up to $N$ for $N\geq 1$ and we are given a sequence $\{b_i\}$ with $b_i=0$ for $i\geq N+1$. We let $c_{N+1}=u_1^{-1}b_N$ and $c_j=0$ for $j\geq N+2$. Apply the induction to the sequence $\{b_i^{\prime}\}_{0\leq i\leq N-1}$ with $b_i^{\prime}=b_i$ if $c_{N+1}$ doesn't show up on the right handside of \cref{equa. find c solutions}, otherwise $b_i^{\prime}=b_i-c_{N+1}u_{N+1-i}$, we get a sequence $\{c_j\}_{1\leq j\leq N}$ solving \cref{equa. find c solutions} for $b_i^{\prime}$ ($0\leq i\leq N-1$), then by construction $\{c_j\}_{1\leq j\leq N+1}$ solving \cref{equa. find c solutions} for $b_i$ ($0\leq i\leq N+1$), we win.
\end{proof}
Similar to \cite[Propostition 3.5.11]{bhatt2022absolute} and \cite[Proposition 2.7]{anschutz2022v}, we can draw the following consequence.
\begin{proposition}\label{prop. ht case fiber seq}
   Assume that $p>2$ or $\alpha>0$. For any $\mathcal{E}\in \mathcal{D}(X^{\HT})$, the natural morphism $\mathrm{R} \Gamma(X^{\HT}, \mathcal{E})\to \rho^*\mathcal{E}\stackrel{\partial_{\mathcal{E}}}{\longrightarrow} \rho^*\mathcal{E}$ constructed in \cref{propt. factor through fiber} is a fiber sequence.
\end{proposition}
\begin{proof}
First we consider $\mathcal{E}=\rho_{*} \mathcal{O}_{X}$. By the discussion at the beginning of this subsection, $\partial_{\mathcal{E}}$ descends to a functor from $\mathcal{E}$ to itself and hence we have a morphism 
\begin{equation*}
    \mathcal{O}_{X^{\HT}}\to \rho_{*} \mathcal{O}_{X} \xrightarrow{\partial} \rho_{*} \mathcal{O}_{X},
\end{equation*}
which is a fiber sequence by \cref{lem. calculate partial on HT} and faithfully flat descent. For a general $\mathcal{E}\in \mathcal{D}(X^{\HT})$, tensoring $\mathcal{E}$ with the above sequence and then applying the projection formula yields a fiber sequence
\begin{equation*}
    \mathcal{E}\to \rho_{*}\rho^{*} \mathcal{E} \xrightarrow{\partial^{\prime}} \rho_{*}\rho^{*} \mathcal{E}
\end{equation*}
After taking cohomology we then get a fiber sequence 
\begin{equation*}
    \mathrm{R} \Gamma(X^{\HT}, \mathcal{E})\to \rho^*\mathcal{E}\stackrel{\partial^{\prime}}{\longrightarrow} \rho^*\mathcal{E}
\end{equation*}
Then one can check that $\partial^{\prime}$ is precisely $\partial_{\mathcal{E}}$ by the usual trick of trivializing Hopf algebra’s comodules.
\end{proof}
\begin{example}[$q$-connections on the ideal sheaf $\mathcal{I}^k$ when restricted to the Hodge-Tate locus]\label{example.action on generators}
     For $\mathscr{E}=\mathcal{I}^k$, one can verify that under the trivialization  $\rho^{*}\mathscr{E}=\mathcal{O}_K\cdot (d^k)$, the $\mathcal{O}_K$-linear $\partial_{\mathscr{E}}$ is given by multiplication by $\sum_{i=1}^{k} \binom{k}{i} e^i \beta^{i-1}=e\cdot\frac{(1+p^{\alpha+1})^k-1}{p^{\alpha+1}}$, this is due to the fact that $\psi(d^k)=(1+\epsilon e)^k d^k$, and $$(1+\epsilon e)^k-1=\epsilon \sum_{i=1}^{k} \binom{k}{i} e^i \beta^{i-1}=\epsilon\frac{(1+e\beta)^k-1}{\beta}=\epsilon\frac{(1+p^{\alpha+1})^k-1}{\beta}=\epsilon e\cdot\frac{(1+p^{\alpha+1})^k-1}{p^{\alpha+1}}.$$ 
     Here we use the fact that $p^{\alpha+1}=e\beta$ obtained in \cref{lem. calculate e}. 
\end{example}
\begin{remark}\label{rem. failure for p=2}
    \cref{prop. ht case fiber seq} fails when $p=2$ and $\alpha=0$. Actually, in this case $e=1$ and $\beta=2$, then if it holds, the previous example implies that $\mathrm{H}^{1}(X^{\HT}, \mathcal{I}^2)=\mathbb{Z}/4$, a contradiction with the fact that $\mathrm{H}^{1}(\WCart^{\HT}, \mathcal{I}^2)=\mathbb{Z}/2$ by \cite[Proposition 3.5.11, Corollary 3.5.14]{bhatt2022absolute}.
\end{remark}

\subsection{$q$-connections beyond the Hodge-Tate locus}
With all of the ingredients in hand, in this subsection, we classify quasi-coherent complexes on $X_n^{\Prism}$.
\begin{proposition}\label{propt. sen calculate cohomology}
    Assume that $p>2$ or $\alpha>0$. Let $n\in \mathbb{N}\cup \{\infty\}$. For any $\mathcal{E}\in \mathcal{D}(X_n^{\Prism})$, the natural morphism $$\mathrm{R} \Gamma(X_n^{\Prism}, \mathcal{E})\to \rho^*\mathcal{E}\stackrel{\partial_{\mathcal{E}}}{\longrightarrow} \rho^*\mathcal{E}$$ constructed in \cref{propt. factor through fiber} is a fiber sequence.
\end{proposition}
\begin{proof}
    For $n\in \mathbb{N}$, to see it induces an identification of $\mathrm{R} \Gamma(X_n^{\Prism}, \mathcal{E})$ with $\fib(\rho^*\mathcal{E}\stackrel{\Theta_{\mathscr{E}}}{\longrightarrow} \rho^*\mathcal{E})$, using standard  d\'evissage (the trick used in the proof of \cite[Proposition 3.15]{liu2024prismatization}), by induction on $n$, it reduces to $n=1$, which follows from \cref{prop. ht case fiber seq}.

Finally for $\mathcal{E}\in \mathcal{D}(X^{\Prism})$, as taking global sections commutes with limits, by writing $\mathcal{E}$ as the inverse limit of $\mathcal{E}_n$ for $\mathcal{E}_n$ the restriction of $\mathcal{E}$ to $X^{\Prism}_n$, we see that 
\[\mathrm{R} \Gamma(X^{\Prism}, \mathcal{E})=\varprojlim \mathrm{R} \Gamma(X^{\Prism}_n, \mathcal{E}_n)=\varprojlim \fib(\rho^*\mathcal{E}_n\stackrel{\partial_{\mathscr{E}}}{\longrightarrow} \rho^*\mathcal{E}_n)=\fib(\rho^*\mathcal{E}\stackrel{\partial_{\mathcal{E}}}{\longrightarrow} \rho^*\mathcal{E}).\]
Here the second equality follows from the first paragraph and the last equality holds as finite limits commute with limits.
\end{proof}

As a byproduct of \cref{propt. sen calculate cohomology}, we conclude that
\begin{corollary}
   Assume that $p>2$ or $\alpha>0$. The global sections functor $$\mathrm{R} \Gamma(X^{\Prism}, \bullet): \mathcal{D}(X^{\Prism}) \rightarrow \widehat{\mathcal{D}}(\mathbb{Z}_p) \quad \text{resp.}\quad \mathrm{R} \Gamma(X^{\Prism}_n, \bullet): \mathcal{D}(X^{\Prism}_n) \rightarrow \widehat{\mathcal{D}}(\mathbb{Z}_p)$$
    commutes with colimits.
\end{corollary}

\begin{proposition}\label{propt.generation}
    Let $n\in \mathbb{N}$. The $\infty$-category $\mathcal{D}(X^{\Prism})$ (resp. $\mathcal{D}(X^{\Prism}_n)$) is generated under shifts and colimits by the invertible sheaves $\mathcal{I}^k$ for $k\in \mathbb{Z}$.
\end{proposition}
\begin{proof}
    Arguing as in \cite[Corollary 3.5.16]{bhatt2022absolute}, this could be reduced to $n=1$, where the results follow from \cite[Proposition 2.9]{anschutz2022v}.
\end{proof}
\begin{theorem}\label{thmt.main classification}
Assume that $p>2$ or $\alpha>0$. Let $n\in \mathbb{N}\cup \{\infty\}$. The functor 
    \begin{align*}
        &\beta_n^{+}: \mathcal{D}(X_n^{\Prism}) \rightarrow  \mathcal{D}(A/d^n[\partial; \gamma_A, \partial_A]), \qquad \mathcal{E}\mapsto (\rho^{*}(\mathcal{E}),\partial_{\mathcal{E}})
    \end{align*}
    constructed in \cref{prop. induce functor} is fully faithful. 
     Moreover, its essential image consists of those objects $M\in \mathcal{D}(A/d^n[\partial; \gamma_A, \partial_A])$ 
     satisfying the following pair of conditions:
    \begin{itemize}
        \item $M$ is $(p,d)$-complete.
        \item The action of $\partial$ on the cohomology $\mathrm{H}^*(M\otimes^{\mathbb{L}}k)$ \footnote{Here the derived tensor product means the derived base change along $A/d^n\to A/(d,q-1)=k$.} is locally nilpotent. 
    \end{itemize}
\end{theorem}
\begin{proof}
The functor is well-defined thanks to \cref{prop. induce functor}. Then we follow the proof of \cite[Theorem 3.5.8]{bhatt2022absolute}. 
    For the full faithfulness, let $\mathcal{E}$ and $\mathcal{F}$ be quasi-coherent complexes on $X^\Prism_n$ and we want to show that the natural map 
    \begin{equation*}
        \Hom_{\mathcal{D}(X^\Prism_n)}(\mathcal{E}, \mathcal{F}) \rightarrow \Hom_{\mathcal{D}(A/d^n[\partial; \gamma_A, \partial_A])} (\rho^{*}(\mathcal{E}), \rho^{*}(\mathcal{F}))
    \end{equation*}
    is a homotopy equivalence. Thanks to \cref{propt.generation}, we could reduce to the case that $\mathcal{E}=\mathcal{I}^k$ for some $k\in \mathbb{Z}$. Replacing $\mathcal{F}$ by the twist $\mathcal{F}(-k)$, we could further assume that $k=0$. Then the desired result follows from \cref{rem. cohomology of the structure sheaf} and \cref{propt. sen calculate cohomology}. 
    
    To check that the action of $\partial$ on the cohomology $\mathrm{H}^*(\rho^{*}(\mathcal{E})\otimes^{\mathbb{L}}k)$ is locally nilpotent for $\mathcal{E}\in \mathcal{D}(X_n^{\Prism})$, again thanks to \cref{propt.generation}, we might assume $\mathcal{E}=\mathcal{I}^r$ for some $r\in \mathbb{Z}$. Then by \cref{example.action on generators}, after base change to $k$ the action of $\partial$ is given by multiplication by $re$,
    which already vanishes as our assumption ($p>2$ or $\alpha>0$) guarantees that $\mathcal{O}_K$ is ramified over $W(k)$, hence $e=0$ in $k$. 

    Let $\mathcal{C}_n\subseteq \mathcal{D}(A/d^n[\partial; \gamma_A, \partial_A])$ be the full subcategory spanned by objects satisfying two conditions listed in \cref{thmt.main classification}. As the source $\mathcal{D}(X_n^{\Prism})$ is generated under shifts and colimits by the invertible sheaves $\mathcal{I}^n$ for $n\in \mathbb{Z}$ by \cref{propt.generation}, to complete the proof it suffices to show that $\mathcal{C}_n$ is also generated under shifts and colimits by $\{\beta_n^{+}(\mathcal{I}^k)\}$ ($k\in \mathbb{Z}$). For this purpose, it can be reduced to the Hodge-Tate locus. Indeed, if $\mathcal{E}$ is in $\mathcal{C}_n$ and $\RHom_{\mathcal{C}_n}(\beta_n^{+}(\mathcal{I}^k), \mathcal{E})\cong \RHom_{\mathcal{C}_n}(\beta_n^{+}(\mathcal{O}),\beta_n^{+}(\mathcal{I}^{-k})\otimes\mathcal{E})=0
    $ for all $k\in \mathbb{Z}$\footnote{Here the underlying complex of $\beta_n^{+}(\mathcal{I}^{-k})\otimes\mathcal{E}$ is $\rho^*(\mathcal{I}^{-k})\otimes_{A/d^n} \mathcal{E}$ and $\partial$ acts on it via the twisted Leibnitz rule stated in \cref{rem.twisted leibnitz in general}.}, then we wish to show that $\mathcal{E}=0$, which is equivalent to the vanishing of $\mathcal{E}|_{\mathcal{C}_1}:=\mathcal{E}/^{\mathbb{L}}d$ due to the $(p,d)$-completeness assumption on $\mathcal{E}$.
    
    By considering the fiber sequence
\begin{equation*} \RHom_{\mathcal{C}_n}(\beta_n^{+}(\mathcal{O}),\beta_n^{+}(\mathcal{I}^{1-k})\otimes\mathcal{E})\to \RHom_{\mathcal{C}_n}(\beta_n^{+}(\mathcal{O}),\beta_n^{+}(\mathcal{I}^{-k})\otimes\mathcal{E}) \to\RHom_{\mathcal{C}_1}(\beta_1^{+}(\mathcal{O}),(\mathcal{I}^{-k}\mathcal{E})|_{\mathcal{C}_1}),
\end{equation*}
we deduce that $\RHom_{\mathcal{C}_{1}}(\beta_1^{+}(\mathcal{I}^k), \mathcal{E}|_{\mathcal{C}_{1}})$ vanishes for all $k$. We will show that this already guarantees the vanishing of $\mathcal{E}|_{\mathcal{C}_{1}}$.
    
    Indeed, we argue that for every nonzero object $\mathcal{E}_0\in \mathcal{C}_1$, $\RHom_{\mathcal{C}_{1}}(\beta_1^{+}(\mathcal{O}), \mathcal{E}_0)\neq 0$. For this purpose, first we observe that $\RHom_{\mathcal{C}_{1}}(\beta_1^{+}(\mathcal{O}), \mathcal{E}_0)\simeq \fib(\mathcal{E}_0\stackrel{\partial}{\longrightarrow} \mathcal{E}_0)$ by \cref{rem. cohomology of the structure sheaf}. Replacing $\mathcal{E}_0$ by $\mathcal{E}_0\otimes k$ (the derived Nakayama guarantees that $L\otimes k$ detects whether $L$ is zero or not for $p$-complete $L$), we may assume that there exists some cohomology group $\mathrm{H}^{-m}(\mathcal{E}_0)$ containing a nonzero element killed by $\partial_{\mathcal{E}_0}$ (this could be done by iterating the action of $\partial$ and then use the nilpotence assumption). 
    It then follows that there exists a non-zero morphism from $\beta_1^{+}(\mathcal{O}[m])$ to $\mathcal{E}_0$, we win.
\end{proof}
When restricting the above theorem to perfect complexes, we get the following result.
\begin{corollary}\label{corot.perfect version of the remain theorem}
Let $n\in \mathbb{N}\cup \{\infty\}$.. The functor $\beta^+_n$ from \cref{thmt.main classification} restricts to a fully faithful functor
  \[
    \beta^+_n\colon \Perf(X_n^{\Prism}) \to \Perf(A/d^n[\partial; \gamma_A, \partial_A])
  \]
  whose essential image consists of $(p,d)$-adically complete perfect complexes $M$ over $A/d^n[\partial; \gamma_A, \partial_A]$ for which  $\partial$ is nilpotent on $H^\ast( M\otimes^{\mathbb{L}}_{A/d^n} k)$. Moreover, $\mathcal{E}\in \Perf(X_n^{\Prism})$ comes from a vector bundle on $X_n^{\Prism}$ if and only if $\beta_n^{+}(\mathcal{E})$ is further required to be a finite projective $A/d^n$-module concentrated on degree $0$. 
\end{corollary}
\begin{proof}
    Given $\mathcal{E}\in \Perf(X_n^{\Prism})$, the underlying complex of $\beta^+_n(\mathcal{E})$ is $\rho^*\mathcal{E}$, hence it is a $(p,d)$-adically complete perfect complex over $A/d^n$. To see it is perfect over $A/d^n[\partial; \gamma_A, \partial_A]$, by (non-commutative) derived Nakayama, it suffices to show that $\beta^+_n(\mathcal{E})\otimes_{A/d^n[\partial; \gamma_A, \partial_A]}^{\mathbb{L}} A/d[\partial; \gamma_A, \partial_A] \simeq \beta^+_1(\mathcal{E}|_{X^{\HT}})$ is perfect over $A/d[\partial; \gamma_A, \partial_A]=\mathcal{O}_K[\partial]$, which follows from its perfectness over $\mathcal{O}_K$ and the regularity of $\mathcal{O}_K[\partial]$. Consequently $ \beta^+_n$ in \cref{thmt.main classification} does restrict to a functor stated in the corollary and its full faithfulness follows from \cref{thmt.main classification}. 

    Conversely, given a $(p,d)$-complete $M \in \Perf(A/d^n[\partial; \gamma_A, \partial_A])$ such that $\partial$ is nilpotent on $H^\ast( M\otimes^{\mathbb{L}}_{A/d^n} k)$, we need to show that $M$ is a perfect $A/d^n$-complex. As $A/d^n$ is complete with respect to the $(p,d)$-topology, it suffices to check this for $M\otimes_{A/d^n} k$. By further noticing that canonical truncations of $M$ are again perfect $A/d^n[\partial; \gamma_A, \partial_A]$ complexes, we reduce to the case that $M$ is concentrated on degree $0$ and lives over $k$. Then $M$ is a finitely generated $k[\partial]/\partial^i$-module for some $i\geq 0$, hence a perfect $A/d^n$-module, as desired. 
    
    For the moreover part, one just observes that $\mathcal{E}\in \mathcal{D}(X_n^{\Prism})$ if and only if $\beta_n^+(\mathcal{E})$ is concentrated on degree $0$ by \cref{prop. t independ}.
\end{proof}

\subsection{Comparison with the $q$-de Rham prism}
In this section we would like to present $X^{\Prism}$ as a quotient of the affine formal scheme $\Spf(A)$, motivated by the construction in \cite[Section 3.8]{bhatt2022absolute} as well as the following consequence of \cref{propt.key automorphism of functors}, namely we can consider the restriction of $\gamma_{b,c}$ to the locus $\epsilon=q(q^{p^{\alpha}}-1)$ and thus obtain the following commutative diagram:
\[\xymatrixcolsep{5pc}\xymatrix{\Spf(A)\ar@/_/[ddr]_{\iota} \ar@/^/[drr]^{\gamma_A} \ar[dr]^{\epsilon=q(q^{p^{\alpha}}-1)}\\ &\Spf(R)\ar[d]^{\iota}\ar[r]^{\psi}& \Spf(A)\ar@{=>}[dl]^{\gamma_{b,c}} \ar[d]_{}^{\rho}
\\& \Spf(A)  \ar^{\rho}[r]&X^{\Prism},}\]
where $\gamma_A$ is the atuomorphism of $A$ sending $q$ to $q^{1+p^{\alpha+1}}$. Briefly speaking, the diagram provides an isomorphism between $\rho\circ \gamma_A$ and $\rho$, viewed as points in $X^{\Prism}(A)$. 

We next upgrade the above diagram to show that $\rho$ factors through a quotient of $\Spf(A)$. First we observe that the $q$-prism $(A,I)=(W(k)[[q-1]], [p]_{q^{p^{\alpha}}})$ admits an action by $\mathbb{Z}_p^{\times}$, where an element $u\in \mathbb{Z}_p^{\times}$ acts by the automorphism $q\mapsto q^u$, which will be denoted as $\gamma_u$ later. Indeed, one could check that for $u\in \mathbb{Z}_p$, $\gamma_u$ preserves the ideal $(d)$ if and only if $u\in \mathbb{Z}_p^{\times}$.
\begin{proposition}\label{prop. group quotient}
    $\rho\circ \gamma_u \in X^{\Prism}(A)$ is isomorphic to $\rho \in X^{\Prism}(A)$ if and only if $u\in (1+p^{\alpha+1}\mathbb{Z}_p)^{\times}$. Moreover, the above $\mathbb{Z}_p^{\times}$-action induces a morphism of stacks
    \begin{equation*}
    [\Spf(A)/(1+p^{\alpha+1}\mathbb{Z}_p)^{\times}]\to X^{\Prism},
    \end{equation*}
    where $[\Spf(A)/(1+p^{\alpha+1}\mathbb{Z}_p)^{\times}]$ is the stack quotient of $\Spf(A)$ by the profinite group $(1+p^{\alpha+1}\mathbb{Z}_p)^{\times}$.
\end{proposition}
\begin{proof}
    Following the strategy of proving \cref{propt.key automorphism of functors}, specifying an isomorphism between $\rho\circ \gamma_u$ and $\rho$ is equivalent to giving the following data:
    \begin{itemize}
        \item An element $b_u\in W(A)$ (it will be automatically in $W(A)^{\times}$ if exists) such that $\Tilde{\gamma_u}(d)=\Tilde{\iota}(d)\cdot b_u$.
        \item An element $c_u\in W(A)$ such that $\Tilde{\gamma_u}(q)-\Tilde{\iota}(q)=\Tilde{\iota}(d)\cdot c_u$.
    \end{itemize}
    Here in above $\Tilde{\gamma_u}: A\to W(A)$ (resp. $\Tilde{\iota}: A\to W(A)$) is the unique $\delta$-ring morphism lifting $\gamma_u: A\to A$ (resp. $\iota=\Id: A\to A$).

    One can check that such a $b_u$ exists if and only if $u\in \mathbb{Z}_p^{\times}$ by mimicking the proof of \cref{lemT.construct b}. We leave it as an exercise for the reader. Heuristically, this is due to the fact that $\gamma_u$ preserves the ideal $(d)$ and $\frac{\gamma_u(d)}{d}\in A^{\times}$ for $u\in \mathbb{Z}_p^{\times}$.

    For the existence of $c_u$, we proceed as in the proof of \cref{lemt.construct homotopy}. First notice that $w_n(\Tilde{\iota}(q))=w_0(\varphi^n(\Tilde{\iota}(q)))=w_0(\Tilde{\iota}(\varphi^n(q)))=q^{p^n}$ and that $w_n(\Tilde{\gamma_u}(q))=q^{up^n}$, $w_n(d)=[p]_{q^{p^{n+\alpha}}}$. Hence there exists $r_n\in A$ such that $w_n(\Tilde{\gamma_u}(q))-w_n(\Tilde{\iota}(q))=w_n(d)\cdot r_n$ if and only if $[p]_{q^{p^{n+\alpha}}} |~q^{up^n}-q^{p^n}$, which is equivalent to that $u\equiv 1 \mod p^{\alpha+1}$. Moreover, if $u\in (1+p^{\alpha+1}\mathbb{Z}_p)^{\times}$, then we define $r_n$ to be $\frac{q^{up^n}-q^{p^n}}{[p]_{q^{p^{n+\alpha}}}}$ and observe that $\varphi(r_n)=r_{n+1}$. By invoking Dwork's lemma (see \cite[4.6 on page 213]{lazard2006commutative} for details), we conclude that $(r_0, r_1,\cdots)\in R^{\mathbb{N}}$ is the ghost coordinate for some element $c_u\in W(A)$ (such $c_u$ is unique as the ghost map is injective). By construction, $c_u$ satisfies that $\Tilde{\gamma_u}(q)-\Tilde{\iota}(q)=\Tilde{\iota}(d)\cdot c_u$, hence finish the proof of the first part in the proposition. 
    
    For the moreover part, we just observe that for $u\in (1+p^{\alpha+1}\mathbb{Z}_p)^{\times}$, the isomorphism between $\rho\circ \gamma_u$ and $\rho$ is unique as both $b_u$ and $c_u$ are unique (as $W(A)$ is $\Tilde{\iota}(d)$-torsion free, which can be checked after applying ghost map), hence the isomorphism $\sigma_u: \rho\circ \gamma_u \simeq \rho$ satisfies the ``higher associativity" condition in \cite[Definition 1.3]{romagny2005group}, hence $\rho$ factors through $[\Spf(A)/(1+p^{\alpha+1}\mathbb{Z}_p)^{\times}]\to X^{\Prism}$.
\end{proof}
\begin{remark}
    Unfortunately, the above morphism $[\Spf(A)/(1+p^{\alpha+1}\mathbb{Z}_p)^{\times}]\to X^{\Prism}$ is not an isomorphism. However, as $\alpha$ increases, it becomes closer to an isomorphism heuristically. Indeed, by varying $\alpha$ and considering the inverse limit of the following commutative diagram (which corresponds to perfection)
    \[\xymatrixcolsep{5pc}\xymatrix{\cdots\ar[r]\ar[d]&[\Spf(A)/(1+p^{2}\mathbb{Z}_p)^{\times}]\ar[d]^{\rho}\ar[r]^{q\mapsto q^p}& [\Spf(A)/(1+p^{}\mathbb{Z}_p)^{\times}] \ar[d]_{}^{\rho}
\\ \cdots\ar[r]&(W(k)[\zeta_{p^{2}}]_p)^{\Prism}  \ar^{}[r]&(W(k)[\zeta_{p}]_p)^{\Prism},}\]
we finally end with an isomorphism $\Spf(\Ainf(W(k)[\zeta_{p^{\infty}}]_p^{\wedge}))\stackrel{\simeq}{\rightarrow} (W(k)[\zeta_{p^{\infty}}]_p^{\wedge})^{\Prism}$.
\end{remark}

In what follows, we will abuse notation by still writing the morphism $[\Spf(A)/(1+p^{\alpha+1}\mathbb{Z}_p)^{\times}]\to X^{\Prism}$ as $\rho$. 

For any $\mathcal{E}\in \mathcal{D}(X^{\Prism})$, each element $u\in (1+p^{\alpha+1}\mathbb{Z}_p)^{\times}$ determines an automorphism of $\rho^*\mathcal{E}$, which we will denote by $\gamma_u$. For $u=1+p^{\alpha+1}$ (which is a topological generator of $(1+p^{\alpha+1}\mathbb{Z}_p)^{\times}$), the automorphism $\gamma_u$ is completely determined by the $q$-connection on $\rho^*\mathcal{E}$.
\begin{proposition}\label{prop. monodromy and group action}
    For $u_0=1+p^{\alpha+1}$, $\gamma_{u_0}$ acts on $\rho^*\mathcal{E}$ via $1+q(q^{p^{\alpha}}-1)\partial_{\mathcal{E}}$.
\end{proposition}
\begin{proof}
    This follows from our construction of $\gamma_u$, $\partial_{\mathcal{E}}$ and the commutative diagram given above \cref{prop. group quotient}.
\end{proof}
\begin{remark}\label{rem. decalage}
    As $q$ is a unit and $u_0$ is a topologically generator for $(1+p^{\alpha+1}\mathbb{Z}_p)^{\times}$, combining \cref{prop. monodromy and group action} with \cref{thmt.main classification}, we see that for $\mathcal{E}\in X^{\Prism}$, 
    \begin{equation*}
        \mathrm{R} \Gamma(X_n^{\Prism}, \mathcal{E})\xrightarrow{\simeq} \fib(\rho^*\mathcal{E}\stackrel{\partial_{\mathcal{E}}}{\longrightarrow} \rho^*\mathcal{E})=L\eta_{q^{p^{\alpha}}-1}\mathrm{R} \Gamma([\Spf(A)/(1+p^{\alpha+1}\mathbb{Z}_p)^{\times}], \rho^*\mathcal{E}),
    \end{equation*}
    where $L\eta$ is the d\'ecalage functor (see \cite{bhatt2018integral} or \cite{berthelot2015notes} for details).
    In particular, pullback from $\mathcal{D}(X^{\Prism})$ to $\mathcal{D}([\Spf(A)/(1+p^{\alpha+1}\mathbb{Z}_p)^{\times}])$ along
 $\rho: [\Spf(A)/(1+p^{\alpha+1}\mathbb{Z}_p)^{\times}]\to X^{\Prism}$ is not fully faithful.
\end{remark}\label{rem. sen and q connection}
\begin{remark}[{Comparing the $q$-connection with Sen operator on the Hodge-Tate locus}]\label{rem. compare two operators}
    We can restrict $\rho$ to the Hodge-Tate locus and obtain $\rho: [\Spf(\mathcal{O}_K)/(1+p^{\alpha+1}\mathbb{Z}_p)^{\times}]\to X^{\HT}$. Consequently, for $u\in (1+p^{\alpha+1}\mathbb{Z}_p)^{\times}$ and $\mathcal{E}\in \mathcal{D}(X^{\HT})$, $\rho^*\mathcal{E}$ is equipped with an automorphism $\gamma_u$. On the other hand, $\mathcal{E}$ is equipped with two operators: the $q$-connection $\partial_{\mathcal{E}}$ from \cref{thmt.main classification} (constructed via the $q$ prism) as well as the Sen operator $\theta_{\mathcal{E}}$ from \cite[Theorem 2.5]{anschutz2022v} (constructed utilizing the Breuil-Kisin prism), hence so is $\rho^*\mathcal{E}$. The relations between the action of $\partial_{\mathcal{E}}$ and $\theta_{\mathcal{E}}$ on 
    $\rho^*\mathcal{E}$ can be connected via the automorphism $\gamma_{u_0}$. Indeed, one can copy the proof of \cite[Proposition 3.7.1]{bhatt2022absolute} to conclude that for $u\in (1+p^{\alpha+1}\mathbb{Z}_p)^{\times}$, $\gamma_u$ acts on $\rho^*\mathcal{E}$ via $u^{\theta_{\mathcal{E}}/E^{\prime}(\pi)}$, but we already know that $\gamma_{u_0}$ acts on $\rho^*\mathcal{E}$ via $1+q(q^{p^{\alpha}}-1)\partial_{\mathcal{E}}$ thanks to \cref{prop. monodromy and group action}. Hence
    \begin{equation*}
(1+p^{\alpha+1})^{\theta_{\mathcal{E}}/E^{\prime}(\pi)}=1+q(q^{p^{\alpha}}-1)\partial_{\mathcal{E}}=1+\frac{p^{\alpha+1}}{e}\partial_{\mathcal{E}},
    \end{equation*}
    where the last identity is due to \cref{lem. calculate e}. As an upsot, we see
    \begin{equation*}
\partial_{\mathcal{E}}=e\cdot\frac{(1+p^{\alpha+1})^{\theta_{\mathcal{E}}/E^{\prime}(\pi)}-1}{p^{\alpha+1}}.
    \end{equation*}
    In this sense, \cref{example.action on generators} is just a special case of the above formula!  
\end{remark}

By examining the construction of the isomorphism $\sigma_u: \rho\circ \gamma_u \simeq \rho$, we can determine the locus where $\sigma_u$ is trivial very explicitly.
\begin{proposition}\label{prop. trivial locus}
    We have the following commutative diagram 
    \[\xymatrixcolsep{5pc}\xymatrix{[\Spf(A/(q^{p^{\alpha}}-1))/(1+p^{\alpha+1}\mathbb{Z}_p)^{\times}]\ar[d]^{}\ar[r]^{q^{p^{\alpha}}\mapsto 1}& [\Spf(A)/(1+p^{\alpha+1}\mathbb{Z}_p)^{\times}] \ar[d]_{}^{\rho}
\\\Spf(A/(q^{p^{\alpha}}-1)) \ar[r]^{\rho_{}}&X^{\Prism}.}\]

\end{proposition}
\begin{proof}
    We observe that the action of $(1+p^{\alpha+1}\mathbb{Z}_p)^{\times}$ on $(A, d)$ is trivial after modulo $q^{p^{\alpha}}-1$. Moreover, for $u\in (1+p^{\alpha+1}\mathbb{Z}_p)^{\times}$, the isomorphism $\sigma_u: \rho\circ \gamma_u\simeq \rho$ constructed in \cref{prop. group quotient} becomes the identify on the vanishing locus of $q^{p^{\alpha}}-1$, which is implied by the vanishing of $r_n$ constructed in \cref{prop. group quotient} on $A/(q^{p^{\alpha}}-1)$ (indeed, this follows from the vanishing of $r_0=q(q^{p^{\alpha}}-1)$). 
\end{proof}
Then the main result of this section is the following theorem:
\begin{theorem}\label{thmt.main classification II}
    Let $p>2$ or $\alpha>0$. Then the diagram in \cref{prop. trivial locus} induces a fully faithful functor of $\infty$-categories 
$$F: \mathcal{D}(X^{\Prism}) \longrightarrow \mathcal{D}([\Spf(A)/(1+p^{\alpha+1}\mathbb{Z}_p)^{\times}])\times_{\mathcal{D}([\Spf(A/(q^{p^{\alpha}}-1))/(1+p^{\alpha+1}\mathbb{Z}_p)^{\times}])} \mathcal{D}(\Spf(A/(q^{p^{\alpha}}-1))).$$
The essential image of $F$ consists of those pairs $(M,\gamma_M)$ with $M\in \mathcal{D}(A)$ and $\gamma_M$ being an automorphism of $M$ which is congruent to the identity modulo $q(q^{p^{\alpha}}-1)$ such that the following holds:
\begin{itemize}
    \item $M\in \mathcal{D}(A)$ is $(p,d)$-complete.
    \item the action of $\partial_M$ on the cohomology $\mathrm{H}^*(M\otimes^{\mathbb{L}}k)$ \footnote{Here the derived tensor product means the derived base change along $A/d^n\to A/(d,q-1)=k$.} is locally nilpotent, where $\partial_M: M\to M$ is an operator satisfying $\gamma_M=\Id+q(q^{p^{\alpha}}-1)\partial_M$.
\end{itemize}
    In particular, for $\mathcal{E}\in \mathcal{D}(X^{\Prism})$, the diagram 
\begin{equation}\label{diagram}
\begin{gathered}
\xymatrixcolsep{5pc}\xymatrix{\mathrm{R} \Gamma(X^{\Prism}, \mathcal{E})\ar[d]^{}\ar[r]^{}&\mathrm{R} \Gamma([\Spf(A)/(1+p^{\alpha+1}\mathbb{Z}_p)^{\times}], \rho^{*} \mathcal{E}) 
\ar[d]_{}^{}
\\\mathrm{R} \Gamma(\Spf(A/(q^{p^{\alpha}}-1)), \rho^{*} \mathcal{E})  \ar^{}[r]&\mathrm{R} \Gamma([\Spf(A/(q^{p^{\alpha}}-1))/(1+p^{\alpha+1}\mathbb{Z}_p)^{\times}], \rho^{*} \mathcal{E})}
\end{gathered}
\end{equation}
    is a pullback square in the $\infty$-category $\hat{\mathcal{D}}(\mathbb{Z}_p)$.
\end{theorem}
\begin{proof}
     We first prove full faithfulness. We use $\mathcal{C}$ to denote the fiber product 
     $$\mathcal{D}([\Spf(A)/(1+p^{\alpha+1}\mathbb{Z}_p)^{\times}])\times_{\mathcal{D}([\Spf(A/(q^{p^{\alpha}}-1))/(1+p^{\alpha+1}\mathbb{Z}_p)^{\times}])} \mathcal{D}(\Spf(A/(q^{p^{\alpha}}-1))).$$
     Let $\mathcal{E}$ and $\mathcal{F}$ be quasi-coherent complexes on $X^\Prism$. We aim to show that the natural map 
    \begin{equation*}
        \Hom_{\mathcal{D}(X^\Prism)}(\mathcal{E}, \mathcal{G}) \rightarrow \Hom_{\mathcal{C}}(F(\mathcal{E}), F(\mathcal{G}))
    \end{equation*}
    is a homotopy equivalence. By \cref{propt.generation}, $\mathcal{D}(X^\Prism)$ is generated under shifts and colimits by $\mathcal{I}^n$ , we could reduce to the case that $\mathcal{E}=\mathcal{I}^k$ for some $k\in \mathbb{Z}$. Replacing $\mathcal{G}$ by the twist $\mathcal{G}(-k)$, we could further assume that $k=0$. Then it suffices to show that \cref{diagram} is indeed a pullback suare for any 
    $\mathcal{E}\in \mathcal{D}(X^{\Prism})$. Let $u=1+p^{\alpha+1}$ be the topological generator for $(1+p^{\alpha+1}\mathbb{Z}_p)^{\times}$. Then $$\mathrm{R} \Gamma([\Spf(A)/(1+p^{\alpha+1}\mathbb{Z}_p)^{\times}], \rho^{*} \mathcal{E})=\fib(\rho^{*} \mathcal{E}\stackrel{\gamma_u-1}{\longrightarrow} \rho^{*} \mathcal{E}),$$
    for which we will denote as $(\rho^{*} \mathcal{E})^{u=1}$ for simplicity. Consequently \cref{diagram} can be rewritten as a diagram of complexes    
    \[\xymatrixcolsep{5pc}\xymatrix{\mathrm{R} \Gamma(X^{\Prism}, \mathcal{E})\ar[d]^{}\ar[r]^{}& (\rho^{*} \mathcal{E})^{u=1}
\ar[d]_{}^{}
\\ (\rho^{*} \mathcal{E})\otimes^{\mathbb{L}}_{A} A/(q^{p^{\alpha}}-1) \ar^{}[r]&(\rho^{*} \mathcal{E})^{u=1}\otimes^{\mathbb{L}}_{A} A/(q^{p^{\alpha}}-1).}\]
By \cref{prop. monodromy and group action}, $$(\rho^{*} \mathcal{E})^{u=1}\otimes^{\mathbb{L}}_{A} A/(q^{p^{\alpha}}-1)=((\rho^{*} \mathcal{E})\otimes^{\mathbb{L}}_{A} A/(q^{p^{\alpha}}-1))\oplus ((\rho^{*} \mathcal{E})\otimes^{\mathbb{L}}_{A} A/(q^{p^{\alpha}}-1))[-1]$$
and hence the cofiber of the bottom map $(\rho^{*} \mathcal{E})\otimes^{\mathbb{L}}_{A} A/(q^{p^{\alpha}}-1)\to (\rho^{*} \mathcal{E})^{u=1}$ is precisely $((\rho^{*} \mathcal{E})\otimes^{\mathbb{L}}_{A} A/(q^{p^{\alpha}}-1))[-1]$. Thanks to \cref{propt. sen calculate cohomology} and \cref{prop. monodromy and group action}, we obtain a commutative diagram in which all of the columns are fiber sequences
\[\xymatrixcolsep{5pc}\xymatrix{\mathrm{R} \Gamma(X^{\Prism}, \mathcal{E})\ar[d]^{}\ar[r]^{}& (\rho^{*} \mathcal{E})^{u=1}
\ar[d]_{}^{}\ar[r] &((\rho^{*} \mathcal{E})\otimes^{\mathbb{L}}_{A} A/(q^{p^{\alpha}}-1))[-1]\ar[d]
\\ \rho^{*} \mathcal{E}\ar^{\Id}[r]\ar^{\partial_{\mathcal{E}}}[d]&\rho^{*} \mathcal{E} \ar[d]^{\gamma_u-\Id}\ar[r] &0\ar[d]
\\ \rho^{*} \mathcal{E} \ar^{\cdot q(q^{p^{\alpha}}-1)}[r] &\rho^{*} \mathcal{E}\ar[r] &(\rho^{*} \mathcal{E})\otimes^{\mathbb{L}}_{A} A/(q^{p^{\alpha}}-1)
,}\]
hence the first row is also a fiber sequence, we thus finish the proof of full faithfulness.

For the essential image of the functor $F$, first we observe that an object $\mathcal{E}\in \mathcal{C}$ can be identified with a pair $(M,\gamma_M)$ where $M\in \mathcal{D}(A)$ and $\gamma_M$ is an automorphism of $M$ congruent to the identity modulo $q(q^{p^{\alpha}}-1)$.

Given $\mathcal{E}\in \mathcal{D}(X^{\Prism})$, to check that the action of $\partial_{\mathcal{E}}$ on the cohomology $\mathrm{H}^*(\rho^{*}(\mathcal{E})\otimes^{\mathbb{L}}k)$ is locally nilpotent, again thanks to \cref{propt.generation}, we might assume $\mathcal{E}=(\mathcal{I})^r$ for some $r\in \mathbb{Z}$. Then by \cref{example.action on generators}, after base change to $k$ the action of $\partial_{\mathcal{E}}$ is given by multiplication by $re$,
    which already vanishes as our assumption ($p>2$ or $\alpha>0$) guarantees that $\mathcal{O}_K$ is ramified over $W(k)$, hence $e=0$ in $k$. 

    Let $\mathcal{J}\subseteq \mathcal{C}$ be the full subcategory spanned by objects satisfying two conditions listed in \cref{thmt.main classification}. As the source $\mathcal{D}(X^{\Prism})$ is generated under shifts and colimits by the invertible sheaves $\mathcal{I}^n$ for $n\in \mathbb{Z}$ by \cref{propt.generation}, to complete the proof it suffices to show that $\mathcal{J}$ is also generated under shifts and colimits by $\{F(\mathcal{I}^k)\}$ ($k\in \mathbb{Z}$). For this purpose, it can be reduced to the Hodge-Tate locus. Indeed, if $\mathcal{E}$ is in $\mathcal{J}$ and $\RHom_{\mathcal{C}}(F(\mathcal{I}^k), \mathcal{E})\cong \RHom_{\mathcal{C}}(F(\mathcal{O}),F(\mathcal{I}^{-k})\otimes\mathcal{E})=0$ for all $k\in \mathbb{Z}$, then we wish to show that $\mathcal{E}=0$, which is equivalent to the vanishing of $\mathcal{E}|_{\mathcal{C}^{\HT}}$ due to the $(p,d)$-completeness assumption on $\mathcal{E}$. Here we denote $\mathcal{C}^{\HT}$ as the fiber product 
     $$\mathcal{D}([\Spf(A/d)/(1+p^{\alpha+1}\mathbb{Z}_p)^{\times}])\times_{\mathcal{D}([\Spf(A/(d,q^{p^{\alpha}}-1))/(1+p^{\alpha+1}\mathbb{Z}_p)^{\times}])} \mathcal{D}(\Spf(A/(d,q^{p^{\alpha}}-1)))
     $$
     and there is a natural restriction functor from $\mathcal{C}$ to $\mathcal{C}^{\HT}$, under which $\mathcal{E}$ is sent to $\mathcal{E}|_{\mathcal{C}^{\HT}}$.
    
    By considering the fiber sequence
\begin{equation*} \RHom_{\mathcal{C}}(F(\mathcal{O}),F(\mathcal{I}^{1-k})\otimes\mathcal{E})\to \RHom_{\mathcal{C}}(F(\mathcal{O}),F(\mathcal{I}^{-k})\otimes\mathcal{E}) \to\RHom_{\mathcal{C}^{\HT}}(F(\mathcal{O})|_{\mathcal{C}^{\HT}},(\mathcal{I}^{-k}\mathcal{E})|_{\mathcal{C}^{\HT}}),
\end{equation*}
we deduce that $\RHom_{\mathcal{C}^{\HT}}(F(\mathcal{I}^k)|_{\mathcal{C}^{\HT}}, \mathcal{E}|_{\mathcal{C}^{\HT}})$ vanishes for all $k$. We will show that this already guarantees the vanishing of $\mathcal{E}|_{\mathcal{C}^{\HT}}$.
    
    Indeed, we argue that for every nonzero object $\mathcal{E}_0=(M, \gamma_M)\in \mathcal{C}^{\HT}$ (here $M\in \mathcal{D}(\mathcal{O}_K)$, $\gamma_M=\Id+q(q^{p^{\alpha}}-1)\partial_M$ is an automorphism of $M$ which is congruent to the identity modulo $q(q^{p^{\alpha}}-1)$) which further satisfies that $M$ is $p$-complete and that the action of $\partial_M$ on the cohomology $\mathrm{H}^*(M\otimes^{\mathbb{L}}k)$ is locally nilpotent, $\RHom_{\mathcal{C}^{\HT}}(F(\mathcal{O}), \mathcal{E}_0)\neq 0$. For this purpose, first we observe that $\RHom_{\mathcal{C}^{\HT}}(F(\mathcal{O}), \mathcal{E}_0)=\fib(M\stackrel{\partial_M}{\longrightarrow} M)$ by repeating the proof of the full faithfulness. Replacing $M$ by $M\otimes k$ (the derived Nakayama guarantees that $L\otimes k$ detects whether $L$ is zero or not for $p$-complete $L$), we may assume that there exists some cohomology group $\mathrm{H}^{-m}(M)$ containing a nonzero element killed by $\partial_M$ (this could be done by iterating the action of $\partial$ and then use the nilpotence assumption). 
    It then follows that there exists a non-zero morphism from $F(\mathcal{O})[m])$ to $\mathcal{E}_0$, hence we finish the proof.
\end{proof}
\begin{remark}\label{rem. truncated main theorem}
    The above theorem still holds if we replace $X^{\Prism}$ with $X^{\Prism}_n$ and substitute $A$ by $A/d^n$ in the statement, hence we get the corresponding classification of truncated prismatic crystals on $X_{\Prism}$ as well. The proof is exactly the same. 
\end{remark}

\subsection{Applications to quasi-coherent complexes on the Cartier-Witt stack}
In this subsection we take $\alpha=0$ and consider the prism $(A,I)=(\mathbb{Z}_p[[q-1]], [p]_q)\in X_{\Prism}$ for $X=\Spf(\mathcal{O}_K)=\Spf(\mathbb{Z}_p[\zeta_p])$. In other words, we specialize the discussion in previous subsections to $k=\mathbb{F}_p, \alpha=0$.

First we recollect some basics for the $\mathbb{F}_p^{\times}$-action on $A$ discussed in \cite[Section 3.8]{bhatt2022absolute}.  Recall that $A=\mathbb{Z}_p[[q-1]]$ admits an action of $\mathbb{F}_p^{\times}$, which carries each element $e\in \mathbb{F}_p^{\times}$ to the automorphism of $\mathbb{Z}_p[[q-1]]$ given by $\gamma_e: q\mapsto q^{[e]}$, for which $[e]$ is the Teichmuller lift of $e$. Following \cite{bhatt2022absolute}, we let $\Tilde{p}\in A$ denote the sum 
\begin{equation*}
    \sum_{e\in \mathbb{F}_p} q^{[e]}=1+\sum_{e\in \mathbb{F}_p^{\times}} q^{[e]},
\end{equation*}
then $\Tilde{p}$ is invariant under the action of $\mathbb{F}_p^{\times}$ and differs a unit from $d$ by \cite[Proposition 3.8.6]{bhatt2022absolute}. Consequently, $\gamma_e$ preserves the ideal $(d)$, hence induces an automorphism $\gamma_e: A/d^n\to A/d^n$ for any $n\geq 1$. In particular, take $n=1$, we have an automorphism $\gamma_e$ form $\mathcal{O}_K$ to itself.

As an upshot, we could define an $\mathbb{F}_p^{\times}$-action on $X^{\Prism}$ as follows
\begin{equation*}
    \begin{split}
        X^{\Prism}(R)&\longrightarrow  X^{\Prism}(R)
        \\ (\alpha: I\to W(R), \eta: \mathcal{O}_K\to \Cone(\alpha))&\longmapsto (\alpha, \eta\circ \gamma_e: \mathcal{O}_K\xrightarrow{\gamma_e} \mathcal{O}_K \xrightarrow{\eta} \Cone(\alpha)).
    \end{split}
\end{equation*}

Similarly, $\mathbb{F}_p^{\times}$ acts on $X^{\Prism}_n$ for any $n\geq 1$.
\begin{proposition}\label{prop. AD descent morphism}
    Let $n\in \mathbb{N}\cup \{\infty\}$. The natural morphism $X^{\Prism}_n \to \WCart_n$ is $\mathbb{F}_p^{\times}$-equivariant, where the group $\mathbb{F}_p^{\times}$ acts trivially on $\WCart_n$. Consequently it descends to a morphism $\pi: [X^{\Prism}_n/\mathbb{F}_p^{\times}] \to \WCart_n$.
\end{proposition}
\begin{proof}
    It follows from that the $\mathbb{F}_p^{\times}$ action on $X^{\Prism}$ fixes the Cartier-Witt divisor $\alpha$. 
\end{proof}

    Heuristically as the coarse moduli space for $[X/\mathbb{F}_p^{\times}]$ is $\Spf(\mathbb{Z}_p)$, $\WCart$ could be viewed as the ``course moduli space" for $[X^{\Prism}/\mathbb{F}_p^{\times}]$. The next result provides some evidence for such a philosophy.
    \begin{theorem}\label{thm. full faithful embedding}
        Assume that $p>2$. Let $n\in \mathbb{N}\cup \{\infty\}$. The pullback along $\pi$ induces a fully faithful functor of $\infty$-categories
        $$\mathcal{D}(\WCart_n) \longrightarrow \mathcal{D}([X^{\Prism}_n/\mathbb{F}_p^{\times}]).$$
    \end{theorem}
    \cref{thm. full faithful embedding} is equivalent to the following a \textit{prior} weaker result:
    \begin{theorem}\label{thm. hodge state full faithful}
        Assume that $p>2$. Let $\mathcal{E}\in \mathcal{D}(\WCart^{\HT})$, then $$\mathrm{R} \Gamma(\WCart^{\HT}, \mathcal{E})\xrightarrow{\simeq}\mathrm{R} \Gamma([X^{\HT}/\mathbb{F}_p^{\times}], \pi^*\mathcal{E}).$$
    \end{theorem}
    \begin{proof}[Proof of \cref{thm. full faithful embedding} from \cref{thm. hodge state full faithful}]
        Let $\mathcal{E}$ and $\mathcal{F}$ be quasi-coherent complexes on $X^\Prism_n$ and we want to show that the natural map 
    \begin{equation*}
        \Hom_{\mathcal{D}(\WCart_n)}(\mathcal{E}, \mathcal{F}) \rightarrow \Hom_{\mathcal{D}([X_n^{\Prism}/\mathbb{F}_p^{\times}])} (\pi^{*}(\mathcal{E}), \pi^{*}(\mathcal{F}))
    \end{equation*}
    is a homotopy equivalence. By \cite[Corollary 3.5.16]{bhatt2022absolute}, $\mathcal{D}(\WCart_{n})$ is generated under colimits by $\mathcal{I}^n$ , we could reduce to the case that $\mathcal{E}=\mathcal{I}^k$ for some $k\in \mathbb{Z}$. Replacing $\mathcal{F}$ by the twist $\mathcal{F}(-k)$, we could further assume that $k=0$. Then it suffices to show that for $\mathcal{E}\in \mathcal{D}(\WCart_{n})$, $$\mathrm{R} \Gamma(\WCart_n, \mathcal{E})=\mathrm{R} \Gamma([X^{\Prism}_n/\mathbb{F}_p^{\times}], \pi^*\mathcal{E}).$$
    For any $\mathcal{E}\in \mathcal{D}(X^{\Prism})$, as taking global sections commutes with limits, by writing $\mathcal{E}$ as the inverse limit of $\mathcal{E}_k$ for $\mathcal{E}_k$ the restriction of $\mathcal{E}$ to $X^{\Prism}_k$ ($k\leq n$), we are reduced to the case $n=1$, which is precisely \cref{thm. hodge state full faithful}.
    \end{proof}

For the proof of \cref{thm. hodge state full faithful}, we need to develop a $\mathbb{F}_p^{\times}$-equivariant version of the isomorphism $\gamma_{b,c}$ constructed in the beginning of this section. 
\begin{proposition}\label{prop. cyclic group action on q prism}
    The covering morphism $\rho: \Spf(A)\to X^{\Prism}$ is $\mathbb{F}_p^{\times}$-equivariant, where the source and target are equipped with the $\mathbb{F}_p^{\times}$-action discussed at the beginning at this subsection. Consequently, it descends to a morphism $ [\Spf(A)/\mathbb{F}_p^{\times}]\to [X^{\Prism}/\mathbb{F}_p^{\times}]$, which will still be denoted as $\rho$ by abuse of notation.
\end{proposition}
\begin{proof}
    By \cite[Definition 1.3]{romagny2005group} and the discussion above that, we need to specify a $\sigma$ satisfying certain certain ``higher associativity" condition such that the following diagram commutes
    \[\xymatrixcolsep{5pc}\xymatrix{\mathbb{F}_p^{\times}\times \Spf(A)\ar[d]^{\rho}\ar[r]^{}\ar@{=>}[dr]^{\sigma}& \Spf(A) \ar[d]_{}^{\rho}
\\\mathbb{F}_p^{\times}\times X^{\Prism}  \ar[r]&X^{\Prism}}\]
Given $u\in \mathbb{F}_p^{\times}$, it induces an automorphism $\gamma_u$ on $A=\mathbb{Z}_p[[q-1]]$ by sending $q$ to $q^{[u]}$, where $[u]$ is the Teichmuller lift of $u$. Moreover, one can check that $\gamma_u$ preserves the $\delta$-structure. Then unwinding the definitions, we see that the composition of $\rho$ and the top arrow corresponds to the following point in $X^{\Prism}(A)$
$$(\alpha^{\prime}:(d) \otimes_{A,\Tilde{\gamma}_u} W(A)\to W(A),\eta^{\prime}: \Cone((d)\to A)\xrightarrow{\Tilde{\gamma}_u} \Cone(\alpha^{\prime})).$$
Similarly, the composition of the bottom arrow and $\rho$-corresponds to the point 
$$(\alpha:(d) \otimes_{A,\Tilde{\iota}} W(A)\to W(A),\eta: \Cone((d)\to A)\xrightarrow{\Tilde{\iota}\circ \gamma_u} \Cone(\alpha)),$$
where $\Tilde{\gamma}_u$ (resp. $\Tilde{\iota}$) is the unique lift of $\gamma_u: A\to A$ (resp. $\Id: A\to A$) to a $\delta$-ring morphism $A\to W(A)$. Moreover, as $\gamma_u$ is a $\delta$-ring homomorphism, $\Tilde{\gamma}_u=\Tilde{\iota}\circ \gamma_u$.

Suppose $\gamma_u(d)=d\cdot d_u$, where $d_u\in A^{\times}$. Then $\Tilde{\gamma}_u(d)=\Tilde{\iota}\circ \gamma_u(d)=\Tilde{\iota}(d\cdot d_u)=\Tilde{\iota}(d)\Tilde{\iota}(d_u)$, hence we could construct an automorphism of Cartier-Witt divisors $\sigma_u: \alpha^{\prime}\xrightarrow{\simeq} \alpha$ as follows:   
 \[\xymatrixcolsep{5pc}\xymatrix{(d) \otimes_{A,\Tilde{\gamma}_u} W(A) \ar[d]^{(d)\otimes x\mapsto (d)\otimes \Tilde{\iota}(d_u)x}\ar[r]^{}& W(A) \ar[d]_{}^{\Id}
\\(d) \otimes_{A,\Tilde{\iota}} W(A)\ar[r]^{}&W(A)}\]
Here the left vertical map is $W(A)$-linear and the commutativity of the diagram follows from our construction of $d_u$.

Then one could easily see that $\sigma_u \circ \eta^{\prime}=\eta$ as $\Tilde{\gamma}_u=\Tilde{\iota}\circ \gamma_u$, such a data together with $\sigma_u$ then gives the desired $\sigma$.

We leave the reader to check that $\sigma$ satisfies the ``higher associativity" condition in \cite[Definition 1.3]{romagny2005group}.
\end{proof}
\begin{remark}
    By \cite[Corollary 3.8.8]{bhatt2022absolute}, $A^{\mathbb{F}_p^{\times}}=\mathbb{Z}_p[[\Tilde{p}]]$, hence it's not hard to see that $\Spf(\mathbb{Z}_p[[\Tilde{p}]])$ is the coarse moduli space for $[\Spf(A)/\mathbb{F}_p^{\times}]$.
\end{remark}

Recall that the construction of the $q$-connection in the first subsection replies heavily on the natural embedding $A\to R=A[\epsilon]/(\epsilon^2-q(q-1)\epsilon)$, which could be promoted to a $\mathbb{F}_p^{\times}$-equivariant embedding.
\begin{lemma}\label{lem. Fp action on R}
    We can equip $R$ with a $\mathbb{F}_p^{\times}$-action by extending that on $A$ and requiring that $$\gamma_u(\epsilon)=\epsilon\cdot q^{[u]-1}\frac{q^{[u]}-1}{q-1},$$ 
    then both the natural embedding $A\to R=A[\epsilon]/(\epsilon^2-q(q-1)\epsilon)$ and $\psi: A\to R$ defined in \cref{lemt.extend delta} are $\mathbb{F}_p^{\times}$-equivariant.
\end{lemma}
\begin{proof}
    To show that the desired $\mathbb{F}_p^{\times}$ action on $R$ is well defined, we just need to check that $\gamma_u(\epsilon^2)=\gamma_u(q(q-1)\epsilon)$, which is left to the reader. To see $\psi$ is $\mathbb{F}_p^{\times}$-equivariant, it suffices to check that $\psi(\gamma_u(q))=\gamma_u(\psi(q))$, but this follows from the following calculation using \cref{lemt.extend delta}.
    \begin{equation*}
        \begin{split}
\psi(\gamma_u(q))&=q^{[u]}+\epsilon[p[u]]_q q^{[u]-1}=q^{[u]}+\epsilon\cdot q^{[u]-1}\cdot \frac{q^{[u]p}-1}{q-1},
\\\gamma_u(\psi(q))&=\gamma_u(q+\epsilon \frac{q^p-1}{q-1})=q^{[u]}+\epsilon \cdot q^{[u]-1}\frac{q^{[u]}-1}{q-1}\frac{q^{[u]p}-1}{q^{[u]}-1}=q^{[u]}+\epsilon\cdot q^{[u]-1}\cdot \frac{q^{[u]p}-1}{q-1}.
        \end{split}
    \end{equation*}
\end{proof}
By mimicking the proof of \cref{prop. cyclic group action on q prism}, we have that the $\rho: \Spf(R)\to X^{\Prism}$ is also $\mathbb{F}_p^{\times}$-equivariant. As an upshot, we can promote \cref{prop. cyclic group action on q prism} to a $\mathbb{F}_p^{\times}$-equivariant diagram.
\begin{proposition}
    The diagram constructed in \cref{propt.key automorphism of functors} and \cref{cort.truncated} are both $\mathbb{F}_p^{\times}$-equivariant. Hence it descends to a commutative diagram
\[\xymatrixcolsep{5pc}\xymatrix{[\Spf(R)/\mathbb{F}_p^{\times}]\ar[d]^{\rho}\ar[r]^{\psi}& [\Spf(A)/\mathbb{F}_p^{\times}]\ar@{=>}[dl]^{\gamma_{b,c}} \ar[d]_{}^{\rho}
\\[X^{\Prism}/\mathbb{F}_p^{\times}]  \ar^{\Id}[r]&[X^{\Prism}/\mathbb{F}_p^{\times}]}\]
\end{proposition} 
\begin{proof}
    Given \cref{prop. cyclic group action on q prism} and \cref{lem. Fp action on R}, all of the arrows in \cref{propt.key automorphism of functors} are $\mathbb{F}_p^{\times}$-equivariant.
\end{proof}
\begin{corollary}\label{cor.Fp equivariance of theta}
    For $\mathcal{E}\in \mathcal{D}([X^{\HT}/\mathbb{F}_p^{\times}])$, $\gamma_{b,c}: \rho^{*}\mathcal{E}\otimes_{\mathcal{O}_K}\mathcal{O}_K [\epsilon]/(\epsilon^2-\beta \epsilon) \xrightarrow{\simeq} \rho^{*}\mathcal{E}\otimes_{\mathcal{O}_K}\mathcal{O}_K [\epsilon]/(\epsilon^2-\beta \epsilon)$ is $\mathbb{F}_p^{\times}$-equivariant.
\end{corollary}
\begin{proof}
    This follows from last proposition and the construction of $\gamma_{b,c}$.
\end{proof}

Now we are ready to prove \cref{thm. hodge state full faithful}.
\begin{proof}[Proof of \cref{thm. hodge state full faithful}]
 Given $\mathcal{E}\in \mathcal{D}(\WCart^{\HT})$, we wish to show that $$\mathrm{R} \Gamma(\WCart^{\HT}, \mathcal{E})\xrightarrow{\simeq}\mathrm{R} \Gamma([X^{\HT}/\mathbb{F}_p^{\times}], \pi^*\mathcal{E}).$$ As both sides are $p$-complete, it suffices to prove that $$\mathrm{R} \Gamma(\WCart^{\HT}, \mathcal{E})\otimes \mathbb{F}_p \xrightarrow{\simeq}\mathrm{R} \Gamma([X^{\HT}/\mathbb{F}_p^{\times}], \pi^*\mathcal{E})\otimes \mathbb{F}_p.$$
The strategy is to calculate both sides explicitly and then identify them after modulo $p$. Actually, applying \cite[Proposition 3.5.11]{bhatt2022absolute}, we see that
\begin{equation}\label{equa. left sides}
    \mathrm{R} \Gamma(\WCart^{\HT}, \mathcal{E})\otimes \mathbb{F}_p=\fib(\eta^{*}\mathcal{E}/p\xrightarrow{\theta} \eta^{*}\mathcal{E}/p).
\end{equation}
On the other hand, if we abuse the notation by denoting the pullback of $\mathcal{E}$ to $\mathcal{D}(X^{\HT})$ still as $\pi^*\mathcal{E}$, then
\begin{equation*}
    \begin{split}
        &\mathrm{R} \Gamma([X^{\HT}/\mathbb{F}_p^{\times}], \pi^*\mathcal{E})=\mathrm{R}\Gamma(X^{\HT}, \pi^*\mathcal{E})^{\mathbb{F}_p^{\times}}\\=&(\fib(\rho^{*}\pi^{*}\mathcal{E}\xrightarrow{\gamma_{b,c}-\Id=\partial_{\pi^{*}\mathcal{E}}\cdot\epsilon}\rho^{*}\pi^{*}\mathcal{E}\cdot\epsilon))^{\mathbb{F}_p^{\times}}
\\=&\fib((\rho^{*}\pi^{*}\mathcal{E})^{\mathbb{F}_p^{\times}} \xrightarrow{\partial_{\pi^{*}\mathcal{E}}\cdot\epsilon}(\rho^{*}\pi^{*}\mathcal{E}\cdot\epsilon)^{\mathbb{F}_p^{\times}}),
        \end{split}
\end{equation*}
here the second equality follows from \cref{propt. sen calculate cohomology}, the last identity holds since $\mathbb{F}_p^{\times}$ is a finite group of order relatively prime to $p$ (hence no higher cohomology when taking $\mathbb{F}_p^{\times}$ invariants on $p$-complete complexes) and $\gamma_{b,c}$ is $\mathbb{F}_p^{\times}$-equivariant by \cref{cor.Fp equivariance of theta}.

Next we calculate $(\rho^{*}\pi^{*}\mathcal{E}\cdot\epsilon)^{\mathbb{F}_p^{\times}}$. By \cref{lem. commute two prism on carrier witt stack}, the following square is commutative up to the isomorphism we constructed in the proof of it,
\[\xymatrixcolsep{5pc}\xymatrix{[\Spf(\mathcal{O}_K)/\mathbb{F}_p^{\times}]\ar[d]^{f}\ar[r]^{\rho}& [X^{\HT}/\mathbb{F}_p^{\times}]
\ar[d]_{}^{\pi}
\\ \Spf(\mathbb{Z}_p) \ar^{\eta}[r]&\WCart^{\HT}.}\]
From now on we fix this isomorphism $\eta\circ f\simeq \pi\circ \rho$. Then $$(\rho^{*}\pi^{*}\mathcal{E})^{\mathbb{F}_p^{\times}}=Rf_{*}(\rho^{*}\pi^{*}\mathcal{E})\simeq Rf_{*}(f^{*}\eta^{*}\mathcal{E})=\eta^{*}\mathcal{E}.$$
Similarly, considering the commutative (up to an isomorphism) diagram
\[\xymatrixcolsep{5pc}\xymatrix{[\Spf(\mathcal{O}_K[\epsilon]/(\epsilon^2-\beta \epsilon))/\mathbb{F}_p^{\times}]\ar[d]^{f}\ar[r]^{\rho^{\prime}}& [X^{\HT}/\mathbb{F}_p^{\times}]
\ar[d]_{}^{\pi}
\\ \Spf(\mathbb{Z}_p) \ar^{\eta}[r]&\WCart^{\HT}}\]
leads to that
$$(\rho^{*}\pi^{*}\mathcal{E}\otimes_{\mathcal{O}_K} \mathcal{O}_K[\epsilon]/(\epsilon^2-\beta\epsilon))^{\mathbb{F}_p^{\times}}=Rf_{*}(\rho^{\prime *}\pi^{*}\mathcal{E})\simeq Rf_{*}(f^{*}\eta^{*}\mathcal{E})=\eta^{*}\mathcal{E}\otimes_{\mathbb{Z}_p} \mathbb{Z}_p[v]/(v^2-pv).$$
Here $\rho^{\prime}$ is the composition of $[\Spf(\mathcal{O}_K[\epsilon]/(\epsilon^2-\beta \epsilon))/\mathbb{F}_p^{\times}]\to [\Spf(\mathcal{O}_K)/\mathbb{F}_p^{\times}]$ and $\rho$ (hence the commutativity follows from \cref{lem. commute two prism on carrier witt stack} as well), and the last equality is due to the projection formula and \cref{lem. calculate invariants}. In particular, $v$ is identified with $e\epsilon$. 

In summary, we have that 
\begin{equation}\label{equa. divide q connection}
\begin{split}
    \mathrm{R} \Gamma([X^{\HT}/\mathbb{F}_p^{\times}], \pi^*\mathcal{E})&\xrightarrow{\simeq}\fib(\eta^{*}\mathcal{E} \xrightarrow{\partial_{\pi^{*}\mathcal{E}}\cdot \epsilon} \eta^{*}\mathcal{E}\cdot v)=\fib(\eta^{*}\mathcal{E} \xrightarrow{\frac{\partial_{\pi^{*}\mathcal{E}}}{e}\cdot v} \eta^{*}\mathcal{E}\cdot v),
    \\\mathrm{R} \Gamma([X^{\HT}/\mathbb{F}_p^{\times}], \pi^*\mathcal{E})\otimes_{\mathbb{Z}_p} \mathbb{F}_p &\xrightarrow{\simeq} \fib(\eta^{*}\mathcal{E}/p \xrightarrow{\frac{\partial_{\pi^{*}\mathcal{E}}}{e}\cdot v} \eta^{*}\mathcal{E}/p \otimes_{\mathbb{F}_p} \mathbb{F}_p[v]/(v^2)).
\end{split}
\end{equation}
Here we use $\frac{\partial_{\pi^{*}\mathcal{E}}}{e}$ to denote the monodromy operator on $\eta^{*}\mathcal{E}$ itself induced by the $\mathbb{F}_p^{\times}$-invariant parts of $\gamma_{b,c}^{}$. Comparing with \cref{equa. left sides} and unwinding all of our constructions and chosen isomorphism, we have a commutative diagram 
\[\xymatrixcolsep{5pc}\xymatrix{\mathrm{R} \Gamma(\WCart^{\HT}, \mathcal{E})\otimes \mathbb{F}_p\ar[d]^{}\ar[r]^{\simeq}& \fib(\eta^{*}\mathcal{E}/p \ar[d]^{\Id} \ar[r]^{\theta_{\mathcal{E}}}& \eta^{*}\mathcal{E}/p)
\ar[d]_{}^{\Id}
\\ \mathrm{R} \Gamma([X^{\HT}/\mathbb{F}_p^{\times}], \pi^*\mathcal{E})\otimes_{\mathbb{Z}_p} \mathbb{F}_p \ar[r]^{\simeq}&\fib(\eta^{*}\mathcal{E}/p \ar[r]^{\frac{\partial_{\pi^{*}\mathcal{E}}}{e}}& \eta^{*}\mathcal{E}/p).}\]
Hence we are left to show that $\frac{\partial_{\pi^{*}\mathcal{E}}}{e}$ coincides with $\theta_{\mathcal{E}}$ constructed in \cite{bhatt2022absolute} after modulo $p$. This requires a closer review of the construction of such monodromy operators. First we claim that the induced $q$-connection $\frac{\partial_{\pi^{*}\mathcal{E}}}{e}$ on $\eta^{*}\mathcal{E}$ shown up in \cref{equa. divide q connection} precisely corresponds to the monodromy operator $\partial_{v,\mathcal{E}}$ induced by $1+v\in \in \mathbb{G}_m^{\sharp}(\mathbb{Z}_p[v]/(v^2-pv))$, viewed as an element in the automorphism group of \[\WCart^{\HT}\times_{\Spf(\mathbb{Z}_p)} \Spf(\mathbb{Z}_p[v]/(v^2-pv).\]
Indeed, following the discussion above \cref{lem. calculate e}, we could identify $\gamma_{b,c}$ with  $(1+e\epsilon, \epsilon)\in G_\pi(\mathcal{O}_K[\epsilon]/(\epsilon^2-q(q^{}-1)\epsilon))$. Hence under the natural projection $G_\pi \to \mathbb{G}_m^\sharp$, $\gamma_{b,c}\in \mathbb{G}_m^\sharp(\mathcal{O}_K[\epsilon]/(\epsilon^2-q(q^{}-1)\epsilon))$ (by abuse of notation) precisely corresponds to the restriction of $1+v\in \mathbb{G}_m^{\sharp}(\mathbb{Z}_p[v]/(v^2-pv))$ via the structure morphism $\Spf(\mathcal{O}_K[\epsilon]/(\epsilon^2-q(q^{}-1)\epsilon)) \to \Spf(\mathbb{Z}_p[v]/(v^2-pv))$ sending $v$ to $e\epsilon$. This implies that under our fixed isomorphism $\eta\circ f\simeq \pi\circ \rho$, $\partial_{\pi^{*}\mathcal{E}}: \rho^{*}\pi^{*}\mathcal{E}\to \rho^{*}\pi^{*}\mathcal{E}$ is identified with $$\rho^{*}\pi^{*}\mathcal{E}\simeq \eta^{*}\mathcal{E}\otimes_{\mathbb{Z}_p} \mathcal{O}_K\xrightarrow{\partial_{v,\mathcal{E}}\otimes e\cdot\Id} \eta^{*}\mathcal{E}\otimes_{\mathbb{Z}_p} \mathcal{O}_K\simeq \rho^{*}\pi^{*}\mathcal{E},$$
hence $\frac{\partial_{\pi^{*}\mathcal{E}}}{e}=\partial_{v,\mathcal{E}}$ on $\eta^{*}\mathcal{E}$.

Finally we are left to compare $\partial_v$ and $\theta$ after modulo $p$. A key observation is that after modulo $p$, $\gamma=1+v \in \mathbb{G}_m^{\sharp}(\mathbb{F}_p[v]/v^2)$ and all higher divided powers of $v$ vanish, which crucially uses the assumption that $p>2$ to guarantee that $\frac{v^n}{n!}=\frac{p^{n-1}v}{n!}$ is divisible by $p$ for all $n>1$. However, as $\theta$ on $\eta^*\mathcal{E}$ in \cite{bhatt2022absolute} is actually constructed via $1+[v] \in W[F](\mathbb{Z}_p[v]/v^2)$, which corresponds to $1+v\in \mathbb{G}_m^{\sharp}(\mathbb{Z}_p[v]/v^2)$ (with higher divided powers of $v$ vanished) and hence exactly reduces to our $\gamma$ after modulo $p$, we conclude that $\partial_v=\theta$ on $\eta^*\mathcal{E}/p$, hence finish the proof.
\end{proof}
The above proof implicitly tells us we could define a $q$-Higgs connection on $\eta^{*}\mathcal{E}$ for $\mathcal{E}\in \mathcal{D}(\WCart^{\HT})$, which we summarize as the following proposition since it will be used in the proof of \cref{thm. descent for the cartier witt stack to q prism} later.
\begin{proposition}\label{prop. new q derivation on hodge state stack}
    The chosen element $1+v \in \mathbb{G}_m^{\sharp}(\mathbb{Z}_p[v]/(v^2-pv))$ induces an endomorphism $\partial_{v,\mathcal{E}}: \eta^*\mathcal{E}\to \eta^*\mathcal{E} $ for any $\mathcal{E}\in \mathcal{D}(\WCart^{\HT})$ such that $\partial_{\pi^{*}\mathcal{E}}=\partial_{v,\mathcal{E}}\otimes e$ under the isomorphism $\rho^{*}\pi^{*}\mathcal{E}\simeq \eta^{*}\mathcal{E}\otimes_{\mathbb{Z}_p} \mathcal{O}_K$ induced by the chosen isomorphism in the proof of \cref{lem. commute two prism on carrier witt stack}.
    . Moreover, pullback along $\eta: \Spf(\mathbb{Z}_p)\to \WCart^{\HT}$ induces a fully faithful functor 
    \begin{align*}
         \mathcal{D}(\WCart^{\HT}) \rightarrow \mathcal{D}(\mathbb{Z}_p[\partial]), \qquad \mathcal{E}\mapsto (\rho^{*}(\mathcal{E}),\partial_{v,\mathcal{E}}),
    \end{align*} 
    with the essential image consisting of those $M\in \mathcal{D}(\mathbb{Z}_p[\partial])$ such that $M$ is $p$-complete and that $\partial^p-\partial$ acts on $\mathrm{H}^*(M\otimes^{\mathbb{L}}\mathbb{F}_p)$ in a locally nilpotent way.  
\end{proposition}
\begin{proof}
    The statement before the moreover part and the full faithfulness in the moreover part is already given in the proof of \cref{thm. hodge state full faithful}. For the essential image, following the proof of \cref{thmt.main classification}, we just need to check the given nilpotence condition holds for $\mathcal{E}=\mathcal{I}^k$ ($k\in \mathbb{Z}$), which can be checked easily using \cref{example. equiavariant cohomology}.
\end{proof}

The following several lemmas are used in the proof of \cref{thm. hodge state full faithful}.
\begin{lemma}\label{lem. commute two prism on carrier witt stack}
    The following square commutes up to a (non-unique) isomorphism.
    \[\xymatrixcolsep{5pc}\xymatrix{[\Spf(\mathcal{O}_K)/\mathbb{F}_p^{\times}]\ar[d]^{f}\ar[r]^{\rho}& [X^{\HT}/\mathbb{F}_p^{\times}]
\ar[d]_{}^{\pi}
\\ \Spf(\mathbb{Z}_p) \ar^{\eta}[r]&\WCart^{\HT}}\]
\end{lemma}
\begin{proof}
    It suffices to show that $\rho: \Spf(\mathcal{O}_K)\to \WCart$ (which is actually $\pi\circ \rho$, we abuse notation here) is $\mathbb{F}_p^{\times}$-equivariantly isomorphic to $\eta: \Spf(\mathcal{O}_K)\to \WCart$. By construction $\rho (\mathcal{O}_K\xrightarrow{\Id} \mathcal{O}_K)$ corresponds to the Cartier-Witt divisor 
    \[\alpha: (d)\otimes_{A,\Tilde{\iota}} W(\mathcal{O}_K)\to W(\mathcal{O}_K),\]
    where $\Tilde{\iota}: A\to W(\mathcal{O}_K)$ is the unique $\delta$-ring homomorphism lifting $A\to \mathcal{O}_K$. 
    
    On the other hand, $\eta (\mathcal{O}_K\xrightarrow{\Id} \mathcal{O}_K)$ corresponds to the Cartier-Witt divisor
    \[\alpha^{\prime}:  W(\mathcal{O}_K)\xrightarrow{V(1)} W(\mathcal{O}_K).\]

    Notice that $\Tilde{\iota}(d)=\sum_{i=0}^{p-1}[q^i]$ is mapped to $0$ under the projection map $W(\mathcal{O}_K)\to \mathcal{O}_K$, hence $\Tilde{\iota}(d)=V(x)$ for some $x\in W(\mathcal{O}_K)$. As $px=F(V(x))=\sum_{i=0}^{p-1}[q^{ip}]=p$ in $W(\mathcal{O}_K)$, $x$ much be $1$ since $W(\mathcal{O}_K)$ is $p$-torsion free. The following diagram 
\[\xymatrixcolsep{5pc}\xymatrix{(d)\otimes_{A,\Tilde{\iota}} W(\mathcal{O}_K)\ar[d]^{}\ar[r]^{\lambda: d \mapsto V(1)}& (V(1))
\ar[d]_{}^{\eta}
\\ W(\mathcal{O}_K) \ar^{\Id}[r]&W(\mathcal{O}_K)}\]
    then gives an isomorphism between $\rho$ and $\eta$.

Moreover, one can check that $\iota(\gamma_u(d))=V(1)=\gamma_u(V(1))$, hence the isomorphism we constructed above is $\mathbb{F}_p^{\times}$-equivariant, hence finish the proof.
\end{proof}
\begin{lemma}\label{lem. calculate invariants}
    Let $f: [\Spf(\mathcal{O}_K[\epsilon]/(\epsilon^2-\beta \epsilon))/\mathbb{F}_p^{\times}] \to \mathbb{Z}_p$ be the structure morphism, then $$Rf_{*}\mathcal{O}\simeq \mathbb{Z}_p[v]/(v^2-pv),$$ for which $v$ corresponds to $e\epsilon$.
\end{lemma}
\begin{proof}
   Since $\mathbb{F}_p^{\times}$ is a finite group of order relatively prime to $p$, $\mathrm{H}^{i}(\mathcal{O}_K, \mathbb{F}_p^{\times})=0$ for $i>0$, it then suffices to check that the $\mathbb{F}_p^{\times}$-invariants in $\mathcal{O}_K[\epsilon]/(\epsilon^2-\beta \epsilon)$ are given by $\mathbb{Z}_p[v]/(v^2-pv)$ with $v=e\epsilon$. By our notation, $e=\frac{p}{q(q-1)}\in \mathcal{O}_K$ (as usual $q$ means $\zeta_p$ in $\mathcal{O}_K$). Then $\forall u\in \mathbb{F}_p^{\times}$, as $\gamma_u(\epsilon)=\epsilon\cdot q^{[u]-1}\frac{q^{[u]}-1}{q-1}$ by definition, we see that
    \begin{equation*}
        \gamma_u(e\epsilon)=\frac{p}{q^{[u]}(q^{[u]}-1)}\cdot \epsilon\cdot q^{[u]-1}\frac{q^{[u]}-1}{q-1}=\epsilon\frac{p}{q(q-1)}=e\epsilon.
    \end{equation*}
    Hence $\mathbb{Z}_p \cdot e\epsilon \subseteq (\mathcal{O}_K/(\epsilon^2-\beta \epsilon))^{\mathbb{F}_p^{\times}}$. On the other hand, if $\epsilon x$ is fixed by the $\mathbb{F}_p^{\times}$ action, then an easy calculation shows that $\gamma_u(xq(q-1))=xq(q-1)$, hence $xq(q-1)\in \mathcal{O}_K^{\mathbb{F}_p^{\times}}=\mathbb{Z}_p$, but $x\in \mathcal{O}_K$, we see that $xq(q-1)$ must lives in $p\mathbb{Z}_p$, hence $\epsilon x\in \mathbb{Z}_p e\epsilon$. Let $v=e\epsilon$, then $v^2=e^2\epsilon^2=e^2\beta \epsilon=pe\epsilon=pv$ as $p=\beta e$.
    This implies that the $\mathbb{F}_p^{\times}$-invariants in $\mathcal{O}_K[\epsilon]/(\epsilon^2-\beta \epsilon)$ are given by $\mathbb{Z}_p[v]/(v^2-pv)$ with $v=e\epsilon$.
\end{proof}
\begin{example}[Explicit calculation for the ideal sheaf]\label{example. equiavariant cohomology}
    For $\mathscr{E}=\mathcal{I}^k \in \mathcal{D}(\WCart^{\HT})$, by \cref{example.action on generators} and \cref{prop. ht case fiber seq}, we have that 
     \[\mathrm{R} \Gamma(X^{\HT}, \pi^*\mathcal{E})=\fib(\mathcal{O}_K\xrightarrow{\cdot e\epsilon \frac{(1+p)^k-1}{p}} \mathcal{O}_K\cdot\epsilon).\]
    After passing to $\mathbb{F}_p^{\times}$-invariants and invoking \cref{lem. calculate invariants}, we get that
    \[\mathrm{R} \Gamma([X^{\HT}/\mathbb{F}_p^{\times}], \pi^*\mathcal{E})=\fib(\mathbb{Z}_p \xrightarrow{\cdot v \frac{(1+p)^k-1}{p}} \mathbb{Z}_p\cdot v). \]
    We explain why this directly implies that $\mathrm{R} \Gamma([X^{\HT}/\mathbb{F}_p^{\times}], \pi^*\mathcal{E})\simeq \fib(\mathbb{Z}_p \xrightarrow{\cdot k} \mathbb{Z}_p)$ when $p>2$. Actually, as $\frac{(1+p)^k-1}{p}=\sum_{i=1}^k \binom{k}{i}p^{i-1}$, it suffices to prove that $i+v_p(\binom{k}{i})>1+v_p(k)$ for $i>1$ under the assumption that $p>2$. If $v_p(k)=0$, the statement is trivial. Otherwise let $v_p(k)=n>0$. Denote $v_p(i)=t$ for a fixed $i>1$. We separate into two cases. If $t<n$, then 
    \begin{equation*}
        i+v_p(\binom{k}{i})=i+\frac{s(i)+s(n-i)-s(k)}{p-1}\geq i+\frac{(p-1)(n-t)}{p-1}=n+i-t\geq n+\max(p^t,i)-t>n+1,
    \end{equation*}
    here the last inequality holds as $p>2$. 
    
    On the other hand, if $t\geq n$, then $i\geq p^t\geq p^n$, hence $i+v_p(\binom{k}{i})\geq i\geq p^n>1+n$. Hence we finish the proof.
    
\end{example}
\begin{remark}
    \cref{thm. hodge state full faithful} and \cref{thm. full faithful embedding} fails when $p=2$. Indeed, by the previous example, $\mathrm{H}^{1}([X^{\HT}/\mathbb{F}_p^{\times}], \pi^* \mathcal{I}^2)=\mathbb{Z}/4$ while $\mathrm{H}^{1}(\WCart^{\HT}, \mathcal{I}^2)=\mathbb{Z}/2$ by \cite[Proposition 3.5.11, Corollary 3.5.14]{bhatt2022absolute}.
\end{remark}

As an application of \cref{thm. full faithful embedding}, we calculate the cohomology of the structure sheaf on the prismatic site of $\mathbb{Z}_p$.
\begin{proposition}\label{prop. cohomology of Zp}
    Assume that $p>2$. The Cartier-Witt stack $\WCart$ is of cohomological dimension $1$. Moreover, 
    \begin{equation*}
        \mathrm{H}^{0}(\WCart, \mathcal{O})=\mathbb{Z}_p, ~~~\quad \mathrm{H}^{1}(\WCart, \mathcal{O})\simeq \prod_{n\in \mathbb{N}, n\not\equiv p\mod p+1} \mathbb{Z}_p.
    \end{equation*}
\end{proposition}
\begin{proof}
    The first sentence follows from \cref{thmt.main classification} and \cref{thm. full faithful embedding} as there is no higher cohomology when taking $\mathbb{F}_p^{\times}$ invariants on $p$-complete complexes. Moreover, the proof of \cref{thm. full faithful embedding} provides us with a concrete complex to calculate $\mathrm{R} \Gamma(\WCart, \mathcal{O})$, given by 
    \begin{equation}\label{equa. descent cohomology}
        [(\mathbb{Z}_p[[q-1]])^{\mathbb{F}_p^{\times}}\xrightarrow{\partial} (\mathbb{Z}_p[[q-1]]\cdot \epsilon)^{\mathbb{F}_p^{\times}}],
    \end{equation}
    here the $\mathbb{F}_p^{\times}$ action on $\epsilon$ is given in \cref{lem. Fp action on R}. First we would like to write the target of $\partial$ above more explicitly. Notice that there is an injective morphism $r: (\mathbb{Z}_p[[q-1]]\cdot \epsilon)^{\mathbb{F}_p^{\times}}\to \mathbb{Z}_p[[q-1]]$ sending $a\epsilon$ to $aq(q-1)$. Moreover, given the $\mathbb{F}_p^{\times}$ action on $\epsilon$ regulated in \cref{lem. Fp action on R}, it is easy to see that the image of $r$ lands in $\mathbb{Z}_p[[\Tilde{p}]]=(\mathbb{Z}_p[[q-1]])^{\mathbb{F}_p^{\times}}$, where the identity is guaranteed by \cite[Corollary 3.8.8]{bhatt2022absolute}. 
    
    We claim that the image of $r$ consists exactly of those $g\in \mathbb{Z}_p[[\Tilde{p}]]$ such that $g(p)=0$. Indeed, if $g(\Tilde{p})\in \mathbb{Z}_p[[\Tilde{p}]]$ can be factored as $q(q-1)h$ in $\mathbb{Z}_p[[q-1]]$, then it is sent to $0$ under the projection ring homomorphism $\mathrm{pr}: \mathbb{Z}_p[[q-1]]\to \mathbb{Z}_p,\quad, q\mapsto 1$. But $\mathrm{pr}(\Tilde{p})=p$, hence $g(p)=0$. On the other hand, if $g(\Tilde{p})\in \mathbb{Z}_p[[\Tilde{p}]]$ satisfies that $g(p)=0$, then viewing $g$ as an element in $\mathbb{Z}_p[[q-1]]$, we have that $\mathrm{pr}(g)=0$. By the Weierstrass preparation theorem for formal power series, $g\in \mathbb{Z}_p[[q-1]]$ can be factored as $(q-1)h$ for some $h\in \mathbb{Z}_p[[q-1]]$. As $q$ is a unit in $\mathbb{Z}_p[[q-1]]$, $g=q(q-1)(q^{-1}h)$ is a factorization in $\mathbb{Z}_p[[q-1]]$, hence we finish the proof of the claim.

    Let $B\subseteq \mathbb{Z}_p[[\Tilde{p}]]$ be the subring consisting of those $g\in \mathbb{Z}_p[[\Tilde{p}]]$ such that $g(p)=0$. Then the above argument shows that $r$ induces an isomorphism (of $\mathbb{Z}_p$-modules) from $(\mathbb{Z}_p[[q-1]]\cdot \epsilon)^{\mathbb{F}_p^{\times}}$ to $B$. Moreover, note that $B$ could be identified with $\prod_{i\in \mathbb{N}} \mathbb{Z}_p$ by sending $(a_i)_{i\in \mathbb{N}}$ in the latter to $\sum_{i=0}^{\infty} a_ie_i\in B$ with $e_i:=\Tilde{p}^i(\Tilde{p}-p)$, $i\geq 0$.

    Let $f=r\circ \partial: \mathbb{Z}_p[[\Tilde{p}]] \to B$, then applying \cref{lem.leibniz} for $\alpha=0$ and unwinding our definition of $r$, it is easy to check that $f$ is a $\mathbb{Z}_p$-linear morphism satisfying that $f(xy)=f(x)y+xf(y)+f(x)f(y)$.
As $\Tilde{p}=\sum_{u\in \mathbb{F}_p} q^{[u]}=1+\sum_{u\in \mathbb{F}_p^{\times}} q^{[u]}$ by definition, applying \cref{lemt.extend delta} and we see that
\begin{equation*}
    f(\Tilde{p})=q(q-1)(\sum_{u\in \mathbb{F}_p^{\times}} q^{[u]-1}\frac{q^{p[u]}-1}{q-1})=\sum_{u\in \mathbb{F}_p^{\times}} q^{[u](p+1)}-\sum_{u\in \mathbb{F}_p^{\times}} q^{[u]},
\end{equation*}
By considering the degree and the constant term, this implies that there exists $a_i\in \mathbb{Z}_p$ for $1\leq i\leq p-1$ such that $f(\Tilde{p})=e_{p}+\sum_{i=1}^{p-1} (a_i e_i)$. In general, we claim that
\begin{equation}\label{equa. general formula for ast}
    (\ast)\quad f(\Tilde{p}^k)=e_{k(p+1)+p}+\clubsuit_k \quad \text{with}\quad \clubsuit_k\in \oplus_{l=0}^{k(p+1)+p-1}\mathbb{Z}_pe_{l}\subseteq B,~~ \forall k\geq 1.
\end{equation}

To prove $ (\ast)$, we first notice that for $m, j\geq 0$, $e_j\cdot e_m=\Tilde{p}^{j+m}(\Tilde{p}-p)^2=e^{j+m+1}+\spadesuit$ with $\spadesuit$ belongs to the $\mathbb{Z}_p$-subspace in $B$ generated by $\{e_0,\cdots,e_{j+m}\}$. Then we do induction on $k$. The base case $k=1$ was explained in the last paragraph. Suppose $ (\ast)$ holds up to $k$. Then as $f$ satisfies the Leibnitz rule $f(xy)=f(x)y+xf(y)+f(x)f(y)$, we have that
\begin{equation*}
    \begin{split}
f(\Tilde{p}^{k+1})=\Tilde{p}f(\Tilde{p}^{k})+\Tilde{p}^{k+1}f(\Tilde{p})+f(\Tilde{p})f(\Tilde{p}^{k})=e_{(k+1)(p+1)+p}+\clubsuit_{k+1}
    \end{split}
\end{equation*}
for some $\clubsuit_{k+1}\in \oplus_{l=0}^{(k+1)(p+1)+p-1}\mathbb{Z}_pe_{l}\subseteq B$ as desired.

Finally by \cref{equa. descent cohomology}, 
    \begin{equation*}
        \mathrm{R} \Gamma(\WCart, \mathcal{O})=\fib(\mathbb{Z}_p[[\Tilde{p}]]\xrightarrow{f} B),
    \end{equation*}
hence $(\ast)$ implies that $\mathrm{H}^{0}(\WCart, \mathcal{O})=\mathbb{Z}_p$ and $ \mathrm{H}^{1}(\WCart, \mathcal{O})\simeq \prod_{n\in \mathbb{N}, n\not\equiv p\mod p+1} \mathbb{Z}_p.$ We finish the proof.
\end{proof}

Unfortunately the functor constructed in \cref{thm. full faithful embedding} is not an equivalence. Indeed, note that the $\mathbb{F}_p^{\times}$-action on $(A,d)=(\mathbb{Z}_p[[q-1]], [p]_q)$ is trivial after modulo $q-1$, hence we have the following commutative diagram
\[\xymatrixcolsep{5pc}\xymatrix{[\Spf(\mathbb{Z}_p)/\mathbb{F}_p^{\times}]\ar[d]^{f}\ar[r]& [\Spf(A)/\mathbb{F}_p^{\times}] \ar[r]^{\rho} &[X^{\Prism}/\mathbb{F}_p^{\times}]
\ar[d]_{}^{\pi}
\\ \Spf(\mathbb{Z}_p) \ar^{\rho_{\dR}}[rr]&&\WCart^{},}\]
where $f$ is the structure morphism, $\rho_{\dR}$ is the de Rham point introduced in \cite[Example 3.2.6]{bhatt2022absolute}. The commutativity of the above diagram follows from the observation that $[p]_q$ is sent to $p$ after modulo $q-1$, hence $\rho$ coincides with the de Rham point on the vanishing locus of $q-1$. This diagram implies that the pullback functor along $\pi$ actually factors through 
 $$F: \mathcal{D}(\WCart) \longrightarrow \mathcal{D}([X^{\Prism}/\mathbb{F}_p^{\times}])\times_{[\Spf(\mathbb{Z}_p)/\mathbb{F}_p^{\times}]} \mathcal{D}(\Spf(\mathbb{Z}_p)),$$
 which turns out to be an equivalence of categories if we further assume a nilpotence condition on the target.

\begin{theorem}\label{thm. descent for the cartier witt stack to q prism}
    Assume $p>2$. Then $F: \mathcal{D}(\WCart) \longrightarrow \mathcal{D}_{}([X^{\Prism}/\mathbb{F}_p^{\times}])\times_{[\Spf(\mathbb{Z}_p)/\mathbb{F}_p^{\times}]} \mathcal{D}(\Spf(\mathbb{Z}_p))$ is fully faithful. Moreover, the essential image of $F$ consists of those $\mathcal{E}$ in the target whose underlying complex in $ \mathcal{D}_{}([X^{\Prism}/\mathbb{F}_p^{\times}])$ (still denoted as $\mathcal{E}$ by abuse of notation) satisfies the following additional condition:
    \begin{itemize}
        \item the $q$-Higgs connection $\partial_{\mathcal{E}_0}$ of $\mathcal{E}_0:=\mathcal{E}|_{X^{\HT}}$ can be factored as $e\circ \partial_{\mathcal{E}_0}^{\prime}$ such that the action of $(\partial_{\mathcal{E}_0}^{\prime})^p-\partial_{\mathcal{E}_0}^{\prime}$ on the cohomology $\mathrm{H}^*(\rho^*\mathcal{E}\otimes^{\mathbb{L}}k)$ is locally nilpotent, here $e=d^{\prime}(q)\in W(k)[\zeta_p]$.
    \end{itemize}
\end{theorem}
\begin{remark}
    The conditions we impose on the target of $F$ are very natural. Actually, we can consider a baby example: the natural projection $\pi: [\Spf(\mathcal{O}_K)/\mathbb{F}_p^{\times}] \to \Spf(\mathbb{Z}_p)$ induces a fully faithful embedding $\mathcal{D}(\Spf(\mathbb{Z}_p))\to 
    \mathcal{D}([\Spf(\mathcal{O}_K)/\mathbb{F}_p^{\times}])$ as $[\Spf(\mathcal{O}_K)/\mathbb{F}_p^{\times}]$ is a tame stack in the sense of \cite[Definition 3.1]{abramovich2008tame}, and the essential image of such a pullback functor consists of those $M\in [\Spf(\mathcal{O}_K)/\mathbb{F}_p^{\times}]$ such that $f^{*}M$ has a $\mathbb{F}_p^{\times}$-invariant basis, where $f: \Spf(\mathcal{O}_K) \to [\Spf(\mathcal{O}_K)/\mathbb{F}_p^{\times}]$ is the covering map. However, such a condition can be checked after reducing to the special fiber $k=\mathbb{F}_p$ by derived Nakayama lemma, which leads to the following slightly fancy way to state such a result
    \begin{equation*}
        \mathcal{D}(\Spf(\mathbb{Z}_p))\xrightarrow{\simeq} 
    \mathcal{D}([\Spf(\mathcal{O}_K)/\mathbb{F}_p^{\times}])\times_{\mathcal{D}([\Spec(\mathbb{F}_p)/\mathbb{F}_p^{\times}])} \mathcal{D}(\Spec(\mathbb{F}_p)).
    \end{equation*}
\end{remark}
\begin{proof}[Proof of \cref{thm. descent for the cartier witt stack to q prism}]
We use $\mathcal{C}$ to denote the subcategory of $\mathcal{D}_{}([X^{\Prism}/\mathbb{F}_p^{\times}])\times_{[\Spf(\mathbb{Z}_p)/\mathbb{F}_p^{\times}]} \mathcal{D}(\Spf(\mathbb{Z}_p))$ satisfying the additional nilpotence conditions stated in the theorem. Notice that $F$ factors through $\mathcal{C}$ by \cref{prop. new q derivation on hodge state stack}.

\cref{thm. full faithful embedding} already shows that $F$ is fully faithful, hence it suffices to prove essential surjectivity. As pullback preserves colimit and $\mathcal{D}(\WCart)$ is generated under shifts and colimits by the invertible sheaves $\mathcal{I}^k$ for $k\in \mathbb{Z}$ by \cite[Corollary 3.5.16]{bhatt2022absolute}, given the full faithfulness of $F$, it suffices to show the target is also generated by $F(\mathcal{I}^k)$ ($k\in \mathbb{Z}$) under shifts and colimits. For this purpose, it can be reduced to the Hodge-Tate locus. Indeed, if $\mathcal{E}$ is in $\mathcal{C}$ and $\RHom_{\mathcal{C}}(\mathcal{I}^k, \mathcal{E})\cong \mathrm{R} \Gamma([X^{\Prism}/\mathbb{F}_p^{\times}],\mathcal{I}^{-k}\mathcal{E})=0$ for all $k\in \mathbb{Z}$, then we wish to show that $\mathcal{E}=0$, which is equivalent to the vanishing of $\mathcal{E}|_{[X^{\HT}/\mathbb{F}_p^{\times}]}$. Considering the fiber sequence
\begin{equation*}
    \mathrm{R} \Gamma([X^{\Prism}/\mathbb{F}_p^{\times}],\mathcal{I}^{1-k}\mathcal{E})\to \mathrm{R} \Gamma([X^{\Prism}/\mathbb{F}_p^{\times}],\mathcal{I}^{-k}\mathcal{E}) \to \mathrm{R} \Gamma([X^{\HT}/\mathbb{F}_p^{\times}],(\mathcal{I}^{-k}\mathcal{E})|_{[X^{\HT}/\mathbb{F}_p^{\times}]}),
\end{equation*}
we deduce that $\mathrm{R} \Gamma([X^{\HT}/\mathbb{F}_p^{\times}],(\mathcal{I}^{-k}\mathcal{E})|_{[X^{\HT}/\mathbb{F}_p^{\times}]})\cong \RHom(\mathcal{I}^k|_{[X^{\HT}/\mathbb{F}_p^{\times}]}, \mathcal{E}|_{[X^{\HT}/\mathbb{F}_p^{\times}]})$ vanishes for all $k$. We will show that this already guarantees the vanishing of $\mathcal{E}|_{[X^{\HT}/\mathbb{F}_p^{\times}]}$, denoted as $\mathcal{E}_0$ later.

For that purpose, it suffices to show that 
\begin{equation*}
    F_0: \mathcal{D}(\WCart^{\HT}) \longrightarrow \mathcal{D}_{\nil}([X^{\HT}/\mathbb{F}_p^{\times}])\times_{[\Spf(\mathbb{F}_p)/\mathbb{F}_p^{\times}]} \mathcal{D}(\Spf(\mathbb{F}_p))
\end{equation*}
is essential surjective, where $F_0$ is induced by the restriction of the diagram defining $F$:
\[\xymatrixcolsep{5pc}\xymatrix{[\Spf(\mathbb{F}_p)/\mathbb{F}_p^{\times}]\ar[d]^{f}\ar[r]& [\Spf(\mathcal{O}_K)/\mathbb{F}_p^{\times}] \ar[r]^{\rho} &[X^{\HT}/\mathbb{F}_p^{\times}]
\ar[d]_{}^{\pi}
\\ \Spf(\mathbb{F}_p) \ar^{\rho_{\dR}}[rr]&&\WCart^{\HT}.}\]
Indeed, once this is shown, both $\mathcal{I}^k$ and $\mathcal{E}_0$ will lie in the essential image of $F_0$, hence we could assume that $\mathcal{E}_0\cong F_0(\mathcal{G}_0)$ for $\mathcal{G}_0\in \mathcal{D}(\WCart^{\HT})$. The full faithfulness obtained in \cref{thm. full faithful embedding} then implies that 
\begin{equation*}
\RHom_{\WCart^{\HT}}(\mathcal{I}^k, \mathcal{G}_0)\cong \RHom(\mathcal{I}^k|_{[X^{\HT}/\mathbb{F}_p^{\times}]}, \mathcal{E}|_{[X^{\HT}/\mathbb{F}_p^{\times}]})=0
\end{equation*}
for all $k\in \mathbb{Z}$, hence $\mathcal{G}_0\cong 0$ as $\WCart^{\HT}$ is generated by $\mathcal{I}^k$ under shifts and colimits by \cite[Proposition 3.5.15]{bhatt2022absolute}, which finishes the proof.

We are left to prove the essential surjectivity of $F_0$. As $\mathcal{D}_{}([X^{\HT}/\mathbb{F}_p^{\times}])$ is equivalent to the category of $\mathbb{F}_p^{\times}$-equivariant quasi-coherent complexes on $X^{\HT}$ (for example, see \cite[Section 2.1]{abramovich2008tame}), but $\mathcal{D}(X^{\HT})$ is equivalent to the category of 
$\mathcal{D}_{\nil}(\mathcal{O}_K[\partial])$ by \cref{thmt.main classification}, then thanks to \cref{cor.Fp equivariance of theta} we see that objects in $\mathcal{D}_{}([X^{\HT}/\mathbb{F}_p^{\times}])$ can be described as pairs $(M,\partial_M)$, where $M$ is a $\mathbb{F}_p^{\times}$-equivariant $\mathcal{O}_K$ module equipped with a $\mathbb{F}_p^{\times}$-equivariant $q$-Higgs connection $\partial_M: M\to M\cdot \epsilon$ ($\mathbb{F}_p^{\times}$ acts on $\epsilon$ via \cref{lem. Fp action on R} as usual). 

Now given any $\mathcal{H}$ in the target of $F_0$, it corresponds to the pair $(\rho^{*}(\mathcal{H}),\partial_{\mathcal{H}})$ in the above discussion. As $\rho^{*}(\mathcal{H})\in \mathcal{D}_{}([\mathcal{O}_K/\mathbb{F}_p^{\times}])\times_{[\Spf(\mathbb{F}_p)/\mathbb{F}_p^{\times}]} \mathcal{D}(\Spf(\mathbb{F}_p))$, 
\begin{equation*}
    \begin{split}
        \rho^{*}(\mathcal{H})&=N\otimes_{\mathbb{Z}_p} \mathcal{O}_K \quad \text{for}~~N=(\rho^{*}(\mathcal{H}))^{\mathbb{F}_p^{\times}},
        \\(\rho^{*}(\mathcal{H})\cdot \epsilon)^{\mathbb{F}_p^{\times}}&=N\otimes_{\mathbb{Z}_p} (\mathcal{O}_K \cdot \epsilon)^{\mathbb{F}_p^{\times}}=N\cdot v  \quad \text{for}~~v=e\epsilon.
    \end{split}
\end{equation*}
Here the last identity follows from \cref{lem. calculate invariants}. 

As $\partial_{\mathcal{H}}$ is $\mathbb{F}_p^{\times}$-equivariant, it descends to a $\mathbb{Z}_p$-linear morphism $N\xrightarrow{\partial_{N}} N\cdot v$ such that the following diagram commutes 
\[\xymatrixcolsep{5pc}\xymatrix{N\ar[r]^{\partial_{N}} \ar[d]&   N\cdot v
\ar[d]^{\times e}
\\ N\otimes_{\mathbb{Z}_p} \mathcal{O}_K \ar[r]^{\partial_{\mathcal{H}}} &N\otimes_{\mathbb{Z}_p} \mathcal{O}_K\cdot \epsilon.}\]
By assumption on $\mathcal{H}$, $\partial_{N}$ satisfies that $\partial_{N}^p-\partial_{N}$ is locally  nilpotent on the cohomology $\mathrm{H}^*(N\otimes^{\mathbb{L}}\mathbb{F}_p)$. By unwinding all of the constructions and applying \cref{prop. new q derivation on hodge state stack}, we see that $(N, \partial_N)$ determines an object $\mathcal{N}\in \mathcal{D}(\WCart^{\HT})$. such that $F_0(\mathcal{N})\cong \mathcal{H}$, from which we finish the proof of essential surjectivity of $F_0$.
\end{proof}

As a quick corollary, we get the following concrete description of $\mathcal{D}(\WCart)$.
\begin{corollary}\label{cor. main cor 1}
Assume that $p>2$. Then $F$ in \cref{thm. descent for the cartier witt stack to q prism} induces a fully faithful functor 
$$F: \mathcal{D}(\WCart) \longrightarrow \mathcal{D}_{}([\Spf(A[\partial; \gamma_A, \partial_A])/\mathbb{F}_p^{\times}])\footnote{By abuse of notation, we write $\mathcal{D}_{}([\Spf(A[\partial; \gamma_A, \partial_A])/\mathbb{F}_p^{\times}])$ for the full subcategory of $\mathcal{D}(A[\partial; \gamma_A, \partial_A])/\mathbb{F}_p^{\times}])$ consisting of those objects which are $(p,d)$-complete.}\times_{[\Spf(\mathbb{Z}_p)/\mathbb{F}_p^{\times}]} \mathcal{D}(\Spf(\mathbb{Z}_p)).$$
Here $u\in \mathbb{F}_p^{\times}$ acts on $A[\partial; \gamma_A, \partial_A]$ by extending $\gamma_u$ on $A$ and further requiring that $$\gamma_u(\partial)=q^{1-[u]}\frac{q-1}{q^{[u]}-1}\partial.$$   
Moreover, the essential image of $F$ consists of those $\mathcal{E}$ in the target whose underlying complex in $ \mathcal{D}_{}([\Spf(A[\partial; \gamma_A, \partial_A])/\mathbb{F}_p^{\times}])$ (still denoted as $\mathcal{E}$ by abuse of notation) satisfies the following additional condition:
    \begin{itemize}
        \item the action of $\partial$ on $\mathcal{E}_0:=\mathcal{E}|_{X^{\HT}}$ can be factored as $e\circ \partial_{}^{\prime}$ such that the action of $(\partial_{}^{\prime})^p-\partial_{}^{\prime}$ on the cohomology $\mathrm{H}^*(\rho^*\mathcal{E}\otimes^{\mathbb{L}}k)$ is locally nilpotent, here $e=d^{\prime}(q)\in W(k)[\zeta_p]$.
    \end{itemize}
\end{corollary}
\begin{proof}
    By \cref{cor.Fp equivariance of theta}, $\gamma_{b,c}$ is $\mathbb{F}_p^{\times}$-equivaraint. For $\mathcal{E}\in \mathcal{D}(X^{\Prism})$, unwinding our construction of $\partial_{\mathcal{E}}$ on $\rho^*\mathcal{E}$ and our regulation of the $\mathbb{F}_p^{\times}$-action on $\epsilon$ in \cref{lem. Fp action on R}, we see that for $u\in \mathbb{F}_p^{\times}$, $$\gamma_u(\partial_{\mathcal{E}})= q^{1-[u]}\frac{q-1}{q^{[u]}-1}\partial_{\mathcal{E}}.$$
    The statement of the corollary then follows by combining \cref{thmt.main classification} with \cref{thm. descent for the cartier witt stack to q prism}.
\end{proof}

\section{Complexes on the relative prismatization over a $q$-prism base}
In this section we work in the following setting: $(A, I)$ is a transversal prism living over the $q$-prism $(W(k)[[q-1]], [p]_{q^{}})$. Hence $I$ is generated by $d=[p]_{q^{}}$ (viewed as an element in $A$). We still denote $\frac{q^p-1}{d}=q-1 \in A$ as $\beta$. Let $X$ be a $p$-adic smooth formal scheme over $A/I$, we would like to classify quasi-coherent complexes on $\TXA$ and $\TXAn$.
\begin{remark}\label{rem. relative change to alpha}
    In the last section we work with the (generalized) $q$-prism by allowing $\alpha>0$. In this section since we work with the relative setting, and actually $(W(k)[[q-1]], [p]_{q^{p^\alpha}})$ could be viewed as a prism over $(W(k)[[q-1]], [p]_{q^{}})$ via 
    \begin{equation*}
        \begin{split}
            f: (W(k)[[q-1]], [p]_{q^{}}) &\longrightarrow (W(k)[[q-1]], [p]_{q^{p^\alpha}})
            \\ q&\longmapsto q^{p^\alpha}
        \end{split}
    \end{equation*}
    hence all of the results in this section works for $(A,I)=(W(k)[[q-1]], [p]_{q^{p^\alpha}})$ by replacing $q$ with $q^{p^\alpha}$ in the statements.
\end{remark}

\subsection{Local computation}
We consider $X=\Spf(R)$, a smooth $p$-adic formal scheme over $\overline{A}=A/I$. We assume that $X$ is a small affine, in the sense that we  fix a $p$-completely \'etale chart 
\[\square:=\overline{A}\left\langle T_1, \ldots, T_m\right\rangle \rightarrow R\]
By deformation theory, this \'etale chart map uniquely lifts to a prism $(\Tilde{R}, I)$ over $(A\left\langle \underline{T}\right\rangle, I)$, here $A\left\langle \underline{T}\right\rangle$ is equipped with the $\delta$ structure respecting that on $A$ and sending $T_i$ to $0$, then we apply \cite[Lemma 2.18]{bhatt2021prismatic} to uniquely extend such a $\delta$-structure to $\Tilde{R}$. Moreover, the induced map $A\left\langle \underline{T}^{}\right\rangle \to \Tilde{R}$ is $(p, I)$-completely \'etale. In particular, $(\Tilde{R}, I)$ is a prism in $(X/A)_{\Prism}$ and $\Tilde{R}/I=R$. In this case, $\Omega_{\Tilde{R}/A}=\Tilde{R}\otimes_{A\left\langle \underline{T}^{}\right\rangle} \Omega_{A\left\langle \underline{T}^{}\right\rangle/A}$ is a finite free $R$-module of rank $m$ with a basis given by $\{dT_1,\cdots,dT_m\}$.

For any $i\in \{1,\cdots,m\}$, we define the automorphism $\gamma_i$ of $A\left\langle \underline{T}^{}\right\rangle$ fixing $A$ and sending $T_i$ to $q^pT_i$ and $T_j$ to $T_j$ for $j\neq i$. As $A\left\langle \underline{T}^{}\right\rangle \to \Tilde{R}$ is $(p, I)$-completely \'etale, these automorphisms lift uniquely to automorphisms $\gamma_i$ of $\Tilde{R}$. Also, $\gamma_i$ on $A\left\langle \underline{T}^{}\right\rangle$  is congruent to $1$ modulo the topologically nilpotent ideal $(q^p T_i-T_i)$, hence so is $\gamma_i$ on $\Tilde{R}$, which enables us to
define a twisted version of $q$-connection studied in \cite{scholze2017canonical}. Namely, for $f\in \Tilde{R}$, we set
$$\nabla_{i}(f)=d\frac{\gamma_i(f)-f}{q^p T_i-T_i}, \quad \text{and}\quad \nabla(f)=d\sum_{i=1}^m\frac{\gamma_i(f)-f}{q^p T_i-T_i}dT_i \in \Omega_{\Tilde{R}/A}.$$


\begin{lemma}\label{lem. twist S in the relative case}
    Let $S=\Tilde{R}\oplus \epsilon \Omega_{\Tilde{R}/A}$ and we promote $S$ to an $A$-algebra by requiring that \begin{equation*}
        (\epsilon dT_i)^2=\beta T_i\cdot\epsilon dT_i,\quad \epsilon dT_i\cdot \epsilon dT_j=0 ~~\text{for}~~ i\neq j.
    \end{equation*}
    Then the map 
    \begin{equation*}
        \begin{split}
            \psi : \Tilde{R} &\longrightarrow  S
            \\ f&\longmapsto f+\epsilon \nabla_q(f)=f+\epsilon\cdot d\sum_{i=1}^m\frac{\gamma_i(f)-f}{q^p T_i-T_i}dT_i
        \end{split}
    \end{equation*}
    defines a ring homomorphism of $A$-algebras, where the canonical $A$ algebra structure on $S$ is induced by $A\to S$ sending $a$ to $a$.
\end{lemma}
\begin{proof}
    Clearly $\psi$ is $A$-linear (as $\gamma_i$'s are all $A$-linear) and additive, we just need to check that $\psi_x(f_1f_2)=\psi(f_1)\cdot \psi(f_2)$ for $f_1, f_2\in \Tilde{R}$. But we have that
    \begin{equation*}
        \begin{split}
         \psi(f_1)\cdot \psi(f_2)&=(f_1+\epsilon\cdot d\sum_{i=1}^m\frac{\gamma_i(f_1)-f_1}{q^p T_i-T_i}dT_i)\cdot (f_2+\epsilon\cdot d\sum_{i=1}^m\frac{\gamma_i(f_2)-f_2}{q^p T_i-T_i} dT_i) 
         \\&=f_1f_2+d\epsilon\sum_{i=1}^m \frac{dT_i}{q^p T_i-T_i}(f_1(\gamma_i(f_2)-f_2)+f_2(\gamma_i(f_1)-f_1)+(\gamma_i(f_1)-f_1)(\gamma_i(f_2)-f_2))
         \\&=f_1f_2+d\epsilon\sum_{i=1}^m \frac{dT_i}{q^p T_i-T_i}(\gamma_i(f_1f_2)-f_1f_2)=\psi(f_1f_2).
        \end{split}
    \end{equation*}
Here the second identity holds as $\epsilon dT_i\cdot \epsilon dT_j=0$ for $i\neq j$ and 
\begin{equation*}
    \begin{split}
        &(d\epsilon \frac{\gamma_i(f_1)-f_1}{q^p T_i-T_i}dT_i)\cdot (d\epsilon \frac{\gamma_i(f_2)-f_2}{q^p T_i-T_i}dT_i)=d^2(\beta T_i\cdot\epsilon dT_i)\frac{(\gamma_i(f_1)-f_1)(\gamma_i(f_2)-f_2)}{(q^p T_i-T_i)^2}
        \\=&\frac{d\epsilon dT_i}{q^p T_i-T_i}(T_i\beta d \frac{(\gamma_i(f_1)-f_1)(\gamma_i(f_2)-f_2)}{q^p T_i-T_i})=\frac{d\epsilon dT_i}{q^p T_i-T_i} \cdot (\gamma_i(f_1)-f_1)(\gamma_i(f_2)-f_2),
    \end{split}
\end{equation*}
where the last equality follows from $\beta d=q^p-1$ by definition.
\end{proof}
\begin{remark}\label{rem. replace psi by i}
    For the reasons that will be clear later, we would like to point out that the above statements still hold if we replace $\psi$ by $\psi_i$ defined as follows
    \begin{equation*}
        \begin{split}
            \psi_i: \Tilde{R} &\longrightarrow  S_i=\Tilde{R}[\epsilon_i]/(\epsilon_i^2-\beta T_i\epsilon_i)=\Tilde{R}\oplus \Tilde{R}\cdot \epsilon_i
            \\ f&\longmapsto f+\epsilon_i \nabla_{i}(f)=f+\epsilon_i\cdot d\frac{\gamma_i(f)-f}{q^p T_i-T_i}
        \end{split}
    \end{equation*}
    Moreover, one can verify that $\nabla_{i}\circ \nabla_{j}=\nabla_{j}\circ \nabla_{i}$ for $i\neq j$, which essentially follows from the commutativity of $\gamma_i$ and $\gamma_j$, i.e. $\gamma_i\circ \gamma_j=\gamma_j\circ \gamma_i$.
\end{remark}
\begin{remark}\label{rem. delta ring structure on s relative}
We can upgrade $S$ to a $\delta$-ring extending that structure on $\Tilde{R}$ by requiring that $\varphi(\epsilon dT_i)=T_i^{p-1}\frac{q^{p}-1}{q^{}-1}\epsilon dT_i$. Actually, one can calculate that $$\varphi(\epsilon dT_i)-(\epsilon dT_i)^p=T_i^{p-1}\epsilon dT_i(\frac{q^{p}-1}{q^{}-1}-(q-1)^{p-1})$$
is zero modulo $p$, hence $\varphi$ is indeed a lift of the Fribenius modulo $p$. Moreover, under such a regulation $\psi$ defined in \cref{lem. twist S in the relative case} is actually a $\delta$-ring homomorphism, see next proposition.
\end{remark}
\begin{proposition}\label{prop. rela psi preserve delta structure}
    $\psi$ (resp. $\psi_i$) constructed in \cref{lem. twist S in the relative case} (resp. \cref{rem. replace psi by i}) is actually a $\delta$-ring homomorphism, where the target is equipped with the $\delta$-structure given in \cref{rem. delta ring structure on s relative}.
\end{proposition}
\begin{proof}
    We only prove the statement for $\psi$. First we claim that it suffices to show that $\psi: A\left\langle \underline{T}^{}\right\rangle \to S$ preserves the $\delta$-structure. Indeed, once this is shown, the two $\delta$-structures on $\psi(\Tilde{R})$, among which one is induced by that on $\Tilde{R}$ and the other is inherited from $S$, must coincide by \cite[Lemma 2.18]{bhatt2022prisms} as both of them extend $\delta$-structure on $\psi(A\left\langle \underline{T}^{}\right\rangle)$ (here we notice that $\Tilde{R}\hat{\otimes}_{A\left\langle \underline{T}^{}\right\rangle,\psi} A\left\langle \underline{T}^{}\right\rangle \xrightarrow[\psi]{\simeq} \psi(\Tilde{R})$ is $(p, I)$-completely \'etale over $A\left\langle \underline{T}^{}\right\rangle$). Given that claim, we are left to show that $\psi(\varphi(T_i))=\varphi(\psi(T_i))$ for all $1\leq i\leq m$ as both sides of $\psi$ are $p$-torsion free. But as $d=\frac{q^p-1}{q-1}$ and $\varphi(q)=q^p$, this follows from 
    \begin{equation*}
        \begin{split}
\psi(\varphi(T_i))&=\psi(T_i^p)=T_i^p+d\epsilon \frac{(q^pT_i)^p-T_i^p}{q^pT_i-T_i} dT_i=T_i^p+ \frac{q^{p^2}-1}{q-1}T_i^{p-1}\epsilon dT_i,
\\\varphi(\psi(T_i))&=\varphi(T_i+d\epsilon dT_i)=T_i^p+\epsilon \frac{q^{p^2}-1}{q^p-1}T_i^{p-1}\frac{q^{p}-1}{q^{}-1}\epsilon dT_i=T_i^p+ \frac{q^{p^2}-1}{q-1}T_i^{p-1}\epsilon dT_i.
        \end{split}
    \end{equation*}
\end{proof}

Next proposition specifies a homotopy between $\Tilde{\psi}: \Tilde{R}\to W(S)/d$ and $\Tilde{\iota}: \Tilde{R}\to W(S)/d$, which will be used later to construct the desired $q$-Higgs derivation.
\begin{proposition}\label{prop. relative c to construct homotopy}
Let $S=\Tilde{R}\oplus \epsilon \Omega_{\Tilde{R}/A}$ as in \cref{lem. twist S in the relative case}, then for any $x\in \Tilde{R}$ there exists a unique $c_{\psi}(x)$ in $W(S)$ such that $$\Tilde{\psi}(x)-\Tilde{\iota}(x)=d \cdot c_{\psi}(x).$$
Here $\Tilde{\psi}$ is the unique $\delta$-ring morphism such that the following diagram commutes:
    \[\begin{tikzcd}[cramped, sep=large]
& W(S) \arrow[d, "p_0"]\\ 
\Tilde{R} \arrow[ru, "\Tilde{\psi}"] \arrow[r, "\psi"]& S
\end{tikzcd}\]
and $\Tilde{\iota}$ is defined similarly with the bottom line in the above diagram replaced with the natural inclusion $\iota: \Tilde{R}\hookrightarrow S$. 
\end{proposition}
\begin{proof}
As $W(S)$ is $d$-torsion free by \cref{lem. s torsion freeness}, it suffices to show existence. First we would like to reduce to the case that $x\in A\left\langle \underline{T}^{}\right\rangle \to \Tilde{R}$. For this purpose, we consider the following two commutative diagrams: 
\[\xymatrixcolsep{5pc}\xymatrix{\Tilde{R}\ar[r]^{\iota}\ar@{>}[dr]^{\Tilde{\iota}}& S  \ar[r] &S/d
\\A\left\langle \underline{T}^{}\right\rangle  \ar[u]\ar[r]^{\Tilde{\iota}}&W(S)\ar[u]\ar[r]^{\pi} & W(S)/d \ar[u]} \]
and 
\[\xymatrixcolsep{5pc}\xymatrix{\Tilde{R}\ar[r]^{\psi}\ar@{>}[dr]^{\Tilde{\psi}}& S  \ar[r] &S/d
\\A\left\langle \underline{T}^{}\right\rangle  \ar[u]\ar[r]^{\Tilde{\psi}}&W(S)\ar[u]\ar[r]^{\pi} & W(S)/d \ar[u]} \]
Once the statement is proven for $x\in A\left\langle \underline{T}^{}\right\rangle$, then both $\pi\circ \Tilde{\iota}$ and $\pi\circ \Tilde{\psi}$ make the following diagram commutes
\[\xymatrixcolsep{5pc}\xymatrix{\Tilde{R}\ar[r]^{\psi=\iota}\ar@{.>}[dr]^{}& S/d   
\\A\left\langle \underline{T}^{}\right\rangle  \ar[u]\ar[r]^{\Tilde{\psi}=\Tilde{\iota}}&W(S)/d\ar[u] } \]
However, as $A\left\langle \underline{T}^{}\right\rangle \to \Tilde{R}$ is formally \'etale and $W(S)/d\to S/d$ is a pro-thickening, such lift must be unique, hence $\pi\circ \Tilde{\iota}=\pi\circ \Tilde{\psi}$, the desired result follows.

To verify the existence of $c_{\psi}(x)$ for $x\in \AT$, without loss of generality, we can assume $x=T_i$ as the set of $x$ satisfying the statements is closed under addition and multiplication. For simplicity, we omit $i$ from the subscript. We wish to construct $c=(c_0,c_1,\ldots)$ such that $\Tilde{\psi}(T)-\Tilde{\iota}(T)=\Tilde{\iota}(d)\cdot c$. As $S$ is $p$-torsion free, the ghost map is injective, hence this identity is equivalent to that 
\begin{align}\label{equa for relative.ghost identity}
    \forall n\geq 0, w_n(\Tilde{\psi}(T))-w_n(\Tilde{\iota}(T))=w_n(c)\cdot w_n(\Tilde{\iota}(d)),
\end{align}
where $w_n$ denotes the $n$-th ghost map.
Notice that
\begin{equation*}
    \begin{split}
    &w_n(\Tilde{\iota}(T))=w_0(\varphi^n(\Tilde{\iota}(T)))=w_0(\Tilde{\iota}(\varphi^n(T)))=T^{p^n},
        \\&w_n(\Tilde{\psi}(T))=w_0(\varphi^n(\Tilde{\psi}(T)))=w_0(\Tilde{\psi}(\varphi^n(T)))=\psi(T^{p^n})=T^{p^n}+\epsilon\cdot d\sum_{i=1}^m\frac{\gamma_i(T^{p^n})-T^{p^n}}{q^p T-T}dT,
        \\& w_n(\Tilde{\iota}(d))=\varphi^n(d)=[p]_{q^{p^{n}}}.
    \end{split}
\end{equation*}
Consequently,
\begin{equation*}
    \begin{split}
        &w_n(\Tilde{\psi}(T))-w_n(\Tilde{\iota}(T))=\epsilon\cdot d\frac{q^{p^{n+1}} T^{p^n}-T^{p^n}}{q^p T-T}dT
        \\=&\epsilon d\cdot T^{p^n-1}\frac{q^{p^{n+1}}-1}{q^p-1}dT
        \\=& \epsilon T^{p^n-1} \frac{q^p-1}{q-1} \cdot \frac{q^{p^{n+1}}-1}{q^{p^{n}}-1} \cdot \frac{q^{p^{n}}-1}{q^{p^{}}-1} dT
        \\=&\epsilon T^{p^n-1} [p]_{q^{p^{n}}} \cdot [p^n]_{q^{}} dT
    \end{split}
\end{equation*}
We define $r_n=T^{p^n-1} [p^n]_{q} \cdot\epsilon dT$. As $S$ is $\varphi^n(d)$-torsion free for all $n$, \cref{equa for relative.ghost identity} holds if and only if the ghost coordinates of $c$ is precisely given by $(r_0, r_1,\cdots)$. In other words, we aim to show that $(r_0, r_1,\cdots)\in S^{\mathbb{N}}$ is the ghost coordinate for some element $c\in W(S)$ (such $c$ is unique if exists as ghost map is injective). As $S$ is a $\delta$-ring thanks to \cref{rem. delta ring structure on s relative}, by invoking Dwork's lemma (see \cite[4.6 on page 213]{lazard2006commutative} for details), it suffices to show that for any $n\geq 0$,
$$p^{n+1}|\varphi(r_n)-r_{n+1}.$$
But $\varphi(r_n)-r_{n+1}$ actually vanishes as
\begin{equation*}
    \begin{split}
        &\varphi(r_n)-r_{n+1}=\varphi(T^{p^n-1} [p^n]_{q} \cdot\epsilon dT)-T^{p^{n+1}-1} [p^{n+1}]_{q} \cdot\epsilon dT
        \\=&T^{p-1}\frac{q^{p}-1}{q^{}-1}T^{p^{n+1}-p}\cdot \frac{q^{p^{n+1}}-1}{q^p-1}\cdot\epsilon dT- T^{p^{n+1}-1} \cdot \frac{q^{p^{n+1}}-1}{q^{}-1}\cdot\epsilon dT
        \\=&0.
    \end{split}
\end{equation*}
Hence we finish the proof.
\end{proof}
The following lemma is used in the above proof.
\begin{lemma}\label{lem. s torsion freeness}
    Let $S$ be as above, then $W(S)$ is $d$-torsion free.
\end{lemma}
\begin{proof}
    As $S$ is $p$-torsion free, the ghost map is injective, hence it suffices to show $\varphi^n(d)$ is a non-zero divisor in $S$, which follows from \cite[Lemma 3.3]{anschutz2020p}.
\end{proof}
Now we are ready to construct the $q$-Higgs derivation for quasi-coherent complexes on $X^{\Prism}$, based on the following key proposition.
\begin{proposition}\label{propt.key automorphism of functors relative case}
    $c_{\psi}$ constructed in \cref{prop. relative c to construct homotopy} induces an isomorphism $\gamma_{c_{\psi}}$ between functors $\rho: \Spf(S) \xrightarrow{\iota} \Spf(\Tilde{R}) \rightarrow X_{/A}^{\Prism}$
    and $\rho\circ \psi: \Spf(S) \xrightarrow{\psi} \Spf(\Tilde{R})\rightarrow X_{/A}^{\Prism}$, i.e. we have the following commutative diagram:
\[\xymatrixcolsep{5pc}\xymatrix{\Spf(S)\ar[d]^{\iota}\ar[r]^{\psi}& \Spf(\Tilde{R})\ar@{=>}[dl]^{\gamma_{c_{\psi}}} \ar[d]_{}^{\rho}
\\X_{/A}^{\Prism}  \ar^{\rho}[r]&X_{/A}^{\Prism}}\]
\end{proposition}
\begin{proof}
      By abuse of notation we write $\iota: A\to S$ for the composition of $A\to \Tilde{R}\xrightarrow{\iota} S$, then $\rho(S)$ corresponds to the point $$(\alpha:(d) \otimes_{A,\Tilde{\iota}} W(S)\to W(S),\eta: \Cone((d)\to \Tilde{R})\xrightarrow{\Tilde{\iota}} \Cone(\alpha))$$
      in $X^{\Prism}(S)$, while $(\rho\circ \psi)(S)$ corresponds to the point $$(\alpha:(d) \otimes_{A,\Tilde{\iota}} W(S)\to W(S), \eta^{\prime}: \Cone((d)\to \Tilde{R}) \xrightarrow{\Tilde{\psi}} \Cone(\alpha)).$$ 
      
      Then clearly $c_{\psi}$ constructed in \cref{prop. relative c to construct homotopy} serves a homotopy between $\eta^{\prime}$ and $\eta$ by drawing the following diagram as maps of quasi-ideals (in the sense of \cite{drinfeld2020prismatization}), which finishes the proof.
\[\xymatrixcolsep{5pc}\xymatrix{(d) \ar@<-.5ex>[d]_{\Tilde{\psi}} \ar@<.5ex>[d]^{\iota}\ar[r]^{\iota}& \Tilde{R} \ar[d]_{\Tilde{\psi}}
 \ar@<-.5ex>[d]^{\iota}
\\(d) \otimes_{A,\Tilde{\iota}} W(S)\ar[r]^{\iota}&W(S).}\]
\end{proof}

Let $?\in \{\emptyset, 1,\cdots, m\}$ and $n\in \mathbb{N}$. Recall that we have $\psi_{?}: \Tilde{R}\to S_{?}$ constructed in \cref{lem. twist S in the relative case} and \cref{rem. replace psi by i}. Then we are ready to construct a $q$-Higgs derivation $\nabla_{\mathcal{E},?}$ on $\rho^*\mathscr{E}$ for $\mathscr{E}\in \mathcal{D}(X_{/A}^{\Prism})$ (resp. $ \mathcal{D}(X_{/A,n}^{\Prism})$). Based on \cref{propt.key automorphism of functors relative case}, we have an isomorphism $\gamma_{c_{\psi_{?}}}: \rho\circ \psi_{?} \stackrel{\simeq}{\longrightarrow} \rho$. Consequently, for $\mathscr{E}\in \mathcal{D}(X_{/A}^{\Prism})$ (resp. $ \mathcal{D}(X_{/A,n}^{\Prism})$), we have an isomorphism
\[\gamma_{c_{\psi_{?}}}: \psi_{?}^{*}\rho^{*}(\mathscr{E}) \stackrel{\simeq}{\longrightarrow} \rho^{*}\mathscr{E}\otimes_{\Tilde{R}} S_?.\]
Unwinding the definitions, this could be identified with a $\psi_?$-linear morphism 
\begin{align}\label{equat. relative truncated induced morphism}
    \gamma_{c_{\psi_{?}}}: \rho^{*}(\mathscr{E})\rightarrow \rho^{*}(\mathscr{E})\otimes_{\Tilde{R}} S_?.
\end{align} 
Moreover, our definition of $c_{\psi_?}$ in \cref{prop. relative c to construct homotopy} implies that $\gamma_{c_{\psi_{?}}}$ in \cref{equat. relative truncated induced morphism} reduces to the identity modulo $\epsilon_{?}$, hence could be written as $\Id+\epsilon_{?}\nabla_{\mathscr{E},?}$ for some operator $\nabla_{\mathscr{E},?}: \rho^{*}\mathscr{E} \rightarrow \rho^{*}\mathscr{E}\otimes_{\Tilde{R}} \Omega_{\Tilde{R}/A,?}$ (as $S_?=\Tilde{R}\oplus \epsilon \Omega_{\Tilde{R}/A, ?}$). Moreover, the construction implies that $\nabla_{\mathcal{E}}=\sum_{i=1}^m \nabla_{\mathcal{E},i}\otimes dT_i$. We will call $\nabla_{\mathcal{E}}$ the \textit{$q$-Higgs derivation} on the complex $\rho^*\mathscr{E}$.  

Next we show that for $\mathcal{E}\in \mathcal{D}(X_{/A}^{\Prism})$, $(\rho^{*}(\mathscr{E}),\gamma_{c_{\psi_{}}}=\Id+\epsilon \nabla_{\mathcal{E}})$ actually enjoys an additional property.
\begin{proposition}\label{prop. rel two derivation vanish}
For $\mathcal{E}\in \mathcal{D}(X_{/A}^{\Prism})$, $\nabla_{\mathcal{E}}$ on $\rho^*\mathcal{E}$ satisfies the following property: $$\nabla_{\mathcal{E}}\wedge \nabla_{\mathcal{E}}: \rho^*\mathcal{E}\to \rho^*\mathcal{E}\otimes_{\Tilde{R}} \Omega_{\Tilde{R}/A}^2$$ vanishes.
\end{proposition}
\begin{proof}
   As $\nabla_{\mathcal{E}}=\sum_{i=1}^m \nabla_{\mathcal{E},i}\otimes dT_i$, it suffices to show that $\nabla_{\mathcal{E},i}\circ \nabla_{\mathcal{E},j}=\nabla_{\mathcal{E},j}\circ \nabla_{\mathcal{E},i}$ for $i\neq j$. We may assume that $i=1$ and $j=2$. For $i=1,2$, let $S_i=\Tilde{R}[\epsilon_i]/(\epsilon_i^2-\beta T_i\epsilon_i)=\Tilde{R}\oplus \Tilde{R}\cdot \epsilon_i$ and $\psi_i:\Tilde{R}\to S_i=\Tilde{R}[\epsilon_i]$ be as in \cref{rem. replace psi by i}. Denote $T=S_1\otimes_{\Tilde{R}} S_2$ and let $\psi_2^{\prime}: S_1\to T$ be the base change of $\psi_2$, i.e. $\psi_2^{\prime}(x_1+\epsilon_1 y_1)=(\Id+\epsilon_2 \nabla_{q,2})\circ (x_1+\epsilon_1 y_1)=x_1+\epsilon_1 y_1+\epsilon_2(\nabla_{q,2}(x_1))+\epsilon_1\epsilon_2 \nabla_{q,2}(y_1)$. Similarly we define $\psi_1^{\prime}: S_2\to T$ be the base change of $\psi_1$.
   
   We consider the following commutative diagram 
   \[ \xymatrix{ \Spf(T)  \ar@/_2pc/[dd]_{\psi_2} \ar@/^2pc/[rr]^{\psi_1} \ar[r]^{\psi_1^{\prime}} \ar[d]^{\psi_2^{\prime}} & \Spf(S_2)\ar[r]^{\iota} \ar[d]^{\psi_2} &\Spf(\Tilde{R}) \ar[dd]^{\rho}\\ \Spf(S_1)\ar[r]^{\psi_1} \ar[d]^{\iota}
 &  \Spf(\Tilde{R}) \ar@{=>}[dl]^{\gamma_{c_{\psi_1}}} \ar@{=>}[ru]^{\gamma_{c_{\psi_2}}} \ar[rd]^{\rho} &\\  \Spf(\Tilde{R}) \ar[rr]^{\rho} &&X_{/A}^{\Prism}.}\]
 Notice that \begin{equation*}
     \rho\circ \psi_1 \circ \psi_2^{\prime} \xrightarrow[\gamma_{c_{\psi_{1}}}] {\simeq} \rho\circ \psi_2^{\prime}=\rho\circ \psi_2 \xrightarrow[\gamma_{c_{\psi_{2}}}] {\simeq} \rho,
 \end{equation*}
 here the last $\rho$ actually means $\rho: \Spf(T) \xrightarrow{\iota} 
 \Spf(\Tilde{R}) \to X_{/A}^{\Prism}$, and the first map is the canonical morphism. Hence it induces $\gamma_{c_{\psi_{1,2}}}: \psi_2^{\prime,*}(\psi_1^{*}(\rho^{*}\mathcal{E}))\cong \rho^{*}\mathcal{E}\otimes_{\Tilde{R}} T$, which by unwinding definitions is the $T$-linear extension of
 \begin{equation*}
     \rho^{*}\mathcal{E} \xrightarrow{\Id\otimes 1} \rho^{*}\mathcal{E}\otimes_{\Tilde{R}, \psi_1} S_1 \xrightarrow[\gamma_{c_{\psi_{1}}}]{\simeq} \rho^{*}\mathcal{E}\otimes_{\Tilde{R}} S_1 \xrightarrow{\Id\otimes 1} (\rho^{*}\mathcal{E}\otimes_{\Tilde{R}} S_1)\otimes_{S_1, \psi_2^{\prime}} T=\rho^{*}\mathcal{E}\otimes_{\Tilde{R}, \psi_2} T \xrightarrow[\gamma_{c_{\psi_{2}}}]{\simeq} \rho^{*}\mathcal{E}\otimes_{\Tilde{R}} T.
 \end{equation*}
 As $\gamma_{c_{\psi_{1}}}=\Id+\epsilon_1 \nabla_{\mathcal{E},1}$ and $\gamma_{c_{\psi_{1}}}=\Id+\epsilon_1 \nabla_{\mathcal{E},2}$, the above morphism $\rho^{*}\mathcal{E} \to \rho^{*}\mathcal{E}\otimes_{\Tilde{R}} T$ is given by 
 $$(\Id+\epsilon_2 \nabla_{\mathcal{E},2})\circ (\Id+\epsilon_1 \nabla_{\mathcal{E},1})=(\Id+\epsilon_1 \nabla_{\mathcal{E},1}+\epsilon_2 \nabla_{\mathcal{E},2}+\epsilon_1 \epsilon_2 \nabla_{\mathcal{E},2}\circ \nabla_{\mathcal{E},1}).$$
On the other hand, as the left-upper square is commutative by \cref{rem. replace psi by i}, $\psi_1 \circ \psi_2^{\prime}=\psi_2 \circ \psi_1^{\prime}$, repeating the argument via $\rho\circ \psi_2 \circ \psi_1^{\prime} $ instead, we see that $\gamma_{c_{\psi_{1,2}}}$ should also be given the $T$-linear base change of 
$$(\Id+\epsilon_1 \nabla_{\mathcal{E},1})\circ (\Id+\epsilon_2 \nabla_{\mathcal{E},2})=\Id+\epsilon_1 \nabla_{\mathcal{E},1}+\epsilon_2 \nabla_{\mathcal{E},2}+\epsilon_1 \epsilon_2 \nabla_{\mathcal{E},1}\circ \nabla_{\mathcal{E},2},$$
which forces $\nabla_{\mathcal{E},2}\circ \nabla_{\mathcal{E},1}=\nabla_{\mathcal{E},1}\circ \nabla_{\mathcal{E},2}$ (as $\gamma_{c_{\psi_{1,2}}}$ is the unique homotopy determined by the two $\Tilde{R}$-algebra structures on $T$, given by $\iota$ and $\psi_2^{\prime}\circ \psi_1=\psi_1^{\prime}\circ \psi_2$ separately).
\end{proof}

In the special case that $\mathcal{E}\in \mathcal{D}(X_{/A}^{\Prism})^{\heartsuit}$, $\rho^{*}(\mathscr{E})$ is a $\Tilde{R}$-module concentrated on degree $0$ and we observe that $\nabla_{\mathcal{E},i}$ satisfies certain twisted Leibnitz rule.

\begin{lemma}[Twisted Leibnitz rule on the $q$-Higgs derivation]\label{lem. relative leibniz}
   Given a pair $(M,\gamma_M)$ such that $M=\rho^{*}(\mathscr{E})$ as above and $\gamma_M=\gamma_{c_{\psi}}: M\otimes_{\Tilde{R},\psi} S\stackrel{\simeq}{\rightarrow} M\otimes_{\Tilde{R}} S$, write $\gamma_M$ as $\Id+\epsilon \nabla_M$ when restricted to $M$ (hence $\nabla_M: M\to M\otimes_{\Tilde{R}} \Omega_{\Tilde{R}/A}$), then $\nabla_M=\sum_{i=1}^m \nabla_{M,i}\otimes dT_i$, and $\nabla_{M,i}$ satisfies the following twisted Leibnitz rule:
$$\nabla_{M,i}(ax)=a\nabla_{M,i}(x)+\nabla_{i}(a)x+\beta T_i \nabla_{i}(a)\nabla_{M,i}(x)=\gamma_i(a)\nabla_{M,i}(x)+\nabla_{i}(a)x$$ 
for $a\in \Tilde{R}$ and $x\in M$, here $\nabla_{i}$ and $\gamma_i$ on $\Tilde{R}$ is defined above \cref{lem. twist S in the relative case}. In particular, 
   $$\nabla_{M,i}(T_ix)=T_i\nabla_{M,i}(x)+dx+\beta T_i d \nabla_{M,i}(x)=dx+T_iq^p\nabla_{M,i}(x).$$
\end{lemma}
\begin{proof} 
    For such a pair $(M,\gamma_M)$,
    \begin{equation*}
        \begin{split}
            &\gamma_M(ax)=\psi(a)\gamma_M(x)=(a+\epsilon\nabla(a))(x+\epsilon\nabla_M(x))
          \\=& (a+\epsilon \sum_{i=1}^m \nabla_{i}(a)dT_i)(x+\epsilon \sum_{i=1}^m \nabla_{i}(x)dT_i)
\\=&ax+\sum_{i=1}^m \epsilon dT_i(\nabla_{i}(a)x+a\nabla_{i}(x))+\sum_{i=1}^m \epsilon  \beta T_i \nabla_{i}(a)\nabla_{i}(x)\cdot dT_i
\\=&ax+\sum_{i=1}^m \epsilon dT_i(\nabla_{i}(a)x+a\nabla_{i}(x)+\beta T_i \nabla_{i}(a)\nabla_{i}(x)),
        \end{split}
    \end{equation*}
here the second equation follows from \cref{lem. twist S in the relative case}. This implies the desired result.
\end{proof}

\begin{remark}\label{rem. rel twisted leibnitz in general}
    More generally, let $\mathscr{E}$ and $\mathscr{E}^{\prime}$ be quasi-coherent complexes on $X_n^{\Prism}$, and let $\mathscr{E}\otimes \mathscr{E}^{\prime}$ denote their derived tensor product. Then the same argument shows that the $q$-Higgs derivation $\nabla_{\mathscr{E}\otimes \mathscr{E}^{\prime},i}$ on $\rho^*(\mathscr{E}\otimes \mathscr{E}^{\prime})\cong \rho^*(\mathscr{E}) \otimes \rho^*(\mathscr{E}^{\prime})$ can be identified with $\nabla_{\mathscr{E}, i}\otimes\Id_{\mathscr{E}^{\prime}}+\Id_{\mathscr{E}}\otimes \nabla_{\mathscr{E}^{\prime}, i}+\beta T_i\nabla_{\mathscr{E}, i}\otimes \nabla_{\mathscr{E}^{\prime}, i}$.
\end{remark}

Motivated by \cref{prop. rel two derivation vanish} and \cref{lem. relative leibniz}, it is natural to introduce the following non-commutative ring.
\begin{definition}\label{rel def.skew polynomial}
    Let $\gamma_{\Tilde{R},i}: \Tilde{R}\to \Tilde{R}$ be the ring automorphism $\gamma_i$ defined above \cref{lem. twist S in the relative case}, then $\nabla_{\Tilde{R},i}: \Tilde{R}\to \Tilde{R}$, denoted as $\nabla_i$ above \cref{lem. twist S in the relative case}, is a $\gamma_{\Tilde{R},i}$-derivation of $\Tilde{R}$, i.e. $\nabla_{\Tilde{R},i}(x_1x_2)=\gamma_{\Tilde{R},i}(x_1)\nabla_{\Tilde{R},i}(x_2)+\nabla_{\Tilde{R},i}(x_1)x_2$. We define the \textit{Ore extension} $\Tilde{R}[\underline{\nabla}; \gamma_{\Tilde{R}}, \nabla_{\Tilde{R}}]$ to be the noncommutative ring obtained by giving the ring of polynomials $\Tilde{R}[\nabla_1,\cdots,\nabla_m]$ a new multiplication law, subject to the identity
    \begin{equation*}
        \begin{split}
           \nabla_i\cdot \nabla_j&=\nabla_j\cdot \nabla_i,  ~~~\forall ~ 1\leq i,j\leq m
           \\ \nabla_i \cdot r&=\gamma_{\Tilde{R},i}(r)\cdot\nabla_i+\nabla_{\Tilde{R},i}(r), ~~~\forall~ r\in \Tilde{R}, 1\leq i\leq m .
        \end{split}
    \end{equation*}
\end{definition}

\begin{remark}
\label{rem. rel cohomology of q higgs}
   Note that $\Tilde{R}$ is isomorphic to the quotient of $\Tilde{R}[\underline{\nabla}; \gamma_{\Tilde{R}}, \nabla_{\Tilde{R}}]$ by the (left) ideal generated by $\nabla_i~(1\leq i\leq m)$. In this sense, $\Tilde{R}$ is a left $\Tilde{R}[\underline{\nabla}; \gamma_{\Tilde{R}}, \nabla_{\Tilde{R}}]$-module and $\nabla_i$ acts on $R$ via $\nabla_{\Tilde{R},i}$ as desired. Moreover, the Koszul complex provides a canonical resolution for $\Tilde{R}$ by finite free $\Tilde{R}[\underline{\nabla}; \gamma_{\Tilde{R}}, \nabla_{\Tilde{R}}]$-modules:
   $$[\Tilde{R}[\underline{\nabla}; \gamma_{\Tilde{R}}, \nabla_{\Tilde{R}}] \stackrel{(\nabla_1,\cdots, \nabla_m)}{\longrightarrow} \bigoplus_{1\leq k\leq m}\Tilde{R}[\underline{\nabla}; \gamma_{\Tilde{R}}, \nabla_{\Tilde{R}}] \cdots \to \bigoplus_{1\leq k_1<\cdots< k_s\leq m}\Tilde{R}[\underline{\nabla}; \gamma_{\Tilde{R}}, \nabla_{\Tilde{R}}]\to \cdots]\simeq \Tilde{R}[0],$$
   where $\bigoplus_{1\leq k_1<\cdots< k_s\leq m}\Tilde{R}[\underline{\nabla}; \gamma_{\Tilde{R}}, \nabla_{\Tilde{R}}]$ lives in (cohomological) degree $s-m$ and the differential from $\Tilde{R}[\underline{\nabla}; \gamma_{\Tilde{R}}, \nabla_{\Tilde{R}}]$ in spot $k_1<\cdots<k_s$ to $\Tilde{R}[\underline{\nabla}; \gamma_{\Tilde{R}}, \nabla_{\Tilde{R}}]$ in spot $p_1<\cdots<p_{s+1}$ is nonzero only if $\{k_1,\cdots,k_s\}\subseteq \{p_1,\cdots,p_{s+1}\}$, in which case it sends $f\in \Tilde{R}[\underline{\nabla}; \gamma_{\Tilde{R}}, \nabla_{\Tilde{R}}]$ to $(-1)^{u-1}f\cdot\nabla_{p_u}$, where $u\in \{1,\cdots,s+1\}$ is the unique integer such that $p_u\notin \{k_1,\cdots,k_s\}$.
   
   Consequently, for $N$ an object in $\mathcal{D}(\Tilde{R}[\underline{\nabla}; \gamma_{\Tilde{R}}, \nabla_{\Tilde{R}}])$, the derived category of (left) $\Tilde{R}[\underline{\nabla}; \gamma_{\Tilde{R}}, \nabla_{\Tilde{R}}]$-modules, we have that
   $$\RHom_{\mathcal{D}(\Tilde{R}[\underline{\nabla}; \gamma_{\Tilde{R}}, \nabla_{\Tilde{R}}])}(\Tilde{R}, N)\simeq[N\xrightarrow{\nabla_N:=\sum \nabla_i\otimes dT_i} N\otimes_{\Tilde{R}} \Omega_{\Tilde{R}/A}\xrightarrow{\nabla_{N}\wedge \nabla_{N}} \cdots \to N\otimes_{\Tilde{R}} \Omega_{\Tilde{R}/A}^{m}\to 
        0].$$
   We will call the right-hand side of the above equation the \textit{$q$-de Rham complex of $N$}, denoted as $\mathrm{DR}(N,\nabla)$. Such an observation will be used to calculate the cohomology later.
\end{remark}
    
Given \cref{rel def.skew polynomial}, the discussion in this subsection so far can be then summarized as follows:
\begin{proposition}\label{prop. relative induce functor}
  For $n\in \mathbb{N}\cup \{\infty\}$, the pullback along $\rho: \Spf(\Tilde{R}/d^n)\to X_{/A,n}^{\Prism}$ induces a functor 
  \begin{equation*}
  \begin{split}
      \mathcal{D}(X_{/A,n}^{\Prism}) &\to \mathcal{D}(\Tilde{R}/d^n[\underline{\nabla}; \gamma_{\Tilde{R}}, \nabla_{\Tilde{R}}])\\
      \mathcal{E} &\mapsto (\rho^{*}\mathcal{E}, \nabla_{\mathcal{E}})
  \end{split}
  \end{equation*}
  which will be denoted as $\beta_n^+$ later.
\end{proposition}
\begin{proof}
Given \cref{prop. rel two derivation vanish} and \cref{lem. relative leibniz}, the proof of \cref{prop. induce functor} still works here, provided we know that $\mathcal{D}(X_{/A,n}^{\Prism})$ is equivalent to the derived category of its heart. But the latter can be proven similarly as \cref{prop. heart generator} by reducing to the Hodge-Tate locus and then taking \cite[Theorem 6.3]{anschutz2023hodge} as an input.
\end{proof}

Let $n\in \mathbb{N}\cup \{\infty\}$. As for $\mathcal{E}\in \mathcal{D}(X_{/A, n}^{\Prism})$, the global section of $\mathcal{E}$ is defined as 
\begin{equation*}
    \mathrm{R} \Gamma(X_{/A, n}^{\Prism}, \mathcal{E}):=\lim_{f: \Spec(R)\to X_{/A, n}^{\Prism}} f^{*}\mathcal{E}
\end{equation*}

In particular, the cover $\rho: \Spf(\Tilde{R}/d^n) \rightarrow X_{/A, n}^{\Prism}$ induces a natural morphism $$\mathrm{R} \Gamma(X_{/A, n}^{\Prism}, \mathcal{E})\to \rho^{*}\mathcal{E}.$$

The next proposition shows that it actually factors through the fiber of $\nabla_{\mathcal{E},i}$ for all $i$.
\begin{proposition}\label{propt. relative factor through fiber}
    For any $\mathcal{E}\in \mathcal{D}(X_{/A, n}^{\Prism})$, the natural morphism $\mathrm{R} \Gamma(X_{/A, n}^{\Prism}, \mathcal{E})\to \rho^{*}\mathcal{E}$ induces a canonical morphism $$
    \mathrm{R} \Gamma(X_{/A, n}^{\Prism}, \mathcal{E})\to \fib(\rho^*\mathcal{E}\stackrel{\nabla_{\mathcal{E},i}}{\longrightarrow} \rho^*\mathcal{E}).$$
    for all $i$. Hence it induces a canonical morphism 
    \[ \mathrm{R} \Gamma(X_{/A, n}^{\Prism}, \mathcal{E})\to (\rho^*\mathcal{E} \xrightarrow{\nabla_{\mathcal{E}}} \rho^*\mathcal{E}\otimes_{\Tilde{R}/I^n}\Omega_{(\Tilde{R}/I^n)/(A/I^n)}\to \cdots \xrightarrow{\nabla_{\mathcal{E}}\wedge \nabla_{\mathcal{E}}} \rho^*\mathcal{E}\otimes_{\Tilde{R}/I^n}\Omega_{(\Tilde{R}/I^n)/(A/I^n)}^{m}\to 0)\]
\end{proposition}
\begin{proof}
By the discussion after \cref{propt.key automorphism of functors relative case}, we have an isomorphism $\gamma_{c_{\psi_i}}: \rho\circ \psi_i \stackrel{\simeq}{\longrightarrow} \rho$ as functors  $\Spf(S_i/I^n) \rightarrow X_{/A,n}^{\Prism}$. Then the definition of $ \mathrm{R} \Gamma(X_{/A,n}^{\Prism}, \mathcal{E})$ implies that the natural morphism $\mathrm{R} \Gamma(X_{/A,n}^{\Prism}, \mathcal{E})\to \rho^{*}\mathcal{E}$ factors through the equalizer of $$\rho^*\mathcal{E}\stackrel{\Id\otimes 1}{\longrightarrow} 
 \rho^*\mathcal{E}\otimes_{\Tilde{R}/d^n} S/d^n$$
and 
$$\rho^*\mathcal{E}\stackrel{\gamma_{c_{\psi_i}}}{\longrightarrow} 
 \rho^*\mathcal{E}\otimes_{\Tilde{R}/d^n, \psi} S/d^n,$$
where $\gamma_{c_{\psi_i}}=\Id+\epsilon \nabla_{\mathcal{E},i}$. This produces a canonical morphism 
\begin{equation*}
    \mathrm{R} \Gamma(X_n^{\Prism}, \mathcal{E})\to \fib(\rho^*\mathcal{E}\stackrel{\nabla_{\mathcal{E},i}}{\longrightarrow} \rho^*\mathcal{E}).
\end{equation*}
for arbitrary $i$. Then the last statement follows as the de Rham complex of $\mathcal{E}$ is the totalization of certain finite diagram admitting a morphism from  $\prod_{i=1}^m \fib(\nabla_{\mathcal{E},i})$.
\end{proof}

 Our goal is to prove that $\beta_n^+$ stated in \cref{prop. relative induce functor} identifies the source with a certain subcategory of the target. For that purpose, similar to the absolute case, we would like to understand the behavior of $\nabla_{\mathcal{E}}$ when restricted to the Hodge-Tate locus first. If we work with canonical Higgs field instead of $q$-Higgs derivation, this was discussed in \cite[Section 6.1]{anschutz2023hodge} using Fourier transform. However, for our purpose, it is more convenient to follow the strategy of \cite[Theorem 3.5.8]{bhatt2022absolute}, in the spirit of \cite[Remark 6.14]{anschutz2023hodge}.

Actually, as $\psi: \Tilde{R}\to S$ reduces to the natural embedding after modulo $d$, we see $\gamma_{c_{\psi}}$ descends to an automorphism of
    \[X_{/A}^{\HT}\times_{\Spf(R)} \Spf(R[\epsilon \Omega_{R/\Bar{A}}]/((\epsilon dT_i)^2-\beta T_i \epsilon dT_i)),\]
    which implies that for $\mathcal{E}\in \mathcal{D}(X_{/A}^{\HT})$, $\nabla_{\mathcal{E}}$ also descends to a functor from $\mathcal{E}$ to itself.

    To simplify the notation, we write $\epsilon_i$ for $\epsilon dT_i$ in the following and hence $$S/d=R[\epsilon_1, \cdots, \epsilon_m]/(\epsilon_i^2-\beta T_i \epsilon_i, \epsilon_i \epsilon_j)_{1\leq i< j\leq m}\cong R\oplus \oplus_{i=1}^m R\epsilon_i.$$
    
    As $X_{/A}^{\HT}$ is the classifying stack of $T_{X/\Bar{A}}^{\sharp}\{1\}$ over $X=\Spf(R)$ by \cite[Proposition 5.12]{bhatt2022prismatization}, and we have already fixed a generalizer $d$ of $I$ as well as a basis for $\Omega_{R/\Bar{A}}$, this group can be trivialized as $(\mathbb{G}_a^{\sharp})^m=\prod_{i=1}^m \mathbb{G}_a^{\sharp}$.
    
Under such an identification, $\gamma_{c_{\psi}}$ corresponds to an element in $(\mathbb{G}_a^{\sharp})^m(S/d)$, and we claim this element is precisely $(\epsilon_i)_{1\leq i\leq m}\in (\mathbb{G}_a^{\sharp})^m(S/d)$.\footnote{ Here by our definition we see that $\epsilon_i^k=T_i^{k-1}\beta^{k-1}\epsilon_i$. In particular, it doesn't vanish for $k\geq 1$. However, the divided powers of $\epsilon_i$ still exist in $S/d$ as $v_p(\beta)=v_p(q-1)=\frac{1}{p-1}$, hence $ v_p(\beta^{n-1})=\frac{n-1}{p-1}\geq v_p(n!)$. Consequently, $\epsilon_i \in \mathbb{G}_a^\sharp(S/d)$.}. 
To see this, unwinding the construction of $\gamma_{c_{\psi}}$, we just need to verify that the image of $c_{\psi}(T_i)$ under the natural morphism $$W(S)\to W(S/d)$$ lies in $\mathbb{G}_a^\sharp(S/d)$ and is precisely given by $\epsilon_i$. 
As $S/d$ is $p$-torsion free, it suffices to check that the first coordinate for $c_{\psi}(T_i)$ is $\epsilon_i$ in $S/d$, which is clear from the proof of \cref{prop. relative c to construct homotopy}. 

The next lemma is an analog of \cref{lem. calculate partial on HT} and helps us understand how $\nabla_i$ acts on $\rho^{*}\rho_{*} \mathcal{O}_{X}$ under the covering morphism $\rho: X\to X_{/A}^{\HT}$, in parallel to the behavior of the Sen operator $\theta$ studied in \cite{bhatt2022prismatization} and \cite{anschutz2023hodge} (see \cite[Section 6.1]{anschutz2023hodge} for details).
\begin{lemma}\label{prop. relative calculate partial on HT}
    Let $\mathcal{E}=\rho_{*} \mathcal{O}_{X}$, then 
    \begin{itemize}
        \item $\rho^{*}\mathcal{E}\cong  R\{a_1,\cdots, a_m\}^{\wedge}_{p}$.
        \item the sequence $0\to 
        (R\{a_j\}_{j\neq i})^{\wedge}_{p} \to \rho^{*}\mathcal{E}\xrightarrow{\nabla_i} \rho^{*}\mathcal{E} \to 0$ is exact.
    \end{itemize}
\end{lemma}
\begin{proof}
    The projection formula tells us that $\rho^{*}\mathcal{E}\cong R\{a_1,\cdots, a_m\}^{\wedge}_{p}$. To prove the second statement, we write $B$ for $(R\{a_j\}_{j\neq i})^{\wedge}_{p}$ and omit $i$ from the subscript of $a_i$ and $T_i$ for the ease of notation. 
    As the formal group law on the $i$-th component of $(\mathbb{G}_a^{\sharp})^m$ is given by $$\Delta: \widehat{\mathcal{O}}_{\mathbb{G}_a^{\sharp}} \to \widehat{\mathcal{O}}_{\mathbb{G}_a^{\sharp}} \hat{\otimes} \widehat{\mathcal{O}}_{\mathbb{G}_a^{\sharp}}, ~~~~a\mapsto a+b$$
and the isomorphism $\gamma_{c_{\psi_i}}$ (hence also the $q$-derivation $\nabla_i$) is constructed via $\epsilon \in \mathbb{G}_a^{\sharp}(S/d)$, hence for $f(a)=\sum_{i=0}^{\infty}c_i\frac{a^n}{n!}\in \rho^*\mathcal{E}=B\{a\}_p^{\wedge}$ with $c_i\in B$, $\gamma_{c_{\psi_i}}(f)=f(a+\epsilon)=f+\epsilon\nabla_i(f)$. We then calculate that 
\begin{equation*}
    \begin{split}
       \epsilon\nabla_i(f)&=f(a+\epsilon)-f(a)=\sum_{n=0}^{\infty}c_n\frac{(a+\epsilon)^n-a^n}{n!}
       \\&=\sum_{n=0}^{\infty}\frac{c_n}{n!}(\sum_{i=0}^{n-1}\binom{n}{i}a^i\epsilon^{n-i})
       =\epsilon \sum_{n=0}^{\infty}\frac{c_n}{n!}(\sum_{i=0}^{n-1}\binom{n}{i}a^i(\beta T)^{n-i-1}),
    \end{split}
\end{equation*}
where the last equality follows from $\epsilon^j=(\beta T)^{j-1}\epsilon$. Consequently
\begin{equation*}
    \begin{split}
        \nabla_i(f)&=\sum_{i=0}^{\infty} a^i(\sum_{n=i+1}^{\infty} \frac{c_n}{n!}\binom{n}{i}(\beta T)^{n-1-i})
        \\&=\sum_{i=0}^{\infty} \frac{a^i}{i!}(\sum_{n=i+1}^{\infty} \frac{c_n}{(n-i)!}(\beta T)^{n-1-i}).
    \end{split}
\end{equation*}
We then consider the infinite dimensional matrix $M$ with $M_{i,j}=\frac{(\beta T)^{j-i}}{(j+1-i)!}$ for $i\leq j$ and $0$ otherwise, then $M$ is an upper triangular matrix with all the diagonal elements being $1$ in $B$, hence $M$ is an invertible matrix. Let $g=\sum_{n=0}^{\infty} b_n\frac{a^n}{n!}$, then the previous calculation tells us that $\nabla_i(f)=g$ if and only if 
\begin{equation}\label{equa. solve c for b}
    M \Vec{c}=\Vec{b}.
\end{equation}
for $\Vec{b}=(b_0,b_1,\cdots)^{\mathrm{T}}$ and $\Vec{c}=(c_1,c_2,\cdots)^{\mathrm{T}}$.
As $M$ is invertible, given any $\Vec{b}$ satisfies that $b_i\to 0$ as $i\to \infty$, there exists a unique $\Vec{c}$ solving \cref{equa. solve c for b}. Moreover, $c_i\to 0$ as $i\to \infty$. This finishes the proof of the second result stated in the proposition.
\end{proof}

As a quick corollary of \cref{prop. relative calculate partial on HT}, we get the following Poincare lemma for $\mathcal{O}_X$.
\begin{corollary}\label{cor. relative poincare lemma}
The de Rham complex for $\mathcal{E}_0=\rho_{*} \mathcal{O}_{X}$ formed via $\nabla_{\mathcal{E}_0}$ is a resolution for $\mathcal{O}_X$. More precisely,
    \begin{equation*}
        R\xrightarrow{\simeq} (\rho^*\mathcal{E}_0 \xrightarrow{\nabla_{\mathcal{E}_0}} \rho^*\mathcal{E}_0\otimes_{R}\Omega_{R/\Bar{A}}\to \cdots \xrightarrow{\nabla_{\mathcal{E}_0}\wedge \nabla_{\mathcal{E}_0}} \rho^*\mathcal{E}_0\otimes_{R}\Omega_{R/\Bar{A}}^{m}\to 0).
    \end{equation*}
\end{corollary}
\begin{proof}
By \cref{prop. relative calculate partial on HT},   $(R\{a_j\}_{j\neq i})^{\wedge}_{p}\xrightarrow{\simeq} \fib(\rho^{*}\mathcal{E}_0\xrightarrow{\nabla_{\mathcal{E}_0,i}} \rho^{*}\mathcal{E}_0)$. Moreover, the proof of it actually implies that for any $1< q\leq m$, $\nabla_{\mathcal{E}_0,q}$ is surjective on $\cap_{j=1}^{q-1} \ker(\nabla_{\mathcal{E}_0,j})$. By a well-known lemma, the Koszul complex for $\{\nabla_{\mathcal{E}_0,i}\}_{1\leq i\leq m}$ is hence concentrated on degree $0$ and is given by $\cap_{j=1}^{m} \ker(\nabla_{\mathcal{E}_0,j})$. Consequently by taking limits with respect to all $i$, we get that
\begin{equation*}
    R\xrightarrow{\simeq} 
 (\rho^*\mathcal{E}_0 \xrightarrow{\nabla_{\mathcal{E}_0}} \rho^*\mathcal{E}_0\otimes_{R}\Omega_{R/\Bar{A}}\to \cdots \xrightarrow{\nabla_{\mathcal{E}_0}\wedge \nabla_{\mathcal{E}_0}} \rho^*\mathcal{E}_0\otimes_{R}\Omega_{R/\Bar{A}}^{m}\to 0).
\end{equation*}
\end{proof}

\begin{proposition}\label{prop. relative ht case fiber seq}
    Assume $\mathcal{E}_0=\rho_{*} \mathcal{O}_{X}$ as in the previous corollary. For any $\mathcal{E}\in \mathcal{D}(X_{/A}^{\HT})$, there is a canonical resolution
    \begin{equation*}
        \mathcal{E} \xrightarrow{\simeq} (\rho_*\rho^*\mathcal{E} \xrightarrow{\nabla_{\mathcal{E}}} \rho_*\rho^*\mathcal{E}\otimes_{R}\Omega_{R/\Bar{A}}\to \cdots \xrightarrow{
        } \rho_*\rho^*\mathcal{E}\otimes_{R}\Omega_{R/\Bar{A}}^{m}\to 0).
    \end{equation*}
Moreover, taking cohomology induces a canonical quasi-isomorphism $$\mathrm{R} \Gamma(X_{/A}^{\HT}, \mathcal{E})\xrightarrow{\simeq} (\rho^*\mathcal{E} \xrightarrow{\nabla_{\mathcal{E}}} \rho^*\mathcal{E}\otimes_{R}\Omega_{R/\Bar{A}}\to \cdots \xrightarrow{\nabla_{\mathcal{E}}\wedge \nabla_{\mathcal{E}}} \rho^*\mathcal{E}\otimes_{R}\Omega_{R/\Bar{A}}^{m}\to 0).$$ Hence the morphism constructed in \cref{propt. relative factor through fiber} is a quasi-isomorphism when $n=1$.
\end{proposition}
\begin{proof}
By the discussion above \cref{prop. relative calculate partial on HT}, $\nabla_{\mathcal{E}_0}$ descends to a functor from $\mathcal{E}_0$ to $\mathcal{E}_0\otimes_R \Omega_{R/\Bar{A}}$ and hence we have a morphism 
\begin{equation*}
    \mathcal{O}_{X_{/A}^{\HT}}\to (\mathcal{E}_0 \xrightarrow{\nabla_{\mathcal{E}_0}} \mathcal{E}_0\otimes_{R}\Omega_{R/\Bar{A}}\to \cdots \xrightarrow{\nabla_{\mathcal{E}_0}\wedge \nabla_{\mathcal{E}_0}} \mathcal{E}_0\otimes_{R}\Omega_{R/\Bar{A}}^{m}\to 0),
\end{equation*}
which actually is a resolution for $\mathcal{O}_{X_{/A}^{\HT}}$ by \cref{cor. relative poincare lemma} and faithfully flat descent. For a general $\mathcal{E}\in \mathcal{D}(X^{\HT})$, tensoring $\mathcal{E}$ with the above sequence and then yields a canonical resolution for $\mathcal{E}$
\begin{equation*}
        \mathcal{E} \xrightarrow{\simeq} (\mathcal{E}\otimes_{R} \mathcal{E}_0 \xrightarrow{\Id\otimes\nabla_{\mathcal{E}_0}} \mathcal{E}\otimes_{R} \mathcal{E}_0\otimes_{R}\Omega_{R/\Bar{A}}\to \cdots \xrightarrow{} \mathcal{E}\otimes_{R} \mathcal{E}_0\otimes_{R}\Omega_{R/\Bar{A}}^{m}\to 0)
    \end{equation*}
By the usual trick of trivializing Hopf algebra’s comodules, the projection formula $\rho_*\rho^*\mathcal{E}\cong \mathcal{E}\otimes_{R}\rho_*(\mathcal{O}_X)$ identifies $\nabla_{\mathcal{E}}$ with $\Id\otimes \nabla_{\mathcal{E},0}$, hence applying it and then taking cohomology we then get a canonical quasi-isomorphism
\begin{equation*}
    \mathrm{R} \Gamma(X_{/A}^{\HT}, \mathcal{E})\xrightarrow{\simeq} (\rho^*\mathcal{E} \xrightarrow{\nabla_{\mathcal{E}}} \rho^*\mathcal{E}\otimes_{R}\Omega_{R/\Bar{A}}\to \cdots \xrightarrow{\nabla_{\mathcal{E}}\wedge \nabla_{\mathcal{E}}} \rho^*\mathcal{E}\otimes_{R}\Omega_{R/\Bar{A}}^{m}\to 0).
\end{equation*}
\end{proof}
\begin{remark}
    The analog of \cref{prop. relative ht case fiber seq} by replacing the $q$-Higgs derivation with the usual canonical Higgs filed (see \cite[Definition 6.15]{anschutz2023hodge}) is \cite[Lemma 6.10]{anschutz2023hodge}.
\end{remark}
\begin{example}[$q$-Higgs derivations on the structure sheaf when restricted to the Hodge-Tate locus]\label{example. relative action on generators}
     For $\mathcal{E}=\mathcal{O}_{X_{/A}^{\HT}}$, $\nabla_{\mathcal{E}}=0$. Indeed, as in \cite[Corollary 3.5.14]{bhatt2022absolute}, \cref{prop. relative ht case fiber seq}  implies that $\mathcal{E}\in \mathcal{D}(X_{/A}^{\HT})$ is isomorphic to $ \mathcal{O}_{X_{/A}^{\HT}}$ if and only if $\rho^*\mathcal{E}\cong \mathcal{O}_X$ and $\nabla_{\mathcal{E}}=0$.
\end{example}

Now we proceed to classify quasi-coherent sheaves on $X_{/A,n}^{\Prism}$ for all $n$.
\begin{proposition}\label{propt. relative sen calculate cohomology}
    Let $n\in \mathbb{N}\cup \{\infty\}$. For any $\mathcal{E}\in \mathcal{D}(X_{/A,n}^{\Prism})$, the natural morphism $$\mathrm{R} \Gamma(X_{/A, n}^{\Prism}, \mathcal{E})\xrightarrow{\simeq} (\rho^*\mathcal{E} \xrightarrow{\nabla_{\mathcal{E}}} \rho^*\mathcal{E}\otimes_{\Tilde{R}/I^n}\Omega_{(\Tilde{R}/I^n)/(A/I^n)}\to \cdots \xrightarrow{\nabla_{\mathcal{E}}\wedge \nabla_{\mathcal{E}}} \rho^*\mathcal{E}\otimes_{\Tilde{R}/I^n}\Omega_{(\Tilde{R}/I^n)/(A/I^n)}^{m}\to 0)$$ constructed in \cref{propt. relative factor through fiber} is a quasi-isomorphism.
\end{proposition}
\begin{proof}
    For $n\in \mathbb{N}$, it reduces to $n=1$ by standard  d\'evissage, which follows from \cref{prop. relative ht case fiber seq}.

Finally for $\mathcal{E}\in \mathcal{D}(X_{/A}^{\Prism})$, as taking global sections commutes with limits, by writing $\mathcal{E}$ as the inverse limit of $\mathcal{E}_n$ for $\mathcal{E}_n$ the restriction of $\mathcal{E}$ to $X^{\Prism}_{/A,n}$, the desired result follows as taking inverse limits commutes with taking cohomology and totalization.
\end{proof}

As a consequence of \cref{propt. relative sen calculate cohomology}, we get that
\begin{corollary}
    The global sections functor $$\mathrm{R} \Gamma(X_{/A}^{\Prism}, \bullet): \mathcal{D}(X_{/A}^{\Prism}) \rightarrow \widehat{\mathcal{D}}(\mathbb{Z}_p) \quad \text{resp.}\quad \mathrm{R} \Gamma(X^{\Prism}_{/A,n}, \bullet): \mathcal{D}(X^{\Prism}_{/A,n}) \rightarrow \widehat{\mathcal{D}}(\mathbb{Z}_p)$$
    commutes with colimits.
\end{corollary}

\begin{proposition}\label{propt. relative generation}
    Let $n\in \mathbb{N}$. The $\infty$-category $\mathcal{D}(X_{/A}^{\Prism})$ (resp. $\mathcal{D}(X^{\Prism}_{/A,n})$) is generated under shifts and colimits by the structure sheaf $\mathcal{O}$.
\end{proposition}
\begin{proof}
    As we work with the relative prismatization, $X_{/A}^{\Prism}$ lives over $A$, hence the ideal sheaf $\mathcal{I}$ is trivialized, i.e. $\mathcal{I}\cong \mathcal{O}$. Let $\mathcal{E}\in \mathcal{D}(X_{/A}^{\Prism})$ such that $\RHom(\mathcal{O}, \mathcal{E})=0$, then $\RHom(\mathcal{I}^n, \mathcal{E})=0$ for all $n$ as $\mathcal{I}^n\cong \mathcal{O}$. This implies that $\RHom(\mathcal{O}|_{X_{/A}^{\HT}}, \mathcal{E}|_{X_{/A}^{\HT}})=0$, hence $\mathcal{E}|_{X_{/A}^{\HT}}=0$ as $\mathcal{D}(X_{/A}^{\HT})$ is generated under shifts and colimits by the structure sheaf thanks to \cite[Lemma 6.12]{anschutz2023hodge}. Consequently $\mathcal{E}=0$.
\end{proof}
\begin{theorem}\label{thmt. relative main classification}
Let $n\in \mathbb{N}\cup \{\infty\}$. The functor 
    \begin{align*}
        &\beta_n^{+}: \mathcal{D}(X_{/A, n}^{\Prism}) \rightarrow \mathcal{D}(\Tilde{R}/d^n[\underline{\nabla}; \gamma_{\Tilde{R}}, \nabla_{\Tilde{R}}]), \qquad \mathcal{E}\mapsto (\rho^{*}(\mathcal{E}),\nabla_{\mathcal{E}})
    \end{align*}
    constructed in \cref{prop. relative induce functor} is fully faithful. 
     Moreover, its essential image consists of those objects $M\in \mathcal{D}(\Tilde{R}/d^n[\underline{\nabla}; \gamma_{\Tilde{R}}, \nabla_{\Tilde{R}}])$ 
     satisfying the following pair of conditions:
    \begin{itemize}
        \item $M$ is $(p,d)$-adically complete.
        \item The action of $\nabla_{i}$ on the cohomology $\mathrm{H}^*(M\otimes_{\Tilde{R}/d^n}^{\mathbb{L}}\Tilde{R}/(d,p))$ 
        is locally nilpotent for all $i$. 
    \end{itemize}
\end{theorem}
\begin{proof}
The functor is well-defined thanks to \cref{prop. relative induce functor}. For the full faithfulness, let $\mathcal{E}$ and $\mathcal{F}$ be quasi-coherent complexes on $X_{/A,n}^\Prism$ and we want to show that the natural map 
    \begin{equation*}
        \Hom_{\mathcal{D}(X^\Prism_{/A, n})}(\mathcal{E}, \mathcal{F}) \rightarrow \Hom_{\mathcal{D}(\Tilde{R}/d^n[\underline{\nabla}; \gamma_{\Tilde{R}}, \nabla_{\Tilde{R}}])} (\rho^{*}(\mathcal{E}), \rho^{*}(\mathcal{F}))
    \end{equation*}
    is a homotopy equivalence. Thanks to \cref{propt. relative generation}, we could reduce to the case that $\mathcal{E}$ is the structure sheaf. Then the desired result follows from \cref{propt. relative sen calculate cohomology} and \cref{rem. rel cohomology of q higgs}. 
    
    By \cref{propt. relative generation}, the essential image of $\beta_n^+$ is generated by $\Tilde{R}/d^n$ (with $\nabla_i$ acting via $\nabla_{\Tilde{R},i}$) under shifts and colimits, and $\nabla_{\Tilde{R},i}$ on $\Tilde{R}/d^n$ satisfies that $\nabla_{\Tilde{R},i}(x)\in d(\Tilde{R}/d^n), \forall x\in \Tilde{R}/d^n$. In particular, the essential image of $\beta_n^+$ is contained in the subcategory of $(p, d)$-complete complexes $M\in \mathcal{D}(\Tilde{R}/d^n[\underline{\nabla}; \gamma_{\Tilde{R}}, \nabla_{\Tilde{R}}])$ such that each $\nabla_{i}$ acts locally nilpotently on $\mathrm{H}^*(M\otimes_{\Tilde{R}/d^n}^{\mathbb{L}}\Tilde{R}/(d,p))$.

    Let $\mathcal{C}_n\subseteq \mathcal{D}(\Tilde{R}/d^n[\underline{\nabla}; \gamma_{\Tilde{R}}, \nabla_{\Tilde{R}}])$ be the full subcategory spanned by objects satisfying two conditions listed in \cref{thmt. relative main classification}. As the source $\mathcal{D}(X_{/A,n}^{\Prism})$ is generated under shifts and colimits by the structure sheaf thanks to \cref{propt. relative generation}, to complete the proof it suffices to show that $\mathcal{C}_n$ is also generated under shifts and colimits by $\beta^+_n(\mathcal{O}_{X_{/A,n}^{\Prism}})$. In other words, we need to show that for every nonzero object $M\in \mathcal{C}_n$, $M$ admits a nonzero morphism from $\beta^+_n(\mathcal{O}_{X_{/A,n}^{\Prism}})[m]$ for some $m \in \mathbb{Z}$. 
    
    Arguing as \cref{thmt.main classification}, we can reduce to the Hodge-Tate case and it suffices to show that for every nonzero object $M\in \mathcal{C}_1$, $\RHom_{\mathcal{C}_{1}}(\beta_1^{+}(\mathcal{O}), M)\neq 0$. For this purpose, first we observe that $\RHom_{\mathcal{C}_{1}}(\beta_1^{+}(\mathcal{O}), M)$ can be calculated by $D_{\dR}(M,\nabla)$, the $q$-de Rham complex of $M$ thanks to \cref{rem. rel cohomology of q higgs}. Replacing $M$ by $M\otimes k$ (the derived Nakayama guarantees that $L\otimes k$ detects whether $L$ is zero or not for $p$-complete $L$), we may assume that there exists some cohomology group $\mathrm{H}^{-m}(M)$ containing a nonzero element $x$ killed by $\nabla_{i}^N$ for all $i$, where $N$ is fixed. Furthermore, we could assume this element is non-zero and is actually killed by $\nabla_i$ for all $i$. Indeed, assume that $N$ is the minimal positive integer such that $\nabla_{i}^N$ kills $x$ for all $i$. We do induction on $N$. If $N=1$, this is already done. Otherwise there exists $i$ such that $\nabla_{i}^{N-1}(x)\neq 0$. we first replace $x$ by $\nabla_{i}^{N-1}(x)$. If $x$ is still not killed by $\nabla_{j}^{N-1}$, we replace $x$ with $\nabla_{j}^{N-1}(x)$ and then repeat the process until reduced to the case that $N=1$. Notice that in this step we crucially use the assumption that $\nabla_{i}$ commutes with $\nabla_{j}$ for all $i, j$. It then follows that $\mathrm{H}^{-m}(D_{\dR}(M,\nabla))\neq 0$, hence there exists a non-zero morphism from $\beta^+_1(\mathcal{O}_{})/p[m]$ to $M$ as desired.
\end{proof}
\subsection{Locally complete intersection case}
In this subsection we would like to extend the classification results in the previous section to the locally complete intersection case.

As in the last subsection, we still assume $X=\Spf(R)$ is a small affine and fix the following diagram
\[\xymatrixcolsep{5pc}\xymatrix{A\left\langle \underline{T}^{}\right\rangle\ar[r]^{\square} \ar[d]&   \Tilde{R}
\ar[d]^{}
\\ \Bar{A}\left\langle \underline{T}^{}\right\rangle\ar[r] \ar[r]^{\square} &R,}\]
where the horizontal maps are \'etale chart maps. But we allow $m=0$, for example, $R=\mathcal{O}_K$ (and $\Tilde{R}=A$) in this subsection.

In this paper we always consider $Y=\Spf(R/(\Bar{x}_1,\cdots,\Bar{x}_r)) \hookrightarrow X$ ($x_i\in \Tilde{R}$ and $\Bar{x}_i$ is its image in $\Bar{A}$) to be a closed embedding such that the prismatic envelope with respect to the morphism of $\delta$-pairs $(\Tilde{R}, (d))\to (\Tilde{R}, (d,x_1,\cdots,x_m))$ exists and is exactly given by $\Tilde{R}\{\frac{x_1,\cdots,x_r}{d}\}^{\wedge}_{\delta}$, obtained by freely adjoining $\frac{x_i}{d}$ to $\Tilde{R}$ in the category of derived $(p,I)$-complete simplicial $\delta$-$A$-algebras\footnote{See \cite[Corollary 2.44]{bhatt2022prisms} for the precise definition.}.

We first verify that the assumption on $Y$ holds for a large class of interesting objects.
\begin{lemma}\label{lem. ARL discrete case}
$\Tilde{R}\{\frac{x_1,\cdots,x_r}{d}\}^{\wedge}_{\delta}$ is discrete and $d$-torsion free (hence satisfies the above condition) in the following two cases: 
\begin{itemize}
    \item $(\Bar{x}_1,\cdots,\Bar{x}_r)$ forms a $p$-completely regular sequence relative to $\mathcal{O}_K$ in the sense of \cite[Definition 2.42]{bhatt2022prisms}.
    \item $x_1=(q-1)^{t}\in A$ ($t \geq 1$), $(\Bar{x}_2,\cdots,\Bar{x}_r)$ forms a $p$-completely regular sequence relative to $\mathcal{O}_K$.
\end{itemize}
\end{lemma}
\begin{proof}
    In the first situation, the desired result is \cite[Proposition 3.13]{bhatt2022prisms} as $(x_1,\cdots,x_r)$ forms $(p,d)$-completely regular sequence relative to $A$. For the second case, it suffices to check for $r=1$ and $x_1=(q-1)^t$. 
    Without loss of generality, we can further assume $\Tilde{R}=A$. To see $C=A\{\frac{(q-1)^{t}}{d}\}^{\wedge}_{\delta}$ is discrete and $d$-torsion free, by derived Nakayama it suffices to show that $C/^{\mathbb{L}} d$ is discrete. Actually we will prove that $C/^{\mathbb{L}} d\simeq \mathcal{O}_{\mathbb{G}_{a,\mathcal{O}_K/\pi^m}^{\sharp}}$ (here $\pi=q-1\in \mathcal{O}_K$). 

    For this purpose, we notice that by the proof of \cref{prop. lci case base to R} (which doesn't require $C/^{\mathbb{L}} d$ to be discrete), $\Spf(C/^{\mathbb{L}} d)$ (in the derived sense) represents the functor sending a test $\mathcal{O}_K/\pi^t$-algebra $S$ to $Y_{/A}^{\HT}(S)=\{y\in W(S)|~(q-1)^t= y d\}$ (here $Y=\Spf(\mathcal{O}_K/\pi^t)$). Hence it suffices to show the latter is represented by $\mathbb{G}_{a,\mathcal{O}_K/\pi^m}^{\sharp}$.

    First we construct a base point in $\Spf(C/^{\mathbb{L}} d)(\mathcal{O}_K/\pi^t)=Y_{/A}^{\HT}(\mathcal{O}_K/\pi^t)$. By definition (see \cite[Corollary 2.44]{bhatt2022prisms}), $C$ is obtained from the pushout diagram 
    \[\xymatrixcolsep{5pc}\xymatrix{A\{x\}_{\delta}^{\wedge}\ar[r]^{w: x\mapsto (q-1)^t} \ar[d]^{v: x\mapsto dz}&   A
\ar[d]^{}
\\ A\{z\}_{\delta}^{\wedge} \ar[r]^{} &C}\]
    in the $\infty$-category of simplicial commutative $\delta$-rings. Let $\iota: A\to A/(d, (q-1)^t)=\mathcal{O}_K/\pi^m$ be the usual quotient map. Then it can be extended to a ring homomorphism $g: A\{z\}_{\delta}^{\wedge}\to \mathcal{O}_K/\pi^m$ by sending $\delta^i(z)$ to $0$ for all $i\geq 0$. Let $\Tilde{g}: A\{z\}_{\delta}^{\wedge}\to W(\mathcal{O}_K/\pi^m)$ (resp. $\Tilde{\iota}: A\to W(\mathcal{O}_K/\pi^m)$) be the unique $\delta$-ring homomorphism lifting $g$ (resp. $\iota$).

    Let $a=q-1\in A=W(k)[[q-1]]$. Then $\varphi(a)=(a+1)^p-1$, hence $$\delta(a)=\frac{\varphi(a)-a^p}{p}=\frac{(a+1)^p-a^p-1}{p}\in (a).$$
As $\delta(xy)=x^p\delta(y)+y^p\delta(x)+p\delta(x)\delta(y)$ for $x,y\in A$, by induction we see that $\delta(a^n)\in (a^n), \forall n\geq 1$. Indeed, a stronger statement that $\delta^k(a^n)\in (a^n), \forall k\geq 1$ also holds. We proceed by induction on $k$. Suppose that the claim holds up to $k$ and let $c\in A$ satisfies that $\delta^k(a^n)=a^nc$, then 
$$\delta^{k+1}(a^n)=\delta(\delta^{k}(a^n))=\delta(a^nc)=\frac{\varphi(a^nc)-a^{np}c^p}{p}=\frac{((a+1)^p-1)^n\varphi(c)-a^{np}c^p}{p}=a^{np}\delta(c)+(a^n)\in (a^n),$$
hence $\delta^{k+1}(a^n)\in (a^n)$ and we finish the induction.

The above analysis shows that $\iota\circ w: A\{x\}_{\delta}^{\wedge}\to 
\mathcal{O}_K/\pi^t$ kills all of $\delta^{i}(x)$, $i\geq 0$, which implies $\Tilde{\iota}\circ w: A\{x\}_{\delta}^{\wedge} \to W(\mathcal{O}_K/\pi^t)$ also kills all of $\delta^{i}(x)$, $i\geq 0$ (by the uniqueness of $\delta$-lifting of $\iota\circ w$). Consequently $\Tilde{l}((q-1)^t)=0$ and we get a base point $y=0\in Y_{/A}^{\HT}(\mathcal{O}_K/\pi^t)$.

Finally we construct a desired isomorphism between functors $Y_{/A}^{\HT}$ and $\mathbb{G}_{a,\mathcal{O}_K/\pi^m}^{\sharp}$ over $\Spf(\mathcal{O}_K/\pi^t)$. By \cref{lem. lci image of d}, there exists a unit $x_d\in W(\mathcal{O}_K)$ such that $d=V(F(x_d))$ holds in $W(\mathcal{O}_K)$. Guaranteeing the existence of such $x_d$, for a test $\mathcal{O}_K/\pi^t$-algebra $S$, sending $y\in \mathbb{G}_{a,\mathcal{O}_K/\pi^m}^{\sharp}(S)\simeq W(S)^{F=0}$ to $y\in Y_{/A}^{\HT}(S)$ defines an isomorphism of functors. The element $y$ satisfies $dy=0$ because $dy=V(F(x_d))\cdot y=V(F(x_d)\cdot F(y))=V(0)=0$. 
\end{proof}
The following lemma is used in the above proof.
\begin{lemma}\label{lem. lci image of d}
    Let $\iota: A=W(k)[[q-1]]\to W(\mathcal{O}_K)$ be the unique $\delta$-ring map lifting the quotient map $A\to \mathcal{O}_K=A/([p]_{q^{p^\alpha}})$. Then there exists a unit $x_{d}\in W(\mathcal{O}_K)$ such that $\iota(d)=V(F(x_{d}))$. 
\end{lemma}
\begin{proof}
    For simplicity, we will just write $d$ for $\iota(d)=\iota([p]_{q^{p^\alpha}})$ in the following proof. First notice that $d$ maps to $0$ under the projection $W(\mathcal{O}_K) \to \mathcal{O}_K$, hence it lies in $V(W(\mathcal{O}_K))$. It then suffices to show that $\lambda=V(F(x))$ has a solution $x_{d}=(x_0,x_1,\cdots)$ in $W(\mathcal{O}_K)$. As $\mathcal{O}_K$ is $p$-torsion free, the ghost map is injective, hence this equation is equivalent to that 
\begin{align}\label{appl.ghost identity}
    \forall n\geq 0, w_n(d)=w_n(V(F(x))).
\end{align}
    We will construct $x_{d}=(x_0,x_1,\cdots)$ inductively on $n$ by showing that the solution exists in $W_n(\mathcal{O}_K)$.

    For $n=0$, $w_0(d)=0 \in \mathcal{O}_K$, $w_0(V(F(x_d)))=0$, hence \cref{appl.ghost identity} always holds.

    For $n\geq 1$, we have that 
    \begin{equation*}
        w_n(d)=w_0(\varphi^n(d))=w_0([p]_{q^{p^{\alpha+n}}})=p,
    \end{equation*}
    and that
    \begin{equation*}
        w_n(V(F(x)))=pw_{n-1}(F(x))=pw_n(x)=p(\sum_{i=0}^n x_i^{p^{n-i}}p^i).
    \end{equation*}
    
    This implies if we take $x_0=1$ and $x_i=0$ for $i\geq 1$ then \cref{appl.ghost identity} holds for all $n$, i.e. $d=V(1)$ in $W(\mathcal{O}_K)$. Clearly such $x_{d}$ is a unit in $W(\mathcal{O}_K)$ as $x_0=1$ by construction.
\end{proof}

\begin{remark}\label{rem. ARL allow second case}
The main reason for considering the second case in \cref{lem. ARL discrete case} is that we are interested in classifying quasi-coherent complexes on the prismatization of $Y=\Spf(\mathcal{O}_K/\pi^t)$ later. Indeed, calculating the cohomology of the structure sheaf is already interesting and it (as well as its filtered versions) plays an important role in studying the $K$-thory of $\mathcal{O}_K/\pi^t$, see \cite{antieau2024k} and \cite{antieau2023prismatic} for this picture.
\end{remark}

Our goal is to study $\mathcal{D}(Y_{/A}^{\Prism})$, where $Y_{/A}^{\Prism}$ is the relative prismatization of $Y$ with respect to $(A,I)$ and hence fits into the following diagram 
\[ \xymatrix{ Y_{/\Tilde{R}}^{\Prism}  \ar[r]^{\rho} \ar[d]^{} & Y_{/A}^{\Prism} \ar[r]^{} \ar[d]^{} &Y^{\Prism} \ar[d]^{}\\ \Spf(\Tilde{R})\ar[r]^{\rho}
 &  X_{/A}^{\Prism} \ar[d]^{}\ar[r]^{} & X^{\Prism} \ar[d]\\ & \Spf(A) \ar[r]^{} &\Spf(\Bar{A})^{\Prism}}\]
where all squares in the diagram are pullback squares.

Our strategy is to use the covering map $Y_{/\Tilde{R}}^{\Prism} \to  Y_{/A}^{\Prism}$ in the above diagram, which will still be denoted as $\rho$. Indeed, $Y_{/\Tilde{R}}^{\Prism}$ is nothing but $\Spf(\Tilde{R}\{\frac{x_1,\cdots,x_r}{d}\}^{\wedge}_{\delta})$, explained by the next proposition. This result is not surprising as $Y_{/\Tilde{R}}^{\Prism}$ could be viewed as the affine stack of $\mathrm{R} \Gamma(Y/\Tilde{R}, \mathcal{O}_{\Prism})$, which is precisely the prismatic envelope by the universal property of the later stated in \cite[Proposition 3.13]{bhatt2022prisms}.
\begin{proposition}\label{prop. lci case base to R}
    $Y_{/\Tilde{R}}^{\Prism}$ (resp. $Y_{/\Tilde{R},n}^{\Prism}$) is represented by $\Spf(\Tilde{R}\{\frac{x_1,\cdots,x_r}{d}\}^{\wedge}_{\delta})$ (resp. $\Spf(\Tilde{R}\{\frac{x_1,\cdots,x_r}{d}\}^{\wedge}_{\delta}/d^n)$).
\end{proposition}
\begin{proof}
    We only prove the statement for $Y_{/\Tilde{R}}^{\Prism}$ as the argument for $Y_{/\Tilde{R},n}^{\Prism}$ is similar. For a $p$-nilpotent test algebra $S$ over $\Tilde{R}$ with structure morphism $f: \Tilde{R} \to S$, by definition we have that 
    $$Y_{/\Tilde{R}}^{\Prism}(S)=\mathrm{Map}_{\Tilde{R}}(R/(\Bar{x}_1,\cdots,\Bar{x}_r), W(S)/{}^{\mathbb{L}} d),$$
    where the mapping space is calculated in $p$-complete animated rings and $W(S)$ is viewed as an algebra over $\Tilde{R}$ by lifting $f$ to the unique $\delta$-ring homomorphism $\Tilde{f}: \Tilde{R}\to W(S)$. As the animated ring $R/(\Bar{x}_1,\cdots,\Bar{x}_r)$ (resp. $W(S)/{}^{\mathbb{L}} d$) is obtained from $\Tilde{R}$ (resp. $W(S)$) by freely setting $(d, x_1,\cdots,x_r)$ (resp. $d$) to be zero and that $\Tilde{R}$ is the initial object in the category of testing algebras, the above then simplifies to
    \begin{equation*}
        Y_{/\Tilde{R}}^{\Prism}(S)=\{(y_1,\cdots,y_r)\in W(S)^r|~x_i= y_i d, ~1\leq i\leq r\}\footnote{here we view $d,x_1,\cdots,x_r$ as elements in $W(S)$ by lifting $f$ to a $\delta$-ring homomorphism $\Tilde{f}: \Tilde{R}\to 
    W(S)$ first and then consider the image of $d, x_i$ in $W(S)$}.
    \end{equation*}
    Given any $y=(y_1,\cdots,y_r) \in Y_{/\Tilde{R}}^{\Prism}(S)$, the unique $\delta$-ring map $\Tilde{f}: \Tilde{R} \to W(S)$ lifting the structure map $f: \Tilde{R}\to S$ extends uniquely to a $\delta$-ring map $\Tilde{f}_y: \Tilde{R}\{\frac{x_1,\cdots,x_r}{d}\}^{\wedge}_{\delta} \to W(S)$ by sending $\delta^{i}(\frac{x_j}{d})$ to $\delta^{i}(y_j)$ ($i\geq 0, 1\leq j\leq r$) due to the universal property of $\Tilde{R}\{\frac{x_1,\cdots,x_r}{d}\}^{\wedge}_{\delta}$. Considering the composition of the projection $W(S)\to S$ and $\Tilde{f}_y$, we obtain a ring homomorphism $f_y: \Tilde{R}\{\frac{x_1,\cdots,x_r}{d}\}^{\wedge}_{\delta} \to S$, corresponding to a point $f_y\in \Spf(\Tilde{R}\{\frac{x_1,\cdots,x_r}{d}\}^{\wedge}_{\delta})$.

    Conversely, given a morphism $f:\Tilde{R}\{\frac{x_1,\cdots,x_r}{d}\}^{\wedge}_{\delta} \to S$, $f$ uniquely lifts to a $\delta$-ring morphism $\Tilde{f}: \mathfrak{S}\{\frac{p^m}{\lambda}\}_{\delta}^{\wedge} \to W(S)$ by the universal property of Witt rings. Let $y_{f,j}=\Tilde{f}(\frac{x_j}{d})$ for $1\leq j\leq r$, then it satisfies that $x_j= y_{f,j}d$, hence $y_{f}:=(y_{f,1},\cdots, y_{f,r})$ determines a point in $Y_{/\Tilde{R}}^{\Prism}(S)$.

    To see $y_{f_y}=y$, we just need to notice that given $y\in  Y_{/\Tilde{R}}^{\Prism}(S)$ , the $\Tilde{f}_y: \Tilde{R}\{\frac{x_1,\cdots,x_r}{d}\}^{\wedge}_{\delta} \to W(S)$ constructed above is precisely the $\delta$-ring morphism lifting $f_y: \Tilde{R}\{\frac{x_1,\cdots,x_r}{d}\}^{\wedge}_{\delta} \to S$ by construction.

    Finally, for the purpose of showing that $f_{y_f}=f$, it suffices to observe that given $f: \Tilde{R}\{\frac{x_1,\cdots,x_r}{d}\}^{\wedge}_{\delta} \to S$, $\Tilde{f}=\Tilde{f}_{y_f}$ by our construction. Then we are done.
\end{proof}
To ease the notation, from now on we denote $\Tilde{R}\{\frac{x_1,\cdots,x_r}{d}\}^{\wedge}_{\delta}$ as $\Prism_X(Y)$. Given \cref{prop. lci case base to R}, we have the following pullback square 
\[\xymatrixcolsep{8pc}\xymatrix{ \Spf(\Prism_X(Y))=\Spf(\Tilde{R}\{\frac{x_1,\cdots,x_r}{d}\}^{\wedge}_{\delta})  \ar[r]^{\rho} \ar[d]^{} & Y_{/A}^{\Prism}  \ar[d]^{} \\ \Spf(\Tilde{R})\ar[r]^{\rho}
 &  X_A^{\Prism} ,}\]
 hence we could apply the strategy in the previous subsection to study $\mathcal{D}(Y_{/A}^{\Prism})$. As usual, we construct $\psi$ first.
\begin{lemma}\label{lemt.lci case appl extend eta}
Let $S=\Tilde{R}\oplus \epsilon \Omega_{\Tilde{R}/A}$ be as that in \cref{lem. twist S in the relative case}. Then the $A$-linear homomorphism $\psi: \Tilde{R} \to  S$ constructed in \cref{lem. twist S in the relative case} uniquely extends to a $\delta$-ring homomorphism $$ \XY \to \XY\otimes_{\Tilde{R}} S,$$
which will still be denoted as $\psi$ by abuse of notation. Moreover, this further induces an $\psi: \Tilde{R}/(d,x_1,\cdots,x_r)^n \to \XY/d^n\otimes_{\Tilde{R}} S$ after modulo $d^n$ for any $n\in \mathbb{N}$.
\end{lemma}
\begin{proof}
Arguing as that in \cref{rem. delta ring structure on s relative}, $\XY\otimes_{\Tilde{R}} S$ could be promoted to a $\delta$-ring extending the $\delta$-structure on $\XY$. Moreover, as $$\psi(x_j)=x_j+\epsilon \nabla_q(x_j)=x_j+\epsilon\cdot d\sum_{i=1}^m\frac{\gamma_i(x_j)-x_j}{q^p T_i-T_i}dT_i, \quad \forall 1\leq j\leq r$$ we have the following holds in $\XY\otimes_{\Tilde{R}} S$:
\begin{equation*}
    \frac{\psi(x_j)}{d}=\frac{x_j}{d}+\epsilon\cdot \sum_{i=1}^m\frac{\gamma_i(x_j)-x_j}{q^p T_i-T_i}dT_i.
\end{equation*}
Hence $\psi:\Tilde{R}\to S$ (which is a $\delta$-ring homomorphism thanks to \cref{prop. rela psi preserve delta structure}) extends to a unique $\delta$-ring homomorphism $\psi: \XY \to \XY\otimes_{\Tilde{R}} S$  sending $\frac{x_j}{d}$ to $\frac{x_j}{d}+\epsilon\cdot \sum_{i=1}^m\frac{\gamma_i(x_j)-x_j}{q^p T_i-T_i}dT_i$ by the universal property of $\XY$. For the moreover part, just notice that $\psi$ preserves the $d$-adic filtration and $\psi(x_j)\in (d)(\XY\otimes_{\Tilde{R}} S)$.
\end{proof}
\begin{remark}\label{rem. lci replace psi by i}
    The same argument shows that $\psi_i$ constructed in \cref{rem. replace psi by i} also extends to $\XY$. Consequently, $\nabla_i=\frac{\psi_i-\Id}{\epsilon_i}$ extends to $\XY$ as well. By defining $\gamma_i=\Id+T_i\beta \nabla_i$, we get an extension of $\gamma_i$ to $\XY$, which is still an automorphism of $\XY$. Moreover, the relations that $\nabla_{i}\circ \nabla_{j}=\nabla_{j}\circ \nabla_{i}$ and that $\gamma_i\circ \gamma_j=\gamma_j\circ \gamma_i$ still hold.
\end{remark}
\begin{warning}
    A key difference with the smooth case studied in the last subsection is that $\nabla_i$ constructed in the last remark doesn't need to vanish on the Hodge-Tate locus $\XY/d$ (for example, $\nabla_i(\frac{x_i}{d})$ is not necessarily in $(d)\XY$). This implies that although $\mathcal{D}(X_{/A}^{\HT})$ can be realized as modules over the commutative ring $R[\underline{\nabla}]$, in general this doesn't hold for $Y$. Instead, certain non-commutative \textit{Weyl-algebras} (which is a special kind of Ore extension) will always enter the picture for the purpose of classifying $\mathcal{D}(Y_{/A}^{\HT})$. This phenomenon already appears for the Hodge-Tate stack of $\mathbb{Z}_p/p
    ^n$, see \cite[Lemma 6.13]{petrov2023non} and \cite[Theorem 6.9, Remark 6.12]{liu2024prismatization} for details.
\end{warning}

Given \cref{lemt.lci case appl extend eta}, we can restrict \cref{propt.key automorphism of functors relative case} to $Y_{/A}^{\Prism}$ as follows.
\begin{proposition}\label{propt. lci key automorphism of functors relative case}
    $\gamma_{c_{\psi}}$ constructed in \cref{propt.key automorphism of functors relative case} restricts to an isomorphism $\gamma_{c_{\psi}}$ between functors $\rho: \Spf(\XY\otimes_{\Tilde{R}} S) \xrightarrow{\iota} \Spf(\XY) \xrightarrow{\rho} Y_{/A}^{\Prism}$
    and $\rho\circ \psi: \Spf(\XY\otimes_{\Tilde{R}} S) \xrightarrow{\psi} \Spf(\XY)\rightarrow Y_{/A}^{\Prism}$, i.e. we have the following commutative diagram:
\[\xymatrixcolsep{5pc}\xymatrix{\Spf(\XY\otimes_{\Tilde{R}} S)\ar[d]^{\iota}\ar[r]^{\psi}& \Spf(\XY)\ar@{=>}[dl]^{\gamma_{c_{\psi}}} \ar[d]_{}^{\rho}
\\\Spf(\XY)  \ar^{\rho}[r]&Y_{/A}^{\Prism}}\]
\end{proposition}
\begin{proof}
      By abuse of notation we write $\iota: \Tilde{R}\to \XY\otimes_{\Tilde{R}} S$ for the composition of $\Tilde{R}\to \XY\xrightarrow{\iota} \XY\otimes_{\Tilde{R}} S$, then $\rho(\XY\otimes_{\Tilde{R}} S)$ corresponds to the point $$(\alpha:(d) \otimes_{A,\Tilde{\iota}} W(\XY\otimes_{\Tilde{R}} S)\to W(\XY\otimes_{\Tilde{R}} S),\eta: \Cone((d,x_1,\cdots,x_r)\to \Tilde{R})\xrightarrow{\Tilde{\iota}} \Cone(\alpha))$$
      in $X^{\Prism}(S)$, while $(\rho\circ \psi)(S)$ corresponds to the point $$(\alpha:(d) \otimes_{A,\Tilde{\iota}} W(\XY\otimes_{\Tilde{R}} S)\to W(\XY\otimes_{\Tilde{R}} S), \eta^{\prime}: \Cone((d,x_1,\cdots,x_r)\to \Tilde{R})\xrightarrow{\Tilde{\psi}} \Cone(\alpha)).$$ 
      Arguing as \cref{lem. s torsion freeness}, we deduce that $W(\XY\otimes_{\Tilde{R}} S)$ is $d$-torsion free, hence the source of $\alpha$ is just the ideal generated by $d$ in $W(\XY\otimes_{\Tilde{R}} S)$. 
      By applying the last statement in \cref{lemt.lci case appl extend eta} and drawing the following diagram as maps of quasi-ideals (in the sense of \cite{drinfeld2020prismatization}),
\[\xymatrixcolsep{5pc}\xymatrix{(d,x_1,\cdots,x_r) \ar@<-.5ex>[d]_{\Tilde{\psi}} \ar@<.5ex>[d]^{\Tilde{\iota}}\ar[r]^{\iota}& \Tilde{R} \ar[d]_{\Tilde{\psi}}
 \ar@<-.5ex>[d]^{\Tilde{\iota}}
\\(d) \ar[r]^{\iota}&W(\XY\otimes_{\Tilde{R}} S).}\]
we derive that $c_{\psi}$ constructed in \cref{prop. relative c to construct homotopy} serves a homotopy between $\eta^{\prime}$ and $\eta$, hence finish the proof.
\end{proof}
Then following the discussion between \cref{prop. rel two derivation vanish} and \cref{prop. relative induce functor} by taking \cref{propt. lci key automorphism of functors relative case} as the input replacing \cref{propt.key automorphism of functors relative case}, we obtain the following analog of \cref{prop. relative induce functor}:
\begin{proposition}\label{prop. lci induce functor}
    For $n\in \mathbb{N}\cup \{\infty\}$, the pullback along $\rho: \Spf(\XY/d^n)\to Y_{/A,n}^{\Prism}$ induces a functor 
  \begin{equation*}
  \begin{split}
      \mathcal{D}(Y_{/A,n}^{\Prism}) &\to \mathcal{D}(\XY/d^n[\underline{\nabla}; \gamma_{\XY}, \nabla_{\XY}])\\
      \mathcal{E} &\mapsto (\rho^{*}\mathcal{E}, \nabla_{\mathcal{E}})
  \end{split}
  \end{equation*}
  which will be denoted as $\beta_n^+$ later.
\end{proposition}
\begin{remark}\label{rem. lci. explicit definition}
    The non-commutative ring $\XY/d^n[\underline{\nabla}; \gamma_{\XY}, \nabla_{\XY}]$ stated in the last proposition is defined similarly to \cref{rel def.skew polynomial}. More precisely, let $\gamma_{\XY,i}: \XY\to \XY$ be the ring automorphism $\gamma_i$ defined in \cref{rem. lci replace psi by i} extending $\gamma_{\Tilde{R},i}$ on $\Tilde{R}$, then $\nabla_{\XY,i}: \XY\to \XY$, denoted as $\nabla_i$ in \cref{rem. lci replace psi by i}, is a $\gamma_{\XY,i}$-derivation of $\XY$, i.e. $\nabla_{\XY,i}(x_1x_2)=\gamma_{\XY,i}(x_1)\nabla_{\XY,i}(x_2)+\nabla_{\XY,i}(x_1)x_2$. We define the \textit{Ore extension} $\XY[\underline{\nabla}; \gamma_{\XY}, \nabla_{\XY}]$ to be the noncommutative ring obtained by giving the ring of polynomials $\XY[\nabla_1,\cdots,\nabla_m]$ a new multiplication law, subject to the identity
    \begin{equation*}
        \begin{split}
           \nabla_i\cdot \nabla_j&=\nabla_j\cdot \nabla_i,  ~~~\forall ~ 1\leq i,j\leq m
           \\ \nabla_i \cdot r&=\gamma_{\XY,i}(r)\cdot\nabla_i+\nabla_{\XY,i}(r), ~~~\forall~ r\in \XY, 1\leq i\leq m .
        \end{split}
    \end{equation*}
\end{remark}

Then finally we can state our classification results for relative prismatic crystals on $(Y/A)_{\Prism}$, generalizing \cref{thmt. relative main classification} to the locally complete intersection case.
\begin{theorem}\label{thmt. lci main classification}
Let $n\in \mathbb{N}\cup \{\infty\}$. The functor 
    \begin{align*}
        &\beta_n^{+}: \mathcal{D}(Y_{/A, n}^{\Prism}) \rightarrow \mathcal{D}(\XY/d^n[\underline{\nabla}; \gamma_{\XY}, \nabla_{\XY}]), \qquad \mathcal{E}\mapsto (\rho^{*}(\mathcal{E}),\nabla_{\mathcal{E}})
    \end{align*}
    constructed in \cref{prop. lci induce functor} is fully faithful. 
     Moreover, its essential image consists of those objects $M\in \mathcal{D}(\XY/d^n[\underline{\nabla}; \gamma_{\XY}, \nabla_{\XY}])$ 
     satisfying the following pair of conditions:
    \begin{itemize}
        \item $M$ is $(p,d)$-adically complete.
        \item The action of $\nabla_{i}$ on the cohomology $\mathrm{H}^*(M\otimes_{\XY/d^n}^{\mathbb{L}}\XY/(d,p))$ 
        is locally nilpotent for all $i$. 
    \end{itemize}
\end{theorem}
\begin{proof}
Given \cref{prop. lci induce functor}, the strategy we employed proving \cref{thmt. relative main classification} still works once we show that the theorem holds for $n=1$, i.e. the Hodge-Tate case. 

For this purpose, first we notice that $\rho: X=\Spf(R) \to X_{/A}^{\HT}$ is a cover with automorphism group $T_{X/\Bar{A}}^{\sharp}\{1\}\cong (\mathbb{G}_a^{\sharp})^m$ (see \cite[Proposition 5.12]{bhatt2022prismatization} and discussion after \cref{propt. relative factor through fiber}, here $\mathbb{G}_a^{\sharp}$ is defined over $X$), hence by restricting the pullback square above \cref{lemt.lci case appl extend eta} to the Hodge-Tate locus, we obtain the following pullback diagram
\[\xymatrixcolsep{8pc} \xymatrix{ \Spf(\XY/d) \ar[r]^{\rho^{\prime}} \ar^{\pi^{\prime}}[d] &  Y_{/A}^{\HT}=\Spf(\XY/d)/(\mathbb{G}_{a}^{\sharp})^m_{Y} \ar^{\pi}[d] \\
X=\Spf(R) \ar[r]^{\rho} & X_{/A}^{\HT}=X/(\mathbb{G}_a^{\sharp})^m=\Spf(R)/(\mathbb{G}_a^{\sharp})^m.}\]
Here $(\mathbb{G}_{a}^{\sharp})^m_{Y}$ is the base change of $(\mathbb{G}_{a}^{\sharp})^m$ along the closed embedding $Y\to X$.

Consequently, we have that 
\begin{equation}\label{equa. lci identify}
    \mathcal{D}(Y_{/A}^{\HT})=\Mod_{\pi_*\mathcal{O}}(\mathcal{D}(X/(\mathbb{G}_a^{\sharp})^m))=\Mod_{\beta_1^+(\pi_*\mathcal{O})}(\mathcal{D}_{\nil}(R[\underline{\nabla}])).
\end{equation}
Here the first identity holds as $\pi$ is affine, and the second equality follows from \cref{thmt. relative main classification} by taking $n=1$ there. Also, we use $\mathcal{D}_{\nil}(R[\underline{\nabla}])$\footnote{Notice that when $n=1$, $\Tilde{R}/d^n[\underline{\nabla}; \gamma_{\Tilde{R}}, \nabla_{\Tilde{R}}]$ is commutative and is just the free polynomial over $R$ with $m$-variables $\nabla_i (1\leq i\leq m)$, hence we just write $R[\underline{\nabla}]$ for simplicity.} to denote the essential image of $\beta_1^+$ stated in \cref{thmt. relative main classification} for simplicity. 

But we claim that the right-hand side of \cref{equa. lci identify} is exactly the category stated in \cref{thmt. lci main classification}.

For this purpose, we first notice that the underlying complex of $\beta_1^+(\pi_*\mathcal{O})$ is given by $\rho^{*}\pi_*\mathcal{O}=\pi^{\prime}_*\rho^{\prime *}\mathcal{O}=\XY/d$. Moreover, unwinding the identification of $Y_{/\Tilde{R}}^{\HT}(S)$ with $\Spf(\XY/d)$ in \cref{prop. lci case base to R} as well as the identification of $\mathrm{Aut}(\rho)$ with $(\mathbb{G}_a^{\sharp})^m$, we see that the $(\mathbb{G}_a^{\sharp})^m$-action on $\XY/d$ is given (hence is also determined) by the usual addition action on $\frac{x_i}{d}$, namely given a test algebra $Q$ over $Y=\Spf(R/(\Bar{x}_1,\cdots,\Bar{x}_r))$,
\begin{equation*}
    \begin{split}
        (\mathbb{G}_a^{\sharp})^m(Q)\times Y_{/\Tilde{R}}^{\HT}(Q)&\longrightarrow Y_{/\Tilde{R}}^{\HT}(Q)
        \\((t_1,\cdots,t_r), (y_1,\cdots,y_r))&\longmapsto (t_1+y_1,\cdots,t_r+y_r).
    \end{split}
\end{equation*}
This implies that the $q$-Higgs derivation associated with $\pi_*\mathcal{O}$ via \cref{thmt. relative main classification} is precisely $\nabla_{\XY}$ described in \cref{rem. lci. explicit definition}. Consequently, we see that $$\beta_1^+(\pi_*\mathcal{O})=(\XY/d, \nabla_{\XY}) 
\in \mathcal{D}_{\nil}(R[\underline{\nabla}]).$$

Finally, we notice that there is a canonical tensor structure on $\mathcal{D}_{\nil}(R[\underline{\nabla}])$ inherited from that on $\mathcal{D}(X^{\HT})$ via \cref{thmt. relative main classification}. In particular, for $\mathcal{E}, \mathcal{F}\in \mathcal{D}_{\nil}(R[\underline{\nabla}])$, the underlying complex for $\mathcal{E}\otimes \mathcal{F}$ is $\mathcal{E}\otimes_R \mathcal{F}$ and $\nabla_{\mathcal{E}\otimes \mathcal{F},i}$ is given by $\nabla_{\mathscr{E}, i}\otimes\Id_{\mathscr{E}^{\prime}}+\Id_{\mathscr{E}}\otimes \nabla_{\mathscr{E}^{\prime}, i}+\beta T_i\nabla_{\mathscr{E}, i}\otimes \nabla_{\mathscr{E}^{\prime}, i}$
according to \cref{rem. rel twisted leibnitz in general}.   Consequently $\mathcal{H}\in \mathcal{D}_{\nil}(R[\underline{\nabla}])$ is a $\beta_1^+(\pi_*\mathcal{O})$-module if and only if the underlying complex of $\mathcal{H}$ is a $\XY/d$-module and the $\underline{\nabla}$ action on $\mathcal{H}$ satisfies that $$\nabla_{i}\cdot a=\gamma_{\XY}(a)\cdot\nabla_i+\nabla_{\XY,i}(a), ~~~\forall a\in \XY/d.$$ Such conditions exactly determine an object in $\mathcal{D}(\XY/d[\underline{\nabla}; \gamma_{\XY}, \nabla_{\XY}])$, hence we finish the proof.
\end{proof}

\section{On the absolute prismatization}
\subsection{Smooth case}
In this section we work with $\mathcal{O}_K=W(k)[\zeta_{p^{\alpha+1}}]$ ($\alpha\geq 0$) and the $q$-prism $(A, I)=(W(k)[[q-1]], [p]_{q^{p^\alpha}})$. 
We consider $X=\Spf(R)$, a smooth $p$-adic formal scheme over $\overline{A}=\mathcal{O}_K$. As in the last section, we assume that $X$ is a small affine and fix the following diagram
\[\xymatrixcolsep{5pc}\xymatrix{A\left\langle \underline{T}^{}\right\rangle\ar[r]^{\square} \ar[d]&   \Tilde{R}
\ar[d]^{}
\\ \Bar{A}\left\langle \underline{T}^{}\right\rangle\ar[r] \ar[r]^{\square} &R,}\]
where the horizontal maps are \'etale chart maps. 
By deformation theory, this \'etale chart map uniquely lifts to a prism $(\Tilde{R}, I)$ over $(A\left\langle \underline{T}\right\rangle, I)$, here $A\left\langle \underline{T}\right\rangle$ is equipped with the $\delta$ structure respecting that on $A$ and sending $T_i$ to $0$, then we apply \cite[Lemma 2.18]{bhatt2021prismatic} to uniquely extend such a $\delta$-structure to $\Tilde{R}$. Moreover, the induced map $A\left\langle \underline{T}^{}\right\rangle \to \Tilde{R}$ is $(p, I)$-completely \'etale. In particular, $(\Tilde{R}, I)$ is a prism in $(X/A)_{\Prism}$ and $\Tilde{R}/I=R$. We have the following diagram 
\[ \xymatrix{  \Spf(\Tilde{R})\ar@/^2pc/[rr]^{\rho_X}\ar[r]^{\rho}
 &  X_{/A}^{\Prism} \ar[d]^{}\ar[r]^{\rho_A} & X^{\Prism} \ar[d]\\ & \Spf(A) \ar[r]^{\rho_A} &\Spf(\mathcal{O}_K)^{\Prism},}\]
here the square is a pullback square and we use $\rho_A$ to denote the covering map, $\rho_X=\rho_A\circ \rho$.

In Section 3, we studied the difference between $\mathcal{D}(\Spf(A))$ and $\mathcal{D}(\Spf(\mathcal{O}_K)^{\Prism})$, which turns out to be governed by an operator $\partial_{\mathcal{E}}$ satisfying certain twisted Leibnitz rule for $\mathcal{E}\in \mathcal{D}(\Spf(\mathcal{O}_K)^{\Prism})$ by \cref{thmt.main classification}. Then later in Section 4, we equipped a $q$-Higgs derivation $\nabla_{\mathcal{F}}$ on $\rho^*\mathcal{F}$ for $\mathcal{F}\in \mathcal{D}(X_{/A}^{\Prism})$, which is enough for the purpose of classify $\mathcal{D}(X_{/A}^{\Prism})$. With such results in hand, it is natural to expect that for $\mathcal{G}\in \mathcal{D}({X^{\Prism}})$, $\rho_X^* \mathcal{G}$ is equipped with both a $q$-connection $\partial_{\mathcal{F}}: \rho_X^* \mathcal{G}\to \rho_X^* \mathcal{G}$ and a $q$-Higgs derivation $\nabla_{\mathcal{G}}: \rho_X^* \mathcal{G}\to \rho_X^* \mathcal{G}\otimes_{\Tilde{R}} \Omega_{\Tilde{R}/A}$. Moreover, we should be able to fully understand $\mathcal{D}({X^{\Prism}})$ with the help of $\partial$ and $\nabla$. We aim to show this is indeed the case in this section.

Per \cref{rem. relative change to alpha}, in this section $\beta=q^{p^\alpha}-1$. For any $i\in \{1,\cdots,m\}$, we define the automorphism $\gamma_i$ of $A\left\langle \underline{T}^{}\right\rangle$ fixing $A$ and sending $T_i$ to $q^{p^{\alpha+1}}T_i$ and $T_j$ to $T_j$ for $j\neq i$. We further define an automorphism $\gamma_0$ of $A\left\langle \underline{T}^{}\right\rangle$ fixing $T_i$ ($1\leq i\leq m$) and sending $q$ to $q^{p^{\alpha+1}}\cdot q=q^{p^{\alpha+1}+1}$. Then these automorphisms lift uniquely to automorphisms $\gamma_i$ of $\Tilde{R}$ following the discussions above \cref{lem. twist S in the relative case}. Moreover, one can check by hand that they satisfy the following:
\begin{equation}\label{equa. relation between j i}
    \gamma_i\circ \gamma_j=\gamma_j\circ \gamma_i, ~~1\leq i, j\leq m,\quad \gamma_0\circ \gamma_i=\gamma_i^{p^{\alpha+1}+1}\gamma_0, 1\leq i\leq m.
\end{equation}

Following the notation in the last section (above \cref{lem. twist S in the relative case}), we define $\partial$ and $\nabla_{i}$ on $\Tilde{R}$ via 
\begin{equation}\label{equa. differentials on the structure sheaf}
    \partial(f)=d\frac{\gamma_0(f)-f}{q^{p^{\alpha+1}}\cdot q-q}=\frac{\gamma_0(f)-f}{q\beta}, \quad \nabla_{i}(f)=d\frac{\gamma_i(f)-f}{q^{p^{\alpha+1}} T_i-T_i}=\frac{\gamma_i(f)-f}{\beta T_i}, \quad \nabla_{}(f)=\sum_{i=1}^{m} \nabla_{i}(f)\epsilon_i.
\end{equation}

As a consequence of \cref{equa. relation between j i}, 
\begin{equation}\label{equa. relation between nabla i, j}
\begin{split}
    \nabla_{i}\circ \nabla_{j}&=\nabla_{j}\circ \nabla_{i}\quad  1\leq i, j\leq m,
    \\(\Id+q\beta \partial) \circ (\Id+T_i\beta \nabla_{i})&=(\Id+T_i\beta \nabla_{i})^{p^{\alpha+1}+1}\circ (\Id+q\beta \partial) \quad 1\leq i\leq m.
\end{split}
\end{equation}
Equation \ref{equa. relation between nabla i, j} could then be reinterpreted as the following lemma, which will be used later.
\begin{lemma}\label{lem. AR relation between 0 and i}
 $\forall ~1\leq i\leq m$, 
    \begin{equation*}
   (1+\beta q \mathscr{D}(\nabla_{i})) \nabla_{i}\circ \partial=s_0(\partial-s_0^{-1}\mathscr{D}(\nabla_{i})+s_1)\circ \nabla_{i},
\end{equation*}
where $s_0=\frac{\gamma_0(\beta)}{\beta [k]_{q^{p^{\alpha+1}}}} \in A^{\times}$ and $s_1=q^{-1}\frac{(1+q\partial(\beta))-[k]_{q^{p^{\alpha+1}}}}{\gamma_0(\beta)}$ for $k=p^{\alpha+1}+1$, $\mathscr{D}(\nabla_{i})$ is defined in the proof.
\end{lemma}
\begin{proof}
    For ease of notation, we omit $i$ and write $\nabla_{i}$ as $\nabla$ in the proof. First we show that for all $k\in \mathbb{N}$, there exists a sequence $\{a_{k,j}\}_{1\leq j\leq k}$ such that   \begin{equation*}
        (\Id+T\beta \nabla_{})^k=\Id+\sum_{j=1}^k a_{k,j}(T\beta\nabla)^j q^{\frac{j(j-1)}{2}p^{\alpha+1}}
    \end{equation*}
    and that $$a_{k,1}=[k]_{q^{p^{\alpha+1}}},\quad a_{k+1,j}=a_{k,j}(1+\beta d [j]_{q^{p^{\alpha+1}}})+a_{k,j-1},\quad j\geq 1, \text{here we define}~~a_{k,0}=1.$$
    We do induction on $k$. Clearly it holds when $k=1$. Suppose the claim holds up to $k$. By invoking \cref{lem. twist S in the relative case} and noticing that $\nabla(T)=d, \nabla_{}(T^j)=T^{j-1}d[j]_{q^{p^{\alpha+1}}}$, hence
    \begin{equation*}
        \begin{split}
            (\Id+T\beta \nabla_{})^{k+1}&=(\Id+T\beta \nabla_{})\circ(\Id+T\beta \nabla_{})^k=(\Id+T\beta \nabla_{})\circ (\Id+\sum_{j=1}^k a_{k,j}(T\beta\nabla)^j q^{\frac{j(j-1)}{2}p^{\alpha+1}})
        \\&=\Id+T\beta\nabla+\sum_{j=1}^{k+1} a_{k,j}(T^j\beta^j \nabla^j+\beta^{j+1}T\nabla\circ (T^j \nabla^j))q^{\frac{j(j-1)}{2}p^{\alpha+1}}
        \\&=\Id+\sum_{j=1}^{k+1} (a_{k,j}T\beta\nabla^jq^{\frac{j(j-1)}{2}p^{\alpha+1}}(1+\beta d [j]_{q^{p^{\alpha+1}}})+a_{k,j-1}\beta^j T q^{\frac{j(j-1)}{2}p^{\alpha+1}} \cdot q^{p^{\alpha+1}(j-1)}T^{j-1}\nabla^j)
        \\&=\Id+\sum_{j=1}^{k+1} (a_{k,j}(1+\beta d [j]_{q^{p^{\alpha+1}}})+a_{k,j-1})(T\beta\nabla)^j q^{\frac{j(j-1)}{2}p^{\alpha+1}}.
        \end{split}
    \end{equation*}
    This implies the desired induction formula for $a_{k,j}$. It then follows that $a_{k,1}=[k]_{q^{p^{\alpha+1}}}$ by induction and observing that $\beta d=q^{p^{\alpha+1}}-1$.

    From now we fix $k=p^{\alpha+1}+1$ and rewite $a_{k,j}$ just as $a_j$. Then the above calculation tells us that 
    \begin{equation*}
        \begin{split}
            (\Id+T\beta \nabla_{})^{p^{\alpha+1}+1}\circ (\Id+q\beta \partial)&=(\Id+[k]_{q^{p^{\alpha+1}}} T\beta \nabla+\sum_{j=2}^{k} a_{j}(T\beta\nabla)^j q^{\frac{j(j-1)}{2}p^{\alpha+1}}) (\Id+q\beta \partial)
            \\&= \Id+q\beta \partial+[k]_{q^{p^{\alpha+1}}} (T\beta \nabla + qT\beta^2 \nabla\circ \partial+qT\beta^2\mathscr{D}(\nabla)\nabla+q^2\beta^3 T\mathscr{D}(\nabla)\nabla\partial),
        \end{split}
    \end{equation*}
    here we define $\mathscr{D}(\nabla)=\frac{1}{[k]_{q^{p^{\alpha+1}}}}\sum_{j=2}^{k} a_{j} T^{j-1}\beta^{j-2}\nabla^{j-1} q^{\frac{j(j-1)}{2}p^{\alpha+1}-1}$, which makes sense as $[k]_{q^{p^{\alpha+1}}} \in \Tilde{R}^{\times}$. Indeed, it projects to the unit $k=p^{\alpha+1}+1$ after modulo $q-1$.

    On the other hand, one calculates that
    \begin{equation*}
    \begin{split}
        &(\Id+q\beta \partial) \circ (\Id+T\beta \nabla_{})=\Id+q\beta \partial+T\beta \nabla +qT\beta \partial(\beta \nabla)
        \\=&\Id+q\beta \partial+T\beta (1+q\partial(\beta))\nabla+qT\beta \gamma_0(\beta) \partial\nabla.
    \end{split}
    \end{equation*}

Notice that 
\begin{equation*}
    \frac{\gamma_0(\beta)}{\beta}=\frac{q^{kp^\alpha}-1}{q^{p^{\alpha}}-1}=[k]_{q^{p^{\alpha}}}\in \Tilde{R}^{\times}
\end{equation*}
as its projection after modulo $q-1$ is given by $k=p^{\alpha+1}+1$, which is a unit.

Also, $[k]_{q^{p^{\alpha+1}}}-(1+q\partial(\beta))=[k]_{q^{p^{\alpha+1}}}-q^{p^\alpha} d [p^\alpha]_{q^{p^{\alpha+1}}}-1$ is divisible by $\beta$ as its image after modulo $\beta=q^{p^{\alpha}}-1$ is given by $k-1-p\cdot p^\alpha=0$.

Combining all of the calculations above, Equation \ref{equa. relation between nabla i, j} then states that
\begin{equation*}
\begin{split}
    (\nabla+\beta q \mathscr{D}(\nabla)\nabla)\circ \partial&=(\frac{\gamma_0(\beta)}{\beta [k]_{q^{p^{\alpha+1}}}}\partial-\mathscr{D}(\nabla)+q^{-1}\frac{(1+q\partial(\beta))-[k]_{q^{p^{\alpha+1}}}}{\beta [k]_{q^{p^{\alpha+1}}}})\circ \nabla
    \\&=s_0(\partial-s_0^{-1}\mathscr{D}(\nabla)+s_1)\circ \nabla.
\end{split}
\end{equation*}
for $s_0=\frac{\gamma_0(\beta)}{\beta [k]_{q^{p^{\alpha+1}}}} \in A^{\times}$ and $s_1=q^{-1}\frac{(1+q\partial(\beta))-[k]_{q^{p^{\alpha+1}}}}{\gamma_0(\beta)}$.
\end{proof}
\begin{remark}\label{rem. actual s1}
    By considering the action of two sides of \cref{lem. AR relation between 0 and i} on $T_i\in \Tilde{R}$, we see that $s_1=-\frac{\partial(d)}{d}\in A$.
\end{remark}

The following lemma combines \cref{lemt.extend delta} and \cref{lem. twist S in the relative case}. For simplicity, we rewrite $\epsilon dT_i$ as $\epsilon_i$ for $1\leq i\leq m$ in this section.
\begin{lemma}\label{lem. AR twist S in the relative case}
    Let $S=\Tilde{R}\oplus_{i=0}^{m} \Tilde{R}\epsilon_i$ and we regulate the following algebra structure on $S$: \begin{equation*}
        (\epsilon_0)^2=\beta q\cdot\epsilon_0,\quad (\epsilon_i)^2=\beta T_i\cdot\epsilon_i, ~~1\leq i\leq m,\quad \epsilon_i\cdot \epsilon_j=0 ~~\text{for}~~ i\neq j.
    \end{equation*}
    Then the map 
    \begin{equation*}
        \begin{split}
            \psi : \Tilde{R} &\longrightarrow  S
            \\ f&\longmapsto =f+d\frac{\gamma_0(f)-f}{q^{p^{\alpha+1}}\cdot q-q}\epsilon_0+d\sum_{i=1}^m\frac{\gamma_i(f)-f}{q^{p^{\alpha+1}} T_i-T_i}\epsilon_i
        \end{split}
    \end{equation*}
    defines a ring homomorphism of $W(k)$-algebras. similarly, we get ring homomorphism $\psi_i$ as that in \cref{rem. replace psi by i} for $0\leq i\leq m$.
\end{lemma}
\begin{proof}
    This follows from \cref{lemt.extend delta} and \cref{lem. twist S in the relative case} after noticing that for $f=q^k$, 
    \begin{equation*}
        \begin{split}
            \frac{\gamma_0(f)-f}{q^{p^{\alpha+1}}\cdot q-q}=\frac{(q^{p^{\alpha+1}+1})^k-q^k}{q^{p^{\alpha+1}+1}-q}=q^{k-1}[k]_{q^{p^{\alpha+1}}},
        \end{split}
    \end{equation*}
    hence $\psi(q^k)$ coincides with that in \cref{lemt.extend delta}. 
\end{proof}
\begin{proposition}\label{propt.key automorphism of functors AR}
    The elements $b$ and $c$ constructed in \cref{lemT.construct b} and \cref{lemt.construct homotopy} together with $c_{\psi}$ constructed in \cref{prop. relative c to construct homotopy} induce an isomorphism $\gamma_{b,c_{\psi}}$ between functors $\rho: \Spf(S) \to \Spf(\Tilde{R}) \xrightarrow{\rho_X} X_{}^{\Prism}$
    and $\rho\circ \psi: \Spf(S) \xrightarrow{\psi} \Spf(\Tilde{R})\xrightarrow{\rho_X} X_{}^{\Prism}$, i.e. we have the following commutative diagram:
\[\xymatrixcolsep{5pc}\xymatrix{\Spf(S)\ar[d]^{\iota}\ar[r]^{\psi}& \Spf(\Tilde{R})\ar@{=>}[dl]^{\gamma_{b,c_{\psi}}} \ar[d]_{}^{\rho_X}
\\\Spf(\Tilde{R})  \ar^{\rho_X}[r]&X_{}^{\Prism}}\]
\end{proposition}
\begin{proof}
We write $\iota: \Tilde{R}\to S$ for the canonical map as usual, then $\rho(S)$ corresponds to the point $$(\alpha:(d) \otimes_{\Tilde{R},\Tilde{\iota}} W(S)\to W(S),\eta: \Cone((d)\to \Tilde{R})\xrightarrow{\Tilde{\iota}} \Cone(\alpha))$$
      in $X^{\Prism}(S)$, while $(\rho\circ \psi)(S)$ corresponds to the point $$(\alpha:(d) \otimes_{\Tilde{R},\Tilde{\psi}} W(S)\to W(S), \eta^{\prime}: \Cone((d)\to \Tilde{R}) \xrightarrow{\Tilde{\psi}} \Cone(\alpha)).$$ 

      We need to specify an isomorphism $\gamma_{b}: \alpha^{\prime}\xrightarrow{\simeq} \alpha$ as well as a homotopy $\gamma_{c_{\psi}}$ between $\gamma_{b}\circ \eta^{\prime}$ and $\eta$. $\gamma_{b}$ is constructed as that in \cref{propt.key automorphism of functors} by viewing $b$ constructed in \cref{lemT.construct b} as an element of $W(S)$. To construct the homotopy, we just need to notice that for any $x\in S$, there exists a unique $c_{\psi}(x)$ in $W(S)$ such that $\Tilde{\psi}_(x)-\Tilde{\iota}(x)=d \cdot c_{\psi}(x)$ by combining \cref{lemt.construct homotopy} and \cref{prop. relative c to construct homotopy}, then the proof of \cref{propt.key automorphism of functors} still works here.   
\end{proof}

Then following the discussion before \cref{prop. rel two derivation vanish}, for $\mathcal{E}\in \mathcal{D}(X_{}^{\Prism})$, $\rho_X^*\mathcal{E}$ is equipped with a $q$-connection $\partial_{\mathcal{E}}$ as well as a $q$-Higgs derivation $\nabla_{\mathcal{E}}=\sum_{i=1}^m \nabla_{\mathcal{E},i}\otimes dT_i$. We first prove they satisfy the following relation, as an analog of \cref{prop. rel two derivation vanish}. 
\begin{proposition}\label{prop. AR two derivation vanish}
For $\mathcal{E}\in \mathcal{D}(X_{}^{\Prism})$, $\nabla_{\mathcal{E}}=\sum_{i=1}^m \nabla_{\mathcal{E},i}\otimes dT_i$ and $\partial_{\mathcal{E}}$ on $\rho_X^*\mathcal{E}$ satisfies the following property: 
\begin{equation*}
    \begin{split}
        \nabla_{\mathcal{E}}\wedge \nabla_{\mathcal{E}}&=0,
        \\(\Id+q\beta \partial_{\mathcal{E}}) \circ (\Id+T_i\beta \nabla_{\mathcal{E},i})&=(\Id+T_i\beta \nabla_{\mathcal{E},i})^{p^{\alpha+1}+1}\circ (\Id+q\beta \partial_{\mathcal{E}}) \quad 1\leq i\leq m.
    \end{split}
\end{equation*}
Moreover, $\nabla_{\mathcal{E},i}$ and $\partial_{\mathcal{E}}$ satisfies the relation stated in \cref{lem. AR relation between 0 and i}, i.e. 
\begin{equation*}
    (\nabla_{\mathcal{E},i}+\beta q \mathscr{D}(\nabla_{\mathcal{E},i})\nabla_{\mathcal{E},i})\circ \partial_{\mathcal{E}}=s_0(\partial_{\mathcal{E}}-s_0^{-1}\mathscr{D}(\nabla_{\mathcal{E},i})+s_1)\circ \nabla_{\mathcal{E},i}.
\end{equation*}
\end{proposition}
\begin{proof}
   The first sentence is due to \cref{prop. rel two derivation vanish}. To prove the second statement,  we may assume that $i=1$. Let $S_0=\Tilde{R}[\epsilon_i]/(\epsilon_0^2-\beta q\epsilon_0)=\Tilde{R}\oplus \Tilde{R}\cdot \epsilon_0$ and $S_1=\Tilde{R}[\epsilon_1]/(\epsilon_1^2-\beta T_1\epsilon_0)=\Tilde{R}\oplus \Tilde{R}\cdot \epsilon_1$.
   For $i=0, 1$, we define $\psi_i:\Tilde{R}\to S_i=\Tilde{R}[\epsilon_i]$ as that in \cref{rem. replace psi by i}. 
   
   We consider the following commutative diagram 
   \[ \xymatrixcolsep{5pc}\xymatrix{  \Spf(\Tilde{R})\ar@/^2pc/[rr]^{\Id}\ar@{^{(}->}[r]^{j_1: \epsilon_1\mapsto \beta T_1}\ar[dr]^{\gamma_1}
 &  \Spf(S_1) \ar[d]^{\psi_1}\ar[r]^{\iota} & \Spf(\Tilde{R})\ar@{=>}[dl]^{\gamma_{b,c_{\psi_1}}} \ar[d]^{\rho_X}\\ & \Spf(\Tilde{R}) \ar[r]^{\rho_X} &X^{\Prism},}\]
 where $j_1$ is the closed embedding corresponding to the quotient map $S_1\to \Tilde{R}$ sending $\epsilon_1$ to $\beta T_1$. The commutativity of the left triangle follows from the observation that the image of $$\psi_1(f)=f+\epsilon_1 \nabla_{q,1}(f)=f+\epsilon_1\frac{\gamma_1(f)-f}{\beta T_1}$$
 in $\Tilde{R}$ after modulo $(\epsilon_1-\beta T_1)$ is precisely $\gamma_1(f)$.
 
Unwinding our construction of $\nabla_{\mathcal{E},1}$, this implies that the restriction of $\gamma_{b,c_{\psi_1}}$ along $j$ induces an isomorphism $\gamma_1^*\rho_X^* \mathcal{E}\cong  \rho_X^* \mathcal{E}$ given by 
\begin{equation*}
    \begin{split}
        \rho_X^* \mathcal{E} \otimes_{\Tilde{R},\gamma_1} \Tilde{R}&\xrightarrow{\simeq} \rho_X^* \mathcal{E}
        \\ x\otimes a &\longmapsto (x+\beta T_1\cdot \nabla_{\mathcal{E},1}(x))\cdot a
    \end{split}
\end{equation*}
Playing the same game for $\gamma_0$ and considering the following diagram \footnote{Here $\gamma$ is the unique homotopy makes the diagram commutes, it exists by the previous discussion.}
 \[ \xymatrixcolsep{5pc}\xymatrix{ \Spf(\Tilde{R})  \ar[r]^{\Id} \ar[d]_{\gamma_1\circ \gamma_0=\gamma_0 \circ \gamma_1^{p^{\alpha+1}+1}} & \Spf(\Tilde{R})\ar@{=>}[dl]^{\gamma_{}}\ar[d]^{\iota} \\ \Spf(\Tilde{R})\ar[r]^{\iota} 
 &  X^{\Prism},}\]
we see that the unique isomorphism (induced by $\gamma$) $\gamma_0^{*}(\gamma_1^{*}(\rho^{*}\mathcal{E}))\cong \rho^{*}\mathcal{E}$ could be calculated via the $\Tilde{R}$-linear extension of $(\Id+q\beta \partial_{\mathcal{E}}) \circ (\Id+T_1\beta \nabla_{\mathcal{E},1})$. On the other hand, as $\gamma_1\circ \gamma_0=\gamma_0 \circ \gamma_1^{p^{\alpha+1}+1}$, a similar discussion implies that it can also be calculated by $\Tilde{R}$-linear extension of $(\Id+T_1\beta \nabla_{\mathcal{E},1})^{p^{\alpha+1}+1}\circ (\Id+q\beta \partial_{\mathcal{E}})$, from which we deduce that 
$$(\Id+q\beta \partial_{\mathcal{E}}) \circ (\Id+T_1\beta \nabla_{\mathcal{E},1})=(\Id+T_1\beta \nabla_{\mathcal{E},1})^{p^{\alpha+1}+1}\circ (\Id+q\beta \partial_{\mathcal{E}}).$$

For the moreover part, denote $T=S_0\otimes_{\Tilde{R}} S_1$. Set $\epsilon_0 \partial(\epsilon_1):=q\epsilon_1 \partial(\beta)$ and 
let $\psi_0^{\prime}: S_1\to T$ be the unique extension of $\psi_0: \Tilde{R}\to T$, i.e. $\psi_0^{\prime}(x_1+\epsilon_1 y_1)=(\Id+\epsilon_0 \partial)\circ (x_1+\epsilon_1 y_1)=x_1+\epsilon_1 y_1+\epsilon_0\partial(x_1)+\epsilon_0\partial(\epsilon_1 y_1)$. Similarly, we extend $\psi_1: \Tilde{R}\to T$ to $\psi_1^{\prime}: S_0\to T$ by regulating that $\nabla(\epsilon_0)=0$.

We consider the following commutative diagram 
   \[ \xymatrixcolsep{8pc}\xymatrix{  \Spf(T_0)=\Spf(T/(\epsilon_0\epsilon_1-qT\epsilon_0, \epsilon_0\epsilon_1-q\beta\epsilon_0))\ar@/^2pc/[rr]^{\Id}\ar@{^{(}->}[r]^{j}\ar[dr]^{f_1}
 &  \Spf(T) \ar[d]^{\psi_1\circ \psi_0^{\prime}}\ar[r]^{\iota} & \Spf(\Tilde{R})\ar@{=>}[dl]^{\gamma} \ar[d]^{\rho_X}\\ & \Spf(\Tilde{R}) \ar[r]^{\rho_X} &X^{\Prism},}\]
 where $f_1$ is the composition of $\psi_1\circ \psi_0^{\prime}$ and the closed embedding $j$, while the existence of the homotopy $\gamma$ follows from the proof of \cref{propt.key automorphism of functors AR}. Consequently, the unique isomorphism (induced by $\gamma$) $f_1^*(\rho^{*}\mathcal{E}))\cong \rho^{*}\mathcal{E}$ could be calculated via the $T_0$-linear extension of $(\Id+\epsilon_0 \partial_{\mathcal{E}}) \circ (\Id+\epsilon_1 \nabla_{\mathcal{E},1})$.

 A key observation is that as ring homomorphisms from $\Tilde{R}$ to $T_0$,
 $$(\Id+\epsilon_0 \partial_{}) \circ (\Id+\epsilon_1 \nabla_{1})=(\Id+\epsilon_1 \nabla_{1})^{p^{\alpha+1}+1}\circ (\Id+\epsilon_0 \partial_{}),$$
 which is implied by the proof of \cref{lem. AR relation between 0 and i}.

 Then arguing as in the proof for the second statement, we conclude that
 $$(\Id+\epsilon_0 \partial_{\mathcal{E}}) \circ (\Id+\epsilon_1 \nabla_{\mathcal{E},1})=(\Id+\epsilon_1 \nabla_{\mathcal{E},1})^{p^{\alpha+1}+1}\circ (\Id+\epsilon_0 \partial_{\mathcal{E}}).$$
 As the difference between the left-hand side and the right-hand side equals $$\epsilon_0\epsilon_1((\nabla_{\mathcal{E},1}+\beta q \mathscr{D}(\nabla_{\mathcal{E},1})\nabla_{\mathcal{E},1})\circ \partial_{\mathcal{E}}-(s_0(\partial_{\mathcal{E}}-s_0^{-1}\mathscr{D}(\nabla_{\mathcal{E},1})+s_1)\circ \nabla_{\mathcal{E},1})),$$ 
 hence it must vanish, which finishes the proof.
\end{proof}

In the special case that $\mathcal{E}\in \mathcal{D}(X_{}^{\Prism})^{\heartsuit}$, $\rho_X^{*}(\mathscr{E})$ is a $\Tilde{R}$-module concentrated on degree $0$ and $\partial_{\mathcal{E}}$ (resp. $\nabla_{\mathcal{E},i}$) satisfies the twisted Leibnitz rule stated in \cref{lem.leibniz} (resp. \cref{lem. relative leibniz}).

Given \cref{prop. AR two derivation vanish}, we introduce the following non-commutative ring, which will be used later for the purpose of classifying $\mathcal{D}(X^{\Prism})$.
\begin{definition}\label{absolute def.skew polynomial}
    Let $\gamma_{\Tilde{R},i}: \Tilde{R}\to \Tilde{R}$ ($0\leq i\leq m$) be the ring automorphism $\gamma_i$ defined above \cref{lem. AR relation between 0 and i}, then $\partial_{\Tilde{R}}: \Tilde{R}\to \Tilde{R}$ (resp. $\nabla_{\Tilde{R},i}: \Tilde{R}\to \Tilde{R}$ for $i\geq 1$), denoted as $\partial$ (resp. $\nabla_i$) above \cref{lem. AR relation between 0 and i}, is a $\gamma_{\Tilde{R},0}$-derivation (resp. $\gamma_{\Tilde{R},i}$-derivation) of $\Tilde{R}$. We define the \textit{Ore extension} $\Tilde{R}[\partial, \underline{\nabla}; \gamma_{\Tilde{R}}]$ to be the noncommutative ring obtained by giving the ring of polynomials $\Tilde{R}[\partial, \nabla_1,\cdots,\nabla_m]$ a new multiplication law, subject to the identity
    \begin{equation*}
        \begin{split}
           \nabla_i\cdot \nabla_j&=\nabla_j\cdot \nabla_i,  ~~~\forall ~ 1\leq i,j\leq m,
           \\ \nabla_i \cdot r&=\gamma_{\Tilde{R},i}(r)\cdot\nabla_i+\nabla_{\Tilde{R},i}(r), ~~~\forall~ r\in \Tilde{R}, 1\leq i\leq m,
           \\ \partial\cdot r&=\gamma_{\Tilde{R},0}(r)\cdot\partial+\partial_{\Tilde{R}}(r), ~~~\forall~ r\in \Tilde{R},
           \\ \partial\cdot \nabla_{i}&=s_0^{-1}(1+\beta q \mathscr{D}(\nabla_{i}))\nabla_{i}\cdot \partial+(s_0^{-1}\mathscr{D}(\nabla_{i})-s_1)\cdot \nabla_{i},~~~\forall ~ 1\leq i\leq m.
        \end{split}
    \end{equation*}
\end{definition}

We observe that $\Tilde{R}[\underline{\nabla}; \gamma_{\Tilde{R}}]$ (denoted as $\Tilde{R}[\underline{\nabla}; \gamma_{\Tilde{R}}, \nabla_{\Tilde{R}}]$ in \cref{rel def.skew polynomial}) is isomorphic to the quotient of $\Tilde{R}[\partial, \underline{\nabla}; \gamma_{\Tilde{R}}]$ by the left ideal generated by $\partial$. Combining with \cref{rem. rel cohomology of q higgs}, we obtain the following:
\begin{lemma}\label{lem. cohomology for modules over the noncommutative ring}
    There is a canonical resolution of $\Tilde{R}$ by finite free $\Tilde{R}[\partial, \underline{\nabla}; \gamma_{\Tilde{R}}]$-modules:
    \[\xymatrixcolsep{2pc}\xymatrix{\Tilde{R}[\partial, \underline{\nabla}; \gamma_{\Tilde{R}}]\ar[d]^{\times\partial^{[m]}}\ar[r]^-{}& \bigoplus_{1\leq k\leq m}\Tilde{R}[\partial, \underline{\nabla}; \gamma_{\Tilde{R}}] \ar[d]_{}^{\times\partial^{[m-1]}}\ar[r] &\cdots \ar[d]\ar[r] & \bigoplus_{1\leq k_1<\cdots< k_s\leq m}\Tilde{R}[\partial, \underline{\nabla}; \gamma_{\Tilde{R}}] \ar[d]^{\times\partial^{[m-s]}}\ar[r] &\cdots\ar[r]&\Tilde{R}[\partial, \underline{\nabla}; \gamma_{\Tilde{R}}]\ar[d]^{\times\partial}
\\ \Tilde{R}[\partial, \underline{\nabla}; \gamma_{\Tilde{R}}]\ar[r]^-{}& \bigoplus_{1\leq k\leq m}\Tilde{R}[\partial, \underline{\nabla}; \gamma_{\Tilde{R}}]\ar[r]& \cdots \ar[r] & \bigoplus_{1\leq k_1<\cdots< k_s\leq m}\Tilde{R}[\partial, \underline{\nabla}; \gamma_{\Tilde{R}}] \ar[r] &\cdots\ar[r] &\Tilde{R}[\partial, \underline{\nabla}; \gamma_{\Tilde{R}}],}\]
   where $\bigoplus_{1\leq k_1<\cdots< k_s\leq m}\Tilde{R}[\partial, \underline{\nabla}; \gamma_{\Tilde{R}}]$ in the first row (resp. the second row) lives in (cohomological) bi-degree $(s-m,-1)$ (resp. $(s-m,0)$) and the horizontal differential in the second row (resp. the first row) from $\Tilde{R}[\partial, \underline{\nabla}; \gamma_{\Tilde{R}}]$ in spot $k_1<\cdots<k_s$ to $\Tilde{R}[\partial, \underline{\nabla}; \gamma_{\Tilde{R}}]$ in spot $p_1<\cdots<p_{s+1}$ is nonzero only if $\{k_1,\cdots,k_s\}\subseteq \{p_1,\cdots,p_{s+1}\}$, in which case it sends $f\in \Tilde{R}[\partial, \underline{\nabla}; \gamma_{\Tilde{R}}]$ to $(-1)^{u-1}f\cdot\nabla_{p_u}$ (resp. $(-1)^{u-1}f\cdot(1+\beta q \mathscr{D}(\nabla_{p_u}))\nabla_{p_u}$), where $u\in \{1,\cdots,s+1\}$ is the unique integer such that $p_u\notin \{k_1,\cdots,k_s\}$. The vertical map $\partial^{[m-s]}$ is defined in the proof. 
\end{lemma}
\begin{proof}
   Given \cref{rem. rel cohomology of q higgs}, it suffices to construct $\partial^{[t]}$ to make the diagram commute and moreover, each vertical map gives a resolution for $\Tilde{R}[\underline{\nabla}; \gamma_{\Tilde{R}}]$. Let $P_{n}^i$ be the elementary symmetric polynomials in $n$ variables of degree $i$, i.e. $P_{n}^i((x_j)_{1\leq j\leq n})=P_{n}^i(x_1,\cdots,x_n)=\sum_{1\leq j_1<\cdots<j_i\leq n} \prod_{t=1}^{i} x_{j_t}$. Then we define $\partial^{[t]}$ by sending $f\in \Tilde{R}[\partial, \underline{\nabla}; \gamma_{\Tilde{R}}]$ in spot $k_1<\cdots<k_{m-t}$ to the following element in $\Tilde{R}[\partial, \underline{\nabla}; \gamma_{\Tilde{R}}]$ (in spot $k_1<\cdots<k_{m-t}$ as well):
   \begin{equation*}
       f\cdot(s_0^t\partial+(\sum_{i=1}^ts_0^i) s_1-\sum_{i=1}^t \beta^{i-1}q^{i-1}P^i_t(\mathscr{D}(\nabla_{p_{1}}),\cdots,\mathscr{D}(\nabla_{p_{t}}))),
   \end{equation*}
   where $s_0, s_1, \mathscr{D}$ are defined in \cref{lem. AR relation between 0 and i} and $\{p_{1},\cdots p_{t}\}=\{1,\cdots,m\}\backslash \{k_1,\cdots,k_{m-t}\}$. In particular, if $t=0$, then $\partial^{[0]}=\partial$. Clearly each vertical map $\partial^{[t]}$ defined in this way gives a resolution of $\Tilde{R}[\underline{\nabla}; \gamma_{\Tilde{R}}]$ as $s_0$ is a unit in $A$. To verify that it really defines a morphism of complexes, it amounts to checking that for a fixed spot $k_1<\cdots<k_{m-t}$ and any $p_j \in \{p_1,\cdots p_t\}=\{1,\cdots,m\}\backslash \{k_1,\cdots,k_{m-t}\}$, 
   \begin{equation*}
       (s_0^t\partial+(\sum_{i=1}^t (s_0^i s_1-(\beta q)^{i-1}P^i_t(\mathscr{D}(\nabla_{p_{u}}))))\cdot \nabla_{p_j}=((1-\beta q \mathscr{D}(\nabla_{p_j}))\nabla_{p_j})\cdot (s_0^{t-1}\partial+(\sum_{i=1}^{t-1} (s_0^i s_1-(\beta q)^{i-1}P^i_{t-1}(\mathscr{D}(\nabla_{p_{u}}))))),
   \end{equation*}
   here for simplicity we write $P^i_t(\mathscr{D}(\nabla_{p_{u}}))$ for $P^i_t(\mathscr{D}(\nabla_{p_{1}}),\cdots, \mathscr{D}(\nabla_{p_{t}}))$ and $P^i_{t-1}(\mathscr{D}(\nabla_{p_{u}}))$ for\\ $P^i_t(\mathscr{D}(\nabla_{p_{1}}),\cdots,\mathscr{D}(\nabla_{p_{j-1}}),\mathscr{D}(\nabla_{p_{j+1}}),\cdots \mathscr{D}(\nabla_{p_{t}}))$. But this follows from the observation that
   \begin{equation*}
       \begin{split}
           &((1+\beta q \mathscr{D}(\nabla_{p_j}))\nabla_{p_j})\cdot (s_0^{t-1}\partial+(\sum_{i=1}^{t-1} (s_0^i s_1-(\beta q)^{i-1}P^i_{t-1}(\mathscr{D}(\nabla_{p_{u}})))))
           \\=&(s_0^{t-1}(s_0\partial-\mathscr{D}(\nabla_{p_{j}})+s_0s_1)+(\sum_{i=1}^{t-1}s_0^i) s_1+\beta q\mathscr{D}(\nabla_{p_{j}})(\sum_{i=1}^{t-1}s_0^i) s_1+\mathscr{D}(\nabla_{p_{j}})-\sum_{i=1}^t \beta^{i-1}q^{i-1}P^i_t(\mathscr{D}(\nabla_{p_{u}})))\cdot \nabla_{p_j}
            \\=&(s_0^t\partial+(\sum_{i=1}^t (s_0^i s_1-(\beta q)^{i-1}P^i_t(\mathscr{D}(\nabla_{p_{u}}))))\cdot \nabla_{p_j}+(1-s_0^{t-1}+\beta q \sum_{i=1}^{t-1}s_0^is_1)\mathscr{D}(\nabla_{p_{j}})\cdot \nabla_{p_j}
             \\=&(s_0^t\partial+(\sum_{i=1}^t (s_0^i s_1-(\beta q)^{i-1}P^i_t(\mathscr{D}(\nabla_{p_{u}}))))\cdot \nabla_{p_j}.
       \end{split}
   \end{equation*}
   Here the first equality follows from \cref{lem. AR relation between 0 and i} and the commutativity of $\mathscr{D}(\nabla_{a})$ and $\mathscr{D}(\nabla_{b})$ for $a\neq b$ and the last identify follows from the vanishing of $s_0(1-s_1\beta q)-1$:
   \begin{equation*}
       \begin{split}
           s_0(1-s_1\beta q)=\frac{\gamma_0(\beta)}{\beta [k]_{q^{p^{\alpha+1}}}} (1-\beta q\cdot q^{-1}\frac{(1+q\partial(\beta))-[k]_{q^{p^{\alpha+1}}}}{\gamma_0(\beta)})=\frac{\frac{\gamma_0(\beta)}{\beta}-(1+q\partial(\beta))+[k]_{q^{p^{\alpha+1}}}}{[k]_{q^{p^{\alpha+1}}}}=1,
       \end{split}
   \end{equation*}
   where the last identity holds as $\gamma_0(\beta)=\beta+q\beta \partial(\beta)$.
\end{proof}
As a quick corollary, we obtain the following result. 
\begin{corollary}\label{absolute rem. rel cohomology of q higgs}
    Let $N$ be an object in $\mathcal{D}(\Tilde{R}[\partial, \underline{\nabla}; \gamma_{\Tilde{R}}])$, the derived category of (left) $\Tilde{R}[\underline{\nabla}; \gamma_{\Tilde{R}}, \nabla_{\Tilde{R}}]$-modules. If we further assume that $N$ is derived $(p,d)$-complete, then 
    \begin{equation*}
        \RHom_{\mathcal{D}(\Tilde{R}[\partial, \underline{\nabla}; \gamma_{\Tilde{R}}])}(\Tilde{R}, N)\simeq \fib(\mathrm{DR}(N,\nabla)\stackrel{\Tilde{\partial}^{[\bullet]}}{\longrightarrow} \mathrm{DR}(N,\nabla)),
    \end{equation*}
    where $\mathrm{DR}(N,\nabla)$ is the $q$-de Rham complex of $N$ defined in \cref{rem. rel cohomology of q higgs} and $\Tilde{\partial}^{[t]}$ acts on $N$ sitting in spot $l_1<\cdots<l_t$ via $$(\prod_{i=1}^t (1+\beta q \mathscr{D}(\nabla_{l_{i}}))^{-1})(s_0^t\partial+(\sum_{i=1}^ts_0^i) s_1-\sum_{i=1}^t \beta^{i-1}q^{i-1}P^i_t(\mathscr{D}(\nabla_{l_{1}}),\cdots,\mathscr{D}(\nabla_{l_{t}}))).$$
\end{corollary}
\begin{proof}
For any $N\in \mathcal{D}(\Tilde{R}[\partial, \underline{\nabla}; \gamma_{\Tilde{R}}])$, \cref{lem. cohomology for modules over the noncommutative ring} implies that \begin{equation*}
        \RHom_{\mathcal{D}(\Tilde{R}[\partial, \underline{\nabla}; \gamma_{\Tilde{R}}])}(\Tilde{R}, N)\simeq \fib(\mathrm{DR}(N,\nabla)\stackrel{\partial^{[\bullet]}}{\longrightarrow} \mathrm{DR}(N,\nabla^{\prime}))
    \end{equation*}
    with $\mathrm{DR}(N,\nabla^{\prime}):=[N\xrightarrow{\nabla_N^{\prime}:=\sum (1+\beta q \mathscr{D}(\nabla_{i}))\nabla_i\otimes dT_i} N\otimes_{\Tilde{R}} \Omega_{\Tilde{R}/A}\xrightarrow{\nabla^{\prime}_{N}\wedge \nabla^{\prime}_{N}} \cdots \to N\otimes_{\Tilde{R}} \Omega_{\Tilde{R}/A}^{m}\to 
        0]$.

    If $N$ is further $(p,d)$-complete, then $1+\beta q \mathscr{D}(\nabla_{i})$ is invertible on $N$ (as $\beta$ is topologically nilpotent with respect to the $(p,d)$-topology), hence we could replace $\mathrm{DR}(N,\nabla^{\prime})$ (resp. $\partial^{[\bullet]}$) with $\mathrm{DR} (N,\nabla^{})$ (resp. $\Tilde{\partial}^{[\bullet]}$) to make the diagram commute.
\end{proof}

\begin{remark}\label{rem. AR compare notation with MW}
 Heuristically one might want to view an object in $\mathcal{D}(\Tilde{R}[\partial, \underline{\nabla}; \gamma_{\Tilde{R}}])$ as an ``enhanced $q$-Higgs module", where such a notation is motivated by \cite[Definition 4.1]{min2022p}\footnote{Notice that $d^{\prime}(q)$ in our setting plays the role of $E^{\prime}(\pi)$ in \cite[Definition 4.1]{min2022p}} and the following reinterpretation of the last condition in \cref{absolute def.skew polynomial} when $d=0$, which is a twisted version of the requirement that $\phi\circ \theta-\theta\circ \phi=E^{\prime}(\pi)\theta$ given in \cite[Definition 4.1]{min2022p}.
\end{remark}
\begin{lemma}\label{lem. reinterpret the last condition}
    If we work with $\Tilde{R}/d[\partial, \underline{\nabla}; \gamma_{\Tilde{R}}]$ instead, then the last condition in \cref{absolute def.skew polynomial} is equivalent to that 
    \begin{equation*}
        (1+q\beta \mathscr{D}(\nabla_{i}))\nabla_{i}\cdot \partial_{}=(\frac{1}{p^{\alpha+1}+1}\partial_{}-\mathscr{D}(\nabla_{i})-\frac{e}{p^{\alpha+1}+1})\cdot \nabla_{i}
    \end{equation*}
    for which $e=d^{\prime}(q)\in \mathcal{O}_K$ and $\mathscr{D}(\nabla_{i})=\frac{1}{q(p^{\alpha+1}+1)}\sum_{j=2}^{p^{\alpha+1}+1} \binom{p^{\alpha+1}+1}{j} T_i^{j-1}\beta^{j-2}\nabla_{i}^{j-1}$.
\end{lemma}
\begin{proof}
First we notice that after modulo $d$, $q^{p^{\alpha+1}}$ is identified with $1$, hence the inductive formula for $a_{k,j}$ in \cref{lem. AR relation between 0 and i} turns into that $a_{k+1,j}=a_{k,j}+a_{k,j-1}$ with $a_{k,1}=k$, hence $a_{k,j}=\binom{k}{j}$. Applying it to $k=p^{\alpha+1}+1$, we get the desired formula for $\mathscr{D}(\nabla_{i})$. Then by \cref{lem. calculate e}, $e q\beta=p^{\alpha+1}$ (be aware that $\beta$ in \cref{lem. calculate e} is $q\beta$ here) in $\mathcal{O}_K$, hence
\begin{equation*}
    q^{-1}\frac{(1+q\partial(\beta))-[k]_{q^{p^{\alpha+1}}}}{\beta})=\frac{1-(p^{\alpha+1}+1)}{q\beta}=-e.
\end{equation*}
The desired formula then follows by \cref{lem. AR relation between 0 and i}.
\end{proof}
\begin{remark}\label{rem. AR relation modulo q-1 and d}
    The above calculation implies that if we further modulo $q-1$, i.e. working with $\Tilde{R}/(p,d)[\partial, \underline{\nabla}; \gamma_{\Tilde{R}}]$, then the above relation reduces to that 
    \[\nabla\cdot \partial=\partial\cdot \nabla\]
    as $\beta=0$ and $\binom{p^{\alpha+1}+1}{2}=0$ in $A/(d,q-1)$, which implies the vanishing of $\mathscr{D}(\nabla)$.
\end{remark}

Summarizing the discussion so far, we get the following analog of \cref{prop. relative induce functor}. 
\begin{proposition}\label{prop. AR induce functor}
    For $n\in \mathbb{N}\cup \{\infty\}$, the pullback along $\rho_X: \Spf(\Tilde{R})\to X^{\Prism}$ induces a functor 
  \begin{equation*}
  \begin{split}
      \mathcal{D}(X^{\Prism}) &\to \mathcal{D}(\Tilde{R}/d^n[\partial, \underline{\nabla}; \gamma_{\Tilde{R}}])\\
      \mathcal{E} &\mapsto (\rho_X^{*}\mathcal{E}, \partial_{\mathcal{E}}, \nabla_{\mathcal{E}}),
  \end{split}
  \end{equation*}
  which will be denoted as $\beta_n^+$ later.
\end{proposition}

By combining \cref{prop. ht case fiber seq} and \cref{prop. relative ht case fiber seq}, we get the following result, which is a combination of \cref{propt. relative sen calculate cohomology} and \cref{propt. sen calculate cohomology}.
\begin{proposition}\label{prop. AR fully faithful}
     Assume that $p>2$ or $\alpha>0$. Let $n\in \mathbb{N}\cup \{\infty\}$. For any $\mathcal{E}\in \mathcal{D}(X_n^{\Prism})$, the pullback functor $\beta_n^+: \mathcal{D}(X^{\Prism}) \to \mathcal{D}(\Tilde{R}/d^n[\partial, \underline{\nabla}; \gamma_{\Tilde{R}}])$ constructed in \cref{prop. AR induce functor} is fully faithful. Consequently,
     \begin{equation*}
        \mathrm{R} \Gamma(X_{n}^{\Prism}, \mathcal{E})\xrightarrow{\simeq} \fib(\mathrm{DR}(\rho_X^*\mathcal{E},\nabla_{\mathcal{E}})\xrightarrow{\Tilde{\partial}_{\mathcal{E}}^{[\bullet]}} \mathrm{DR}(\rho_X^*\mathcal{E},\nabla_{\mathcal{E}})),
    \end{equation*}
    where $\Tilde{\partial}_{\mathcal{E}}^{[\bullet]}$ is constructed in \cref{absolute rem. rel cohomology of q higgs}.
\end{proposition}
\begin{proof}
    Arguing as that in \cref{propt. relative sen calculate cohomology}, by standard d\'evissage, it suffices to prove the statement for $n=1$. By restricting the diagram shown up at the beginning of this section to the Hodge-Tate locus, we get the following commutative diagram
    \[ \xymatrixcolsep{5pc}\xymatrix{  \Spf(R)\ar@/^2pc/[rr]^{\rho_X}\ar[r]^{\rho}
 &  X_{/A}^{\HT} \ar[d]^{f^{\prime}}\ar[r]^{\rho_A} & X^{\HT} \ar[d]^{f}\\ & \Spf(\mathcal{O}_K) \ar[r]^{\rho_A} &\Spf(\mathcal{O}_K)^{\HT},}\]
 where $f$ is the structure morphism. Moreover, $f$ is affine by \cite[Lemma 2.10]{anschutz2023hodge}, hence $f_*=Rf_*$. Consequently, for $\mathcal{E}\in \mathcal{D}(X^{\HT})$,
 \begin{equation}\label{equa. abc cohomology}
     \mathrm{R} \Gamma(X^{\HT}, \mathcal{E})= \mathrm{R} \Gamma(\Spf(\mathcal{O}_K)^{\HT}, f_*\mathcal{E})\cong\fib(\rho_A^*f_*\mathcal{E}\xrightarrow{\partial_{f_{_*}\mathcal{E}}} \rho_A^*f_*\mathcal{E}),
 \end{equation}
 where the last equality is due to \cref{prop. ht case fiber seq}. On the other hand, we notice that $$\rho_A^*f_*\mathcal{E}\cong f^{\prime}_*\rho_A^*\mathcal{E}=\mathrm{R} \Gamma(X_{/A}^{\HT}, \rho_A^*\mathcal{E})\cong \RHom_{\mathcal{D}(\Tilde{R}[\underline{\nabla}; \gamma_{\Tilde{R}}])}(\Tilde{R}, \rho_X^*\mathcal{E})\cong \mathrm{DR}(\rho_X^*\mathcal{E},\nabla_{\mathcal{E}}).$$
Here the second to last isomorphism follows from \cref{thmt. relative main classification} and the last isomorphism is due to \cref{rem. rel cohomology of q higgs}. Moreover, by the proof of \cref{absolute rem. rel cohomology of q higgs} and unwinding the construction of $\partial_{f_{_*}\mathcal{E}}$ and $\partial_{\mathcal{E}}$,  we see that under the above isomorphism $\partial_{f_{_*}\mathcal{E}}$ could be identified with $\Tilde{\partial}^{[\bullet]}_{\mathcal{E}}$ on $\mathrm{DR}(\rho_X^*\mathcal{E},\nabla_{\mathcal{E}})$. Combining with \cref{equa. abc cohomology}, we conclude that

\begin{equation*}
        \mathrm{R} \Gamma(X_{}^{\HT}, \mathcal{E})\xrightarrow{\simeq} \fib(\mathrm{DR}(\rho_X^*\mathcal{E},\nabla_{\mathcal{E}})\xrightarrow{\Tilde{\partial}_{\mathcal{E}}^{[\bullet]}} \mathrm{DR}(\rho_X^*\mathcal{E},\nabla_{\mathcal{E}})),
\end{equation*}
hence finish the proof by \cref{absolute rem. rel cohomology of q higgs}.
\end{proof}

Next we proceed to show that $\mathcal{D}(X_{}^{\Prism})$ is generated under shifts and colimits by the $\mathcal{I}^k$, $k\in \mathbb{Z}$, generalizing \cite[Corollary 3.5.16]{bhatt2022absolute}. A key ingredient is the geometry of $X^{\HT}$ studied in \cite{bhatt2022prismatization} and \cite{anschutz2023hodge}, which is summarized as \cite[Corollary 3.11]{anschutz2023hodge} in our setting.
\begin{lemma}[{\cite[Corollary 3.11]{anschutz2023hodge}}]\label{lem. AR HT stack geometric structure}
    Let $G_X$ be the group sheaf of automorphisms $\mathrm{Aut}(\rho_X)$ of $\rho_X: X=\Spf(R)\to X^{\HT}$, the restriction of $\rho_X$ to the Hodge-Tate locus. Then 
    \[G_X \cong (\mathbb{G}_a^{\sharp})^{m+1}=\prod_{i=0}^m \mathbb{G}_a^{\sharp}.\]
    Moreover, the group structure on $G_X$ transfers through this to the map
    \begin{equation*}
        \prod_{i=0}^m \mathbb{G}_a^{\sharp} \times \prod_{i=0}^m \mathbb{G}_a^{\sharp},\quad ((c_i)_{i=0,\cdots,m},(a_i)_{i=0,\cdots,m})\mapsto (a_k+c_k(1+ea_0))_{k=0,\cdots,m}
    \end{equation*}
    with $e$ defined to be the image of $d^{\prime}(q)$ in $\mathcal{O}_K$ as usual.
\end{lemma}
\begin{proof}
    This is \cite[Corollary 3.11]{anschutz2023hodge}. Note that by unwinding all of the constructions in \textit{loc.cit.}, $E^{\prime}(\pi)$ there can be replaced with $d^{\prime}(q)$  if one works with the $q$-prism instead of the Breuil-Kisin prism.
\end{proof}

As $\psi: \Tilde{R}\to S$ constructed in \cref{lem. AR twist S in the relative case} reduces to the natural embedding after modulo $d$, we see $\gamma_{b, c_{\psi}}$ constructed in \cref{propt.key automorphism of functors AR} descends to an automorphism of
    \[X_{}^{\HT}\times_{\Spf(R)} \Spf(R\oplus_{i=0}^m R\epsilon_i),\]
    which implies that for $\mathcal{E}\in \mathcal{D}(X_{}^{\HT})$, $\partial_{\mathcal{E}}, \nabla_{\mathcal{E}}$ also descends to a functor from $\mathcal{E}$ to itself.

Under such an identification, $\gamma_{b, c_{\psi}}$ corresponds to an element in $(\mathbb{G}_a^{\sharp})^m(S/d)$, which is precisely $(\epsilon,\epsilon,\cdots,\epsilon)_{0\leq i\leq m}\in (\mathbb{G}_a^{\sharp})^{m+1}$ after identifying the latter with $G_X$ via \cref{lem. AR HT stack geometric structure}. Indeed, this follows from the discussion above \cref{lem. calculate e} and \cref{prop. relative calculate partial on HT}. Consequently, we can describe $\partial_{\mathcal{E}}$ and $\nabla_{\mathcal{E}}$ for $\mathcal{E}=\rho_{X,*}(\mathcal{O}_X)$ very explicitly, which could be viewed as a combination of \cref{lem. calculate partial on HT} and \cref{prop. relative calculate partial on HT}.
\begin{lemma}\label{prop. AR calculate partial on HT}
    Let $\mathcal{E}=\rho_{X,*} \mathcal{O}_{X}$, then 
    \begin{itemize}
        \item $\rho_X^{*}\mathcal{E}\cong  R\{a_0, a_1,\cdots, a_m\}^{\wedge}_{p}$.
        \item Suppose that $p>2$ or $\alpha>0$,then the sequence $0\to 
        (R\{a_j\}_{j\neq i})^{\wedge}_{p} \to \rho_X^{*}\mathcal{E}\xrightarrow{\nabla_i} \rho_X^{*}\mathcal{E} \to 0$ is exact for all $0\leq i\leq m$, here $\nabla_0=\partial$ by abuse of notation.
    \end{itemize}
\end{lemma}
\begin{proof}
    \cref{lem. AR HT stack geometric structure} and the projection formula tells us that $\rho_X^{*}\mathcal{E}\cong R\{a_0, a_1,\cdots, a_m\}^{\wedge}_{p}$. For the second statement, the case that $i=0$ was already treated in \cref{lem. calculate partial on HT} (note that the group formula used there are the same as that stated in \cref{lem. AR HT stack geometric structure}). For any $i\geq 1$,
    we write $B$ for $(R\{a_j\}_{j\neq i})^{\wedge}_{p}$ and omit $i$ from the subscript of $a_i$ and $T_i$ for the ease of notation. By \cref{lem. AR HT stack geometric structure}, the formal group law on the $i$-th component of $(\mathbb{G}_a^{\sharp})^m$ is given by $$\Delta: \widehat{\mathcal{O}}_{\mathbb{G}_a^{\sharp}} \to \widehat{\mathcal{O}}_{\mathbb{G}_a^{\sharp}} \hat{\otimes} \widehat{\mathcal{O}}_{\mathbb{G}_a^{\sharp}}, ~~~~a\mapsto a+b(1+ea_0)$$
and the isomorphism $\gamma_{b,c_{\psi_i}}$ (hence also the $q$-derivation $\nabla_i$) is constructed via $\epsilon \in \mathbb{G}_a^{\sharp}(S/d)$, hence for $f(a)=\sum_{i=0}^{\infty}c_i\frac{a^n}{n!}\in \rho^*\mathcal{E}=B\{a\}_p^{\wedge}$ with $c_i\in B$, $\gamma_{b,c_{\psi_i}}(f)=f(a+\epsilon (1+ea_0))=f+\epsilon\nabla_i(f)$. If we denote $1+ea_0\in B$ as $u$ for simplicity, then it follows that 
\begin{equation*}
    \begin{split}
       \epsilon\nabla_i(f)&=f(a+\epsilon u)-f(a)=\sum_{n=0}^{\infty}c_n\frac{(a+\epsilon u)^n-a^n}{n!}
       \\&=\sum_{n=0}^{\infty}\frac{c_n}{n!}(\sum_{i=0}^{n-1}\binom{n}{i}a^i\epsilon^{n-i} u^{n-i})
       =\epsilon \sum_{n=0}^{\infty}\frac{c_n}{n!}(\sum_{i=0}^{n-1}\binom{n}{i}a^i u^{n-i}(\beta T)^{n-i-1}),
    \end{split}
\end{equation*}
where the last equality follows from $\epsilon^j=(\beta T)^{j-1}\epsilon$. Consequently
\begin{equation*}
    \begin{split}
        \nabla_i(f)&=\sum_{i=0}^{\infty} a^i(\sum_{n=i+1}^{\infty} \frac{c_n}{n!}\binom{n}{i}u^{n-i}(\beta T)^{n-1-i})
        \\&=\sum_{i=0}^{\infty} \frac{a^i}{i!}(\sum_{n=i+1}^{\infty} \frac{c_n}{(n-i)!}u^{n-i}(\beta T)^{n-1-i}).
    \end{split}
\end{equation*}
We then consider the infinite dimensional matrix $M$ with $M_{i,j}=\frac{u^{j+1-i}(\beta T)^{j-i}}{(j+1-i)!}$ for $i\leq j$ and $0$ otherwise, then $M$ is an upper triangular matrix with all the diagonal elements being $u\in B^{\times}$, hence $M$ is an invertible matrix. Let $g=\sum_{n=0}^{\infty} b_n\frac{a^n}{n!}$, then the previous calculation tells us that $\nabla_i(f)=g$ if and only if 
\begin{equation}\label{equa. AR solve c for b}
    M \Vec{c}=\Vec{b}.
\end{equation}
for $\Vec{b}=(b_0,b_1,\cdots)^{\mathrm{T}}$ and $\Vec{c}=(c_1,c_2,\cdots)^{\mathrm{T}}$.
As $M$ is invertible, given any $\Vec{b}$ satisfies that $b_i\to 0$ as $i\to \infty$, there exists a unique $\Vec{c}$ solving \cref{equa. AR solve c for b}. Moreover, $c_i\to 0$ as $i\to \infty$. This finishes the proof of the second result stated in the proposition.
\end{proof}

Thanks to \cref{prop. AR calculate partial on HT}, we get the following Poincare lemma for $\mathcal{O}_X$.
\begin{corollary}\label{cor. AR poincare lemma}
Let $\mathcal{E}_0=\rho_{X,*} \mathcal{O}_{X}\in \mathcal{D}(x^{\HT})$. Suppose that $p>2$ or $\alpha>0$. Then there is a canonical resolution from $R$ to the total complex of 
    \[\xymatrixcolsep{2pc}\xymatrix{\rho_X^*\mathcal{E}_0\ar[d]^{\partial_{\mathcal{E}_0}}\ar[r]^-{\nabla_{\mathcal{E}_0}}& \rho_X^*\mathcal{E}_0\otimes_{R} \Omega^1_{R/\Bar{A}} \ar[d]_{}^{\Tilde{\partial}^{[1]}_{\mathcal{E}_0}}\ar[r] &\cdots \ar[d]\ar[r] & \rho_X^*\mathcal{E}_0\otimes_{R} \Omega^m_{R/\Bar{A}} \ar[d]^{\Tilde{\partial}^{[m]}_{\mathcal{E}_0}}\ar[r] &0
\\ \rho_X^*\mathcal{E}_0\ar[r]^-{\nabla_{\mathcal{E}_0}}& \rho_X^*\mathcal{E}_0\otimes_{R} \Omega^1_{R/\Bar{A}}\ar[r]& \cdots \ar[r] & \rho_X^*\mathcal{E}_0\otimes_{R} \Omega^m_{R/\Bar{A}} \ar[r] &0.}\]
\end{corollary}
\begin{proof}
    Given \cref{prop. AR calculate partial on HT}, the proof is similar to that for \cref{cor. relative poincare lemma}. For simplicity, we treat the case that $m=1$, which amounts to showing that the total complex of 
    \[\xymatrixcolsep{2pc}\xymatrix{\rho_X^*\mathcal{E}_0\ar[d]^{\partial_{\mathcal{E}_0}}\ar[r]^-{\nabla_{\mathcal{E}_0}}& \rho_X^*\mathcal{E}_0\ar[d]_{}^{\Tilde{\partial}^{[1]}_{\mathcal{E}_0}}
\\ \rho_X^*\mathcal{E}_0\ar[r]^-{\nabla_{\mathcal{E}_0}}& \rho_X^*\mathcal{E}_0}\]
is concentrated on degree $0$ and is given by $R$. The only nontrivial part is to check that given a pair $(x,y)\in \rho_X^*\mathcal{E}_0\oplus \rho_X^*\mathcal{E}_0$ such that $\Tilde{\partial}^{[1]}_{\mathcal{E}_0}(x)-\nabla(y)=0$, then there exists $t\in \rho_X^*\mathcal{E}_0$ satisfying that $\nabla(t)=x$ and that $\partial(t)=y$. For this purpose, first notice that $\nabla$ is surjective by \cref{prop. AR calculate partial on HT}, hence we could find $t_1\in \rho_X^*\mathcal{E}_0$ such that $\nabla(t_1)=x$. Then $\nabla(y-\partial(t_1))=\nabla(y)-\Tilde{\partial}^{[1]}_{\mathcal{E}_0}\circ \nabla(t_1)=0$, hence $y-\partial(t_1) \in \ker(\nabla)$. By the proof of \cref{prop. AR calculate partial on HT}, $\partial: \ker(\nabla)\to \ker(\nabla)$ is also surjective, hence there esists $f\in \ker(\nabla)$ such that $\partial(f)=y-\partial(t_1)$, then $t=t_1+f$ satisfies that $\nabla(t_1)=x$ and that $\partial(t)=y$, we are done.
\end{proof}

As a byproduct, we get the following refinement of \cref{prop. AR fully faithful} when $n=1$, which is also an analog of \cref{prop. relative ht case fiber seq}. Moreover, \cref{prop. AR ht case fiber seq} could lead to a direct proof of \cref{prop. AR fully faithful} by standard  d\'evissage.
\begin{proposition}\label{prop. AR ht case fiber seq}
    Suppose that $p>2$ or $\alpha>0$.. For any $\mathcal{E}\in \mathcal{D}(X_{}^{\HT})$, there is a canonical resolution
    from $\mathcal{E}$ to the total complex of 
    \[\xymatrixcolsep{2pc}\xymatrix{\rho_{X,*}\rho_X^*\mathcal{E}\ar[d]^{\partial_{\mathcal{E}}}\ar[r]^-{\nabla_{\mathcal{E}}}& \rho_{X,*}\rho_X^*\mathcal{E}\otimes_{R} \Omega^1_{R/\Bar{A}} \ar[d]_{}^{\Tilde{\partial}^{[1]}_{\mathcal{E}}}\ar[r] &\cdots \ar[d]\ar[r] & \rho_{X,*}\rho_X^*\mathcal{E}\otimes_{R} \Omega^m_{R/\Bar{A}} \ar[d]^{\Tilde{\partial}^{[m]}_{\mathcal{E}}}\ar[r] &0
\\ \rho_{X,*}\rho_X^*\mathcal{E}\ar[r]^-{\nabla_{\mathcal{E}}}& \rho_{X,*}\rho_X^*\mathcal{E}\otimes_{R} \Omega^1_{R/\Bar{A}}\ar[r]& \cdots \ar[r] & \rho_{X,*}\rho_X^*\mathcal{E}\otimes_{R} \Omega^m_{R/\Bar{A}} \ar[r] &0.}\]
Moreover, taking cohomology induces a canonical quasi-isomorphism from $\mathrm{R} \Gamma(X_{}^{\HT}, \mathcal{E})$ to the total complex of 
\[\xymatrixcolsep{2pc}\xymatrix{\rho_X^*\mathcal{E}\ar[d]^{\partial_{\mathcal{E}}}\ar[r]^-{\nabla_{\mathcal{E}}}& \rho_X^*\mathcal{E}\otimes_{R} \Omega^1_{R/\Bar{A}} \ar[d]_{}^{\Tilde{\partial}^{[1]}_{\mathcal{E}}}\ar[r] &\cdots \ar[d]\ar[r] & \rho_X^*\mathcal{E}\otimes_{R} \Omega^m_{R/\Bar{A}} \ar[d]^{\Tilde{\partial}^{[m]}_{\mathcal{E}}}\ar[r] &0
\\ \rho_X^*\mathcal{E}\ar[r]^-{\nabla_{\mathcal{E}}}& \rho_X^*\mathcal{E}\otimes_{R} \Omega^1_{R/\Bar{A}}\ar[r]& \cdots \ar[r] & \rho_X^*\mathcal{E}\otimes_{R} \Omega^m_{R/\Bar{A}} \ar[r] &0.}\]
\end{proposition}
\begin{example}[$q$-Higgs connections and derivations on the Breuil-Kisin twists when restricted to the Hodge-Tate locus]\label{example. AR relative action on generators}
     For $\mathcal{E}=\mathcal{O}_{X^{\HT}}\{k\}$, combining \cref{example.action on generators} with \cref{example. relative action on generators}, we see that $\nabla_{\mathcal{E}}=0$ and $\partial_{\mathcal{E}}$ is given by multiplication by $e\frac{(1+p^{\alpha+1})^k-1}{p^{\alpha+1}}$. Moreover, as in \cite[Corollary 3.5.14]{bhatt2022absolute}, \cref{prop. AR ht case fiber seq} implies that $\mathcal{E}\in \mathcal{D}(X_{}^{\HT})$ is isomorphic to $ \mathcal{O}_{X^{\HT}}\{k\}$ if and only if $\rho^*\mathcal{E}\cong \mathcal{O}_X$, $\nabla_{\mathcal{E}}=0$ and $\partial_{\mathcal{E}}$ is given by multiplication by $e\frac{(1+p^{\alpha+1})^k-1}{p^{\alpha+1}}$.
\end{example}
The resolution above essentially leads to the following desired result.
\begin{proposition}\label{prop. AR generate by Ik}
    Let $n\in \mathbb{N}\cup \{\infty\}$. The $\infty$-category $\mathcal{D}(X_{}^{\Prism})$ is generated under shifts and colimits by the $\mathcal{I}^k$, $k\in \mathbb{Z}$. 
\end{proposition}
\begin{proof}
    Arguing as \cite[Corollary 3.5.16]{bhatt2022absolute}, we are reduced to proving that $\infty$-category $\mathcal{D}(X_{}^{\HT})$ is generated under shifts and colimits by the $\mathcal{I}^k$, $k\in \mathbb{Z}$. We follow the proof of \cite[Proposition 3.5.15]{bhatt2022absolute}. Without loss of generality, we assume that $m=1$, which already reflects the non-commutativity of $G_X$ for which $X_{}^{\HT}=BG_X$. For general $m$, a similar argument works. By replacing $\mathcal{E}$ with $\rho_{X,*}\rho_X^*\mathcal{E}$ thanks to \cref{prop. AR ht case fiber seq}, it suffices to show that $\rho_{X,*}\mathcal{O}_X$ lies in the full subcategory $\mathcal{C}\subseteq \mathcal{D}(X_{}^{\HT})$ defined to be the full subcategory generated under shifts and colimits by the $\mathcal{I}^k$. By \cref{lem. AR HT stack geometric structure}, we can identify $\mathcal{D}(X_{}^{\HT})$ with the $\infty$-category of $\mathcal{O}_{G_X}$-comodule objects in $\widehat{\mathcal{D}}(R)$. Under this identification, $\rho_{X,*}\mathcal{O}_X$ corresponds to the $p$-complete regular representation of $\widehat{\mathcal{O}}_{\prod_{i=0}^1\mathbb{G}_a^{\sharp}}=R\{a_0, a_1\}^{\wedge}_{p}$ thanks to \cref{prop. AR calculate partial on HT}. For each $n\geq 0, 0\leq k\leq n$, Let $V_{\frac{n(n+1)}{2}+k}$ denote the $R$-submodule of $R\{a_0, a_1\}^{\wedge}_{p}$ generated by $\frac{a_0^i}{i!}\cdot \frac{a_1^j}{j!}$ for $i+j<n$ or $i+j=n, j\leq k$. Then the calculation in the proof of \cref{lem. calculate partial on HT} and \cref{prop. AR calculate partial on HT} implies that 
    \begin{equation*}
    \begin{split}
        \partial(\frac{a_0^{n-k}}{(n-k)!} \frac{a_1^k}{k!})=\frac{a_0^{n-k}}{(n-k)!}\frac{a_1^k}{k!}\cdot (e\frac{(1+p^{\alpha+1})^{n-k}-1}{p^{\alpha+1}}) \quad &\text{mod} \quad V_{\frac{n(n+1)}{2}+k-1},
        \\ \nabla(\frac{a_0^{n-k}}{(n-k)!} \frac{a_1^k}{k!})=0 \quad &\text{mod} \quad V_{\frac{n(n+1)}{2}+k-1}.
    \end{split}
    \end{equation*}
Invoking \cref{example. AR relative action on generators}, we obtain fiber sequences 
\[V_{\frac{n(n+1)}{2}+k-1}\to V_{\frac{n(n+1)}{2}+k}\to \mathcal{O}_{X^{\HT}}\{n-k\}\]
It follows by induction on $j$ that each $V_{\leq j}$ belongs to the category $\mathcal{C}$. Taking colimit on $j$ implies that $\rho_{X,*}\mathcal{O}_X$ belongs to $\mathcal{C}$.
\end{proof}
\begin{remark}\label{rem. Xht generates holds in general}
    Although we assume that $X$ is small affine over $\mathcal{O}_K=W(k)[\zeta_{p^{\alpha+1}}]$ in this section, \cref{prop. AR generate by Ik} holds for such an $X$ over a general $\mathcal{O}_K$. This follows as \cref{prop. AR calculate partial on HT} holds in this general setting by replacing $\nabla_i$ with the Higgs field $\theta_i$ studied in \cite{anschutz2023hodge} (the proof of \cref{prop. AR calculate partial on HT} works verbatimly by replacing $e=d^{\prime}(q)$ with $E^{\prime}(\pi)$). 
\end{remark}

Finally we can state and prove the main results in this section.
\begin{theorem}\label{thmt. AR main classification}
Assume that $p>2$ or $\alpha>0$. Let $n\in \mathbb{N}\cup \{\infty\}$. The functor 
    \begin{align*}
        &\beta_n^{+}: \mathcal{D}(X_{ n}^{\Prism}) \rightarrow \mathcal{D}(\Tilde{R}/d^n[\partial, \underline{\nabla}; \gamma_{\Tilde{R}}]),
       \qquad \mathcal{E} \mapsto (\rho_X^{*}\mathcal{E}, \partial_{\mathcal{E}}, \nabla_{\mathcal{E}})
    \end{align*}
    constructed in \cref{prop. AR induce functor} is fully faithful. 
     Moreover, its essential image consists of those objects $M\in \mathcal{D}(\Tilde{R}/d^n[\partial, \underline{\nabla}; \gamma_{\Tilde{R}}])$ 
     satisfying the following pair of conditions:
    \begin{itemize}
        \item $M$ is $(p,d)$-adically complete.
        \item The action of $\partial$  and $\nabla_i$ on the cohomology $\mathrm{H}^*(M\otimes_{\Tilde{R}/d^n}^{\mathbb{L}}\Tilde{R}/(d,q-1))$ is locally nilpotent for all $i$.
    \end{itemize}
\end{theorem}
\begin{proof}
The functor is well-defined thanks to \cref{prop. AR induce functor} and the full faithfulness follows from \cref{prop. AR fully faithful} directly. 

    To see that the image of $\beta_n^{+}$ satisfies the stated conditions, we just notice that by \cref{prop. AR generate by Ik}, the essential image of $\beta_n^+$ is generated by $\beta_n^{+}(\mathcal{I}^k)$ ($k\in \mathbb{Z}$) under shifts and colimits, and the action of $\partial$ and $\nabla_{i}$ on the underlying module $\rho_X^{*}$ satisfies the nilpotence condition due to \cref{example. AR relative action on generators}.

    Let $\mathcal{C}_n\subseteq \mathcal{D}(\Tilde{R}/d^n[\partial, \underline{\nabla}; \gamma_{\Tilde{R}}])$ be the full subcategory spanned by objects satisfying two conditions listed in \cref{thmt. relative main classification}. As the source $\mathcal{D}(X_{n}^{\Prism})$ is generated under shifts and colimits by the structure sheaf thanks to \cref{propt. relative generation}, to complete the proof it suffices to show that $\mathcal{C}_n$ is also generated under shifts and colimits by $\beta^+_n(\mathcal{I}^{k})$, $k\in \mathbb{Z}$. In other words, we need to show that for every nonzero object $M\in \mathcal{C}_n$, $M$ admits a nonzero morphism from $\beta^+_n(\mathcal{I}^{k})[m]$ for some $k, m \in \mathbb{Z}$. For this purpose, one can reduce to the Hodge-Tate case first and then copy the proof of \cref{thmt. relative main classification} by replacing \cref{rem. rel cohomology of q higgs} with \cref{absolute rem. rel cohomology of q higgs}.
\end{proof}

\subsection{Locally complete intersection case}
In this subsection we would like to classify quasi-coherent complexes on the prismatization of $Y$, for $Y$ a locally complete intersection over $\mathcal{O}_K$. We still work with $\mathcal{O}_K=W(k)[\zeta_{p^{\alpha+1}}]$ ($\alpha\geq 0$) and the $q$-prism $(A, I)=(W(k)[[q-1]], [p]_{q^{p^\alpha}})$. 

As in the last subsection, $X=\Spf(R)$ is a small affine over $\mathcal{O}_K=W(k)[\zeta_{p^{\alpha+1}}]$ ($\alpha\geq 0$) and we fix the following diagram
\[\xymatrixcolsep{5pc}\xymatrix{A\left\langle \underline{T}^{}\right\rangle\ar[r]^{\square} \ar[d]&   \Tilde{R}
\ar[d]^{}
\\ \Bar{A}\left\langle \underline{T}^{}\right\rangle\ar[r] \ar[r]^{\square} &R,}\]
where the horizontal maps are \'etale chart maps. 

As usual, we assume that $Y=\Spf(R/(\Bar{x}_1,\cdots,\Bar{x}_r)) \hookrightarrow X$ ($x_i\in \Tilde{R}$ and $\Bar{x}_i$ is its image in $\Bar{A}$) is a closed embedding such that the prismatic envelope with respect to the morphism of $\delta$-pairs $(\Tilde{R}, (d))\to (\Tilde{R}, (d,x_1,\cdots,x_m))$ exists and is exactly given by $\Tilde{R}\{\frac{x_1,\cdots,x_r}{d}\}^{\wedge}_{\delta}$. Main examples for such $Y$ are given in \cref{lem. ARL discrete case}.

To study $\mathcal{D}(Y_{}^{\Prism})$, we analyze the following diagram shown above \cref{prop. lci case base to R}:
\[ \xymatrix{ Y_{/\Tilde{R}}^{\Prism} \ar@/^2pc/[rr]^{\rho_Y} \ar[r]^{\rho} \ar[d]^{} & Y_{/A}^{\Prism} \ar[r]^{\rho_A} \ar[d]^{} &Y^{\Prism} \ar[d]^{}\\ \Spf(\Tilde{R})\ar[r]^{\rho}
 &  X_{/A}^{\Prism} \ar[d]^{}\ar[r]^{\rho_A} & X^{\Prism} \ar[d]\\ & \Spf(A) \ar[r]^{\rho_A} &\Spf(\Bar{A})^{\Prism},}\]
where all squares in the diagram are pullback squares and $\rho_Y=\rho_A\circ \rho$.

By \cref{prop. lci case base to R}, $Y_{/\Tilde{R}}^{\Prism}$ is represented by $\Prism_X(Y)=\Tilde{R}\{\frac{x_1,\cdots,x_r}{d}\}^{\wedge}_{\delta}$. We then proceed by extending $\psi$ in \cref{lem. AR twist S in the relative case} to $\Prism_X(Y)$ first.
\begin{lemma}\label{lemt.ARL lci case appl extend eta}
Let $S=\Tilde{R}\oplus_{i=0}^{m} \Tilde{R}\epsilon_i$ be as that in \cref{lem. AR twist S in the relative case}. Then the $W(k)$-linear homomorphism $\psi: \Tilde{R} \to  S$ constructed in \cref{lem. AR twist S in the relative case} uniquely extends to a $\delta$-ring homomorphism $$ \XY \to \XY\otimes_{\Tilde{R}} S,$$
which will still be denoted as $\psi$ by abuse of notation. Moreover, this further induces an $\psi: \XY/d^n \to \XY/d^n\otimes_{\Tilde{R}} S$ after modulo $d^n$ for any $n\in \mathbb{N}$.
\end{lemma}
\begin{proof}
    The proof is the same as that for \cref{lemt.lci case appl extend eta}.
\end{proof}
\begin{remark}\label{rem. ARL RELATIONS OF GAMMA}
    As the automorphisms $\gamma_i$ ($0\leq i\leq m$) on $\Tilde{R}$ is $\psi$-equivariant and $\gamma_i(x_j)$ belongs to the ideal generated by $d$ in $\XY$, running a similar argument as that for $\psi$, we see that all of these $\gamma_i$'s extend uniquely to automorphisms of $\XY$ by the universal property of the latter. Moreover, the relations given in Equation (\ref{equa. relation between j i}) and \cref{lem. AR relation between 0 and i} still holds on $\XY$. Then similarly as \cref{absolute def.skew polynomial}, we can define the \textit{Ore extension} $\XY[\partial, \underline{\nabla}; \gamma_{\XY}]$.
\end{remark}

Given \cref{lemt.ARL lci case appl extend eta}, we then restrict \cref{propt.key automorphism of functors AR} to $Y_{}^{\Prism}$ as follows.
\begin{proposition}\label{propt. ARL lci key automorphism of functors relative case}
    $\gamma_{b,c_{\psi}}$ constructed in \cref{propt.key automorphism of functors AR} restricts to an isomorphism $\gamma_{b, c_{\psi}}$ between functors $\rho: \Spf(\XY\otimes_{\Tilde{R}} S) \xrightarrow{\iota} \Spf(\XY) \xrightarrow{\rho_Y} Y_{}^{\Prism}$
    and $\rho\circ \psi: \Spf(\XY\otimes_{\Tilde{R}} S) \xrightarrow{\psi} \Spf(\XY)\xrightarrow{\rho_Y} Y_{}^{\Prism}$, i.e. we have the following commutative diagram:
\[\xymatrixcolsep{5pc}\xymatrix{\Spf(\XY\otimes_{\Tilde{R}} S)\ar[d]^{\iota}\ar[r]^{\psi}& \Spf(\XY)\ar@{=>}[dl]^{\gamma_{c_{\psi}}} \ar[d]_{}^{\rho_Y}
\\\Spf(\XY)  \ar^{\rho_Y}[r]&Y_{}^{\Prism}}\]
\end{proposition}
\begin{proof}
    By restricting $\gamma_b$ used in the proof of \cref{propt.key automorphism of functors AR} to $W(\XY\otimes_{\Tilde{R}} S)$, we get the desired isomorphism of Cartier-Witt divisors obtained via $\rho$ and $\rho\circ \psi$ separately. The rest follows from the proof of \cref{propt. lci key automorphism of functors relative case}.
\end{proof}

Similarly as the discussion in the previous sections, \cref{propt. ARL lci key automorphism of functors relative case} and \cref{rem. ARL RELATIONS OF GAMMA} imply that for $\mathcal{E}\in \mathcal{D}(Y^{\Prism})$, $\rho_Y^*\mathcal{E}$ is equipped with $\nabla_{\mathcal{E}}$ and $\partial_{\mathcal{E}}$ satisfying a sequence of relations. 
With these preliminaries in hand, we extend  \cref{thmt. AR main classification} to the locally complete intersection case.
\begin{theorem}\label{thmt. ARL main classification}
Assume that $p>2$ or $\alpha>0$. For $n\in \mathbb{N}\cup \{\infty\}$, the covering map $\rho_Y: \Spf(\Prism_X(Y)/d^n)\to Y_{n}^{\Prism}$ induces a functor
    \begin{align*}
        &\beta_n^{+}: \mathcal{D}(Y_{ n}^{\Prism}) \rightarrow \mathcal{D}(\XY/d^n[\partial, \underline{\nabla}; \gamma_{\Tilde{R}}]),
       \qquad \mathcal{E} \mapsto (\rho_X^{*}\mathcal{E}, \partial_{\mathcal{E}}, \nabla_{\mathcal{E}})
    \end{align*}
    which is fully faithful. 
     Moreover, its essential image consists of those objects $M\in \mathcal{D}(\XY/d^n[\partial, \underline{\nabla}; \gamma_{\Tilde{R}}])$ 
     satisfying the following pair of conditions:
    \begin{itemize}
        \item $M$ is $(p,d)$-adically complete.
        \item The action of $\partial$  and $\nabla_{i}$ on the cohomology $\mathrm{H}^*(M\otimes_{\XY/d^n}^{\mathbb{L}}\XY/(d,q-1))$ 
        is locally nilpotent for all $i$.
    \end{itemize}
\end{theorem}
\begin{proof}
Given \cref{propt. ARL lci key automorphism of functors relative case} and \cref{rem. ARL RELATIONS OF GAMMA}, an analogous result of \cref{prop. AR two derivation vanish} for $\mathcal{E}\in \mathcal{D}(Y_n^{\Prism})$  still holds, hence $\beta_n^+$ is well-produced. Moreover, the inductive process for proving \cref{thmt. AR main classification} still works once we show that the theorem holds for $n=1$, i.e. the Hodge-Tate case. 

For this purpose, first we notice that $\rho_X: X=\Spf(R) \to X_{}^{\HT}$ is a cover with automorphism group $G_X$ by \cref{lem. AR HT stack geometric structure} (here $G_X$ is defined over $X$). Then by restricting the pullback square above \cref{lemt.ARL lci case appl extend eta} to the Hodge-Tate locus, we obtain the following pullback diagram
\[\xymatrixcolsep{8pc} \xymatrix{ \Spf(\XY/d) \ar[r]^{\rho_Y^{\prime}} \ar^{\pi^{\prime}}[d] &  Y_{}^{\HT}=\Spf(\XY/d)/(G_X)_{Y} \ar^{\pi}[d] \\
X=\Spf(R) \ar[r]^{\rho_Y} & X_{}^{\HT}=X/G_X=\Spf(R)/G_X.}\]
Here $(G_X)_{Y}$ is the base change of $G_X$ along the closed embedding $Y\to X$.

Consequently, we have that 
\begin{equation}\label{equa. ARlci identify}
    \mathcal{D}(Y_{}^{\HT})=\Mod_{\pi_*\mathcal{O}}(\mathcal{D}(X/G_X)=\Mod_{\beta_1^+(\pi_*\mathcal{O})}(\mathcal{D}_{\nil}(\XY/d[\partial, \underline{\nabla}; \gamma_{\Tilde{R}}])).
\end{equation}
Here we use $\mathcal{D}_{\nil}(\XY/d[\partial, \underline{\nabla}; \gamma_{\Tilde{R}}])$ to denote the essential image of $\beta_1^+$ stated in \cref{thmt. AR main classification} for simplicity. 

But arguing as that in the proof of \cref{thmt. lci main classification}, we conclude that the right-hand side of \cref{equa. ARlci identify} is exactly the wanted category in \cref{thmt. ARL main classification}.
\end{proof}

\bibliographystyle{amsalpha}
\bibliography{main,preprints}

\end{document}